\documentclass[a4paper,10pt]{report}

\usepackage[utf8]{inputenc}
\usepackage[T1]{fontenc}

\usepackage[xindy]{imakeidx}

\PassOptionsToPackage{hyphens}{url}
\usepackage[
	colorlinks,
	bookmarksnumbered
]{hyperref}
\hypersetup{
	linkcolor=blue,
	citecolor=red,
	urlcolor=teal,
	pdftitle = {Zeta integrals, Schwartz spaces and local functional equations},
	pdfauthor = {Wen-Wei Li}
}

\usepackage{amsthm}
\usepackage{amssymb}
\usepackage{mathtools}
\usepackage{stmaryrd}
\usepackage{mathrsfs}
\usepackage{tikz-cd}
\usetikzlibrary{positioning, shapes.geometric, patterns, graphs}

\usepackage{paralist} 

\DeclareSymbolFont{bbold}{U}{bbold}{m}{n}
\DeclareSymbolFontAlphabet{\mathbbold}{bbold}

\newcommand{\Z}{\ensuremath{\mathbb{Z}}}
\newcommand{\Q}{\ensuremath{\mathbb{Q}}}
\newcommand{\R}{\ensuremath{\mathbb{R}}}
\newcommand{\CC}{\ensuremath{\mathbb{C}}}
\newcommand{\A}{\ensuremath{\mathbb{A}}}
\newcommand{\F}{\ensuremath{\mathbb{F}}}

\newcommand{\rank}{\operatorname{rk}}

\newcommand{\Aut}{\operatorname{Aut}}
\newcommand{\Tr}{\operatorname{tr}}
\newcommand{\topwedge}{\ensuremath{\bigwedge^{\mathrm{max}}}}
\newcommand{\Ind}{\operatorname{Ind}}

\newcommand{\Resotimes}{\ensuremath{{\bigotimes}'}}

\newcommand{\dd}{\mathop{}\!\mathrm{d}}
\newcommand{\Schw}{\ensuremath{\mathcal{S}}}	

\newcommand{\mes}{\operatorname{mes}}
\newcommand{\angles}[1]{\ensuremath{\langle #1 \rangle}}

\newcommand{\Stab}{\operatorname{Stab}}
\newcommand{\demi}{\ensuremath{\frac{1}{2}}}
\newcommand{\relgeq}[1]{\ensuremath{\underset{#1}{\geq}}}	
\newcommand{\relgg}[1]{\ensuremath{\underset{#1}{\gg}}}		

\newcommand{\identity}{\ensuremath{\mathrm{id}}}

\newcommand{\Hom}{\operatorname{Hom}}
\newcommand{\End}{\operatorname{End}}
\newcommand{\rightiso}{\ensuremath{\stackrel{\sim}{\rightarrow}}}
\newcommand{\leftiso}{\ensuremath{\stackrel{\sim}{\leftarrow}}}

\newcommand{\Ker}{\operatorname{ker}}

\newcommand{\Image}{\operatorname{im}}
\newcommand{\dotimes}[1]{\ensuremath{\underset{#1}{\otimes}}}

\newcommand{\Ad}{\operatorname{Ad}}
\newcommand{\Spec}{\operatorname{Spec}}
\newcommand{\Gm}{\ensuremath{\mathbb{G}_\mathrm{m}}}
\newcommand{\Ga}{\ensuremath{\mathbb{G}_\mathrm{a}}}
\newcommand{\Res}{\operatorname{Res}}
\newcommand{\utimes}[1]{\ensuremath{\overset{#1}{\times}}}

\newcommand{\Supp}{\operatorname{Supp}}
\newcommand{\divisor}{\operatorname{div}}
\newcommand{\GmC}{\ensuremath{\mathbb{G}_{\mathrm{m}, \mathbb{C}}}}
\newcommand{\pr}{\ensuremath{\mathbb{P}}}	

\newcommand{\GL}{\operatorname{GL}}
\newcommand{\SO}{\operatorname{SO}}

\newcommand{\SL}{\operatorname{SL}}
\newcommand{\Sp}{\operatorname{Sp}}

\theoremstyle{plain}
\newtheorem{proposition}{Proposition}
\newtheorem{lemma}[proposition]{Lemma}
\newtheorem{theorem}[proposition]{Theorem}
\newtheorem{corollary}[proposition]{Corollary}

\theoremstyle{definition}
\newtheorem{definition}[proposition]{Definition}
\newtheorem{hypothesis}[proposition]{Hypothesis}
\newtheorem{conjecture}[proposition]{Conjecture}
\newtheorem{axiom}[proposition]{Axiom}
\newtheorem{example}[proposition]{Example}

\theoremstyle{remark}
\newtheorem{remark}[proposition]{Remark}

\newtheorem{notation}[proposition]{Notation}

\numberwithin{equation}{section}

\newcommand{\HC}{\ensuremath{\mathbf{H}}}	
\newcommand{\relint}{\operatorname{rel.int}}	
\newcommand{\rec}{\operatorname{rec}}	
\newcommand{\SVexp}{\ensuremath{\mathfrak{exp}}}	
\newcommand{\amspace}{\ensuremath{\mathfrak{X}}}	

\renewcommand{\Re}{\operatorname{Re}}
\renewcommand{\Im}{\operatorname{Im}}
\renewcommand{\emptyset}{\ensuremath{\varnothing}}

\newcommand{\MyQED}{}	

\usepackage{amsmath}
\usepackage[nottoc]{tocbibind}

\usepackage{geometry}
\usepackage{lmodern}

\numberwithin{proposition}{section}

\title{Zeta integrals, Schwartz spaces and local functional equations}
\author{Wen-Wei Li}
\date{}

\makeindex

\begin{document}

\maketitle

\begin{abstract}
	According to Sakellaridis, many zeta integrals in the theory of automorphic forms can be produced or explained by appropriate choices of a Schwartz space of test functions on a spherical homogeneous space, which are in turn dictated by the geometry of affine spherical embeddings. We pursue this perspective by developing a local counterpart and try to explicate the functional equations. These constructions are also related to the $L^2$-spectral decomposition of spherical homogeneous spaces in view of the Gelfand--Kostyuchenko method. To justify this viewpoint, we prove the convergence of $p$-adic local zeta integrals under certain premises, work out the case of prehomogeneous vector spaces and re-derive a large portion of Godement--Jacquet theory. Furthermore, we explain the doubling method and show that it fits into the paradigm of $L$-monoids developed by L.\ Lafforgue, B.\ C.\ Ngô \textit{et al.}, by reviewing the constructions of Braverman and Kazhdan (2002). In the global setting, we give certain speculations about global zeta integrals, Poisson formulas and their relation to period integrals. 
\end{abstract}

\tableofcontents

\chapter{Introduction}
Zeta integrals are indispensable tools for studying automorphic representations and their $L$-functions. In broad terms, it involves integrating automorphic forms (global case) or ``coefficients'' of representations (local case) against suitable test functions such as Eisenstein series or Schwartz--Bruhat functions; furthermore, these integrals are often required to have meromorphic continuation in some complex parameter $\lambda$ and satisfy certain symmetries, i.e.\ functional equations. These ingenious constructions are usually crafted on a case-by-case basis; systematic theories are rarely pursued with a notable exception \cite{Sak12}, which treats only the global zeta integrals. The goal of this work is to explore some aspects of this general approach, with an emphasis on the local formalism.

\section{Review of prior works}
To motivate the present work, let us begin with an impressionist overview of several well-known zeta integrals.

\subsection*{From Tate's thesis to Godement--Jacquet theory}
The most well-understood example is Tate's thesis (1950). Consider a number field $F$ and a Hecke character $\pi: F^\times \backslash \A_F^\times \to \CC^\times$. Fix a Haar measure $\dd^\times x$ on $\A_F^\times$ and define
\begin{align*}
	Z^\text{Tate}(\lambda, \pi, \xi) & := \int_{\A_F^\times} \pi(x)|x|^\lambda \xi(x) \dd^\times x = \prod_v  Z^\text{Tate}_v(\lambda, \pi_v, \xi_v), \\
	Z^\text{Tate}_v(\lambda, \pi_v, \xi_v) & := \int_{F_v^\times} \pi_v(x_v) |x_v|_v^\lambda \xi_v(x_v) \dd^\times x_v
\end{align*}
where $\xi = \prod_v \xi_v$ belongs to the Schwartz--Bruhat space $\Schw = \Resotimes_v \Schw_v$ of $\A_F$, and the local and global measures are chosen in a compatible manner. The adélic integral and its local counterparts converge for $\Re(\lambda) \gg 0$, and they admit meromorphic continuation to all $\lambda \in \CC$ (rational in $q_v^\lambda$ for $v \nmid \infty$). When the test functions $\xi$ vary, the greatest common divisor of $Z^\text{Tate}(\lambda, \pi, \cdot)$ (resp.\  $Z^\text{Tate}_v(\lambda, \pi_v, \cdot)$), which exists in some sense, gives rise to the $L$-factor $L(\lambda, \pi)$ (resp.\  $L(\lambda, \pi_v)$). Computations at unramified places justify that this is indeed the complete Hecke $L$-function for $\pi$. Moreover, the functional equation for $L(\lambda, \pi)$ stems from those for $Z^\text{Tate}$ and $Z^\text{Tate}_v$: let $\mathcal{F} = \mathcal{F}_\psi: \Schw \to \Schw$ be the Fourier transform with respect to an additive character $\psi = \prod_v \psi_v$ of $F \backslash \A_F$, the local functional equation reads
\[ Z^\text{Tate}_v(1-\lambda, \check{\pi}_v, \mathcal{F}_v(\xi_v)) = \gamma^\text{Tate}_v(\lambda, \pi_v) Z^\text{Tate}_v(\lambda, \pi_v, \xi_v) \]
where $\check{\pi}_v := \pi_v^{-1}$, i.e.\ the contragredient.

A.\ Weil \cite{Weil66} reinterpreted $Z^\text{Tate}_v(\lambda, \pi, \cdot)$ as a family of $F_v^\times$-equivariant tempered distributions on the affine line $F_v$, baptized \emph{zeta distributions}, and similarly for the adélic case. Specifically, $Z^\text{Tate}_v(\lambda, \pi, \cdot)$ may be viewed as elements in $\Hom_{F_v^\times}(\pi_v \otimes |\cdot|^\lambda, \Schw_v^\vee)$, varying in a meromorphic family in $\lambda$. The local functional equation can then be deduced from the uniqueness of such families, using the Fourier transform. In a broad sense, the $\Gm$-equivariant affine embedding $\Gm \hookrightarrow \Ga$ is the geometric backdrop of Tate's thesis. Convention: the groups always act on the right of the varieties.

Tate's approach is successfully generalized to the standard $L$-functions of $\GL(n)$ by Godement and Jacquet \cite{GJ72} for any $n \geq 1$; they also considered the inner forms of $\GL(n)$. The affine embedding $\Gm \hookrightarrow \Ga$ is now replaced by $\GL(n) \hookrightarrow \text{Mat}_{n \times n}$, which is $\GL(n) \times \GL(n)$-equivariant under the right action $x(g_1, g_2) = g_2^{-1} x g_1$. To simplify matters, let us consider only the local setting and drop the index $v$. Denote now by $\Schw$ the Schwartz--Bruhat space of the affine $n^2$-space $\text{Mat}_{n \times n}(F)$, where $F$ stands for a local field. Given an irreducible smooth representation $\pi$ of $\GL(n, F)$, the Godement--Jacquet zeta integrals are
\begin{equation}\label{eqn:GJ-first}\begin{gathered}
	Z^\text{GJ}(\lambda, v \otimes \check{v}, \xi) := \int_{\GL(n,F)} \angles{\check{v}, \pi(x) v} |\det x|^{\lambda - \demi + \frac{n}{2}} \xi(x) \dd^\times x, \\
	v \otimes \check{v} \in V_\pi \otimes V_{\check{\pi}}, \; \xi \in \Schw
\end{gathered}\end{equation}
where $\check{\pi}$ still denotes the contragredient representation of $\pi$, the underlying vector spaces being denoted by $V_\pi$, $V_{\check{\pi}}$, etc. Note that we adopt the Casselman--Wallach representations, or SAF representations (abbreviation for smooth admissible Fréchet, see \cite{BK14}) in the Archimedean case. The integral is convergent for $\Re(\lambda) \gg 0$, with meromorphic/rational continuation to all $\lambda \in \CC$. The shift in the exponent of $|\det x|$ ensures that taking greatest common divisor yields the standard $L$-factor $L(\lambda, \pi)$ in the unramified case. Furthermore, we still have the local functional equation
\[ Z^\text{GJ}_v(1-\lambda, \check{v} \otimes v, \mathcal{F}(\xi)) = \gamma^\text{GJ}(\lambda, \pi) Z^\text{GJ}(\lambda, v \otimes \check{v}, \xi), \]
where $\mathcal{F} = \mathcal{F}_\psi: \Schw \rightiso \Schw$ is the Fourier transform. Note that Godement and Jacquet did not explicitly use Weil's idea of characterizing equivariant distributions in families.

Let us take a closer look at the local integral \eqref{eqn:GJ-first} and its functional equation. The shift by $-\demi$ in the exponent of $|\det x|$ is harmless: if we shift the zeta integrals to $\lambda \mapsto Z^\text{GJ}(\lambda + \demi, \cdots)$, the local functional equation will interchange $\lambda$ with $-\lambda$ and the greatest common divisor will point to $L(\demi+\lambda, \pi)$; the evaluation at $\lambda=0$ then yields the central $L$-value, which is more natural. On the other hand, the shift by $\frac{n}{2}$ remains a mystery: the cancellation $-\demi + \demi = 0$ in the case $n=1$ turns out to be just a happy coincidence.

What is the \textit{raison d'être} of $n/2$ here? Later we shall see in the discussion on $L^2$-theory that this amateurish question does admit a sensible answer.

\subsection*{The doubling method}
In \cite{GPSR87,PSR86}, Piatetski-Shapiro and Rallis proposed a generalization of Godement--Jacquet integrals for classical groups, called the \emph{doubling method}. For ease of notations, here we consider only the case $G = \Sp(V)$ over a local field $F$, where $V$ is a symplectic $F$-vector space. Doubling means that we take $V^\Box := V \oplus (-V)$ where $-V$ stands for the $F$-vector space $V$ carrying the symplectic form $-\angles{\cdot|\cdot}_V$, so that $G \times G$ embeds naturally into $G^\Box := \Sp(V^\Box)$. The diagonal image of $V$ is a Lagrangian subspace of $V^\Box$, giving rise to a canonical Siegel parabolic $P \subset G^\Box$. It is shown in \cite{GPSR87} that $P \cap (G \times G) = \text{diag}(G)$ and $G$ embeds equivariantly into $P \backslash G^\Box$ as the unique open $G \times G$-orbit.

There is a natural isomorphism $\sigma: P_\text{ab} \rightiso \Gm$. Given a continuous character $\chi: P_\text{ab}(F) \to \CC^\times$, one can form the normalized parabolic induction $I^{G^\Box}_P(\chi \otimes |\sigma|^\lambda)$ for $\lambda \in \CC$. The doubling zeta integral of Piatetski-Shapiro and Rallis takes the form
\[ Z_\lambda(\lambda, s, v \otimes \check{v}) = \int_{G(F)} s_\lambda(x) \angles{\check{v}, \pi(x)v} \dd x, \quad \Re(\lambda) \gg 0 \]
where $\pi$ is an irreducible smooth representation of $G(F)$, $v \otimes \check{v} \in \pi \otimes \check{\pi}$ and $s_\lambda$ is a \emph{good section} in $I^{G^\Box}_P(\chi \otimes |\sigma|^\lambda)$ varying meromorphically in $\lambda$. In \textit{loc.\ cit.}, unramified computations show that these zeta integrals point to $L(\demi + \lambda, \chi \times \pi)$. Furthermore, one still obtains a parametrized family of integrals that admits meromorphic/rational continuation and functional equations. These properties are deduced from the corresponding properties of suitably normalized intertwining operators for $I^{G^\Box}_P(\chi \otimes |\sigma|^\lambda)$.

It is explained in \cite{GPSR87} that the doubling method actually subsumes the Godement--Jacquet integral as a special case. In comparison with the prior constructions, here the compactification $G \hookrightarrow P \backslash G^\Box$ replaces the affine embedding $\GL(n) \hookrightarrow \text{Mat}_{n \times n}$, and the good sections $s_\lambda \in I^{G^\Box}_P(\chi \otimes |\sigma|^\lambda)$ replace Schwartz--Bruhat functions. The occurrence of matrix coefficients in both cases can be explained by the presence of the homogeneous $G \times G$-space $G$; this will be made clear later on.

On the contrary, Braverman and Kazhdan \cite{BK02} reinterpreted the doubling method à la Godement--Jacquet, namely by allowing $\chi$ to vary and by replacing the good sections $s_\lambda$ by functions from some Schwartz space $\Schw(X_P)$, where $X_P := P_\text{der} \backslash G^\Box$ is a quasi-affine homogeneous $G^\Box$-space; elements from $\Schw(X_P)$ may be viewed as universal good sections. The embedding in question becomes $P_\text{ab} \times G \hookrightarrow X_P$ which is $P_\text{ab} \times G \times G$-equivariant. This turns out to be closely related to the theory of $L$-monoids to be discussed below. See also \cite{GL17}.

\subsection*{Braverman--Kazhdan: $L$-monoids}
In \cite{BK00}, Braverman and Kazhdan proposed a bold generalization of Godement--Jacquet theory that should yield more general $L$-functions. To simplify matters, we work with a local field $F$ of characteristic zero and a split connected reductive $F$-group $G$; the $L$-factor in question is $L(\lambda, \pi, \rho)$, where $\pi$ is an irreducible smooth representation of $G(F)$ and $\rho: \hat{G} \to \GL(N,\CC)$ is a homomorphism of $\CC$-groups. The basic idea is to generalize $(\GL(n) \hookrightarrow \text{Mat}_{n \times n}, \det)$ into $(G \hookrightarrow X_\rho, \det_\rho)$, where
\begin{compactitem}
	\item $\det_\rho: G \to \Gm$ is a surjective homomorphism, dual to an inclusion $\CC^\times \hookrightarrow \hat{G}$, and we assume that $\rho(z) = z\cdot\identity$ for each $z \in \CC^\times$;
	\item $X_\rho$ is a \emph{normal reductive algebraic monoid} with unit group $G$, constructed in a canonical fashion from $\rho$, so that $G \hookrightarrow X_\rho$ is a $G \times G$-equivariant open immersion;
	\item furthermore, we assume that $\det_\rho$ extends to $X_\rho \to \Ga$ such that $\det_\rho^{-1}(0)$ equals $\partial X_\rho := X_\rho \smallsetminus G$ (set-theoretically at least).
\end{compactitem}
Given these data, one forms the zeta integral $\int_{G(F)} \xi(x) \angles{\check{v}, \pi(x)v} |\det_\rho(x)|^\lambda \dd x$ as before, except that $\xi$ is now taken from a conjectural Schwartz space $\Schw$ attached to $(G, \rho)$. Expectations include
\begin{compactenum}[(i)]
	\item $C^\infty_c(G(F)) \subset \Schw \subset C^\infty(G(F))$ is $G \times G(F)$-stable;
	\item these zeta integrals should converge for $\Re(\lambda) \gg 0$ and admit meromorphic/rational continuations;
	\item in the unramified situation, there should be a distinguished element $\xi^\circ \in \Schw$, called the basic function for $X_\rho$, whose zeta integral yields $L(\lambda, \pi, \rho)$ (or with some shift in $\lambda$, cf.\ the previous discussions);
	\item upon fixing an additive character $\psi$, there should exist a Fourier transform $\mathcal{F}_\rho: \Schw \to \Schw$, which fits into the local functional equation as in the Godement--Jacquet case, and preserves the basic function.
\end{compactenum}
By \cite{BK00}, prescribing $\mathcal{F}_\rho$ is equivalent to prescribing the $\gamma$-factor $\gamma(\lambda, \pi, \rho)$ that figures in the conjectural local functional equation. Furthermore, in the global setting one may form the adélic Schwartz space as a restricted $\otimes$-product $\Schw = \Resotimes_v \Schw_v$ using the basic functions, and one expects a ``Poisson formula'' for $\mathcal{F}_\rho$, at least when $\xi_v \in C^\infty_c(G(F_v))$ for some $v$.

The idea of embedding $G$ into a reductive monoid to produce $L$-factors is then refined by L.\ Lafforgue \cite{Laf14} and B.\ C.\ Ngô \cite{Ng14}, leading up to the concept of $L$-monoids. The precise definition of $\Schw$ is expected to come from the geometry of $X_\rho$. This is illustrated by the following idea of Ngô: in the unramified setting, the values of the basic function $\xi_v^\circ$ is expected to come from the traces of Frobenius of some version of IC-complex on the formal arc space of $X_\rho$, upon passing to the equal-characteristic setting. The main difficulty turns out to be the singularity of $X_\rho$. We refer to \cite{Ng14,BNS16} for further discussions and developments in this direction.

In \cite{Laf14}, L.\ Lafforgue is able to define $\Schw$ using the Plancherel formula for $G$, whose elements are called functions of $L$-type in \cite[Définition II.15]{Laf14}, and the Fourier transform $\mathcal{F}_\rho$ also admits a natural definition in spectral terms. The crucial premise here is to have the correct local factors $L(\lambda, \pi, \rho)$, $\varepsilon(\lambda, \pi, \rho)$ for $\pi$ in the tempered dual of $G(F)$.

\subsection*{Integral representations and spherical varieties}
The group $G$ endowed with $G \times G$-action is just an instance of homogeneous spaces, called the ``group case''. Many basic problems in harmonic analysis, such as the $L^2$-spectral decomposition, can be formulated on homogeneous spaces, and a generally accepted class of spaces is the \emph{spherical homogeneous spaces}. Indeed, this can be justified from considerations of finiteness of multiplicities \cite{KO13} and the work of Sakellaridis--Venkatesh \cite{SV17}. Assume $G$ split. A $G$-variety $X$ is called spherical if it is normal with an open dense Borel orbit.

In \cite{Sak12}, Sakellaridis proposed to understand many global zeta integrals in terms spherical homogeneous $G$-spaces $X^+$ and their affine equivariant embeddings $X^+ \hookrightarrow X$. As in \cite{BK00,BK02}, his construction relies on a conjectural Schwartz space $\Schw = \Resotimes_v \Schw_v$, whose structure is supposed to be dictated by the geometry of $X$. Specifically, the restricted $\otimes$-product is taken with respect to \emph{basic functions} $\xi_v^\circ \in \Schw_v$ at each unramified place $v$. Notice that these functions are rarely compactly supported and they should be related to $L$-factors (cf.\ \cite{Sak18}). His construction subsumes many known global zeta integrals by using the setting of ``preflag bundles''.

As in the predecessors \cite{BK99,BK02}, the Schwartz spaces and basic functions here are partly motivated by the geometric Langlands program. For example, the basic functions in \textit{loc.\ cit.} are closely related to the IC-sheaves of Drinfeld's compactification $\overline{\text{Bun}_P}$ of the moduli stack $\text{Bun}_P$ over a smooth projective curve over some $\F_q$, where $P \subset G$ is a parabolic; it arises in the study of geometric Eisenstein series and provides a global model of the singularities of local spaces in question. 

\subsection*{Other integrals}
There are other sorts of zeta integrals that emerged in the pre-Langlands era, many of them have direct arithmetic significance. We take our lead from the local Igusa zeta integrals: let $F$ be a local field of characteristic zero, $X$ be an affine space over $F$ with a prescribed additive Haar measure on $X(F)$, and $f \in F[X]$. Pursuing the ideas of Gelfand \textit{et al.}, Igusa considered
\[ \int_{X(F)} \xi(x)|f(x)|^\lambda \dd x, \quad \xi \in \Schw(X(F)) \]
where $\Schw(X(F))$ stands for the Schwartz--Bruhat space of $X(F)$, and the integral converges whenever $\Re(\lambda) > 0$. It turns out that this extends to a meromorphic/rational family in $\lambda$ of tempered distributions, also known as \emph{complex powers} of $f$ as it interprets $|f|^\lambda$ as a tempered distribution for all $\lambda$ away from the poles. An extensive survey can be found in \cite{Ig00}.

To bring harmonic analysis into the picture, we switch to the case of a reductive \emph{prehomogeneous vector space} $(G, \rho, X)$, where $G$ is a split connected reductive group, $\rho: G \to \GL(X)$ is an algebraic representation --- thus $G$ acts on the right of $X$ --- and we assume there exists an open $G$-orbit $X^+$ in $X$. To simplify matters, let us suppose temporarily that $\partial X := X \smallsetminus X^+$ is the zero locus of some irreducible $f \in F[X]$, then such $f$ is unique up to $F^\times$ and defines a relative invariant, i.e.\ $f(xg)=\omega(g)f(x)$ for an eigencharacter $\omega \in X^*(G) := \Hom(G, \Gm)$. These prehomogeneous zeta integrals along with their global avatars have been studied intensively by M.\ Sato, T.\ Shintani and their school; we refer to \cite{Sa89,Ki03} for detailed surveys.

In comparison with the local Godement--Jacquet integrals, corresponding to the case $G := \GL(n) \times \GL(n)$, $X := \text{Mat}_{n \times n}$ and $f := \det$, a natural generalization is to replace Igusa's integrand $|f|^\lambda \xi$ by $\varphi(v)|f|^\lambda \xi$, where $\varphi \in \mathcal{N}_\pi := \Hom_{G(F)}(\pi, C^\infty(X^+(F)))$ and $v \in \pi$, where $\pi$ is an irreducible smooth representation of $G(F)$; the function $\varphi(v)$ is called a (generalized) coefficient of $\pi$ in $X^+$. In the Godement--Jacquet case we have $\dim \mathcal{N}_\pi \leq 1$, with equality if and only if $\pi \simeq \sigma \boxtimes \check{\sigma}$ for some $\sigma$, in which case $\mathcal{N}_\pi$ is generated by the matrix coefficient $v \otimes \check{v} \mapsto \angles{\check{v}, \sigma(\cdot)v}$. This approach is pioneered by Bopp and Rubenthaler \cite{BR05} in a classical setting. The upshot here is to have a good understanding of the harmonic analysis on $X^+(F)$, including the asymptotics for the coefficients $\varphi(v)$.

\section{Our formalism}
\subsection*{More general zeta integrals}
Over a local field $F$ of characteristic zero, we synthesize the previous constructions by considering
\begin{compactitem}
	\item $G$: a split connected reductive $F$-group,
	\item $X^+$: an affine spherical homogeneous $G$-space, and
	\item $X^+ \hookrightarrow X$: an equivariant open immersion into an affine spherical $G$-variety $X$.
\end{compactitem}

As in the prehomogeneous zeta integrals, we wish to be able to ``twist'' the functions on $X^+(F)$ by complex powers of some relative invariants. To this end we introduce Axiom \ref{axiom:geometric}. Roughly speaking, its content is:
\begin{enumerate}[(A)]
	\item the $G$-eigenfunctions in $F(X) = F(X^+)$ span a lattice $\Lambda$, and the eigenfunctions in $F[X]$ generate a monoid $\Lambda_X$ in $\Lambda$, we require that $\Lambda_{X,\Q} := \Q_{\geq 0}\Lambda_X$ is a simplicial cone in $\Lambda_{\Q} := \Q \Lambda$;
	\item take minimal integral generators $\omega_1, \ldots, \omega_r$ of $\Lambda_{X,\Q}$ with eigenfunctions (i.e.\ relative invariants) $f_1, \ldots, f_r$ (unique up to $F^\times$), we require that their zero loci cover $\partial X$.
\end{enumerate}
We may define
\begin{compactitem}
	\item the continuous character $|\omega|^\lambda = \prod_{i=1}^r |\omega_i|^{\lambda_i}$ of $G(F)$, and
	\item the $|\omega|^\lambda$-eigenfunction $|f|^\lambda = \prod_{i=1}^r |f_i|^{\lambda_i}$ on $X^+(F)$
\end{compactitem}
for $\lambda = \sum_{i=1}^r \lambda_i \omega_i \in \Lambda_{\CC} := \Lambda_{\Q} \otimes \CC$. We have $r=1$ in most classical cases.

Let us turn to harmonic analysis. The guiding principle here is to set the $L^2$-theory as our background. Specifically, we consider functions on $X^+(F)$ with values in a $G(F)$-equivariant vector bundle $\mathscr{E}$ endowed with a $G(F)$-equivariant, positive definite hermitian pairing
\[ \mathscr{E} \otimes \overline{\mathscr{E}} \to \mathscr{L} \]
where $\mathscr{L}$ stands for the density bundle on $X^+(F)$. Sections of $\mathscr{L}$, called the ``densities'' on $X^+(F)$, are functions that carry measures: their integrals are canonically defined provided that a Haar measure on $F$ is chosen, say by fixing a non-trivial unitary character $\psi$ of $F$. Densities can be expressed locally as $\phi |\omega|$ where $\phi$ is a $C^\infty$-function and $\omega$ is an algebraic volume form. The typical choice of $\mathscr{E}$ is the bundle of \emph{half-densities} $\mathscr{L}^\demi$, whose local sections take the form $\phi |\omega|^\demi$. Indeed, we have $\mathscr{L}^\demi = \overline{\mathscr{L}^\demi}$ together with an isomorphism $\mathscr{L}^\demi \otimes \mathscr{L}^\demi \rightiso \mathscr{L}$, which we shall denote simply by multiplication. Note that $\mathscr{L}$ and $\mathscr{L}^\demi$ are often equivariantly trivializable, but we shall refrain from doing so; the benefit will be seen shortly.

Define
\begin{align*}
	C^\infty(X^+) & := C^\infty(X^+(F), \overline{\mathscr{E}}) \simeq C^\infty(X^+(F), \mathscr{E}^\vee \otimes \mathscr{L}), \\
	C^\infty_c(X^+) & := C^\infty_c(X^+(F), \mathscr{E}), \\
	L^2(X^+) & := L^2(X^+(F), \mathscr{E}).
\end{align*}
The formalism of densities implies that $G(F)$ acts canonically on all these spaces, and $L^2(X^+)$ becomes a unitary representations of $G(F)$. Equipping $C^\infty(X^+)$ and $C^\infty_c(X^+)$ with natural topologies, they become smooth continuous $G(F)$-representations. All topological vector spaces in this work are locally convex by assumption.

Given an irreducible smooth representation $\pi$ of $G(F)$, taken to be an SAF representation for Archimedean $F$, we define $\mathcal{N}_\pi := \Hom_{G(F)}(\pi, C^\infty(X^+))$ (the continuous Hom). To achieve this, $\pi$ has to be viewed as a continuous representation for all $F$. This is unavoidable in the Archimedean case, whereas in the non-Archimedean case it can be naturally done by introducing the notion of algebraic topological vector spaces (Definition \ref{def:algebraic-tvs}).

At this stage, we incorporate the basic Axiom \ref{axiom:finiteness} that asserts
\[ \dim \mathcal{N}_\pi < +\infty \quad \text{for all}\; \pi, \]
which has been established in many cases. Note that the representations with nonzero $\mathcal{N}_\pi$ are the basic objects in the harmonic analysis over $X^+(F)$. Write $\pi_\lambda := \pi \otimes |\omega|^\lambda$ for $\lambda \in \Lambda_{\CC}$. Upon fixing a choice of $f_1, \ldots, f_r$, elements of $\mathcal{N}_\pi$ may be twisted by $\Lambda_{\CC}$ by
\begin{align*}
	\mathcal{N}_\pi & \stackrel{\sim}{\longrightarrow} \mathcal{N}_{\pi_\lambda} \\
	\varphi & \longmapsto \varphi_\lambda := |f|^\lambda \varphi.
\end{align*}

The Schwartz space $\Schw$ is taken to be a subspace of $C^\infty(X^+(F), \mathscr{E}) \cap L^2(X^+)$ in which $C^\infty_c(X^+)$ embeds continuously; moreover $\Schw$ should be a smooth continuous $G(F)$-representation. The desired zeta integrals take the form
\[ Z_{\lambda,\varphi}(v \otimes \xi) = \int_{X^+(F)} \underbracket{\varphi_\lambda(v) \xi}_{\in \mathscr{L}}, \quad v \in \pi, \; \xi \in \Schw, \; \varphi \in \mathcal{N}_\pi, \]
which is required to converge for $\lambda \in \Lambda_{\CC}$, $\Re(\lambda) \relgg{X} 0$, the latter notation meaning that $\lambda = \sum_{i=1}^r \lambda_i \omega_i$ with $\Re(\lambda_i) \gg 0$ for all $i$. The analogy with the local zeta integrals reviewed before should be evident.

The precise requirements on $\Schw$ and $Z_\lambda$ are collected in Axiom \ref{axiom:zeta}. Below is a digest.
\begin{itemize}
	\item The topological vector space $\Schw$ is required to be \emph{separable, nuclear and barreled}. We will see the use of nuclearity when discussing the $L^2$-theory; being barreled implies that $\mathcal{S}^\vee$ is quasi-complete with respect to the topology of pointwise convergence, which is used to address continuity after meromorphic continuation.
	\item Set $\Schw_\lambda := |f|^\lambda \Schw$, endowed with the topology from $\Schw$. We assume that when $\Re(\lambda) \relgeq{X} 0$, we still have the inclusion $\Schw_\lambda \hookrightarrow L^2(X^+)$ that is continuous with dense image. Furthermore, the family of embeddings $\alpha_\lambda: \Schw \hookrightarrow L^2(X^+)$, given by $\alpha_\lambda(\xi) = |f|^\lambda \xi$ for $\Re(\lambda) \relgeq{X} 0$, is required to be holomorphic in a weak sense. Again, these properties are motivated by $L^2$ considerations.
	\item In the Archimedean case, functions from $\Schw$ should have rapid decay on $X(F)$ upon twisting by some $|f|^{\lambda_0}$; here the rapid decay is defined in $L^2$ terms, and we defer the precise yet tentative definition to \S\ref{sec:auxiliary}. In the non-Archimedean case, we assume that the closure of $\Supp(\xi)$ in $X(F)$ is compact for each $\xi \in \Schw$.
	\item The zeta integrals $Z_{\lambda, \varphi}$ above are assumed to be convergent and separately continuous for $\Re(\lambda) \relgg{X} 0$, and admit meromorphic (rational for non-Archimedean $F$) continuation to all $\lambda \in \Lambda_{\CC}$.
\end{itemize}
An often overlooked point in the literature is that the meromorphically continued zeta integrals are still separately continuous off the poles, by a principle from Gelfand and Shilov \cite[Chapter I, A.2.3]{GS1} relying on the quasi-completeness of $\Schw^\vee$. Available tools from functional analysis \cite[III]{BoEVT} imply that the $G(F)$-invariant bilinear form $Z_{\varphi,\lambda}: \pi_\lambda \otimes \Schw \to \CC$ is hypocontinuous off the poles, which essentially means that it induces a continuous intertwining operator $\pi_\lambda \to \Schw^\vee$, where $\Schw^\vee$ carries the strong topology. Weil's vision of zeta distributions \cite{Weil66} is therefore revived.

\subsection*{Local functional equation}
Let $\mathcal{T}$ denote the space of characters $|\omega|^\lambda$, $\lambda \in \Lambda_{\CC}$. Define $\mathcal{O}$ to be the space of holomorphic (regular algebraic when $F$ is non-Archimedean) functions on $\mathcal{T}$ and put $\mathcal{K} := \text{Frac}(\mathcal{O})$. Let $\pi$ be an irreducible smooth representation of $G(F)$. Granting Axiom \ref{axiom:zeta}, the zeta integrals induce a $\mathcal{K}$-linear injection
\[ T_\pi: \mathcal{N}_\pi \otimes \mathcal{K} \hookrightarrow \mathcal{L}_\pi \]
where $\mathcal{L}_\pi$ stands for the $\mathcal{K}$-vector space of ``reasonable'' meromorphic families of invariant pairings $\pi_\lambda \otimes \Schw \to \CC$ parametrized by $\lambda \in \Lambda_{\CC}$. Now consider two affine spherical embeddings $X^+_i \hookrightarrow X_i$ under $G$, with lattices of eigencharacters $\Lambda_i$ ($i=1,2$), each satisfying the aforementioned conditions with given Schwartz spaces $\Schw_i$. The motto here is:
\begin{center}
	Intertwining operator $\mathcal{F}: \Schw_2 \to \Schw_1$ \quad $\dashrightarrow$ \quad local functional equation.
\end{center}
In this work we call such an $\mathcal{F}$ a \emph{model transition}, a term borrowed from Lapid and Mao \cite{LM15}. For simplicity, let us assume $\Lambda_1 = \Lambda_2$ so that $X_1^+$, $X_2^+$ share the same objects $\mathcal{O}$ and $\mathcal{K}$. Pulling back by $\mathcal{F}$ gives rise to a natural isomorphism $\mathcal{F}^\vee: \mathcal{L}^{(1)}_\pi \rightiso \mathcal{L}^{(2)}_\pi$ for every $\pi$. Definition \ref{def:lfe} says that the local functional equation attached to $\pi$ and $\mathcal{F}$, if it exists, is the unique $\mathcal{K}$-linear map $\gamma(\pi)$ that renders the following diagram commutative.
\[\begin{tikzcd}
	\mathcal{L}^{(1)}_\pi \arrow{r}{\mathcal{F}^\vee} & \mathcal{L}^{(2)}_\pi \\
	\mathcal{N}^{(1)}_\pi \otimes \mathcal{K} \arrow[hookrightarrow]{u} \arrow{r}[below]{\gamma(\pi)} & \mathcal{N}^{(2)}_\pi \otimes \mathcal{K} \arrow[hookrightarrow]{u}
\end{tikzcd}\]

This formalism is compatible with the local functional equations in the cases of Godement--Jacquet, Braverman--Kazhdan, etc., where the spaces $\mathcal{N}_\pi$ are at most one-dimensional and we retrieve the usual $\gamma$-factors. For prehomogeneous vector spaces, we recover the $\gamma$-matrices when $\pi$ is the trivial representation; see for example \cite{Sa89}.

The existence of $\gamma(\pi)$ is by no means automatic. However, it can be partly motivated by the $L^2$-theory if we assume that $\mathcal{F}$ extends to an isometry $L^2(X_2^+) \rightiso L^2(X_1^+)$; see below.

\subsection*{\texorpdfstring{$L^2$}{L2}-theory}
The decomposition of the unitary $G(F)$-representation $L^2(X^+)$ is a classical concern of harmonic analysis. The abstract Plancherel decomposition gives an isometry
\begin{gather}\label{eqn:Plancherel-first}
	L^2(X^+) \simeq \int^\oplus_{\Pi_\text{unit}(G(F))} \mathcal{H}_\tau \dd\mu(\tau)
\end{gather}
where $\Pi_\text{unit}(G(F))$ stands for the unitary dual, $\mathcal{H}_\tau$ is a $\tau$-isotypic unitary representation on a separable Hilbert space, and $\mu$ is the corresponding Plancherel measure on $\Pi_\text{unit}(G(F))$. The abstract theory says little beyond the existence and essential uniqueness of such a decomposition. The refinements hinge on using finer spaces of test functions on $X^+(F)$.

According to an idea of Gelfand and Kostyuchenko \cite[Chapter I]{GV4} (see also \cite{Ma68}), later rescued from oblivion by Bernstein \cite{Be88}, the scenario here is to seek a separable topological vector space $\Schw$ with a continuous embedding $\alpha: \Schw \hookrightarrow L^2(X^+)$ of dense image, such that there exists a family of intertwining operators $\alpha_\tau: \Schw \to \mathcal{H}_\tau$ such that $\alpha = \int^\oplus \alpha_\tau \dd\mu(\tau)$; in this case we say $\alpha$ is pointwise defined. As a motivating example, recall the classical study of $L^2(\R)$ under $\R$-translations. For the unitary character $\tau(x) = e^{2\pi i ax}$ of $\R$, the required map $\alpha_\tau$ amounts to the usual Fourier transform evaluated at $a$, well-defined only for functions of rapid decay on $\R$. This is the starting point of classical Fourier analysis.

Coming back to the homogeneous $G$-space $X^+$ satisfying the previous axioms, the minimalist choice of $\Schw$ is $C^\infty_c(X^+)$ or the Schwartz space à la Harish-Chandra $\mathscr{C}(X^+)$ \cite[3.5 Definition]{Be88}; the fact that $\alpha$ is pointwise defined is established in \textit{loc.\ cit.} Loosely speaking, given \eqref{eqn:Plancherel-first}, specifying $\Schw \hookrightarrow L^2(X^+)$ together with $\left( \alpha_\tau: \Schw \to \mathcal{H}_\tau \right)_{\tau \in \Pi_\text{unit}(G(F))}$ amounts to a \emph{refined Plancherel decomposition}. Harish-Chandra's theory is a highly successful instance for the group case, so are Sakellaridis--Venkatesh \cite{SV17} for wavefront spherical homogeneous spaces over non-Archimedean $F$ and the works by van den Ban, Schlichtkrull, Delorme \textit{et al.} for real symmetric spaces, just to mention a few.

Write $\mathcal{H}_\tau = \tau \hat{\otimes} \mathcal{M}_\tau$ where $\mathcal{M}_\tau$ is equipped with the trivial $G(F)$-representation. Following \cite{Be88}, we will explain in \S\ref{sec:GK-method} that the dual Hilbert space $\mathcal{M}_\tau^\vee$ injects into $\Hom_{G(F)}(\Schw, \tau)$, which is in turn embedded into $\Hom_{G(F)}(\pi, \Schw^\vee)$ by taking adjoint, where $\pi := \bar{\tau}^\infty$ denotes the smooth part of the complex-conjugate $\bar{\tau}$. Furthermore, it is shown in \cite{Be88} that
\[ \Hom_{G(F)}(C^\infty_c(X^+), \tau) \simeq \Hom_{G(F)}(\pi, C^\infty(X^+)) = \mathcal{N}_\pi. \]
Therefore the $L^2$ theory is connected to the ``smooth theory'', i.e.\ the study of $\mathcal{N}_\pi$, the coefficients, etc.

Our proposal is to understand the Schwartz space $\Schw$ in zeta integrals from the same angle. The first task is to show $\alpha: \Schw \to L^2(X^+)$ is pointwise defined. The Gelfand--Kostyuchenko method asserts that it suffices to assume $\Schw$ nuclear. This result was stated in \cite[Chapter I, 4]{GV4} and \cite[Chapter II, \S 1]{Ma68} in slightly different terms; we will give a quick proof in \S\ref{sec:GK-method} based on \cite{Be88}. The advantage of working with general $\Schw$ is as follows.
\begin{enumerate}
	\item Consider two homogeneous spaces $X_1^+$, $X_2^+$ together with an isomorphism $\mathcal{F}: L^2(X_2^+) \rightiso L^2(X_1^+)$ of unitary $G(F)$-representations. One infers that the same Plancherel measure may be used for both sides. On the other hand, the comparison of refined Plancherel decompositions depends crucially on the choice of test functions.
	\item It is known that $\mathcal{F}$ disintegrates into isometries $\eta(\tau): \mathcal{M}^{(2)}_\tau \rightiso \mathcal{M}^{(1)}_\tau$ between multiplicity spaces. If $\mathcal{F}$ restricts to a continuous isomorphism $\Schw_2 \rightiso \Schw_1$ between given Schwartz spaces, then $\eta(\tau)^\vee$ is restricted from $\mathcal{F}^*: \Hom_{G(F)}(\Schw_1, \tau) \rightiso \Hom_{G(F)}(\Schw_2, \tau)$. For studying $\eta(\tau)$ (thus $\mathcal{F}$), the main hurdle is to understand the spaces $\Hom_{G(F)}(\Schw_i, \tau) \subset \Hom_{G(F)}(\pi, \Schw_i^\vee)$ together with the image of $\mathcal{M}_\tau^{(i),\vee}$ therein.
	\item The spaces $\mathcal{N}^{(1)}_\pi$ and $\mathcal{N}^{(2)}_\pi$ are relatively easier to cope with, but $\mathcal{F}$ seldom preserves $C^\infty_c$. Omitting indices, the problem may thus be reformulated as follows: how to pass from $\mathcal{N}_\pi$ to $\Hom_{G(F)}(\pi, \Schw^\vee)$ in a manner that respects the embeddings of $\mathcal{M}_\tau^\vee$? The situation is complicated by the non-injectivity of $\Schw^\vee \to C^\infty_c(X^+)^\vee \supset C^\infty(X^+)$, since $C^\infty_c(X^+)$ is rarely dense in $\Schw$.
	\item Our tentative answer is to relate them via meromorphic continuation, more precisely, by zeta integrals. We contend that the restriction of $\gamma(\pi): \mathcal{N}^{(1)}_\pi \otimes \mathcal{K} \to \mathcal{N}^{(2)}_\pi \otimes \mathcal{K}$ to the images of $\mathcal{M}^{(i),\vee}_\tau$ can be evaluated at $\lambda=0$, for $\mu$-almost all $\tau$, and it gives $\eta(\tau)^\vee$.
\end{enumerate}

The heuristics here is that upon twisting a coefficient $\varphi(v)$ by $|f|^\lambda$, where $\varphi \in \mathcal{N}_\pi$, $v \in \pi$ and $\Re(\lambda) \relgg{X} 0$, it will decrease rapidly towards $\partial X$, thus behaves like an element of $\Schw^\vee$. We are still unable to obtain complete results in this direction, but we can partially justify some special cases, including the non-Archimedean Godement--Jacquet case, under
\begin{itemize}
	\item a condition \eqref{eqn:generic-inj} of generic injectivity of the restriction map
		\[ \Hom_{G(F)}(\pi_\lambda, \Schw^\vee) \to \Hom_{G(F)}(\pi_\lambda, C^\infty_c(X^+)^\vee); \]
	\item a technical condition \eqref{eqn:L2-holomorphy} which may be regarded as a parametrized Gelfand--Kostyuchenko method; unfortunately, the author is not yet able to establish this property.
\end{itemize}

Recall that in the case of $L$-monoids (Braverman--Kazhdan), Lafforgue's work \cite{Laf14} turns the picture over by taking the local factors and Harish-Chandra's Plancherel formula as inputs, and produces the Schwartz functions and the Fourier transform in spectral terms. On the contrary, our approach starts from a function space $\Schw$ and outputs a refined spectral decomposition for $L^2$. These two approaches should be seen as complementary.

To conclude the discussion, let us revisit the local Godement--Jacquet integral \eqref{eqn:GJ-first}. Fix an additively invariant volume form $\eta \neq 0$ on $\text{Mat}_{n \times n}$ so that $|\Omega| := |\det|^{-n} |\eta|$ gives a Haar measure on $\GL(n,F)$. Let $\pi$ be an irreducible smooth representation of $\GL(n,F)$. According to the formalism above with $\mathscr{E} = \mathscr{L}^\demi$, one should consider the integration of the following smooth sections of $\mathscr{L}^\demi$, multiplied by $|\det|^\lambda$:
\begin{compactitem}
	\item coefficients $\varphi(v \otimes \check{v}) = \angles{\check{v}, \pi(\cdot)v} |\Omega|^\demi$ where $v \otimes \check{v} \in \pi \boxtimes \check{\pi}$, against
	\item a Schwartz--Bruhat half-density $\xi = \xi_0 |\eta|^\demi$ on $\text{Mat}_{n \times n}(F)$, where $\xi_0$ is a Schwartz--Bruhat function on $\text{Mat}_{n \times n}(F)$. This is a reasonable candidate of $\mathscr{L}^\demi$-valued Schwartz space on $\text{Mat}_{n \times n}(F)$ since the bundle $\mathscr{L}^\demi$ extends naturally to $\text{Mat}_{n \times n}(F)$, and it is contained in $L^2(\GL(n)) = L^2(\text{Mat}_{n \times n})$. When $F$ is Archimedean, this coincides with the Schwartz space $\Schw(\text{Mat}_{n \times n}(F), \mathscr{L}^\demi)$ of \cite[Definition 5.1.2]{AG08}.
\end{compactitem}
Since $|\eta|^\demi = |\det|^{\frac{n}{2}} |\Omega|^\demi$, the zeta integral $Z_{\lambda, \varphi}((v \otimes \check{v}) \otimes \xi)$ equals
\[ \int_{\GL(n,F)} \angles{\check{v}, \pi(\cdot)v} \xi_0 |\det|^{\lambda + \frac{n}{2}}  |\Omega| = Z^\text{GJ}\left( \lambda + \demi, v \otimes \check{v}, \xi \right). \]
This explains the shift by $\frac{n}{2}$ in \eqref{eqn:GJ-first}. Another bonus is that the Schwartz--Bruhat half-densities behave much better under Fourier transform, see \S\ref{sec:Fourier}.

\subsection*{On the global case}
Our brief, speculative treatment of the global setting in \S\ref{sec:global} is largely based on \cite[\S 3]{Sak12}; however, extra care is needed to deal with functions with values in a bundle. Let $F$ be a number field and $\A := \A_F$, consider a split connected reductive $F$-group $G$ and an affine spherical embedding $X^+ \hookrightarrow X$ of $G$-varieties over $F$, satisfying Axiom \ref{axiom:geometric} with everything defined over $F$. Furthermore, assume that the lattice $\Lambda$ and the monoid $\Lambda_X$ coincide with their local counterparts at each place $v$ of $F$. Choose relative invariants $f_i$ over $F$ with eigencharacters $\omega_i$ as in the local case, so that $|f|^\lambda$, $|\omega|^\lambda$ make adélic sense.

Suppose that the Schwartz spaces $\Schw_v \subset L^2(X^+_v)$ are given at each place $v$, say with values in a $G(F_v)$-equivariant vector bundle $\mathscr{E}_v$. We need two ingredients.
\begin{enumerate}
	\item The choice of \emph{basic vectors} (Definition \ref{def:basic-vector}, called basic functions in \cite{Sak12}) $\xi^\circ_v \in \Schw_v^{G(\mathfrak{o}_v)}$ for almost all $v \nmid \infty$, at which our data have good reduction. One can then define $\Schw := \Resotimes_v \Schw_v$ and $\mathscr{E} := \Resotimes_v \mathscr{E}_v$ (a bundle over $X^+(\A)$) with respect to $\xi_v^\circ$. Let $\mathscr{L}$ denote the density bundle on the adélic space $X^+(\A)$, and similarly for $\mathscr{L}^\demi$, we require that the local pairings patch into a well-defined $\mathscr{E} \otimes \overline{\mathscr{E}} \to \mathscr{L}$; morally, this is a condition about the convergence of infinite products of local measures.
	\item Suppose that a system of global ``evaluation maps'' $(\text{ev}_\gamma: \Schw \to \CC)_{\gamma \in X^+(F)}$ is given, subject to suitable equivariance properties (Hypothesis \ref{hyp:evaluation-global}). In \S\ref{sec:theta-dist} we  will construct a canonical closed central subgroup $\mathfrak{a}[X] \subset G(F_\infty)$ acting trivially on $X^+(\A)$. We require that
	\[ \vartheta := \sum_{\gamma \in X^+(F)} \text{ev}_\gamma \]
	defines a $G(F) \mathfrak{a}[X]$-invariant continuous linear functional of $\Schw$, which gives rise to a $G(\A)$-equivariant map $\xi \mapsto \vartheta_\xi \in C^\infty(\amspace_{G,X})$ by Frobenius reciprocity; here $\amspace_{G,X} := \mathfrak{a}[X]G(F) \backslash G(\A)$. 
\end{enumerate}
Note that distributions akin to $\vartheta$ are ubiquitous in the theory of automorphic forms.

Choose an invariant measure on $\amspace_{G,X}$ and let $\pi$ be a smooth cuspidal automorphic representation of $G(\A)$ realized on $\amspace_{G,X}$. The global zeta integral takes the form
\[ Z_\lambda(\phi \otimes \xi) = \int_{\amspace_{G,X}} \vartheta_\xi \phi_\lambda, \quad \xi \in \Schw,\; \phi \in \pi,\; \lambda \in \Lambda_{\CC}, \]
where $\phi_\lambda := \phi|\omega|^\lambda$, supposed to converge for $\Re(\lambda) \relgeq{X} \lambda(\pi)$ for some $\lambda(\pi)$; furthermore, we suppose that $Z_\lambda$ admits meromorphic continuation to all $\lambda$. This is the content of Axiom \ref{axiom:theta}.

We will discuss the global functional equation in \S\ref{sec:Poisson}, which takes the form $Z^{(1)}_\lambda(\phi \otimes \mathcal{F}(\xi)) = Z^{(2)}_\lambda(\phi \otimes \xi)$. The conjectural Poisson formula will also be briefly mentioned, but we do not attempt to build a general theory of Poisson formulas.

As in the local situation, the prehomogeneous vector spaces furnish a useful testing ground for the global conjectures. Nonetheless, it has been observed in \cite[p.656]{Sak12} that the resulting zeta integrals do not produce new $L$-functions.

\section{Positive results}
One may wonder if it is possible to prove anything under this generality. Below is a list of local evidence obtained in this work. Let $F$ be a local field of characteristic zero.
\begin{enumerate}
	\item The minimal requirement: convergence of zeta integrals for $\Re(\lambda) \relgg{X} 0$. We prove this for $F$ non-Archimedean, $X^+$ wavefront and $\mathscr{E} = \mathscr{L}^\demi$ in Corollary \ref{prop:conv-gg0}. The natural idea is to use the assumption that $\Supp(\xi)$ has compact closure in $X(F)$, the smooth asymptotics of the coefficients, which make use of a smooth complete toroidal compactification $X^+ \hookrightarrow \bar{X}$, and the Cartan decomposition for spherical homogeneous spaces. One technical point is to make $X^+ \hookrightarrow \bar{X}$ compatible with the geometry of $X$, see \S\ref{sec:Cartan}. The wavefront assumption is imposed in order to use the smooth asymptotics \cite[\S 5]{SV17}.

	For partial results in the Archimedean case, we refer to \cite{Li18}.
	\item The case where $X$ is a $F$-regular prehomogeneous vector space. Indeed, this is the only case with complete definitions of Schwartz space, model transition and Poisson formulas. Another reason is that the difficulty of choosing $\Schw$ often comes from the singularities of $X$, whereas non-singular affine spherical $G$-varieties are fibered in prehomogeneous vector spaces --- this fact is due to Luna, see \cite[Theorem 2.1]{KV06}.
	
		We will prove that
		\begin{inparaenum}[(i)]
			\item the spaces of Schwartz--Bruhat half-densities satisfy all the required properties, except temporarily those concerning Archimedean zeta integrals (Theorem \ref{prop:pvs-Schw}) --- the rationality is established by invoking Igusa's theory;
			\item upon fixing an additive character $\psi$, the Fourier transform $\mathcal{F}: \Schw(X) \rightiso \Schw(\check{X})$ of Schwartz--Bruhat half-densities is well-defined: the resulting theory is actually simpler than the classical one (Theorem \ref{prop:Fourier-vector});
			\item the local functional equation for the Fourier transform holds under a geometric condition (Hypothesis \ref{hyp:lfe}), which includes the well-known case of symmetric bilinear forms on an $F$-vector space, the Godement--Jacquet case $X = \text{Mat}_{n \times n}$ and their dual. In fact, the hypothesis for $X$ implies that $T_\pi: \mathcal{N}_\pi \otimes \mathcal{K} \to \mathcal{L}_\pi$ is an isomorphism.
		\end{inparaenum}

		We will actually re-derive the local Godement--Jacquet theory in \S\ref{sec:GJ} as an application of the general prehomogeneous theory, with improvements on continuity issues in the Archimedean case. The local functional equation so interpreted relates two different spherical $\GL(n) \times \GL(n)$-varieties, namely the prehomogeneous vector space $\text{Mat}_{n \times n}$ and its dual, the latter being isomorphic to $\text{Mat}_{n \times n}$ with the flipped action $x(g_1, g_2) = g_1^{-1} x g_2$. The result is a conceptually cleaner theory.
		
		However, little is known beyond the Godement--Jacquet case, and we give no predictions on the general structure of $\gamma$-factors (or $\gamma$-matrices) in the prehomogeneous case.
	\item The doubling method of Piatetski-Shapiro and Rallis is considered in \S\ref{sec:doubling}, following the reinterpretation by Braverman and Kazhdan \cite{BK02} in terms of Schwartz spaces. Modulo certain conjectures or working hypotheses inherent in \textit{loc.\ cit.}, we will set up a geometric framework and then reconcile their theory with ours in Proposition \ref{prop:doubling-zeta-axiom}.
	
		To simplify matters and stay within the realm of split connected groups, we will mainly work with the symplectic case, then sketch the general setting in \S\ref{sec:doubling-general}; another reason is that the formalism of \cite{BK02} coincides with that of \cite{PSR86, GPSR87} only in the symplectic case.

		We will also interpret the affine spherical embedding $X^+ \hookrightarrow X$ in doubling method as a reductive monoid. In Theorem \ref{prop:PSR-Ngo-compatible}, everything is shown to be compatible with the recipe in \cite{Laf14,Ng14} of attaching $L$-monoids to irreducible representations of the Langlands dual group. Consequently, one obtains a family of examples in Braverman--Kazhdan--Ngô program that goes beyond Godement--Jacquet, whereas the Schwartz space, Fourier transform and Poisson formula (in the global case) are still available to some extent.
\end{enumerate}

As for the global case, let $F$ be a number field.
\begin{enumerate}
	\item We will show that in Theorem \ref{prop:theta-period} that for $\lambda$ in the range of convergence, we have
		\[ Z_\lambda(\phi \otimes \xi)  = \int_{X^+(\A)} \xi |f|^\lambda \mathcal{P}(\phi) \]
		where $\mathcal{P}$ is an intertwining operator $\pi \to C^\infty(X^+(\A), \overline{\mathscr{E}})$ made from period integrals over $H_\gamma := \Stab_G(\gamma)$, for various $\gamma \in X^+(F)$; see \eqref{eqn:period} for the precise recipe. In particular,
		\begin{compactitem}
			\item the global twists by automorphic characters $|\omega|^\lambda$ correspond to local twists by relative invariants $|f|^\lambda$, and
			\item $Z_\lambda$ factorizes into local zeta integrals provided that $\mathcal{P}$ factorizes into local coefficients $\varphi_v \in \mathcal{N}_{\pi_v}$.
		\end{compactitem}
		This partly justifies our local formalism.
	\item To illustrate the use of Poisson formulas, we deduce the global Godement--Jacquet functional equation in Example \ref{eg:GJ-global}. The core arguments are the same as the classical ones.
\end{enumerate}

Notice one obvious drawback of our current formalism: it excludes the theories involving unipotent integration such as the Rankin--Selberg convolutions, automorphic descent, etc. To remedy this, one will need an in-depth study of the Whittaker-type inductions \cite[\S 2.6]{SV17} as well as the unfolding procedure \cite[\S 9.5]{SV17} of Sakellaridis--Venkatesh. Another drawback is that we do not attempt to locate the poles of zeta integrals or to make the $\gamma$-factors explicit.

\paragraph*{Acknowledgements}
The author is grateful to Professors Bill Casselman, Wee Teck Gan, Yiannis Sakellaridis, Freydoon Shahidi, Binyong Sun and Satoshi Wakatsuki for encouragements and fruitful discussions. His thanks also go to the anonymous referees for their helpful remarks.

\section{Structure of this book}
\begin{compactitem}
	\item In \S\ref{sec:geometric-bg} we collect vocabularies of spherical embeddings, the Luna--Vust classification, the Cartan decomposition and boundary degenerations set up in \cite{SV17}, and then state the geometric Axiom \ref{axiom:geometric}.
	\item In \S\ref{sec:analytic-bg} we define the integration of densities, the formalism of direct integrals and the Gelfand--Kostyuchenko method concerning pointwise defined continuous maps $\Schw \hookrightarrow L^2(X^+)$; we also include a brief review of hypocontinuous bilinear forms in order to fix notations.
	\item In \S\ref{sec:Schwartz-zeta} we specify the bundles, sections and coefficients in question, and state the Axioms \ref{axiom:finiteness}, \ref{axiom:zeta} concerning Schwartz spaces and zeta integrals. We then move to model transitions, the notion of local functional equations and $\gamma$-factors. This chapter ends with a heuristic discussion relating the Schwartz spaces to $L^2$-theory.
	\item The convergence of local zeta integrals in the non-Archimedean case is partly addressed in \S\ref{sec:convergence}. We begin by setting up some results from convex geometry, then we review the theory of smooth asymptotics with values in $\mathscr{L}^\demi$, and establish the convergence for $\Re(\lambda) \relgg{X} 0$.
	\item The example of prehomogeneous vector spaces is treated in \S\ref{sec:pvs}. We rewrite the theory of Schwartz--Bruhat functions and their Fourier transform in terms of half-densities, following \cite[\S 9.5]{SV17} closely. After a rapid review of prehomogeneous vector spaces and their dual, we verify the relevant axioms, mainly in the non-Archimedean case, and prove the local functional equation under Hypothesis \ref{hyp:lfe} over non-Archimedean $F$. The local Godement--Jacquet theory is then realized as a special case.
	\item In \S\ref{sec:doubling}, we review the general theory in \cite{BK02} on the Schwartz space and Fourier transforms for $X_P := P_\text{der} \backslash G^\Box$, specialize it to the doubling method for classical groups and discuss the corresponding zeta integrals. We then realize the doubling method as a special case of our formalism, and relate it to $L$-monoids in the sense of Braverman--Kazhdan, Lafforgue and Ngô. Except in \S\ref{sec:doubling-general}, we will focus on the symplectic case of doubling.
	\item The \S\ref{sec:global} is a sketch of the global case. Requirements on the basic vectors, $\vartheta$-distributions and the global integrals $Z_\lambda$ are set up, including the Axiom \ref{axiom:theta} on the convergence and meromorphy of $Z_\lambda$. We relate $Z_\lambda$ to period integrals, and conclude by a discussion on global functional equations and Poisson formulas, with an illustration in the global Godement--Jacquet case.
\end{compactitem}

The necessary hypotheses and conjectures will be summarized in the beginning of each chapter. Below is a dependency graph of the chapters.
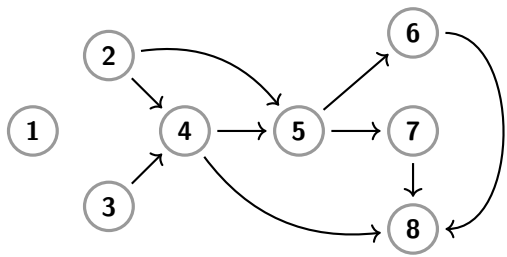
\begin{center} \begin{tikzpicture}[every node/.style = {very thick, draw=black!40, font=\bfseries\sffamily, circle, outer sep=3pt}]
	\node (1) at (0, 0) {1};
	\node (2) at (1, 1) {2};
	\node (3) at (1, -1) {3};
	\node (4) at (2, 0) {4};
	\node (5) at (3.5, 0) {5};
	\node (6) at (5, 1.3) {6};
	\node (7) at (5, 0) {7};
	\node (8) at (5, -1.3) {8};
	\draw[thick, ->] (2) edge (4)	
		  (4) edge (5)
		  (2) edge[bend left=30] (5)
		  (3) edge (4)
		  (4) edge[bend right=30] (8)
		  (5) edge (6)
		  (5) edge (7)
		  (7) edge (8)
		  (6) edge[bend left=90] (8) ;
\end{tikzpicture}\end{center}

\section{Conventions}
\subsection*{Convex geometry}
Let $V$ be a finite-dimensional affine space over $\R$, i.e.\ a torsor under translations by a finite-dimensional $\R$-vector space. It makes sense to speak of affine closures, convexity, etc.\ in $V$. A \emph{polyhedron} in $V$ is defined to be the intersection of finitely many half-spaces $\{v : \alpha(v) \geq 0\}$, where $\alpha: V \to \R$ is an affine form. Bounded polyhedra are called \emph{polytopes}; equivalently, a polytope is the convex hull of finitely many points in $V$. The \emph{relative interior} of a polyhedron $P$ will be denoted by $\relint(P)$. \emph{Faces} of $P$ are defined as subsets of the form $H \cap P$, where $H$ is an affine hyperplane bounding $P$. The union of proper faces of $P$ equals $\partial P = P \smallsetminus \relint(P)$.

Suppose that a base point $0 \in V$ is chosen. A \emph{cone} in $V$ is defined as the intersection of finitely many $\{v : \alpha(v) \geq 0\}$ where $\alpha$ are now taken to be linear forms. Equivalently, a cone takes the form $C = \sum_{i=1}^k \R_{\geq 0} v_i$ for some generators $v_1, \ldots, v_k$; in particular, $C$ is polyhedral and finitely generated by definition. There is a unique minimal system of generators for $C$ up to dilation, namely take $v_i$ from the \emph{extremal rays} (i.e.\ $1$-dimensional faces) of $C$. A cone $C$ is called \emph{strictly convex}\index{cone!strictly convex} if $C \cap (-C) = \{0\}$, \emph{simplicial}\index{cone!simplicial} if $C$ is generated by a set of linearly independent vectors. A \emph{fan} in $V$ is a collection of cones closed under intersections and taking faces.

We will primarily work with \emph{rational} polytopes, cones and fans. This means that the space $V$, the base point, its affine or linear forms are all defined over $\Q$. Faces of rational polytopes are still rational.

\subsection*{Fields}
Given a field $F$, the symbol $\bar{F}$ will stand for some algebraic closure of $F$; the Galois cohomology will be denoted by $H^\bullet(F, \cdot)$.

When $F$ is a local field, the normalized absolute value will be denoted as $|\cdot| = |\cdot|_F$; when $F=\CC$ we set $|z| := z\bar{z}$ so that Artin's product formula for absolute values holds. For a global field $F$, we denote its ring of adèles as $\A = \A_F$. When $F$ is a global field or a non-Archimedean local field, its ring of integers will be denoted by $\mathfrak{o}_F$. Let $v$ be a place of a global field $F$, i.e.\ an equivalence class of valuations of rank $1$, then $F_v$ denotes the completion of $F$ at $v$. If $v \nmid \infty$ (i.e.\ non-Archimedean), we abbreviate $\mathfrak{o}_{F_v}$ as $\mathfrak{o}_v$, and $\mathfrak{m}_v \subset \mathfrak{o}_v$ stands for the maximal ideal.

For a global field $F$ we write $|\cdot| = \prod_v |\cdot|_v: \A_F \to \R_{\geq 0}$ to denote the adélic absolute value. The role of $|\cdot|$ will be clear from the context.

\subsection*{Varieties}
Let $\Bbbk$ be a field. Unless otherwise specified, the $\Bbbk$-varieties are assumed to be separated, geometrically integral schemes of finite type over $\Spec(\Bbbk)$. Recall that a $\Bbbk$-variety $X$ is called \emph{quasi-affine}\index{quasi-affine} if it admits an open immersion into an affine $\Bbbk$-variety. Call $X$ \emph{strongly quasi-affine} if $\Gamma(X, \mathcal{O}_X)$ is a finitely generated $\Bbbk$-algebra and $X \to \Spec(\Gamma(X, \mathcal{O}_X))$ is an open immersion. In the latter case, call $\overline{X}^\text{aff} := \Spec(\Gamma(X, \mathcal{O}_X))$ the \emph{affine closure}\index{affine closure} of $X$.

For general schemes over $\Spec(R)$ where $R$ is a ring, we usually distinguish the scheme $X$ itself and the set 	$X(R)$; the latter is endowed with a natural topology when $R$ is a topological ring. If for every $x \in X$ the local ring $\mathcal{O}_{X,x}$ is a normal domain, we say $X$ is a normal scheme.

\subsection*{Groups}
Let $G$ be an affine algebraic $\Bbbk$-group. Its identity connected component is denoted by $G^\circ$ and its center denoted by $Z_G$. The normalizer (resp.\  centralizer) subgroups in $G$ are denoted as $Z_G(\cdot)$ (resp.\  $N_G(\cdot)$). We adopt the standard notations $\Gm$, $\Ga$, $\GL$, etc.\ for the well-known groups, with one slight abuse: $\Ga$ usually comes with the monoid structure under multiplication aside from the additive one, and $\Ga \supset \Gm$. In a similar spirit, $\GL(n)$ embeds into the space of $n \times n$-matrices denoted by $\text{Mat}_{n \times n}$.

We write $X^*(G) = \Hom_{\Bbbk\text{-grp}}(G, \Gm)$, which is naturally a $\Z$-module. When $G$ is a torus we define $X_*(G) = \Hom(\Gm, G)$. For any commutative ring $R$ we set $X^*(G)_R := X^*(G) \otimes R$, and same for $X_*(G)_R$. Write $G \twoheadrightarrow G_\text{ab} := G/G_\text{der}$ for the abelianization. The unipotent radical of a parabolic subgroup $P \subset G$ is denoted by $U_P$, and we denote opposite parabolic subgroups by $P^-$. The Lie algebras are indicated by $\mathfrak{gothic}$ typeface.

\subsection*{\texorpdfstring{$G$}{G}-varieties}
Let $G$ be an affine algebraic $\Bbbk$-group. A $G$-variety is an irreducible $\Bbbk$-variety $X$ equipped with a \emph{right action} $X \times G \to X$, written as $(x,g) \mapsto xg$. Call $X$ \emph{homogeneous} if it consists of a single $G$-orbit. By choosing $x_0 \in X(\Bbbk)$, a given $G$-space $X$ is isomorphic to $H \backslash G$ (the geometric quotient) for $H := \Stab_G(x_0)$. The group $\Aut_G(X)$\index{Aut_G(X)@$\Aut_G(X)$} of its $G$-automorphisms is isomorphic to $H \backslash N_G(H)$, therefore forms an algebraic $\Bbbk$-group; more precisely, we let a coset $Ha$ in $H \backslash N_G(H)$ act \emph{on the left} of $H \backslash G$ by
\[ Hg \mapsto Hag. \]
Set $\mathcal{Z}(X) := \Aut_G(X)^\circ$\index{Z(X)@$\mathcal{Z}(X)$}. We refer to \cite[\S 1, Appendix D]{Ti11} or \cite[II. \S 4]{AG4} for quotients for $G$-varieties.

\subsection*{Vector spaces}
The topological vector spaces considered in this work are all complex, Hausdorff and locally convex. For a vector space $E$, we denote its dual space as $E^\vee$, and its exterior algebra as $\bigwedge E = \bigoplus_{k \geq 0} \bigwedge^k E$. When $E$ is a topological vector space, $E^\vee$ stands for the \emph{continuous dual} of $E$, i.e.\ the space of continuous linear functionals; depending on the context, it will be equipped with either the strong topology (i.e.\ of \emph{bounded convergence}) or the weak topology (i.e.\ of \emph{pointwise convergence}), see \cite[p.198]{Tr67}. The scalar product on a Hilbert space $E$ will be denoted by $(\cdot|\cdot) = (\cdot|\cdot)_E$. The complex conjugate of a $\CC$-vector space $E$ will be denoted by $\overline{E}$; in the case of Hilbert spaces, $\overline{E}$ is canonically isomorphic to $E^\vee$.\index{topological vector space!strong dual}

When $E_1, E_2$ are topological vector spaces, $\Hom(E_1, E_2)$ will mean the vector space of continuous linear maps $E_1 \to E_2$ unless otherwise specified. For nuclear spaces we will denote the completed tensor product as $\hat{\otimes}$. The same notation also stands for the completed tensor product for Hilbert spaces \cite[V, \S 3]{BoEVT}. The meaning is always clear from the context.

\subsection*{Representations}
Unless otherwise specified, all representations are realized on $\CC$-vector spaces and the group acts from the left. The underlying vector space of a representation $\pi$ will be systematically denoted as $V_\pi$, sometimes identified with $\pi$ itself by abusing notation. A \emph{continuous representation} of a locally compact separated group $\Gamma$ on a topological vector space $V$ corresponds to a continuous linear action $\Gamma \times V \to V$. The \emph{unitary representations} are always realized as isometries on Hilbert spaces. The contragredient (in a suitable category) of a representation $\pi$ is denoted by $\check{\pi}$.

When discussing representations of Lie groups, we will occasionally work in the category of \emph{Harish-Chandra modules} or admissible $(\mathfrak{g}, K)$-modules; such modules are of finite length. The precise definition can be found in \cite[\S 4]{BK14}.

The modular character $\delta_\Gamma: \Gamma \to \R_{>0}^\times$ is specified by $\dd\mu(gxg^{-1}) = \delta_\Gamma(g) \dd\mu(x)$ where $\mu$ is any left Haar measure on $\Gamma$.

Given $\Gamma$-representations $V_1$, $V_2$, denote by $\Hom_\Gamma(V_1, V_2)$ the vector space of continuous $\Gamma$-invariant linear maps $V_1 \to V_2$. The tensor product of two representations $\pi_1$, $\pi_2$ (in suitable categories) of groups $G_1$, $G_2$ is denoted by $\pi_1 \boxtimes \pi_2$, which is a representation $G_1 \times G_2$.

Let $H$ be a subgroup of an affine $F$-group $G$ over a local field, we denote by $\Ind^G_H(\cdot)$ the (unnormalized) induction functor of smooth representations from $H(F)$ to $G(F)$, which will be recalled in \S\ref{sec:coefficients}. If $P \subset G$ is a parabolic subgroup with Levi quotient $M := P/U_P$, we denote by $I^G_P(-) := \Ind^G_P \left(- \otimes \delta_P^\demi \right)$ the functor of \emph{normalized parabolic induction} of smooth representations from $M(F)$ to $G(F)$, where $\delta_P = \delta_{P(F)}$ factorizes through $M(F)$.


\chapter{Geometric background}\label{sec:geometric-bg}
\section{Review of spherical varieties}\label{sec:review-spherical}
The main purpose of this section is to fix notation. We refer to \cite{Kno91} for a detailed treatment.

Let $\Bbbk$ be an algebraically closed field of characteristic zero. Fix a connected reductive $\Bbbk$-group $G$, a Borel subgroup $B \subset G$ and denote by $B \twoheadrightarrow B/U =: A$ its Levi quotient.

A $G$-variety $X$ is \emph{spherical}\index{spherical variety} if $X$ is normal and there exists an open dense $B$-orbit $\mathring{X}$\index{X-ring@$\mathring{X}$} in $X$, which is unique. Hence there exists a unique open dense $G$-orbit $X^+$\index{X-plus@$X^+$} in $X$. By a \emph{spherical embedding}\index{spherical embedding} we mean a $G$-equivariant open immersion $X^+ \hookrightarrow X$ of $\Bbbk$-varieties, where $X^+$ is homogeneous and $X$ is spherical. A spherical embedding $X^+ \hookrightarrow X$ is called \emph{simple}\index{spherical embedding!simple} if there exists a unique closed $G$-orbit in $X$; it is called affine if $X$ is affine\index{spherical embedding!affine}.

The boundary $\partial X$ is defined as $X \smallsetminus X^+$, equipped with the reduced induced scheme structure.

For spherical homogeneous spaces, quasi-affine implies strongly quasi-affine by \cite[Proposition 2.2.3]{Sak12}, thus one can define the affine closure $\overline{H \backslash G}^\text{aff}$ for quasi-affine spherical homogeneous $G$-spaces $H \backslash G$.

Let $P$ be a parabolic subgroup of $G$, with unipotent radical $U_P$ and Levi quotient $\pi: P \twoheadrightarrow P/U_P = M$. For any $M$-variety $Y$, its \emph{parabolic induction} is defined to be the geometric quotient
\begin{gather}\label{eqn:parabolic-induction}
	X := Y \utimes{P} G = \dfrac{Y \times G}{(yp,g) \sim (y,pg), \; p \in P},
\end{gather}
where $Y$ is inflated into a $P$-variety via $\pi$; see \cite[II. \S 4.8]{AG4} for generalities on such contracted products which are often called homogeneous fiber bundles. The class containing $(y,g)$ will be denoted by $[y,g]$. The $G$-action on $X$ is $[y,g]g' = [y,gg']$.
\begin{itemize}
	\item Mapping $[y,g]$ to the coset $Pg$ yields a $G$-equivariant fibration $X \twoheadrightarrow P \backslash G$, whose fiber over $P \cdot 1$ is just $Y$.
	\item If $Y \simeq H_M \backslash M$, then $X$ is isomorphic to $\pi^{-1}(H_M) \backslash G$.
	\item Conversely, if $X \simeq H \backslash G$ is a homogeneous $G$-variety with $U_P \subset H \subset P$, then $X$ is parabolically induced from of $Y := H_M \backslash M$ with $H_M := \pi(H) \subset M$. Indeed, $[H_M m, g] \in Y \utimes{P} G$ corresponds to $H \pi^{-1}(m)g$.
	\item $X$ is spherical if and only if $Y$ is spherical.
\end{itemize}

To a spherical $G$-variety $X$ are attached the following birational invariants.
\begin{enumerate}
	\item Write $\Bbbk(X)^{(B)}$ for the group of $B$-eigenvectors in $\Bbbk(X)^\times$, which are $U$-invariant by stipulation. To each $f \in \Bbbk(X)^{(B)}$ is associated a unique eigencharacter $\lambda \in X^*(B)=X^*(A)$. Let $\Lambda(X) \subset X^*(A)$\index{Lambda(X)@$\Lambda(X)$} be the subgroup formed by these eigencharacters. We deduce a short exact sequence
		\begin{gather}\label{eqn:spherical-ses}
			1 \to \Bbbk^\times \to \Bbbk(X)^{(B)} \to \Lambda(X) \to 0.
		\end{gather}
	\item Denote by $A \twoheadrightarrow A_X$\index{A_X@$A_X$} the homomorphism of tori dual to $\Lambda(X) \hookrightarrow X^*(A)$. Set $\mathcal{Q} := X_*(A_X) \otimes \Q$\index{Q@$\mathcal{Q}$}. Every discrete valuation $v$ of $\Bbbk(X)$ which is trivial on $\Bbbk^\times$ induces an element $\rho(v)$ of $\mathcal{Q}$.
	\item Let $\mathcal{V}$\index{V@$\mathcal{V}$} be the set of $G$-invariant discrete valuations of $\Bbbk(X)$ which are trivial on $\Bbbk^\times$.
	\item Let $\mathcal{D}^B$\index{D_B@$\mathcal{D}^B$} be the set of $B$-stable prime divisors in $X^+$; thus the divisors in $\mathcal{D}^B$ are supported in $X^+ \smallsetminus \mathring{X}$. Since our varieties are normal, to each $D \in \mathcal{D}^B$ we may define $\rho(D) \in \mathcal{Q}$ as the corresponding normalized valuation of $\Bbbk(X)$. Hence a canonical map
		\[ \rho: \mathcal{D}^B \longrightarrow \mathcal{Q}. \]
\end{enumerate}

By \cite[Corollary 1.8]{Kno91}, the map $v \mapsto \rho(v)$ embeds $\mathcal{V}$ as a subset of $\mathcal{Q}$. Moreover, \cite[Corollary 5.3]{Kno91} asserts that $\mathcal{V}$ is a cone in $\mathcal{Q}$ which contains the image of the anti-dominant Weyl chamber under $X_*(A) \otimes \Q \twoheadrightarrow \mathcal{Q}$. Call $\mathcal{V} \subset \mathcal{Q}$ the \emph{valuation cone}\index{valuation cone} of $X$.

\begin{definition}
	Call a spherical variety $X$ \emph{wavefront}\index{spherical variety!wavefront} if $\mathcal{V}$ equals the image of the anti-dominant Weyl chamber in $\mathcal{Q}$.
\end{definition}
The notions above depend only on $X^+$ and not its embedding in $X$.
\begin{definition}
	Given the data above on $X^+$, a \emph{strictly convex colored cone}\index{cone!colored} is a pair $(\mathcal{C}, \mathcal{F})$ where
	\begin{itemize}
		\item $\mathcal{C} \subset \mathcal{Q}$ is a strictly convex cone, finitely generated over $\Q$;
		\item $\mathcal{F} \subset \mathcal{D}^B$ and $\rho(\mathcal{F}) \not\ni 0$;
		\item $\relint(\mathcal{C}) \cap \mathcal{V} \neq \emptyset$;
		\item $\mathcal{C}$ is generated by $\rho(\mathcal{F})$ and finitely many elements from $\mathcal{V}$.
	\end{itemize}
	Elements of the set $\mathcal{F}$ are called the \emph{colors}\index{color}.
\end{definition}

Assume henceforth that the spherical embedding is simple with the closed $G$-orbit $Y$. We denote by $\mathcal{D}(X)$ the set of $B$-stable prime divisors $D$ on $X$ with $D \supset Y$. To each $D \in \mathcal{D}(X)$ we associate the normalized valuation $v_D: \Bbbk(X) \to \Z$, viewed as an element of $\mathcal{Q}$. Set \index{C_X@$\mathcal{C}_X$}\index{F_X@$\mathcal{F}_X$}
\begin{align*}
	\mathcal{V}_X & := \left\{ v_D : D \in \mathcal{D}(X) \text{ is $G$-stable} \right\}, \\
	\mathcal{F}_X & := \left\{ D \in \mathcal{D}(X) : D \cap X^+ \in \mathcal{D}^B \right\} \hookrightarrow \mathcal{D}^B, \\
	\mathcal{C}_X & := \text{the convex cone in $\mathcal{Q}$ generated by $\mathcal{V}_X \cup \rho(\mathcal{F}_X)$}.
\end{align*}

The Luna--Vust classification of simple spherical embeddings can now be stated as follows.
\begin{theorem}[Luna--Vust]\label{prop:LV}
	Given a spherical homogeneous $G$-space $X^+$, there is a bijection
	\begin{align*}
		\left\{ \text{simple spherical embeddings of } X^+ \right\}/\simeq & \stackrel{1:1}{\longleftrightarrow} \{\text{strictly convex colored cones}\} \\
		[X^+ \hookrightarrow X] & \longmapsto (\mathcal{C}_X, \mathcal{F}_X).
	\end{align*}
	Furthermore,
	\begin{itemize}
		\item A valuation $v: \Bbbk(X) \twoheadrightarrow \Z \sqcup \{\infty\}$ lies in $\mathcal{V}_X$ if and only if $\Q_{\geq 0} v$ is an extremal ray of $\mathcal{C}_X$ not intersecting of $\rho(\mathcal{F}_X)$;
		\item every affine spherical embedding $X^+ \hookrightarrow X$ is simple;
		\item $X$ is affine if and only if there exists $\chi \in \mathcal{Q}^\vee$ such that $\chi|_{\mathcal{V}} \leq 0$, $\chi|_{\mathcal{C}_X} = 0$ and $\chi|_{\rho(\mathcal{D}^B \smallsetminus \mathcal{F}_X)} > 0$;
		\item $X^+$ is quasi-affine if and only if $\rho(\mathcal{D}^B)$ does not contain $0$ and generates a strictly convex cone.
	\end{itemize}
\end{theorem}
\begin{proof}
	See \cite[\S 3 and Theorem 6.7]{Kno91}. The description of $\mathcal{V}_X$ can be found in \cite[Lemma 2.4]{Kno91}.
\MyQED\end{proof}

A \emph{face} of a colored cone $(\mathcal{C}, \mathcal{F})$ is a pair $(\mathcal{C}_0, \mathcal{F}_0)$ where $\mathcal{C}_0$ is a face of $\mathcal{C}$ such that $\relint(\mathcal{C}_0) \cap \mathcal{V} \neq \emptyset$, and $\mathcal{F}_0 = \mathcal{F} \cap \rho^{-1}(\mathcal{C}_0)$.

\begin{proposition}\label{prop:G-orbits}
	Let $(\mathcal{C}_X, \mathcal{F}_X)$ be the colored cone attached to $X^+ \hookrightarrow X$. There is an order-reversing bijection
	\begin{gather*}
		\left\{ G\text{-orbits}\; \subset X \right\} \stackrel{1:1}{\longleftrightarrow} \{\text{faces of } \; (\mathcal{C}_X, \mathcal{F}_X) \}
	\end{gather*}
	such that a valuation $v \in \mathcal{V} \cap \mathcal{C}$ lies in $\relint(\mathcal{C}_0)$ of a face $(\mathcal{C}_0, \mathcal{F}_0)$ if and only if the center of $v$ exists in $X$ and equals the orbit closure in $X$ corresponding to $(\mathcal{C}_0, \mathcal{F}_0)$.
\end{proposition}
\begin{proof}
	See \cite[Theorem 2.5 and Lemma 3.2]{Kno91}. Generalities on valuations can be found in \cite[\S 1]{Kno91}.
\MyQED\end{proof}
In particular, the center (see \cite[Definition B.4]{Ti11}) of any nonzero $v \in \mathcal{V} \cap \mathcal{C}$ is a proper $G$-orbit closure, since $\mathcal{C}_0 = \{0\}$ corresponds to the $G$-orbit $X^+$. The morphisms between spherical embeddings also have a combinatorial description as follows.

\begin{theorem}[{\cite[Theorem 4.1]{Kno91}}]\label{prop:LV-morphism}
	Let $X^+ \hookrightarrow X$ and $Y^+ \hookrightarrow Y$ be spherical embeddings under $G$, and denote their Luna--Vust data as $\mathcal{Q}_{X^+}$, $\mathcal{Q}_{Y^+}$ etc. Let $\varphi: X^+ \to Y^+$ be a $G$-equivariant morphism, we have
	\begin{itemize}
		\item $\varphi$ induces $\varphi_*: \Lambda_{X^+} \to \Lambda_{Y^+}$, thus $\varphi_*: \mathcal{Q}_{X^+} \twoheadrightarrow \mathcal{Q}_{Y^+}$;
		\item $\varphi_*$ maps $\mathcal{V}_{X^+}$ onto $\mathcal{V}_{Y^+}$;
		\item $\varphi$ extends to a $G$-equivariant morphism $X \to Y$ if and only if $\varphi_*(\mathcal{C}_X) \subset \mathcal{C}_Y$ and $\varphi_*(\mathcal{F}_X \smallsetminus \mathcal{F}_\varphi) \subset \mathcal{F}_Y$.
	\end{itemize}
	Here $\mathcal{F}_\varphi$ denotes the subset of $D \in \mathcal{D}^B_{X^+}$ that maps dominantly to $Y^+$.
\end{theorem}

\begin{example}\label{eg:typical-cone}
	Consider $X^+ = \GL(n)$ under the $G := \GL(n) \times \GL(n)$-action given by $x(g,h) = g^{-1}xh$. Take $B \subset G$ to be
	\[ B = \left\{ \text{lower triangular} \right\} \times \left\{ \text{upper triangular} \right\}. \]
	Then $A \twoheadrightarrow A_{X_+}$ may be identified with the quotient
	\begin{align*}
		\Gm^n \times \Gm^n & \longrightarrow \Gm^n \\
		(a,b) & \longmapsto a^{-1}b. 
	\end{align*}
	Take the standard basis $\epsilon_1, \ldots \epsilon_n \in X_*(A_{X^+})$ and denote by $\epsilon^*_1, \ldots, \epsilon^*_n$ the dual basis of $\Lambda(X^+)$. This is a wavefront spherical variety, indeed we have $\mathcal{V} = \bigcap_{1 \leq i < j \leq n}\{\epsilon^*_i - \epsilon^*_j \leq 0\}$ inside $\mathcal{Q} = \Q \epsilon_1 \oplus \cdots \oplus \Q \epsilon_n$.
	
	The relevant computations can be found in \cite[Example 24.9]{Ti11}. Consider the affine equivariant embedding into $X = \text{Mat}_{n \times n}$. We assume $n=2$ hereafter, although a general description for any $n$ exists: see \cite[Example 27.21]{Ti11}. We have
	\begin{compactitem}
		\item $\mathcal{C}_X$ is generated by $\{\epsilon_1 - \epsilon_2, \epsilon_2 \}$;
		\item $\mathcal{F}_X$ consists of the prime divisor $D = \begin{pmatrix} 0 & * \\ * & * \end{pmatrix}$ and $\rho(D) = \epsilon_1 - \epsilon_2$;
		\item $\mathcal{V}_X$ consists of the prime divisor $\det=0$, the corresponding normalized valuation is $\epsilon_2$;
		\item the faces of the colored cone are: $(\mathcal{C}_X, \mathcal{F}_X) \supset (\Q_{\geq 0} \epsilon_2 , \emptyset) \supset (0, \emptyset)$, corresponding to the rank stratification of $X$.
	\end{compactitem}
	Note that the divisor $\det=0$ (resp.\  $D$) is defined by a $B$-eigenfunction in $\Bbbk[X]$ with eigencharacter equal to the element $\epsilon_1^* + \epsilon_2^*$ (resp.\  $\epsilon_1^*$) of $\Lambda(X^+)$.
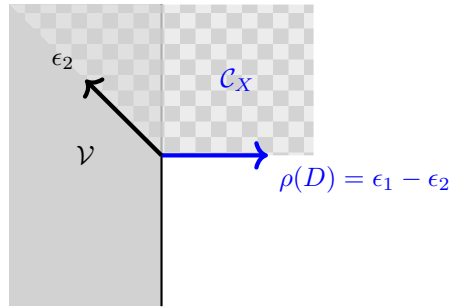
	\begin{center}\begin{tikzpicture}[baseline]
		\fill[gray!35!white] (0,-2) rectangle (-2,2);
		\node at (-1,0) {$\mathcal{V}$};
		\draw[thick] (0,2) -- (0,-2);
		
		\fill[pattern=checkerboard light gray, nearly opaque] (-2,2) -- (0,0) -- (2,0) -- (2,2) -- (-2,2);
		\draw[ultra thick, ->, blue] (0,0) -- (1.4,0) node[below right] {$\rho(D) = \epsilon_1 - \epsilon_2$};
		\draw[ultra thick, ->] (0,0) -- (-1,1) node[above left] {$\epsilon_2$};
		\node[blue] at (1,1) {$\mathcal{C}_X$};
	\end{tikzpicture}\end{center}
\end{example}

Non-simple spherical embeddings can be obtained by patching simple ones, which leads to a classification via strictly convex \emph{colored fans}\index{colored fan}.
\begin{definition}
	Given a spherical homogeneous $G$-space $X^+$, a colored fan is a collection of colored cones in $\mathcal{Q}$ that
	\begin{compactenum}[(i)]
		\item is stable under passing to faces (defined above), and
		\item the intersections with $\mathcal{V}$ of these cones form a fan inside $\mathcal{V}$.
	\end{compactenum}
	Call a colored fan strictly convex if all its cones are strictly convex.
\end{definition}

The bijection of Proposition \ref{prop:G-orbits} and the description of morphisms in Theorem \ref{prop:LV-morphism} generalize to this context. Such generality will not be needed except for the case of toroidal embeddings, which is to be discussed in the next section.

Attach a parabolic subgroup to $X^+$ as follows. Take $X^+ = X$ to be spherical. Set $P(X) := \left\{ g \in G: \mathring{X}g = \mathring{X} \right\}$\index{P(X)@$P(X)$}; clearly $P(X) \supset B$. We make the following choices:
\begin{itemize}
	\item a Levi factor $L(X) \subset P(X)$\index{L(X)@$L(X)$}, say by choosing a base point $x_0 \in \mathring{X}$ together with a $P(X)$-eigenfunction $f$ with zero locus $X \smallsetminus \mathring{X}$, as in \cite[\S 2.1]{SV17};
	\item a maximal torus $A \hookrightarrow B \cap L(X)$.
\end{itemize}
Set $H := \Stab_G(x_0)$ so that we deduce an isomorphism $X^+ \simeq H \backslash G$ from the choice of $x_0$. It is known that $L(X) \cap H \supset L(X)_\text{der}$. Together with the choices above, by \cite[\S 2.1]{SV17} we have an identification of $F$-tori
\begin{gather}\label{eqn:iden-tori}
	L(X)/L(X) \cap H = A/A \cap H \rightiso A_X;
\end{gather}
in particular, we may identify $A_X$ as a closed subvariety of $\mathring{X}$ via the orbit map $A_X \ni a \mapsto x_0 a$.

\section{Boundary degenerations}\label{sec:boundary}
Assume $\text{char}(\Bbbk) = 0$ and consider a spherical homogeneous $G$-space $X$.

\begin{notation}
	A spherical embedding $X \hookrightarrow \bar{X}$ is called \emph{toroidal}\index{spherical embedding!toroidal} if it is colorless, i.e.\ no divisor in $\mathcal{D}^B$ can contain a $G$-orbit in its closure inside $\bar{X}$. As in \cite{SV17}, the term \emph{complete toroidal compactification}\index{complete toroidal compactification} will mean a toroidal embedding with $X$ proper.

	Luna--Vust theory specializes to give the natural bijection
	\begin{align}\label{eqn:toroidal-classification}
		\left\{ \begin{array}{l}
			X \hookrightarrow \bar{X}: \\
			\text{toroidal embeddings}
		\end{array} \right\} & \stackrel{1:1}{\longleftrightarrow}
		\left\{
			\text{strictly convex fans in $\mathcal{V}$}
		\right\}, \\
		\left\{ \begin{array}{l}
			X \hookrightarrow \bar{X}: \\
			\text{toroidal compactifications}
		\end{array} \right\} & \stackrel{1:1}{\longleftrightarrow}
		\left\{ \begin{array}{l}
			\text{subdivisons of $\mathcal{V}$ by} \\
			\text{a strictly convex fan}
		\end{array} \right\}.
	\end{align}
	A strictly convex fan subdividing $\mathcal{V}$ means a strictly convex fan in $\mathcal{Q}$ whose union of cones equals $\mathcal{V}$.
\end{notation}

Smooth complete toroidal compactifications of $X$ always exist \cite[\S 29.2]{Ti11}, but there is no canonical choice unless when $\mathcal{V}$ is strictly convex. Below is a summary of some notions introduced in \cite[\S 2]{SV17}.

\begin{itemize}
	\item There is a finite crystallographic reflection group $W_X$, called the \textit{little Weyl group} of $X$, such that $\mathcal{V}$ is a fundamental domain for the $W_X$-action on $\mathcal{Q}$. There is a natural way to embed $W_X$ into the Weyl group $W$ associated to $G$ and $B$.
	\item Let $\Delta_X$ be the set of integral generators of the extremal rays of
	\[ \left\{\alpha \in \Lambda_{\Q} : \angles{\alpha, \mathcal{V}} \leq 0 \right\}; \]
	its elements are called the (normalized, simple) \emph{spherical roots}\index{Delta_X@$\Delta_X$}. Together with the $W_X$-action on $\mathcal{Q}$, they are known to form a based root system.
	\item Choose a smooth complete toroidal compactification $X \hookrightarrow \bar{X}$. By \cite[2.3.6]{SV17}, we have a notion of correspondence (non-bijective in general) between
		\begin{inparaenum}[(i)]
			\item $G$-orbits $Z \subset \bar{X}$ or their closures, and
			\item subsets $\Theta \subset \Delta_X$.
		\end{inparaenum}
		Specifically, to the complete toroidal compactification is associated a subdivision of $\mathcal{V}$ into a strictly convex fan by \eqref{eqn:toroidal-classification}; each cone therein intersects $X_*(A_X)$ in a free monoid. To each $Z$ is attached a face $\mathcal{F}_Z$ of that fan, by Proposition \ref{prop:G-orbits}. We stipulate that $Z \leftrightarrow \Theta$ if $\mathcal{F}_Z$ is orthogonal to $\Theta$, but not orthogonal to any subset $\Theta' \supsetneq \Theta$ of $\Delta_X$. For example, $\Theta = \Delta_X$ corresponds to $Z = \bar{X}$.

		For any orbit closure $Z \leftrightarrow \Theta$, the normal cone $N_Z \bar{X}$ is known to be spherical. Denote by $X_\Theta$ its open $G$-orbit: this homogeneous $G$-space is called the \emph{boundary degeneration}\index{boundary degeneration} of $X$ attached to $\Theta$. It is wavefront if $X$ is.
	\item Note that one $\Theta$ may correspond to several $Z$. As indicated in \cite[Proposition 2.5.3]{SV17}, the resulting boundary degenerations $X_\Theta$ are all isomorphic; they are also independent of the choice of $\bar{X}$.
	\item To each $\Theta$ is attached a face $\Theta^\perp \cap \mathcal{V}$ of $\mathcal{V}$, as well as the subtorus $A_\Theta := A_{X, \Theta} \subset A_X$\index{A_Theta@$A_\Theta$} characterized by $X_*(A_\Theta) = \Theta^\perp \cap X_*(A_X)$. By \cite[2.4.6]{SV17}, there is a canonical identification
		\[ A_\Theta \rightiso \mathcal{Z}(X_\Theta) := \Aut_G(X_\Theta)^0. \]
	\item Furthermore, there is an embedding $\Gm^{\Delta_X \smallsetminus \Theta} \hookrightarrow A_\Theta$, whose image in $\mathcal{Z}(X_\Theta)$ acts by dilating the fibers of $\text{pr}_Z: N_Z \bar{X} \to Z$; the GIT quotient of $X_\Theta$ by $\Gm^{\Delta_X \smallsetminus \Theta}$ is isomorphic to $Z$ via the projection morphism $\text{pr}_Z$.
\end{itemize}

\begin{remark}\label{rem:local-structure-thm}
	The basic tool for establishing these properties is the Local Structure Theorem due to Brion, Luna and Vust (see \cite[Theorem 2.3.4]{SV17}). Morally, its effect is to reduce things to the toric case $G = A$. For a general toroidal embedding $X \hookrightarrow \bar{X}$ in characteristic zero, it asserts that $\bar{X}$ is covered by $G$-translates of the $P(X)$-stable open subset
	\begin{gather}\label{eqn:local-structure-thm}
		\bar{X}_B := \bar{X} \smallsetminus \bigcup_{D \in \mathcal{D}^B} \bar{D} \leftiso Y \utimes{L(X)} P(X)  \simeq Y \times U_{P(X)}
	\end{gather}
	where
	\begin{compactitem}
		\item $\bar{D}$ stands for the closure in $\bar{X}$ of the prime divisor $D \subset X$,
		\item $Y$ is the closure of $A_X$ in $\bar{X}_B$ which is an $L(X)_\text{ab}$-variety, as a toric variety it corresponds to the fan in $\mathcal{V}$ attached to $X^+ \hookrightarrow \bar{X}$,
		\item $Y \utimes{L(X)} P(X)$ is a contracted product (cf.\ \eqref{eqn:parabolic-induction}) and the morphisms are the obvious ones.
	\end{compactitem}
	Furthermore, $Y$ meets every $G$-orbit in $\bar{X}$. Therefore $\bar{X}$ is smooth if and only if $Y$ is. In view of \eqref{eqn:local-structure-thm} and Theorem \ref{prop:LV-morphism}, the existence of smooth complete toroidal compactifications for $X$ is now clear: simply take any complete toroidal embedding into $\bar{X}$ and desingularize $\bar{X}$ by subdividing the corresponding fan in $\mathcal{V}$.
\end{remark}

\begin{remark}\label{rem:iden-B-orbits}
	The open $B$-orbits of $X$ and $X_\Theta$ can be identified as follows (see \cite[2.4.5]{SV17}):
	\begin{gather}\label{eqn:iden-B-orbits}
		\mathring{X} \simeq \mathring{X}_\Theta \quad \text{as $B$-varieties}.
	\end{gather}
	If we express $\mathring{X}$ as $A_{X^+} \times U_{P(X)}$ using the Local Structure Theorem, then \eqref{eqn:iden-B-orbits} is unique up to a $B$-isomorphism of the form $(y,u) \mapsto (y, y^{-1}vyu)$, for some $v \in U_{P(X)}$ fixed by $\Ker[L(X) \twoheadrightarrow A_X]$.
	
	The identification \eqref{eqn:iden-B-orbits} of $B$-varieties pins down the isomorphism $A_\Theta \rightiso \mathcal{Z}(X_\Theta)$ as follows: $a \in A_\Theta$ acts by $x_0 b \mapsto x_0 a b$, where $b \in B$. This is well-defined by \eqref{eqn:iden-tori}.
\end{remark}

The reader may convince himself of all those claims about $X_\Theta$ by checking the case of toric varieties. For example, the identification \eqref{eqn:iden-B-orbits} should become evident in the toric setting.

\begin{remark}\label{rem:Whittaker-induction}
	Non-complete smooth toroidal embeddings have to be used in certain situations, and the notion of boundary degenerations has to be modified accordingly. This is the case for ``Whittaker inductions'' dealt in \cite[\S 2.6]{SV17}.
\end{remark}

\section{Cartan decomposition}\label{sec:Cartan}
Consider now a local field $F$ of characteristic zero.

Let $G$ be a connected reductive $F$-group. We assume $G$ split in order to apply the results of \cite{SV17}, although some other cases will be discussed in Remark \ref{rem:symmetric-spaces}. As before, we fix a Borel subgroup $B \subset G$ with $A := B/U$. Let $X^+$ be a spherical homogeneous $G$-space defined over $F$, with the open $B$-orbit $\mathring{X}$. The first observation is the existence of $F$-rational points.

\begin{proposition}[{Cf.\ \cite[Theorem 1.5.1]{Sak08} and its proof}]\label{prop:existence-point}
	Under the assumptions above, we have $X^+(F) \neq \emptyset$. Here $F$ can be any field of characteristic zero.
\end{proposition}

Henceforth we fix $x_0 \in \mathring{X}(F)$ and some regular function on $X^+$ with vanishing locus $X^+ \smallsetminus \mathring{X}$. These choices define $P(X)$ and its Levi factor $L(X)$. To apply the Luna--Vust classification (Theorem \ref{prop:LV}) we have to work over an algebraic closure $\bar{F}$. However, since $\text{Gal}(\bar{F}/F)$ acts trivially on $\mathcal{Q}$, the Luna--Vust data not involving colors do not require the algebraic-closeness of $F$; cf.\ the remarks in \cite[\S 2.2]{Sak12}.

For the next definition, recall that in \S\ref{sec:review-spherical} we have defined a quotient torus $A_{X^+}$ of $A$ which may be regarded as a closed subvariety of $\mathring{X}$ via $A_{X^+} \ni a \mapsto x_0 a$; furthermore, $X^*(A_{X^+}) = \Lambda(X^+)$. All these constructions are defined over $F$.

\begin{definition}
	Let $\HC: A_{X^+}(F) \to \Hom(\Lambda(X^+), \Q) = \mathcal{Q}$ be the homomorphism characterized by \index{H@$\HC$}
	\[ \angles{\chi, \HC(a)} = -\log |\chi(a)|, \quad \chi \in \Lambda(X^+), \; a \in A_X(F). \]
	Here $\log = \log_q$ when $F$ is non-Archimedean with residual field $\mathbb{F}_q$. Note that $\HC$ is trivial on the maximal compact subgroup of $A_{X^+}(F)$, and its image equals $X_*(A_{X^+})$ when $F$ is non-Archimedean. Define \index{A_X(F)-plus@$A_{X^+}(F)^+$}
	\[ A_{X^+}(F)^+ := \HC^{-1}(\mathcal{V}). \]
	In the literature $-\HC$ is often used. Our convention here agrees with that of \cite[\S 1.7]{SV17}.
\end{definition}

\begin{theorem}\label{prop:Cartan}\index{Cartan decomposition}
	The \emph{Cartan decomposition} holds for $X^+(F)$, namely there exists a compact subset $K \subset G(F)$ such that
	\[ A_{X^+}(F)^+ K = X^+(F). \]
	In the non-Archimedean case, it is customary to enlarge $K$ so that $K$ is compact open.
\end{theorem}
Alternatively, there exist $x_0, \ldots, x_k \in \mathring{X}(F)$ such that $\bigcup_{i=0}^k x_i A(F)^+ K = X^+(F)$, where $A(F)^+$ is the inverse image of $A_{X^+}(F)^+$ in $A(F)$.
\begin{proof}
	It is established under the following circumstances:
	\begin{inparaenum}[(i)]
		\item $F$ is Archimedean and $X^+$ is a symmetric space \cite[Theorem 4.1]{Fl78}, in which case $K$ can be taken to be a maximal compact subgroup in good position relative to $A$;
		\item the general Archimedean case for ``real spherical spaces'' is addressed in \cite[Theorem 5.13]{KKSS15}, cf.\ the (2.8) of \textit{loc.\ cit.};
		\item $F$ is non-Archimedean and $X^+$ is any spherical homogeneous $G$-space \cite[Lemma 5.3.1]{SV17}.
	\end{inparaenum}
\MyQED\end{proof}

\begin{remark}\label{rem:symmetric-spaces}
	Recall that $X^+$ is a symmetric space means that there exists an involution $\sigma: G \to G$, with fixed subgroup $G^\sigma$, such that $H = \Stab_G(x_0)$ satisfies $(G^\sigma)^\circ \subset H \subset G^\sigma$. Symmetric spaces are wavefront spherical varieties. Cartan decomposition holds for symmetric spaces over a local field $F$ with $\text{char}(F) \neq 2$, and we do not need to assume $G$ split in that case. See \cite{BO07} for the non-Archimedean setting.
\end{remark}

Let us turn to the notion of boundaries in the Cartan decomposition.
\begin{itemize}
	\item Take a smooth complete toroidal compactification $X^+ \hookrightarrow \bar{X}$ as in \S\ref{sec:boundary}, which is automatically defined over $F$. For any subset $\Theta \subset \Delta_{X^+}$ we define \index{A_Theta(F)-plus@$A_\Theta(F)^{>0}$}
		\begin{align*}
			A_\Theta(F)^{>0} & := \left\{ a \in A_\Theta(F) : \forall \alpha \in \Delta_{X^+} \smallsetminus \Theta, \; \angles{-\alpha, \HC(a)} = \log |\alpha(a)| > 0 \right\} \\
			& = \HC^{-1} \left( \relint (\Theta^\perp \cap \mathcal{V}) \right).
		\end{align*}
		Then $A_\Theta(F)^{>0}$ is precisely the subset of elements in $A_\Theta(F)$ that steer every $x \in X_\Theta$ towards some $Z \leftrightarrow \Theta$. In short, $A_\Theta(F)^{>0}$ contracts the fibers of various $N_Z(\bar{X}) \to Z$. Cf.\ \cite[Lemma 2.4.9]{SV17} and Remark \ref{rem:iden-B-orbits}. The avatars $A_\Theta(F)^{\geq 0}$ are defined by using $\geq 0$ or by omitting $\relint$ in the formulas above.
	\item A similar description exists for an affine spherical embedding $X^+ \hookrightarrow X$, cf.\ \cite[p.624]{Sak12}. Here we have to choose a proper birational $G$-equivariant morphism $\hat{X} \to X$, such that $X^+ \hookrightarrow \hat{X}$ is toroidal. This is addressed combinatorially by applying Theorem \ref{prop:LV-morphism} with $\varphi := \identity: X^+ \to X^+$: a canonical choice is to take $\hat{X} = \hat{X}_\text{can}$ to be the ``decolorization'' of $X$, by taking the cone $\mathcal{C}_{\hat{X}} := \mathcal{C}_X \cap \mathcal{V}$ and $\mathcal{F}_{\hat{X}} := \emptyset$.

		Given a face $\mathcal{F}$ of the colored cone $(\mathcal{C}_X, \mathcal{F}_X)$, the subgroup $\mathcal{F} \cap X_*(A_{X^+})$ defines a subtorus $A_\mathcal{F} \subset A_{X^+}$. Set \index{A_F(F)-plus@$A_{\mathcal{F}}(F)^{>0}$} \index{A_F@$A_{\mathcal{F}}$}
		\[ A_\mathcal{F}(F)^{>0} := \HC^{-1}(\relint(\mathcal{F} \cap \mathcal{V})). \]
		Translation by elements of $A_\mathcal{F}(F)^{>0}$ steers elements in $X$ to the $G$-orbit closure corresponding to $\mathcal{F}$, cf.\ Proposition \ref{prop:G-orbits}. Indeed, this can be seen on the level of $\hat{X}$ by Remark \ref{rem:local-structure-thm}. Replacing $\relint(\mathcal{F} \cap \mathcal{V})$ by $\mathcal{F} \cap \mathcal{V}$ gives rise to $A_{\mathcal{F}}(F)^{\geq 0}$
\end{itemize}
Again, the reader is invited to check these properties in the case $G=A$. The rigorous argument goes by reducing to the toric setting via Remark \ref{rem:local-structure-thm}; in the case of $X$ one has to pass to the toroidal embedding $\hat{X}$.

Next, we explain a way of choosing the spherical embeddings $\hat{X} \twoheadrightarrow X \supset X^+$ and $X^+ \hookrightarrow \bar{X}$ compatibly.
\begin{itemize}
	\item Choose a fan $\mathfrak{F}$ subdividing $\mathcal{V}$, such that there exists a subfan $\mathfrak{F}'$ that subdivides $\mathcal{C}_X \cap \mathcal{V}$. Upon taking refinements, we may assume that both $\mathfrak{F}$ and $\mathfrak{F}'$ give rise to smooth toroidal embeddings of $X^+$.
	\item Let $X^+ \hookrightarrow \bar{X}$ and $X^+ \hookrightarrow \hat{X}$ be the toroidal embeddings determined by $\mathfrak{F}$ and $\mathfrak{F}'$, respectively. Notice that $\hat{X}$ is not necessarily equal to the $\hat{X}_\text{can}$ determined by $\mathcal{C}_{\hat{X}} := \mathcal{C}_X \cap \mathcal{V}$; nevertheless, Theorem \ref{prop:LV-morphism} connects them by proper birational equivariant morphisms $\hat{X} \to \hat{X}_\text{can} \to X$, whose composite we denote as $p$. On the other hand, Luna--Vust theory says that $\hat{X}$ embeds as a $G$-stable open subvariety of $\bar{X}$.
	\item All the morphisms above induce identity on $X^+$. The situation is summarized as
		\begin{equation}\label{eqn:hat-vs-bar} \begin{tikzcd}
			\bar{X} & \hat{X} \arrow[twoheadrightarrow]{d}[right]{p} \arrow[hookrightarrow]{l} \\
			X^+ \arrow[hookrightarrow]{u} \arrow[hookrightarrow]{r} & X
		\end{tikzcd}\end{equation}
		in which every $\hookrightarrow$ is an open immersion.
	\item As in Remark \ref{rem:local-structure-thm}, define the open $P(X)$-subvariety $\bar{X}_B := \bar{X} \smallsetminus \bigcup_{D \in \mathcal{D}^B} \bar{D}$, where $\bar{D}$ denotes the closure of $D$ in $\bar{X}$; here we are implicitly working over the algebraic closure $\bar{F}$. Let $\bar{Y}$ stand for the closure of $A_{X^+}$ in $\bar{X}_B$. Recall that it is a toric variety defined by the fan $\mathfrak{F}$ in $\mathcal{Q}$; thus $\bar{Y}$ is defined over $F$.
	\item Similarly, set $\hat{X}_B := \hat{X} \smallsetminus \bigcup_{D \in \mathcal{D}^B} \hat{D}$ with $\hat{D}$ standing for the closure in $\hat{X}$. We obtain a toric variety $\hat{Y}$ in the same manner, whose fan is $\mathfrak{F}' \subset \mathcal{Q}$. There is a natural morphism $\hat{Y} \hookrightarrow Y$ of toric varieties: they share the common open stratum $A_{X^+}$. Thus it makes sense to write $\hat{Y}(\mathfrak{o}_F) \hookrightarrow \bar{Y}(\mathfrak{o}_F)$.
\end{itemize}

\begin{proposition}\label{prop:comparison-Cartan}
	For non-Archimedean $F$, the compact open subset $K \subset G(F)$ in the Cartan decomposition $A_{X^+}(F)^+ K = X^+(F)$ can be taken such that
	\[ (\bar{Y}(\mathfrak{o}_F) \cap \hat{Y}(F)) K = \hat{X}(F) \]
\end{proposition}
\begin{proof}
	We shall plug $X^+ \hookrightarrow \bar{X}$ into the elegant proof of \cite[Lemma 5.3.1]{SV17}; this is justified since $\bar{X}$ is a \emph{wonderful compactification} in the sense of \cite[\S 2.3]{SV17}. In that proof, a compact open subset $K \subset G(F)$ is constructed such that
	\begin{gather}\label{eqn:Cartan-Z}
		(\bar{Y}(\mathfrak{o}_F) \cap Z(F))K = Z(F), \quad Z \subset \bar{X}: G\text{-orbit}.
	\end{gather}
	The construction proceeds by induction: one starts with the closed $G$-orbits. Taking $Z=X^+$ gives the Cartan decomposition since $\bar{Y}(\mathfrak{o}_F) \cap X^+ = \bar{Y}(\mathfrak{o}_F) \cap A_{X^+}(F) = \HC^{-1}(\mathcal{V})$, as the fan $\mathfrak{F}$ subdivides $\mathcal{V}$.

	Now we take the union over all $G$-orbits $Z \subset \hat{X}$ in \eqref{eqn:Cartan-Z}. It implies that $\hat{X}(F)$ equals
	\begin{align*}
		\bigcup_{Z \subset \hat{X}} (\bar{Y}(\mathfrak{o}_F) \cap Z(F)) K & = \left( \bigcup_{Z \subset \hat{X}} \bar{Y}(\mathfrak{o}_F) \cap Z(F) \right) K \\
		& = (\bar{Y}(\mathfrak{o}_F) \cap \hat{X}(F)) K.
	\end{align*}
	We contend that $\hat{Y} = \bar{Y} \cap \hat{X}$. Indeed, $\hat{X} \subset \bar{X}$ is open dense, so it is routine to verify (over the algebraic closure) that
	\begin{inparaenum}[(i)]
		\item $\bar{D} \cap \hat{X} = \hat{D}$ for each $D \in \mathcal{D}^B$, thus
		\item $\bar{X}_B \cap \hat{X} = \hat{X}_B$,
		\item therefore $\bar{Y} \cap \hat{X} = \bar{Y} \cap \hat{X}_B$ is an irreducible closed subvariety of $\hat{X}_B$ containing $A_{X^+}$, from which $\hat{Y} = \bar{Y} \cap \hat{X}$.
	\end{inparaenum}
	
	Conclusion: $\bar{Y}(\mathfrak{o}_F) \cap \hat{X}(F) = \bar{Y}(\mathfrak{o}_F) \cap \hat{Y}(F)$, and our assertion follows immediately.
\MyQED\end{proof}

\begin{corollary}\label{prop:compact-exhaustion}
	With previous notations, $\hat{X}(F)$ is a union of compact open subsets of the form $C_{\hat{X}} = C_{\hat{Y}} \cdot K$ where $C_{\hat{Y}} \subset \hat{Y}(F)$ satisfies
	\[ C_{\hat{Y}} \cap A_{X^+}(F) = \bigcup_{i=0}^m \HC^{-1}\left( v_i + \mathcal{C}_X \cap \mathcal{V} \right) \]
	for some $m$ and $v_0, \ldots, v_m \in \mathcal{V}$.
\end{corollary}
\begin{proof}
	First, $\hat{Y}(F)$ can be expressed a union of compact open subsets $C_{\hat{Y}}$ whose intersection with $A_{X^+}(F)$ takes the form above, except that $v_0, \ldots, v_m$ are taken from $\mathcal{Q}$. This is an immediate consequence of the topology on $\hat{Y}(F)$ and of the toric structure on $\hat{Y}$; cf.\ \cite[\S I.1]{AMRT10} for the real case.
	
	To finish the proof, it suffices to intersect the compact open subsets constructed above with $\bar{Y}(\mathfrak{o}_F)$. Since $\bar{Y}(\mathfrak{o}_F) \cap A_{X^+}(F) = \HC^{-1}(\mathcal{V})$, the effect is to limit $v_0, \ldots, v_m$ to $\mathcal{V}$.
\MyQED\end{proof}

An easy yet useful observation that
\[ C_{\hat{X}} \cap X^+(F) = (C_{\hat{Y}} \cap A_{X^+}(F)) K = \bigcup_{i=0}^m \HC^{-1}\left( v_i + \mathcal{C}_X \cap \mathcal{V} \right) K \]
with the notation of the proof.

\section{Geometric data}\label{sec:geometric-data}
Let $F$ be a local field of characteristic zero. Fix a split connected reductive $F$-group $G$. We may still define the group of $F$-rational characters $X^*(G)$ and $X^*(G)_{\Q} := X^*(G) \otimes_{\Z} \Q$, $X^*(G)_{\R} := X^*(G) \otimes_{\Z} \R$, etc. They coincide with those defined over $\bar{F}$ since $G_\text{ab}$ is split.

Let $X^+ \hookrightarrow X$ be an affine spherical embedding. Define \index{Lambda@$\Lambda$}
\begin{align*}
	F(X^+)^{(G)}_\lambda & := \left\{f \in F(X^+) : \forall g \in G, \; f(\bullet g) = \lambda(g)f(\bullet)  \right\}, \quad \lambda \in X^*(G), \\
	\Lambda & := \left\{ \lambda \in X^*(G) : F(X^+)^{(G)}_\lambda \neq (0) \right\} \quad \text{(a subgroup of $X^*(G)$)}, \\
	F(X^+)^{(G)} & := \bigcup_{\lambda \in \Lambda} F(X^+)^{(G)}_\lambda, \\
	r & := \rank(\Lambda).
\end{align*}
Note that $\dim_F F(X^+)^{(G)}_\lambda = 1$ for all $\lambda \in \Lambda$, and $F(X^+)^{(G)}_0 = F(X^+)^G = F$. For $f \in F(X^+)^{(G)}$, $f \neq 0$, the ``eigencharacter'' $\lambda$ for which $f \in F(X^+)^{(G)}_\lambda$ is uniquely determined. These are birational invariants and depend only on the homogeneous $G$-space $X^+$.

The affine embedding enters into the picture via the objects \index{Lambda_X@$\Lambda_X$}
\begin{align*}
	F[X]^{(G)}_\lambda & := F(X^+)^{(G)}_\lambda \cap F[X], \quad \lambda \in \Lambda, \\
	\Lambda_X & := \left\{ \lambda \in \Lambda : F[X]^{(G)}_\lambda \neq (0) \right\}.
\end{align*}

\begin{lemma}\label{prop:Lambda-generation}
	The group $\Lambda$ is generated by its submonoid $\Lambda_X$.
\end{lemma}
\begin{proof}
	Form the GIT quotient $X\sslash G_\text{der}$, which is an affine $G_\text{ab}$-variety, and note that
	\begin{align*}
		X^*(G) &= X^*(G_\text{ab}), \\
		F(X^+)^{(G)}_\lambda & = F(X)^{(G)}_\lambda = F(X\sslash G_\text{der})^{(G_\text{ab})}_\lambda, \\
		F[X]^{(G)}_\lambda & = F[X\sslash G_\text{der}]^{(G_\text{ab})}_\lambda.
	\end{align*}
	The required result now follows from \cite[Lemma D.7]{Ti11} since $G_\text{ab}$ is an $F$-torus.
\MyQED\end{proof}

Set $\Lambda_{\Q} := \Lambda \otimes {\Q}$, $\Lambda_{\R} := \Lambda \otimes {\R}$, etc. Therefore $\Lambda_X$ generates a cone $\Lambda_{X,\Q} = \Q_{>0} \Lambda_X$ inside $\Lambda_{\Q}$.

\begin{definition}\label{def:relative-invariant} \index{relative invariant}
	Call nonzero $f \in F[X]$ a \emph{relative invariant} with eigencharacter $\lambda \in \Lambda_X$ if $f \in F[X]^{(G)}_\lambda$. In this case we write $f \leftrightarrow \lambda$. The Weil divisor $\divisor(f)$ on $X$ is then $G$-invariant.
\end{definition}

\begin{axiom}\label{axiom:geometric}
	Hereafter, we shall assume that the affine spherical embedding $X^+ \hookrightarrow X$ satisfies
	\begin{enumerate}[(G1)]
		\item $\Lambda_{X,\Q}$ is a simplicial cone inside $\Lambda_{\Q}$; consequently, the minimal integral generators $\omega_1, \ldots, \omega_r \in \Lambda_X$ of the extremal rays of $\Lambda_{X,\Q}$ form a basis of $\Lambda_{\Q}$;
		\item Choose relative invariants $f_1, \ldots, f_r \in F[X]$ such that $f_i \leftrightarrow \omega_i$ (each is unique up to $F^\times$); assume that
			\[ X \smallsetminus X^+ = \bigcup_{i=1}^r \Supp(\divisor(f_i)). \]
	\end{enumerate}
	In practice, the vanishing loci of $f_i$ are often irreducible or even geometrically irreducible.
\end{axiom}

The lattice $\Lambda$ can be described in terms of the stabilizer of a chosen base point as follows.
\begin{lemma}[Cf.\ {\cite[Proposition 2.11]{Ki03}}]\label{prop:char-Lambda}
	Fix $x_0 \in \mathring{X}^+(F)$ and let $H := \Stab_G(x_0)$. Then $\Lambda = \{\omega \in X^*(G) = X^*(G_\mathrm{ab}) : \omega|_H = 1 \}$.
\end{lemma}
\begin{proof}
	Suppose $\omega \in \Lambda$ and select a relative invariant $f \leftrightarrow \omega$. Considerations of $H \ni h \mapsto f(x_0 h)$ show that $\omega|_H = 1$. Conversely, given $\omega \in X^*(G)$ such that $\omega|_H = 1$, we define a regular function $f$ on $H \backslash G$ by $f(x_0 g) = \omega(g)$. Since $\text{char}(F)=0$, we have $H \backslash G \rightiso X^+$ via $g \mapsto x_0 g$, therefore $f$ affords the required relative invariant.
\MyQED\end{proof}

\begin{notation}
	Every $\lambda \in \Lambda_{\CC}$ can be written uniquely as $\lambda = \sum_{i=1}^r \lambda_i \omega_i$ with $\lambda_i \in \CC$. We write
	\begin{itemize}
		\item $\Re(\lambda) \relgeq{X} 0$ if $\Re(\lambda_i) \geq 0$ for all $i$;
		\item $\Re(\lambda) \relgg{X} 0$ if $\Re(\lambda_i) \gg 0$ for all $i$;
		\item $|\omega|^\lambda := \prod_{i=1}^r |\omega_i|^{\lambda_i}$, an unramified character of $G(F)$; \index{omega-lambda@$\lvert \omega\rvert^\lambda$}
		\item $|f|^\lambda := \prod_{i=1}^r |f_i|^{\lambda_i}$, a nowhere vanishing $C^\infty$ function on $X^+(F)$. \index{f-lambda@$\lvert f\rvert^\lambda$}
	\end{itemize}
	Note that $|f|^\lambda$ depends on the choice of $f_1, \ldots, f_r$, and $|f|^\lambda(\bullet g) = |f|^\lambda(\bullet) |\omega|^\lambda(g)$ for every $g \in G(F)$.
\end{notation}

To $X^+ \hookrightarrow X$ is associated the colored cone $(\mathcal{C}_X, \mathcal{F}_X)$, $\mathcal{C}_X \subset \mathcal{Q}$; one may pass to $\bar{F}$ if necessary. Every $\omega \in \Lambda \subset X^*(G)$ induces an element $\chi_\omega \in \Lambda(X^+) \subset \mathcal{Q}^\vee$ by taking the class of $f \leftrightarrow \omega$, since $f \in F(X^+)^{(G)} \subset F(X^+)^{(B)}$. This is compatible with the restriction map $X^*(G) \to X^*(B)$.

The next result will be useful in \S\ref{sec:convergence}.
\begin{lemma}\label{prop:Lambda-positivity}
	Assume $\mathcal{C}_X \neq \{0\}$. For any $\omega \in \Lambda_X$ we have $\chi_\omega|_{\mathcal{C}_X} \geq 0$. Moreover,
	\begin{align*}
		\omega \neq 0 & \implies \angles{\chi_\omega, \rho} >0, \quad \forall  \rho \in \relint(\mathcal{C}_X), \\
		\omega \in \relint(\Lambda_{X,\Q}) & \implies \angles{\chi_\omega, \rho} > 0, \quad \forall \rho \in \mathcal{C}_X \cap \mathcal{V},\; \rho \neq 0.
	\end{align*}
\end{lemma}
\begin{proof}
	Let $f \leftrightarrow \omega$ be a relative invariant. Since $f \in F[X]$, by definition we have $\angles{\chi_\omega, \rho(D)} = v_D(f) \geq 0$ for any prime divisor $D$ on $X \times_F \bar{F}$, where $v_D$ stands for the normalized valuation attached to $D$. Taking $D$ from $\mathcal{F}_X \cup \mathcal{V}_X$ yields the first assertion.

	Assume $r = \rank(\Lambda) \geq 1$ in what follows. For the second assertion, notice that $\divisor(f) \neq 0$ implies $\angles{\chi_\omega, w} = w(f) > 0$ for some $w \in \mathcal{V}_X$. On the other hand, by the description of $\mathcal{V}_X$ in Theorem \ref{prop:LV}, we may express $\rho \in \relint(\mathcal{C}_X)$ as
	\[ \rho = \sum_{v \in \mathcal{V}_X} a_v v + \sum_{D \in \mathcal{F}_X} a_D \rho(D), \quad a_D \geq 0, \; a_v > 0. \]
	Consequently we have $\rho = \epsilon w + \rho'$ with $\epsilon > 0$ sufficiently small so that $\rho' \in \mathcal{C}_X$. Now the positivity of $\angles{\chi_\omega, \rho}$ follows from the first part.
	
	To prove the last assertion, write $\omega = \sum_{i=1}^r \lambda_i \omega_i$ with $\lambda_i > 0$ for all $i$. We may even assume $\lambda_i \in \Z$ so that $\Supp(\divisor(f)) = \partial X$ by Axiom \ref{axiom:geometric}. By Proposition \ref{prop:G-orbits}, every nonzero $\rho \in \mathcal{C}_X \cap \mathcal{V}$ (as a valuation) has center equal to some $G$-orbit closure $Z \subsetneq X$. This means that the valuation ring $\mathcal{O}_\rho := \{h \in F(X) : \rho(h) \geq 0\}$ dominates the local ring $\mathcal{O}_{X,Z}$ of $Z$; as $f$ lies in the maximal ideal $\mathfrak{m}_{X,Z}$, we see $\angles{\chi_\omega, \rho} = \rho(f) > 0$ as required.
\MyQED\end{proof}
The reader is invited to check this in the situation Example \ref{eg:typical-cone}, in which case $\Lambda_X$ is generated by $\epsilon_1^* + \epsilon_2^*$.

\chapter{Analytic background}\label{sec:analytic-bg}
\section{Integration of densities}\label{sec:integration-density}
Let $F$ be a local field of characteristic zero and $G$ be an affine $F$-group. Consider a smooth $G$-variety $Y$ over $F$, so that $Y(F)$ becomes an $F$-analytic manifold with right $G(F)$-action. The constructions below can be performed for a broader class of geometric objects, such as the semi-algebraic $F$-manifolds, but we shall restrict to the smooth algebraic case.

Fix a Haar measure on $F$ (eg.\ by fixing a nontrivial unitary additive character), which induces a Haar measure on $F^n$ for every $n$. We begin by clarifying the meaning of integration over $Y(F)$, using the language of densities as illustrated in \cite[p.29]{BGV04}.

\begin{definition}\label{def:density}\index{density}
	There is a canonical line bundle $\mathscr{L}$ over $Y(F)$, called the \emph{density bundle}, with the following structures.
	\begin{itemize}
		\item $\mathscr{L}$ is obtained from an $\R_{>0}$-torsor on $Y(F)$, therefore it makes sense to talk about its real or positive sections.
		\item $\mathscr{L}$ is $G(F)$-equivariant, i.e.\ equipped with a morphism $\text{pr}_1^* \mathscr{L} \to a^* \mathscr{L}$ where $a$ (resp.\  $\text{pr}_1$) is the action (resp.\  projection) morphism $Y(F) \times G(F) \to Y(F)$, subject to the usual constraints. Groups act on the right of bundles, therefore act on the left of their spaces of sections.
		\item For any open set $\mathcal{U} \subset Y(F)$, there is a canonical linear functional (the ``integration'')
		\[ \int_{\mathcal{U}}: C_c(\mathcal{U}, \mathscr{L}) \longrightarrow \CC \]
		where $C_c(\mathcal{U}, \cdot)$ stands for the space of compactly-supported continuous sections, such that
		\begin{compactitem}
			\item if $a$ is a real (resp.\  positive) section, then $\int_{\mathcal{U}} a$ is real (resp.\  positive when $\mathcal{U} \neq \emptyset$);
			\item $\int_{\mathcal{U}} ga = \int_{\mathcal{U}g} a$ for all $g \in G(F)$;
			\item $\int_{\mathcal{U} \sqcup \mathcal{U}'} a = \int_{\mathcal{U}} a|_{\mathcal{U}} + \int_{\mathcal{U}'} a|_{\mathcal{U}'}$.
		\end{compactitem}
		The integration functional depends on the choice of a measure on $F$. Cf.\ the construction below.
	\end{itemize}
	In particular, one has the notion of integrable or $L^1$ sections of $\mathscr{L}$. The group $G(F)$ acts on the left of the $L^1$-space.
\end{definition}

The density bundle can be constructed concretely as follows.
\begin{enumerate}
	\item Let $\mathcal{G} := \underline{\text{Isom}}(\mathcal{O}_Y, \topwedge \Omega_{Y/F})$ be the $\Gm$-torsor over $Y$ corresponding to the line bundle $\topwedge \Omega_{Y/F}$. We pass to the $F$-analytic topos by using $|\cdot|: \Gm(F) = F^\times \to \R_{>0}$, therefore obtain an $\R_{>0}$-torsor $|\mathcal{G}|$ on $Y(F)$. The density bundle is simply
	\[ \mathscr{L} := |\mathcal{G}| \utimes{\R_{>0}} \CC. \]
	\item Over a small open subset $\mathcal{U} \subset Y(F)$, the sections of $\mathscr{L}$ can be written as $\eta = \varphi |\omega|$ where $\varphi$ is a $C^\infty$ function on $\mathcal{U}$, and $\omega$ is a nowhere vanishing section of $\topwedge \Omega_{Y/F}$ over some Zariski open subset $U$ with $U(F) \supset \mathcal{U}$. The choice of $|\omega|$ trivializes $\mathscr{L}|_{\mathcal{U}}$, and our notation ``explains'' the transition law
		\[ \varphi|\omega| = \varphi \left| \frac{\omega}{\omega'} \right| \cdot |\omega'| \]
		where $\omega/\omega' \in \Gamma(U, \mathcal{O}_Y^\times)$.
	\item Locally around each point $y \in Y(F)$, there exists an $F$-analytic chart $\simeq F^n$. To define the integration functional, it suffices to define $\int_{\mathcal{U}} \eta$ for $\mathcal{U}$ open in $F^n$, by the usual procedure, cf.\ \cite[\S 2.2.1]{Weil82}. Express a section $\eta$ of $\mathscr{L}$ as
	\[ \eta = \varphi |\dd x_1 \wedge \cdots \wedge \dd x_n| \]
	and set $\int_{\mathcal{U}} \eta := \int_{\mathcal{U}} \varphi \dd x_1 \cdots \dd x_n$ with respect to the chosen measure on $F^n$. The formula of change of variables asserts that $\int_{\mathcal{U}}$ is well-defined; furthermore, it also implies that $\int_{\mathcal{U}}$ is invariant under all automorphisms. There exists a $G(F)$-invariant Radon measure on $Y(F)$ if and only if $\mathscr{L}$ is trivial as an equivariant line bundle.
	\item More generally, one can define the $\R_{>0}$-torsor $|\mathcal{G}|^t$ using $|\cdot|^t: F^\times \to \R_{>0}$, for every $t \in \R$, thereby obtains the line bundle of $t$-densities $\mathscr{L}^t$ on $Y(F)$. Observe that\index{L-s-density@$\mathscr{L}^s$}
		\begin{compactitem}
			\item $\mathscr{L}^1 = \mathscr{L}$ and $\mathscr{L}^0$ is identified with the trivial line bundle $\CC$;
			\item for any $t,t' \in \R$, there is a canonical isomorphism
				\begin{gather}\label{eqn:density-mult}
					\mathscr{L}^t \otimes \mathscr{L}^{t'} \rightiso \mathscr{L}^{t+t'}
				\end{gather}
				satisfying the usual constraints.
		\end{compactitem}
		The density bundles inherit their equivariant structures from that of $\Omega_{Y/F}$.
	\item Let us explicate the case of a homogeneous space $Y = H \backslash G$. It suffices to describe $\mathscr{L}$ on each $G(F)$-orbit in $(H \backslash G)(F)$, say the one containing the coset $x_0 := H(F) \cdot 1$. The fiber of $\topwedge \Omega_{Y / F}$ at $x_0$ is
		\[ ( \topwedge\mathfrak{g}^\vee ) \otimes ( \topwedge \mathfrak{h}^\vee )^{-1}. \]
		Thus $H$ acts on that fiber via $(\det \Ad_G\big|_H)^{-1} \det\Ad_H: H \to \Gm$. Upon taking $|\cdot|$ on $F$-points, one sees that $\mathscr{L}^t\big|_{x_0 G(F)}$ corresponds to the induced representation
		\begin{gather}\label{eqn:induced-density}
			\Ind^G_H \left( \delta_H \cdot \delta_G^{-1}|_H \right)^t, \quad t \in \R.
		\end{gather}
		We recover the familiar criterion of the existence of invariant Radon measures ($t=1$) on homogeneous spaces.
\end{enumerate}

Consider a complex vector bundle $\mathscr{E}$ on $Y(F)$ (always assumed to be of finite rank) equipped with a positive definite hermitian pairing\index{E@$\mathscr{E}$}
\begin{gather}\label{eqn:hermitian-vb}
	\mathscr{E} \otimes \overline{\mathscr{E}} \to \mathscr{L}
\end{gather}
on the fibers. Define the separable Hilbert space $L^2(Y(F), \mathscr{E})$ as the completion of $C_c(Y(F), \mathscr{E})$ with respect to the scalar product
\[ (\xi|\xi') := \int_{Y(F)} \underbracket{\xi \overline{\xi'}}_{\mathscr{L}\text{-valued}}. \]
When $\mathscr{E}$ is $G(F)$-equivariant and the pairing \eqref{eqn:hermitian-vb} is invariant, $G(F)$ acts on $L^2(Y(F), \mathscr{E})$ via isometries.

\begin{remark}\label{rem:half-densities}\index{density!half}
	Vector bundles with a pairing \eqref{eqn:hermitian-vb} can be constructed as follows. Let $\mathscr{E}_0$ be an equivariant hermitian vector bundle on $Y(F)$. Since $\mathscr{L}^{1/2} \otimes \mathscr{L}^{1/2} \rightiso \mathscr{L}$ is defined over $\R$, the vector bundle
	\[ \mathscr{E} := \mathscr{E}_0 \otimes \mathscr{L}^{1/2} \]
	is equivariant and equipped with a pairing as in \eqref{eqn:hermitian-vb}. Hence we get a canonical unitary representation of $G(F)$ on $L^2(Y(F), \mathscr{E})$, in view of \eqref{eqn:density-mult}. They typical case is $\mathscr{E}_0 = \CC$, so that $\mathscr{E} = \mathscr{L}^{1/2}$.
	
	Such an approach has already been adopted in \cite[3.7]{Be88}.
\end{remark}

\section{Direct integrals and \texorpdfstring{$L^2$}{L2}-spectral decomposition}\label{sec:direct-integrals}
What follows is fairly standard: the details can be found in \cite[II.1, II.2]{Di96} or \cite[Chapter 14]{Wall92}, for example.

Let $\Pi$ be a topological space with a positive Borel measure $\mu$. Consider a family (also known as a ``field'') of separable Hilbert spaces $\mathcal{H}: \tau \mapsto \mathcal{H}_\tau$, where $\tau \in \Pi$. Assume $(\mathcal{H}_\tau)_\tau$ is \emph{measurable}, which means that we are given $\CC$-vector space $\mathfrak{S}$ of sections $s: \tau \mapsto s(\tau) \in \mathcal{H}_\tau$, verifying
\begin{compactenum}[(i)]
	\item for all $s \in \mathfrak{S}$, the function $\tau \mapsto \|s(\tau)\|$ is measurable on $\Pi$;
	\item if $t$ is a section such that $\tau \mapsto (t(\tau)|s(\tau))$ is measurable for all $s \in \mathfrak{S}$, then $t \in \mathfrak{S}$;
	\item there exists a sequence $s_1, s_2, \ldots$ in $\mathfrak{S}$ such that for every $\tau \in \Pi$, the subset $\{s_i(\tau)\}_{i \geq 1}$ is dense in $\mathfrak{H}_\tau$; call it a \emph{fundamental sequence}.
\end{compactenum}
Sections belonging to $\mathfrak{S}$ are called measurable. Consider two basic examples.
\begin{compactitem}
	\item The constant family $\mathcal{H}_\tau = \mathcal{H}_0$ for some fixed $\mathcal{H}_0$: here $\mathfrak{S}$ is taken to be the space of $s$ such that $\tau \mapsto (s(\tau)|v)$ is measurable for every $v \in \mathcal{H}$.
	\item Let $(\mathcal{H}_\tau)_\tau$ and $(\mathcal{H}'_\tau)_{\tau'}$ be measurable families. The family of completed tensor products $(\mathcal{H}_\tau \hat{\otimes} \mathcal{H}'_\tau)_\tau$ of Hilbert spaces is measurable by requiring that $s \otimes s'$ is measurable whenever $s$ and $s'$ are.
\end{compactitem}

For every $s \in \mathfrak{S}$, put
\[ \|s\|^2 := \int_\Pi \|s(\tau)\|^2 \dd\mu(\tau). \]
Elements $s \in \mathfrak{S}$ with $\|s\| < +\infty$ are called $L^2$ sections; they form a $\CC$-vector space. Define
\[ \int^\oplus_\Pi \mathcal{H}_\tau \dd\mu(\tau) := \left\{ s \in \mathfrak{S} : L^2\text{-section} \right\} \big/ \{ s \in \mathfrak{S}: \|s\| = 0\}. \]
It turns out that $\int^\oplus_\Pi \mathcal{H}_\tau \dd\mu(\tau)$ equipped with $\|\cdot\|^2$ becomes a separable Hilbert space, called the \emph{direct integral}\index{direct integral} of the measurable family $(\mathcal{H}_\tau)_\tau$.

Let $\mathcal{H} = (\mathcal{H}_\tau)_\tau$, $\mathcal{H}' = (\mathcal{H}'_\tau)_\tau$ be two measurable families of separable Hilbert spaces over $(\Pi, \mu)$. Consider a family $\mathcal{F}$ of continuous linear maps from $\mathcal{H}$ to $\mathcal{H}'$, that is,
\[ \mathcal{F}_\tau \in \Hom(\mathcal{H}_\tau, \mathcal{H}'_\tau), \quad \tau \in \Pi. \]
Call $\mathcal{F}$ measurable if it preserves measurable sections. If the operator norms $\tau \mapsto \|\mathcal{F}_\tau\|$ form an $L^\infty$ function on $\Pi$, then we obtain a continuous linear map
\[ \int^\oplus_{\Pi} \mathcal{F}_\tau \dd\mu(\tau): \int^\oplus_\Pi \mathcal{H}_\tau \dd\mu(\tau) \to \int^\oplus_\Pi \mathcal{H}'_\tau \dd\mu(\tau) \]
of norm equal to the $L^\infty$-norm of $\tau \mapsto \|\mathcal{F}_\tau \|$. Here comes more definitions.

\begin{itemize}
	\item Operators of the form $\int^\oplus_{\Pi} \mathcal{F}_\tau \dd\mu(\tau)$ above are called \emph{decomposable}.
	\item There is a continuous injection from $L^\infty(\Pi,\mu)$ to the space of decomposable operators, namely by sending $f$ to the measurable family $\mathcal{F}_f: \tau \mapsto f(\tau)\cdot \identity$. Operators arising in this manner are called \emph{diagonalizable}.
\end{itemize}

Let $\Gamma$ be a separable locally compact group which is CCR (see \cite[14.6.9]{Wall92}). Denote by $\Pi_{\mathrm{unit}}(\Gamma)$ the unitary dual of $\Gamma$ equipped with the Fell topology. We pick a representative for each equivalence class of unitary representations in $\Pi_{\mathrm{unit}}(\Gamma)$.

\begin{theorem}[Abstract Plancherel decomposition]\label{prop:direct-integral}
	Let $L$ be a unitary representation of $\Gamma$ on separable Hilbert spaces. There exists an isomorphism of unitary representations of $\Gamma$
	\[ \Phi: L \rightiso \int^\oplus_{\Pi_{\mathrm{unit}}(\Gamma)} \tau \hat{\otimes} \mathcal{M}_\tau \dd\mu(\tau) \]
	for some positive Borel measure $\mu$ on $\Pi_{\mathrm{unit}}(\Gamma)$, called the Plancherel measure\index{Plancherel measure}, and a measurable family of separable Hilbert spaces $\tau \mapsto \mathcal{M}_\tau$, called the multiplicity spaces, on which $G(F)$ acts trivially.   

	The data $\mu$, $(\mathcal{M}_\tau)_\tau$ and $\Phi$ are unique in the following sense. If $\Phi': L \rightiso \int^\oplus_{\Pi_{\mathrm{unit}}(\Gamma)} \tau \hat{\otimes} \mathcal{M}'_\tau \dd\mu'(\tau)$ is another decomposition, then:
	\begin{compactitem}
		\item the measures $\mu$, $\mu'$ are equivalent: thus there exists a measurable $a: \Pi_{\mathrm{unit}}(\Gamma) \to \R_{>0}$ (the Radon--Nikodym derivative) such that $\dd\mu' = a\dd\mu$;
		\item there exists a measurable family of continuous linear maps $H(\tau): \tau \hat{\otimes} \mathcal{M}_\tau \rightiso \tau \hat{\otimes} \mathcal{M}'_\tau$ such that $a(\tau)H(\tau)$ is an isometry for $\mu$-almost all $\tau$, and
			\[ \Phi' \Phi^{-1} = \int^\oplus_{\Pi_{\mathrm{unit}}(\Gamma)} H(\tau) \dd\mu(\tau). \]
	\end{compactitem}
\end{theorem}
\begin{proof}
	For the existence of direct integral decompositions, see \cite[14.10.* and 14.13.8]{Wall92}.

	As for the uniqueness, let $\mathcal{F} := \Phi' \Phi^{-1}$. In view of \cite[II.6.3 Théorème 4]{Di96}, it suffices to show that $\mathcal{F}$ induces an isomorphism of the $C^*$-algebras of diagonalizable operators.	Denote by $C^*(\Gamma)$ the $C^*$-algebra attached to $\Gamma$ \cite[p.274]{Wall92} and let $\mathcal{A}$ be its strong closure in $\End(L)$. By \cite[p.330]{Wall92}, the center $Z(\mathcal{A})$ gets identified with the algebra of diagonalizable operators under $\Phi$. The same holds true under $\Phi'$. Since $\mathcal{F}$ transports these structures, it preserves the algebras of diagonalizable operators.
\MyQED\end{proof}

\begin{corollary}[Disintegration of intertwining operators]\label{prop:disintegration-F}
	Let $L$, $L'$ be unitary representations of $\Gamma$ as before, and $\mathcal{F}: L \rightiso L'$ is an isomorphism of unitary representations of $\Gamma$, then:
	\begin{itemize}
		\item There exist a Borel measure $\mu$ on $\Pi_{\mathrm{unit}}(\Gamma)$, measurable families of multiplicity spaces $\mathcal{M}_\tau$, $\mathcal{M}'_\tau$, together with isomorphisms
		\[
			L \rightiso \int^\oplus \tau \hat{\otimes} \mathcal{M}_\tau \dd\mu(\tau), \quad L' \rightiso \int^\oplus \tau \hat{\otimes} \mathcal{M}'_\tau \dd\mu(\tau);
		\]
		\item Given the decompositions above, there exists a measurable family of isometries $\eta(\tau): \mathcal{M}_\tau \rightiso \mathcal{M}'_\tau$ such that $\mathcal{F}$ corresponds to $\int^\oplus (\identity_\tau \otimes \eta(\tau)) \dd\mu(\tau)$, provided that $\dim_{\CC} \mathcal{M}_\tau$ is finite for $\mu$-almost all $\tau$.
	\end{itemize}
\end{corollary}
\begin{proof}
	Apply the uniqueness part of Theorem \ref{prop:direct-integral}, and note that $H(\tau)$ must be of the form $\identity_\tau \otimes \eta(\tau)$ for some $\eta(\tau)$ whenever $\dim_{\CC} \mathcal{M}_\tau$ is finite.
\MyQED\end{proof}

We shall apply these results to a $G(F)$-equivariant vector bundle $\mathscr{E}$ on $Y(F)$ together with a pairing \eqref{eqn:hermitian-vb}, thereby making $L^2(Y(F), \mathscr{E})$ into a unitary representation of $G(F)$, where $F$, $Y$, $G$ are as in \S\ref{sec:integration-density} and $G$ is reductive. In this case $G(F)$ is known to be CCR.

\section{Gelfand--Kostyuchenko method}\label{sec:GK-method}
The main reference here is \cite{Be88}. We consider
\begin{itemize}
	\item a separable Hilbert space $L$ which equals a direct integral $\int^\oplus_\Pi \mathcal{H}_\tau \dd\mu(\tau)$ of Hilbert spaces;
	\item a separable topological vector space $\Schw$, with continuous dual $\Schw^\vee$ endowed with the strong topology.
\end{itemize}

\begin{definition}\label{def:pointwise}
	A continuous linear map $\alpha: \Schw \to L$ is said to be \emph{pointwise defined}\index{pointwise defined} if there exists a family
	\[ \alpha_\tau: \Schw \to \mathcal{H}_\tau, \quad \tau \in \Pi \]
	such that for all $\xi \in \Schw$, the section $\tau \mapsto \alpha_\tau(\xi)$ is measurable and represents $\alpha(\xi) \in L$.
\end{definition}
Evidently, one can neglect those $\alpha_\tau$ for $\tau$ in a set of $\mu$-measure zero. We collect some more facts below, cf.\ \cite[1.3 Lemma]{Be88}.
\begin{enumerate}
	\item Any two families $(\alpha_\tau)_\tau$, $(\alpha'_\tau)_\tau$ coincide off a subset of $\mu$-measure zero.
	\item Let $\Gamma$ be a locally compact separable group acting on $\Schw$, $\mathcal{H}_\tau$ (for each $\tau \in \Pi$) and thus on $L$ by isometries, such that $\alpha$ is $\Gamma$-equivariant, then one can choose $\alpha_\tau$ to be $\Gamma$-equivariant for all $\tau$.
	\item If $\alpha$ has dense image, then $\alpha_\tau$ has dense image in $\mathcal{H}_\tau$ for $\mu$-almost all $\tau \in \Pi$.
	\item Given continuous linear maps $\varphi: \Schw \to \Schw'$ and $\alpha': \Schw' \to L$. If $\alpha'$ is pointwise defined, then so is $\alpha := \alpha' \varphi: \Schw \to L$;
\end{enumerate}

Denote by $\Schw^\dagger$ the conjugate of $\Schw^\vee$. If $\alpha$ if injective of dense image, taking hermitian adjoint yields
\[ \Schw \hookrightarrow L \hookrightarrow \Schw^\dagger. \]
This is naturally connected to the idea of \emph{rigged Hilbert spaces} or \emph{Gelfand triples} \cite{GV4} when $\alpha$ is injective, pointwise defined and of dense image. In fact, let $\alpha_\tau^\dagger: \mathcal{H}_\tau \to \Schw^\dagger$ be the hermitian adjoint of $\alpha_\tau$ for $\mu$-almost all $\tau$. In \cite[p.667 (**)]{Be88} it is shown that
\begin{gather}\label{eqn:generalized-expansion}
	\xi = \int_{\tau \in \Pi} \alpha_\tau^\dagger \alpha_\tau (\xi) \dd\mu(\tau), \quad \xi \in \Schw \subset \Schw^\dagger.
\end{gather}
This may be viewed as an expansion for $L$ into generalized eigenvectors, with the help of ``test vectors'' from $\Schw$. These observations lead to the following description of the components $\mathcal{H}_\tau$ in $L$, for $\mu$-almost all $\mu$.

\begin{proposition}
	Suppose that $\alpha: \Schw \to L$ is pointwise defined of dense image. Denote by $\|\cdot\|_\tau$ the continuous semi-norm $\xi \mapsto \|\alpha_\tau(\xi)\|_{\mathcal{H}_\tau}$ on $\Schw$. For $\mu$-almost all $\tau$, the map $\alpha_\tau$ induces an isomorphism
	\[ \left( \text{the completion of $\Schw$ relative to $\|\cdot\|_\tau$} \right) \rightiso \mathcal{H}_\tau \]
	between topological vector spaces. 
\end{proposition}

\begin{proposition}\label{prop:multiplicity-spaces}\index{$\mathcal{M}_\tau$}
	Suppose in addition that
	\begin{compactitem}
		\item we are given a locally compact separable group $\Gamma$, acting unitarily on $L$ such that $\alpha: \Schw \to L$ is $\Gamma$-equivariant;
		\item the direct integral for $L$ is the spectral decomposition in Theorem \ref{prop:direct-integral}, i.e.\ $\Pi = \Pi_{\mathrm{unit}}(\Gamma)$ and $\mathcal{H}_\tau = \tau \hat{\otimes} \mathcal{M}_\tau$.
	\end{compactitem}
	For $\mu$-almost all $\tau$, there is a canonical inclusion of vector spaces
	\[ \Hom(\mathcal{M}_\tau, \CC) \hookrightarrow \left\{ a \in \Hom_\Gamma(\Schw, \tau) : \; \exists C > 0, \; \|a(\cdot)\| \leq C \|\cdot\|_\tau \right\} \]
	where $\Hom(\cdots)$ and $\Hom_\Gamma(\cdots)$ stand for the continuous $\Hom$-spaces. It is an isomorphism when $\dim_{\CC} \mathcal{M}_\tau$ is finite.
\end{proposition}
\begin{proof}
	We may view $\Hom(\mathcal{M}_\tau, \CC)$ as a subspace of $\Hom_\Gamma(\tau \hat{\otimes} \mathcal{M}_\tau, \tau)$, which is full when $\mathcal{M}_\tau$ is finite-dimensional, by Schur's Lemma. Now apply the previous result.
\MyQED\end{proof}

The following result asserts that $\alpha: \Schw \to L$ is pointwise defined when $\Schw$ is \emph{nuclear}\index{topological vector space!nuclear}. It has been stated in \cite[I.4, Theorem 5]{GV4} and \cite[Chapter II, \S 1]{Ma68} in a different flavor. For the theory of nuclear spaces, see \cite[\S 50]{Tr67}.

\begin{theorem}\label{prop:nuclear-spectral}
	Assume that $\Schw$ is nuclear. Then any continuous linear map $\alpha: \Schw \to L$ is pointwise defined. In particular, any continuous linear map issuing from $\Schw$ is fine in the sense of \cite[1.4]{Be88}.
\end{theorem}
\begin{proof}
	By the Gelfand--Kostyuchenko Theorem \cite[1.5, Theorem]{Be88}, Hilbert--Schmidt operators between separable Hilbert spaces are pointwise defined. Therefore it suffices to factorize $\alpha$ into
	\[ \Schw \xrightarrow{\alpha_1} L_1 \xrightarrow{\alpha_2} L \]
	where $L_1$ is a separable Hilbert space, $\alpha_1$ is continuous linear, and $\alpha_2$ is Hilbert--Schmidt.
	
	By \cite[Theorem 50.1]{Tr67}, $\alpha$ is a nuclear mapping: there exist
	\begin{compactitem}
		\item $f_i \in \Schw^\vee$: an equicontinuous family \cite[Definition 14.3]{Tr67} of linear functionals, for $i=1,\ldots$ (at most countable),
		\item $\eta_i \in L$ with $\|\eta_i\| \leq 1$ for $i=1,\ldots$,
		\item $(\lambda_1, \ldots)$: an $\ell^1$-sequence of complex numbers
	\end{compactitem}
	(see also \cite[Proposition 47.2]{Tr67} and its proof), such that for each $\xi \in \Schw$
	\[ \alpha(\xi) = \sum_{i \geq 1} \lambda_i \angles{f_i, \xi} \eta_i \qquad \text{(convergent sum)}. \]
	Set $L_1$ to be the Hilbert space with an orthonormal basis $(\tilde{\eta}_i)_i$, with the same indices $i$ as before. We may also write $\lambda_i = \tilde{\lambda}_i \mu_i$ in such a manner that $(\tilde{\lambda}_i)_i$, $(\mu_i)_i$ are both $\ell^2$-sequences. We define first
	\begin{align*}
		\alpha_1: \Schw & \longrightarrow L_1 \\
		\xi & \longmapsto \sum_{i \geq 1} \tilde{\lambda}_i \angles{f_i, \xi} \tilde{\eta}_i,
	\end{align*}
	which is well-defined and continuous since $(f_i)_i$ is equicontinuous, $(\tilde{\eta}_i)_i$ is orthonormal and $(\tilde{\lambda}_i)_i$ is $\ell^2$.
	
	Secondly, define
	\begin{align*}
		\alpha_2: L_1 & \longrightarrow L \\
		\tilde{\eta}_i & \mapsto \mu_i \eta_i, \quad \forall i \geq 1,
	\end{align*}
	which is Hilbert--Schmidt by \cite[1.5 Lemma]{Be88} since $\|\eta_i\| \leq 1$ and $(\mu_i)_i$ is $\ell^2$. Obviously $\alpha = \alpha_2 \alpha_1$.
\MyQED\end{proof}

Most spaces of test functions in harmonic analysis are nuclear spaces. See \cite[\S 51]{Tr67} for a few examples in the Archimedean case.

\section{Hypocontinuity and barreled spaces}
The following review is meant to fix notations. We follow the definition in \cite[III, \S 5]{BoEVT}.

Let $X, Y, Z$ be topological vector spaces. A bilinear map $B: X \times Y \to Z$ is called \emph{separately continuous} if $B(x,\cdot): Y \to Z$ and $B(\cdot,y): X \to Z$ are both continuous linear maps for all $x, y$; it is called \emph{jointly continuous} if $B$ itself is a continuous map.

\begin{definition}\label{def:hypocontinuity}\index{hypocontinuous}
	Let $\mathfrak{S}$ be a family of bounded sets of $Y$. A separately continuous bilinear map $B: X \times Y \to Z$ is called \emph{$\mathfrak{S}$-hypocontinuous} if for every open neighborhood $\mathcal{W} \subset Z$ of $0$ and every $\mathcal{V} \in \mathfrak{S}$, there exists an open neighborhood $\mathcal{U} \subset X$ of $0$ such that
	\[ B(\mathcal{U} \times \mathcal{V}) \subset \mathcal{W}. \]
	It is known \cite[III, \S 5.3, Proposition 3]{BoEVT} that the condition is equivalent to: $B$ induces a continuous linear map $\theta: X \to \Hom^{\mathfrak{S}}(Y, Z)$ characterized by $\theta(x)(y) = B(x,y)$, where $\Hom^{\mathfrak{S}}(Y, Z)$ is the space of continuous linear maps $Y \to Z$ with the locally convex topology given by uniform convergence over subsets in $\mathfrak{S}$ (cf.\ \cite[\S 19, \S 32]{Tr67}); $\Hom^{\mathfrak{S}}(Y, Z)$ is Hausdorff when $\bigcup \mathfrak{S}$ is dense in $Y$.
\end{definition}

The $\mathfrak{S}$-hypocontinuity interpolates between joint and separate continuity. When $\mathfrak{S}$ equals the set of all bounded subsets of $Y$, we say that $B: X \times Y \to Z$ is a \emph{hypocontinuous} bilinear form. We will mainly be interested in the case $Z=\CC$; in this case, a hypocontinuous $X \times Y \to \CC$ gives rise to a continuous linear map $X \to Y^\vee$ where $Y^\vee = \Hom^{\mathfrak{S}}(Y, \CC)$ is the strong dual of $Y$.

The notion of \emph{barreled} topological vector spaces is standard; the precise definition may be found in \cite[\S 33]{Tr67}. It suffices for us to collect some of their basic properties.\index{topological vector space!barreled}

\begin{proposition}\label{prop:bb-properties}
	The property of being barreled is preserved by arbitrary $\varinjlim$. Furthermore
	\begin{itemize}
		\item Fréchet spaces are barreled, therefore LF-spaces are barreled as well;
		\item the dual of a barreled space is \emph{quasi-complete}\index{topological vector space!quasi-complete} with respect to the topology of pointwise convergence. Here quasi-completeness means that every bounded and closed subset is complete.
	\end{itemize}
\end{proposition}
\begin{proof}
	The permanence under $\varinjlim$ can be found in \cite[III, \S 4.1, Corollaire 3]{BoEVT}. The case of Fréchet spaces is also discussed therein. The quasi-completeness of the dual is proved in \cite[Corollary 2 to Theorem 34.2]{Tr67}. 
\MyQED\end{proof}

\begin{proposition}[{\cite[III, \S 5.3, Proposition 6]{BoEVT}}]\label{prop:hypocontinuity}
	Suppose that $X$ is a barreled space, then every separately continuous bilinear map $B: X \times Y \to Z$ is $\mathfrak{S}$-hypocontinuous, for every family $\mathfrak{S}$ of bounded sets in $Y$.
\end{proposition}

We record a well-known condition for joint continuity.
\begin{proposition}[{\cite[Corollary to Theorem 34.1]{Tr67}}]\label{prop:joint-continuity}
	Let $B: X \times Y \to Z$ be a separately continuous bilinear map. If $X$ is Fréchet and $Y$ is metrizable, then $B$ is jointly continuous.
\end{proposition}

The notion of barreled spaces are used in this work to study holomorphic families of linear functionals. Consider an open subset $U \subset \CC^r$, a locally convex topological vector space $V$, its continuous dual $V^\vee$ and a map $T: U \to V^\vee$, written as $\lambda \mapsto T_\lambda$. For the following application, we say $T$ is \emph{holomorphic} if $\lambda \mapsto T_\lambda(v)$ is holomorphic for all $v \in V$. This also implies that $T: \lambda \mapsto T_\lambda$ is continuous if $V^\vee$ is endowed with the topology of pointwise convergence.

The following \emph{method of analytic continuation} is due to Gelfand--Shilov \cite[Chapter I, A.2.3]{GS1},  see also \cite[Proposition 5.2.1]{Ig00}. The following formulation is proved in \cite[\S 8]{Li18}.
\begin{theorem}\label{prop:GS-principle}
	Suppose the space $V$ is barreled. Given a map
	\begin{align*}
		T: \CC^r & \longrightarrow \Hom_\CC(V, \CC) \\
		\lambda & \longmapsto T_\lambda
	\end{align*}
	where $\Hom_{\CC}$ stands for the algebraic $\Hom$.	Assume that
	\begin{itemize}
		\item for each $v \in V$, the function $\lambda \mapsto T_\lambda(v)$ is holomorphic on $\CC^r$;
		\item there exists an open subset $U \neq \emptyset$ of $\CC^r$ such that $T$ restricts to a holomorphic map $U \to V^\vee$.
	\end{itemize}
	Then $T$ is actually a holomorphic map $\CC^r \to V^\vee$.
\end{theorem}

\chapter{Schwartz spaces and zeta integrals}\label{sec:Schwartz-zeta}
Throughout this chapter, we fix
\begin{compactitem}
	\item a local field $F$ of characteristic zero,
	\item a split connected reductive $F$-group $G$,
	\item an affine spherical embedding $X^+ \hookrightarrow X$ satisfying Axiom \ref{axiom:geometric},
	\item relative invariants $f_i \in F[X]$ of eigencharacter $\omega_i$ (Definition \ref{def:relative-invariant}) for $1 \leq i \leq r := \rank(\Lambda)$.
\end{compactitem}

The Hypothesis \ref{hyp:zeta-0} will be made in the discussion of the $L^2$-aspect, which serves to motivate the overall framework.

\section{Coefficients of smooth representations}\label{sec:coefficients}
The algebraic $\Hom$ between topological vector spaces will be denoted by $\Hom_\text{alg}$ when confusion may arise.

\begin{definition}\label{def:algebraic-tvs}\index{topological vector space!algebraic}
	We call a topological vector space $V$ \emph{algebraic} if it is a countable inductive limit of finite-dimensional vector spaces, each equipped with its usual topology.
\end{definition}

\begin{lemma}\label{prop:autocont}
	Let $V, W$ be topological vector spaces. Suppose that $V$ is algebraic, then $\Hom_{\mathrm{alg}}(V, W) = \Hom(V, W)$.
\end{lemma}
\begin{proof}
	Write $V = \varinjlim_i V_i$ where each $V_i$ is finite-dimensional. Then we have
	\[ \Hom_{\mathrm{alg}}(V, W) = \varprojlim_i \Hom_{\text{alg}}(V_i, W) = \varprojlim_i \Hom(V_i, W), \]
	the last equality resulting from \cite[Theorem 9.1]{Tr67}. The universal property of topological inductive limits yields $\Hom(V, W)$.
\MyQED\end{proof}

\begin{lemma}\label{prop:algebraic-nuclear}
	Algebraic topological vector spaces are separable, nuclear, and barreled. 
\end{lemma}
\begin{proof}
	The required properties are preserved by countable $\varinjlim$:
	\begin{compactitem}
		\item separable: clear,
		\item nuclear: \cite[(50.8)]{Tr67},
		\item barreled: by Proposition \ref{prop:bb-properties}.
	\end{compactitem}
	Obviously, finite-dimensional vector spaces have these virtues.
\MyQED\end{proof}

We review the notion of \emph{smooth representations} below. Let $(\pi, V_\pi)$ be a continuous representation of $G(F)$. We define its smooth part $(\pi^\infty, V_\pi^\infty)$ as follows.\index{representation!smooth}
\begin{itemize}
	\item (For $F$ Archimedean) Take $V_\pi^\infty$ to be the subspace of smooth vectors, i.e.\ $v \in V_\pi^\infty$ if and only if $\gamma_v: g \mapsto \pi(g)v$ is smooth; so the universal enveloping algebra of $\mathfrak{g}$ acts on $V_\pi^\infty$. Given a continuous semi-norm $p$ on $V_\pi$, the corresponding $k$-th Sobolev semi-norm is given by
		\[ p_k(v) = \left( \sum_{m_1 + \cdots + m_n \leq k} p(\pi( X_1^{m_1} \cdots X_n^{m_n} ) v)^2 \right)^\demi \]
		where $X_1, \ldots, X_n$ is a basis of $\mathfrak{g}$. We endow $V_\pi^\infty$ with the locally convex topology induced by $p_k$ for all $p$ and all $k \in \Z_{\geq 1}$, as in \cite[2.4.3]{BK14}.
	\item (For $F$ non-Archimedean) Take $V_\pi^\infty = \bigcup_{J \subset G(F)} V_\pi^J$ where $J$ ranges over the compact open subgroups. It is a smooth representation of $G(F)$ in the usual sense. For each $J$, the subspace $V_\pi^J$ inherits its topology from $V_\pi$, therefore $V_\pi^\infty$ acquires the topology of $\varinjlim$. Let us check that $(\pi^\infty, V_\pi^\infty)$ is a continuous representation of $G(F)$. Recall that a local base at $0 \in V_\pi^\infty$ consists of subsets of the form
	\[ \sum_{J \in \mathcal{J}} \Image \left[ \mathcal{U}_J \to V_\pi^\infty \right] \]
	where $\mathcal{J}$ is a finite set of subgroups $J$ and $\mathcal{U}_J$ is a balanced convex neighborhood of $0 \in V_\pi^J$, for each $J \in \mathcal{J}$. By taking $J'$ so small that $J' \subset J$ for all $J \in \mathcal{J}$, we can assure
	\[ J' \times \sum_{J \in \mathcal{J}} \Image \left[ \mathcal{U}_J \to V_\pi^\infty \right] \to \sum_{J \in \mathcal{J}} \Image \left[ \mathcal{U}_J \to V_\pi^\infty \right] \]
	under the action map, which implies continuity by linearity. 
	
	When $\pi$ is admissible, the topology on each $V_\pi^J$ (finite-dimensional) is unique, and $V_\pi^\infty$ becomes algebraic. There is no need to worry about topologies in such a setting.
\end{itemize}
Call $\pi$ smooth if $\pi = \pi^\infty$. As observed above, this is compatible with the standard notion for non-Archimedean $F$.

\begin{definition}\label{def:nice-smooth-rep}\index{representation!nice}
	In this work, by a \emph{nice representation} of $G(F)$ we mean:
	\begin{itemize}
		\item (For $F$ Archimedean) A smooth representation on a Fréchet space that is admissible of moderate growth, also known as Casselman--Wallach representations\index{representation!Casselman--Wallach} or \emph{SAF representations}\index{representation!SAF} in \cite{BK14}; they are the unique ``SF-globalizations'' of Harish-Chandra modules.
		\item (For $F$ non-Archimedean) A representation $(\pi, V_\pi)$ of $G(F)$ which is smooth admissible of finite length.
	\end{itemize}
\end{definition}

For example, let $\tau$ be an irreducible unitary representation of $G(F)$. It is well known that its smooth part $\tau^\infty$ is a nice representation.

\begin{lemma}\label{prop:nice-smooth-nuclear}
	The underlying space $V_\pi$ of a nice representation $\pi$ is separable, nuclear, and barreled.
\end{lemma}
\begin{proof}
	In the Archimedean case, \cite[Corollary 5.6]{BK14} implies that $V_\pi$ is Fréchet as well as nuclear, whereas its separability is well-known (pass to Harish-Chandra modules). The non-Archimedean case follows from Lemma \ref{prop:algebraic-nuclear}.
\MyQED\end{proof}

Next, consider the data
\begin{itemize}
	\item an equivariant vector bundle $\mathscr{E}$ on $X^+(F)$, equipped with a $G(F)$-invariant pairing $\mathscr{E} \otimes \overline{\mathscr{E}} \to \mathscr{L}$ as in \eqref{eqn:hermitian-vb};
	\item the following functions spaces on $X^+(F)$ with values in vector bundles:
	\begin{align*}
		C^\infty(X^+) & := C^\infty(X^+(F), \overline{\mathscr{E}}) \simeq C^\infty(X^+(F), \mathscr{E}^\vee \otimes \mathscr{L}), \\
		C^\infty_c(X^+) & := C^\infty_c(X^+(F), \mathscr{E}), \\
		L^2(X^+) & := L^2(X^+(F), \mathscr{E}),
	\end{align*}
	where the isomorphism results from \eqref{eqn:hermitian-vb}.
\end{itemize}

We topologize $C_c^\infty(X^+)$ and $C^\infty(X^+)$ in the standard manner \cite[2.2]{Be88}, making them into continuous representations of $G(F)$. Specifically,
\begin{compactitem}
	\item $C_c(X^+)$ is realized as the $\varinjlim$ of $C_\Omega(X^+) := \{\xi \in C_c(X^+): \Supp(\xi) \subset \Omega\}$, carrying the norms $\sup_\Omega \|\cdot\|$, where $\Omega \subset X^+(F)$ ranges over the compact subsets \footnote{Here we take $\|\cdot\|$ to be any continuous family of norms on the fibers of $\mathscr{E}|_{\Omega}$, the choice being immaterial because $\Omega$ is compact.};
	\item $C(X^+)$ carries the semi-norms $\sup_\Omega\|\cdot\|$ with the $\Omega$ above;
\end{compactitem}
Taking smooth parts yields the topological vector spaces $C_c^\infty(X^+)$ and $C^\infty(X^+)$. One can verify that $C^\infty_c(X^+)$ is actually algebraic when $F$ is non-Archimedean.

Note that $C^\infty_c(X^+) \subset L^2(X^+)$, and the integration of densities furnishes an invariant pairing $C^\infty(X^+) \otimes C^\infty_c(X^+) \to \CC$. The next result furnishes a bridge between $L^2$ and smooth theories.

\begin{theorem}\label{prop:C-infty}
	Let $\tau$ be an irreducible unitary representation of $G(F)$ and put $\pi := \check{\tau}^\infty$, where $\check{\tau}$ stands for the contragredient of $\tau$ as a Hilbert representation. There is a $\CC$-linear isomorphism
	\begin{align*}
	\Hom_{G(F)}(C^\infty_c(X^+), \tau) & \rightiso \Hom_{G(F)}(\pi, C^\infty(X^+)) \\
	\varphi & \mapsto \check{\varphi}|_{\check{\tau}^\infty},
	\end{align*}
	where $\check{\varphi}: \check{\tau} \to C^\infty_c(X^+)^\vee$ is the adjoint of $\varphi$.
\end{theorem}
\begin{proof}
	This is just a paraphrase of \cite[Proposition 2.4]{Be88}, where $\mathscr{E} = \CC$ and the hermitian adjoint is used; the arguments in \textit{loc.\ cit.} carry over verbatim.
\MyQED\end{proof}
\begin{remark}\label{rem:C-infty}
	If we consider the hermitian adjoint as in \textit{loc.\ cit.}, the result will be an anti-linear isomorphism onto $\Hom_{G(F)}(\tau^\infty, C^\infty(X^+(F), \mathscr{E}))$.
\end{remark}

\begin{notation}\index{N_pi@$\mathcal{N}_\pi$}
	For an irreducible nice representation $\pi$ of $G(F)$, we adopt the notation
	\begin{gather*}
		\mathcal{N}_\pi := \Hom_{G(F)}(\pi, C^\infty(X^+)).
	\end{gather*}
	The representations with $\mathcal{N}_\pi \neq \{0\}$ are the main objects in the study of \emph{distinguished representations} \index{representation!distinguished}. The elements $\varphi(v) \in C^\infty(X^+)$ for $\varphi \in \mathcal{N}_\pi$, $v \in V_\pi$ are often called the \emph{coefficients}\index{representation!coefficients of} of $\pi$.
\end{notation}

\begin{axiom}\label{axiom:finiteness}
	For any irreducible nice representation $\pi$, the space $\mathcal{N}_\pi$ is finite dimensional.
\end{axiom}

\begin{remark}\label{rem:finiteness}
	Under the Axiom \ref{axiom:geometric}, this condition is almost known to hold. Indeed, the Archimedean case is covered by \cite[Theorem A]{KO13} since we may deal with the finitely many $G(F)$-orbits in $X^+(F)$ separately. As for the non-Archimedean case, there is no need to impose continuity and when $X^+$ is wavefront and $\mathscr{E}$ is trivializable, the finiteness statement is just \cite[Theorem 5.1.5]{SV17}. That result can be extended to the case $\mathscr{E} = \mathscr{L}^\demi$ (cf.\ Remark \ref{rem:half-densities}), which will be covered by Theorem \ref{prop:finiteness-density}. The case of ``Whittaker induction'' is also known: see Remark \ref{rem:Whittaker-induction}.
\end{remark}

\section{The group case}\label{sec:group-case}
In this section, $H$ stands for a connected reductive $F$-group and $G := H \times H$. \emph{The group case} signifies the homogeneous $G$-variety $X^+ := H$ equipped with the action \index{group case}
\begin{gather}\label{eqn:group-case-action1}
	x(g_1, g_2) = g_2^{-1} x g_1, \quad x \in X^+, \; (g_1, g_2) \in G.
\end{gather}
Take $x_0 = 1 \in X^+(F)$, then its stabilizer equals the diagonal image $\text{diag}(H) \subset G$. When $H$ is quasi-split, we will work with Borel subgroups of $G$ of the form $B^- \times B$. It follows from Bruhat decomposition that $X^+$ is a spherical homogeneous $G$-space; in fact $X^+$ is a symmetric space defined by the involution $(g_1, g_2) \mapsto (g_2, g_1)$, therefore $X^+$ is wavefront. Our aim is to describe the ``coefficients'' of \S\ref{sec:coefficients} in this setting.

The $G$-action above is by no means the canonical one. We may flip the two components of $G$ and obtain a new action
\begin{gather}\label{eqn:group-case-action2}
	x(g_1, g_2) = g_1^{-1} x g_2, \quad x \in X^+, \; (g_1, g_2) \in G.
\end{gather}
The actions \eqref{eqn:group-case-action1} and \eqref{eqn:group-case-action2} are intertwined by $x \mapsto x^{-1}$. We shall denote the case of \eqref{eqn:group-case-action2} by $\check{X}^+$ when confusion may arise.

Hereafter, $F$ is a local field and $H$ is split with a chosen Borel subgroup $B$. Take the bundle $\mathscr{E} := \mathscr{L}^\demi$ of half-densities on $X^+(F)$ and set $C^\infty(X^+) :=C^\infty(X^+, \mathscr{L}^\demi)$ as usual. As before, $G(F)$ acts in two ways, thus we write $C^\infty(X^+)$, $C^\infty(\check{X}^+)$ to distinguish. Consider an irreducible nice representation $\Pi$ of $G(F)$ together with the continuous $\Hom$-spaces
\begin{align*}
	\mathcal{N}_\Pi & := \Hom_{G(F)}(\Pi, C^\infty(X^+)), \\
	\check{\mathcal{N}}_\Pi & := \Hom_{G(F)}(\Pi, C^\infty(\check{X}^+)).
\end{align*}
\begin{itemize}
	\item The equivariant bundle $\mathscr{L}^\demi$ on $X^+(F)$ can be trivialized by taking some Haar measure $|\Omega|$ on $H(F)$ and form the invariant global section $|\Omega|^\demi$ of $\mathscr{L}^\demi$. The construction works for $\check{X}^+$ as well.
	\item The space $\mathcal{N}_\Pi$ is nonzero if and only if $\Pi \simeq \pi \boxtimes \check{\pi}$ for some irreducible nice representation $\pi$ of $H(F)$. In this case, $\mathcal{N}_\Pi$ is in bijection with the space of Haar measures $|\Omega|$, by taking matrix coefficients
	\begin{align*}
	\varphi_{|\Omega|}: \pi \boxtimes \check{\pi} & \longrightarrow C^\infty(X^+) \\
	v \otimes \check{v} &\longmapsto \angles{\check{v}, \pi(\cdot)v} \cdot |\Omega|^\demi.
	\end{align*}
	Some words on the proof are in order. Upon trivializing $\mathscr{L}^\demi$, the case of non-Archimedean $F$ is routine: everything is algebraic. As for the Archimedean case, a similar description of $\mathcal{N}_\Pi$ is well-known in the algebraic setting of Harish-Chandra modules; one can pass to the setting of nice representations by Frobenius reciprocity and the ``group case'' of automatic continuity \cite[\S 11.2]{BK14}. The underlying space of $\pi \boxtimes \check{\pi}$ is actually the completed tensor product $V_\pi \hat{\otimes} V_{\check{\pi}}$ of nuclear spaces.
	\item The same holds for $\check{\mathcal{N}}_\Pi$ except that $\mathcal{N}_\Pi$ is now spanned by
	\begin{align*}
	\check{\varphi}_{|\Omega|}: \pi \boxtimes \check{\pi} & \longrightarrow C^\infty(\check{X}^+) \\
	v \otimes \check{v} & \longmapsto \angles{\check{\pi}(\check{v}), v} \cdot |\Omega|^\demi.
	\end{align*}
	\item All in all, we deduce a canonical isomorphism of lines $\mathcal{N}_\Pi \rightiso \check{\mathcal{N}}_\Pi$ given by $\varphi_{|\Omega|} \mapsto \check{\varphi}_{|\Omega|}$. Furthermore, the following diagram of $G(F)$-representations commutes
	\[ \begin{tikzcd}
	C^\infty(X^+) \arrow{r}{\beta \mapsto \beta^\vee} & C^\infty(\check{X}^+) \\
	\Pi \otimes \mathcal{N}_\Pi \arrow{u} \arrow{r}{\sim} & \Pi \otimes \check{\mathcal{N}}_\Pi \arrow{u}
	\end{tikzcd} \quad \beta^\vee(x) := \beta(x^{-1}). \]
\end{itemize}

Finally, we remark that when $H$ is a torus, there is no need to distinguish left and right translations. Hence we may view $X^+ = H$ as a homogeneous $H$-space via $H = H \times \{1\} \hookrightarrow G$. For any continuous character $\chi$ of $H(F)$ we have $\dim \mathcal{N}_\chi = 1$, with generator $\varphi_{|\Omega|}: 1 \mapsto \chi(\cdot)|\Omega|^\demi \in C^\infty(X^+)$. Furthermore, $X^+ = H$ is still wavefront under this setting.

\section{Auxiliary definitions}\label{sec:auxiliary}
The following notions will enter into the Axiom \ref{axiom:zeta}. As they are more or less intuitive, we collect them here to facilitate the reading.

We begin by clarifying the meaning of meromorphic or rational (in the non-Archimedean case) families. Our formulation is based on Bernstein's approach; for the general theory we recommend \cite[VI.8]{Re10}.

\begin{definition}\label{def:T}\index{T@$\mathcal{T}$}\index{O@$\mathcal{O}$}\index{K@$\mathcal{K}$}
	Define $\mathcal{T}$ to be the complex manifold of unramified characters $\{|\omega|^\lambda : \lambda \in \Lambda_{\CC} \}$. Notice that $\mathcal{T}$ also parametrizes the functions $\{|f|^\lambda : \lambda \in \Lambda_{\CC}\}$ on $X^+(F)$. When $F$ is non-Archimedean, $\mathcal{T}$ is the complex algebraic torus corresponding to the group algebra of $\Hom(\Lambda, \Z)$. Define $\mathcal{O}$ to be the algebra of
	\begin{itemize}
		\item holomorphic functions on $\mathcal{T}$ for $F$ Archimedean;  
		\item regular algebraic functions on $\mathcal{T}$, for $F$ non-Archimedean.
	\end{itemize}
	In either case, the map ``evaluation at $g \in G(F)$'' furnishes a group homomorphism
	\[ |\omega|^\text{univ}: G(F) \to \mathcal{O}^\times. \]
	On the other hand, for every $|\omega|^\lambda \in \mathcal{T}$, the homomorphism ``evaluation at $|\omega|^\lambda$'' is denoted by
	\[ \text{ev}_\lambda: \mathcal{O} \to \CC.\]
	Finally, denote the fraction field of $\mathcal{O}$ by
	\[ \mathcal{K} := \text{Frac}(\mathcal{O}). \]
\end{definition}

Suppose that $(\pi, V_\pi)$ is a nice representation of $G(F)$. Let $G(F)$ act on $V_\pi \otimes \mathcal{O}$ by
\[ \widehat{\pi}(g): v \otimes b \mapsto \pi(g)v \otimes |\omega|^\text{univ}(g)b, \]
which turns out to be $\mathcal{O}$-linear, and whose reduction via $\text{ev}_\lambda$ yields $\pi_\lambda = \pi \otimes |\omega|^\lambda$. Set $\widehat{V}_\pi := V_\pi \otimes \mathcal{K}$ which still carries a $G(F)$-action $\widehat{\pi}$. It captures the idea of the family of representations $\pi_\lambda$ parametrized by $\mathcal{T}$.

\begin{definition}\label{def:mero-rat}\index{L-pi@$\mathcal{L}_{\pi,t}$}
	Suppose given a continuous representation of $G(F)$ on a space $\Schw$. For every nonzero $t \in \mathcal{O}$ (the ``denominator''), denote by $\mathcal{L}_{\pi,t}^\circ$ the $\mathcal{K}$-vector space of $\mathcal{K}$-linear maps $B: \widehat{\pi} \dotimes{\CC} \Schw \longrightarrow \mathcal{K}$ that
	\begin{itemize}
		\item $B$ is $G(F)$-invariant: $B(\widehat{\pi}(g)\widehat{v}, g\xi) = B(\widehat{v}, \xi)$ for all $g \in G(F)$;
		\item $B\left( (V_\pi \otimes t\mathcal{O}) \otimes \Schw \right) \subset \mathcal{O}$.
	\end{itemize}
	Denote by $\mathcal{L}_{\pi,t} \subset \mathcal{L}_{\pi,t}^\circ$ the subspace of $B: \widehat{\pi} \otimes \Schw \to \mathcal{K}$ such that the ``evaluation of $tB$''
	\[ v \otimes \xi \longmapsto \text{ev}_\lambda \left( B(v \otimes t, \xi)\right) \]
	is a hypocontinuous bilinear form $V_\pi \times \Schw \to \CC$ for every $\lambda$.

	Therefore, $B \in \mathcal{L}_{\pi,t}$ induces a meromorphic family $\lambda \mapsto \theta_\lambda$ of elements of $\Hom_{G(F)}(\pi_\lambda, \Schw^\vee)$ (with strong topology on $\Schw^\vee$), characterized by
	\begin{gather}\label{eqn:family-theta}
		t(\lambda)\theta_\lambda(v): \xi \mapsto \text{ev}_\lambda \left( B(v \otimes t, \xi)\right),
	\end{gather}
	where $v \in V_\pi, \lambda \in \Lambda_{\CC}, \xi \in \Schw$. Meromorphy (say with denominator $t$) means that the function $\lambda \mapsto t(\lambda)\theta_\lambda(v)$ is holomorphic in $\lambda$ for every $v$. Recall that the notion of holomorphy is defined prior to Theorem \ref{prop:GS-principle}.

	When $t(\lambda) \neq 0$, one can extend $\text{ev}_\lambda$ to $t^{-1}\mathcal{O} \to \CC$, therefore the reduction at $\lambda$ of $B$ can be defined as a $G(F)$-invariant bilinear form $B_\lambda: \pi_\lambda \otimes \Schw \to \CC$. Finally, we accommodate arbitrary denominators by setting
	\begin{align*}
		\mathcal{L}_\pi^\circ & := \varinjlim_t \mathcal{L}_{\pi,t}^\circ, \\
		\mathcal{L}_\pi & := \varinjlim_t \mathcal{L}_{\pi,t}
	\end{align*}
	where the indices $t \in \mathcal{O} \smallsetminus \{0\}$ are directed by divisibility.
\end{definition}

The passage from $\mathcal{L}_{\pi,t}^\circ$ to $\mathcal{L}_{\pi,t}$ is greatly facilitated by the following device.
\begin{lemma}\label{prop:circ-removal}
	Suppose $\Schw$ is barreled. Let $t \in \mathcal{O} \smallsetminus \{0\}$, $B \in \mathcal{L}_{\pi,t}^\circ$. If the evaluations of $tB$ (see above) are separately continuous at every $\lambda$ in some nonempty open subset $\mathcal{U} \subset \Lambda_{\CC}$, then $B \in \mathcal{L}_{\pi,t}$.
\end{lemma}
\begin{proof}
	By Lemma \ref{prop:nice-smooth-nuclear} we know $V_\pi$ is barreled. By Proposition \ref{prop:hypocontinuity}, it suffices to establish the separate continuity of the evaluations of $tB$ at every $\lambda$. In what follows, we equip $\Schw^\vee$ with the topology of pointwise convergence (i.e.\ weak topology).

	Fix $v \in V_\pi$, we obtain as in \eqref{eqn:family-theta} a family in $\lambda \in \Lambda_{\CC}$ of linear functionals $t(\lambda)\theta_\lambda(v): \Schw \to \CC$. They belong to $\Schw^\vee$ whenever $\lambda \in \mathcal{U}$. The first task is to propagate this property to all $\lambda$.

	Given $\xi \in \Schw$, the function $\lambda \mapsto \angles{t(\lambda)\theta_\lambda(v), \xi}$ lies in $\mathcal{O}$ by the definition of $\mathcal{L}_{\pi,t}^\circ$, hence holomorphic in $\lambda \in \Lambda_{\CC}$. Since $\Schw$ is barreled, the required continuity of $t(\lambda)\theta_\lambda(v)$ follows from the \emph{method of analytic continuation} in Theorem \ref{prop:GS-principle}.

	By the same reasoning, for every fixed $\xi \in \Schw$ the continuity in $v \in V_\pi$ for $\lambda \in \mathcal{U}$ propagates to the whole $\Lambda_{\CC}$, since $V_\pi$ is barreled. This completes our proof.
\MyQED\end{proof}
Hence the defining condition of $\mathcal{L}_{\pi, t}$ is independent of the denominator $t$: for nonzero $t' \in \mathcal{O}$:
\[ B \in \mathcal{L}^\circ_{\pi,t} \cap \mathcal{L}_{\pi, tt'} \implies B \in \mathcal{L}_{\pi, t}. \]
Indeed, it suffices to evaluate $B$ over some nonempty open $\mathcal{U}$ outside the zero locus of $tt'$.

As for the $L^2$-part of our axioms, a (nonstandard) variant of the notion of holomorphy will be used.
\begin{definition}\label{def:L2-holomorphy}
	Let $W$ be a topological vector space and $H$ be a separable Hilbert space. A family of continuous linear maps $\beta_\lambda: W \to H$ indexed by $\mathcal{D} := \left\{ \lambda \in \Lambda_{\CC} : \Re(\lambda) \relgeq{X} 0 \right\}$ is called holomorphic if it satisfies
	\begin{gather*}
	\forall \xi \in W, \quad \left[ \lambda \mapsto \beta_\lambda(\xi) \right] \; \text{is continuous}, \\
	\forall \xi \in W, \; \forall b \in H, \quad \left[ \lambda \mapsto (\beta_\lambda(\xi)|b)_H \right] \; \text{is holomorphic}
	\end{gather*}
	where $\lambda \in \mathcal{D}$ (resp.\  the interior of $\mathcal{D}$) in the first (resp.\  the second) condition.
\end{definition}
To see how this leads to a stronger version of holomorphy, we refer to \cite[Theorem 3.31]{Ru91}.

Assume $F$ is Archimedean now; it is harmless to work with $F = \R$. There exist natural structures of Nash manifolds and Nash bundles on $X^+(F)$ and $\mathscr{L}$, respectively; for the latter construction, see \cite[A.1.1]{AG08}. We recommend \cite[\S 1]{AG08} for generalities on semi-algebraic geometry over $\R$.

\begin{definition}\label{def:rapid-decay}\index{rapid decay}
	A smooth section $\Xi$ of $\mathscr{L}$ is said to have \emph{rapid decay}, if for any semi-algebraic function $p$ on the affine variety $X(F)$ we have
	\[ \nu_p(\Xi) := \int_{X^+(F)} \left( 1 + |p|^2 \right) |\Xi| < + \infty. \]
\end{definition}
Note that the definition depends on $X^+ \hookrightarrow X$. The notion of rapid decay of $\mathscr{L}^{1/2}$-valued sections can be defined similarly. The main differences from the definition of Schwartz functions in \cite[Remark 4.1.5]{AG08} are that
\begin{inparaenum}[(i)]
	\item we do not consider differential operators, and
	\item $X$ is singular in general, thus $X(F)$ is not a Nash manifold.
\end{inparaenum}

\section{Schwartz spaces: desiderata}\label{sec:Schwartz}
Fix a $G(F)$-stable subspace $\Schw$ of $C^\infty(X^+(F), \mathscr{E})$ (the ``Schwartz space'') subject to the following conditions:
\begin{itemize}
	\item $\Schw$ carries a topology so that $\Schw$ is a smooth continuous $G(F)$-representation.
	\item $C^\infty_c(X^+)$ is continuously included in $\Schw$.
	\item For every $\lambda \in \Lambda_{\CC}$ we set\index{S_lambda@$\Schw_\lambda$}
		\[ \Schw_\lambda := |f|^\lambda \Schw \]
		which still lies between $C^\infty(X^+(F), \mathscr{E})$ and $C^\infty_c(X^+)$. We topologize $\Schw_\lambda$ by requiring that
		\[
			m_\lambda: \Schw \stackrel{\sim}{\longrightarrow} \Schw_\lambda, \quad \xi \mapsto |f|^\lambda \xi
		\]
		is a continuous isomorphism. Upon recalling how the topology of $C^\infty_c(X^+)$ is defined, one sees that the foregoing conditions for $\Schw$ still hold true for $\Schw_\lambda$.
\end{itemize}

The geometry of $X$ intervenes now. For every irreducible nice representation $\pi$, we put\index{pi_lambda@$\pi_\lambda$}
\[ \pi_\lambda := \pi \otimes |\omega|^\lambda, \quad \lambda \in \Lambda_{\CC} \]
realized on the same underlying vector space as $\pi$. Elements of $\mathcal{N}_\pi$ can be twisted by relative invariants of $X$: for every $\varphi \in \mathcal{N}_\pi$, define\index{phi_lambda@$\varphi_\lambda$}
\begin{align*}
	\varphi_\lambda: \pi_\lambda & \longrightarrow C^\infty(X^+) \\
	v & \longmapsto |f|^\lambda \varphi(v).
\end{align*}
It is clearly equivariant, and we have $\varphi_\lambda \in \mathcal{N}_{\pi_\lambda}$ --- indeed, the continuity of $\varphi_\lambda$ follows from the way $C^\infty(X^+)$ is topologized. Furthermore, $\varphi_{\lambda+\mu} = (\varphi_\lambda)_{\mu}$ and $\varphi_0 = \varphi$. Notice that $\varphi_\lambda$ depends on the choice of relative invariants, albeit in a mild manner.

\begin{axiom}\label{axiom:zeta}
	We assume the validity of Axiom \ref{axiom:finiteness}.
	\begin{itemize}
		\item The functional-analytic axioms (mainly used in \S\ref{sec:connection-L2}):
		\begin{enumerate}[(F1)]
			\item $\Schw$ is a separable nuclear space.
			\item $\Schw$ is barreled. 
			\item We have $\Schw_\lambda \subset L^2(X^+)$ whenever $\Re(\lambda) \relgeq{X} 0$. Denote the inclusion map $\Schw \hookrightarrow L^2(X^+)$ by $\alpha$, so that $\Schw_\lambda \hookrightarrow L^2(X^+)$ may be viewed as a family of inclusions
				\begin{align*}
					\alpha_\lambda: \Schw & \hookrightarrow L^2(X^+) \\
					\xi & \mapsto |f|^\lambda \alpha(\xi).
				\end{align*}
				We require that $\alpha_\lambda$ is continuous whenever $\Re(\lambda) \relgeq{X} 0$. Note that $\alpha_\lambda$ is no longer equivariant: the character $|\omega|^\lambda$ intervenes. By Theorem \ref{prop:nuclear-spectral} $\alpha_\lambda$ is pointwise defined, and $\alpha_\lambda$ has dense image in $L^2(X^+)$ since $C^\infty_c(X^+)$ does.
			\item The family $\alpha_\lambda: \Schw \hookrightarrow L^2(X^+)$ is holomorphic in the sense of Definition \ref{def:L2-holomorphy}\index{alpha_lambda@$\alpha_\lambda$}.
		\end{enumerate}
		\item The smooth axioms:
		\begin{enumerate}[(S1)]
			\item For Archimedean $F$, there exists $\lambda_0 \in \Lambda_{\R}$ such that $|f|^{2\lambda} \|\xi\|^2$ is of rapid decay (Definition \ref{def:rapid-decay}) for all $\xi \in \Schw$ when $\Re(\lambda) \relgeq{X} \lambda_0$. Furthermore, we require that the semi-norms $\xi \mapsto \nu_p \left( |f|^{2\lambda}\|\xi\|^2 \right)$ in Definition \ref{def:rapid-decay} are continuous in $\xi \in \Schw$, for all $p$.
			\item For non-Archimedean $F$, the support of any $\xi \in \Schw$ has compact closure in $X(F)$.
			\item Given an irreducible nice representation $(\pi, V_\pi)$ and $\varphi \in \mathcal{N}_\pi$, we assume that the $G(F)$-invariant bilinear form\index{Z_lambda@$Z_\lambda$}
				\begin{equation}\label{eqn:zeta-integral}\begin{aligned}
					Z_\lambda = Z_{\lambda, \varphi}: V_\pi \otimes \Schw & \longrightarrow \CC \\
					v \otimes \xi & \longmapsto \int_{X^+(F)} \varphi_\lambda(v) \xi
				\end{aligned}\end{equation}
				is well-defined by a convergent integral whenever $\Re(\lambda) \relgg{X} 0$, and is separately continuous.
			\item Furthermore, $Z_\lambda$ admits a meromorphic (rational for non-Archimedean $F$) continuation to the whole $\mathcal{T}$ in the sense of Definition \ref{def:mero-rat}: there exists a homomorphism of $\mathcal{K}$-vector spaces
				\[ \widehat{Z} = \widehat{Z}_\varphi: \widehat{\pi} \dotimes{\CC} \Schw \longrightarrow \mathcal{K} \]
				verifying:
				\begin{enumerate}
					\item $\widehat{Z} \in \mathcal{L}_\pi^\circ$, hence $\widehat{Z} \in \mathcal{L}_\pi$ by the previous condition together with Lemma \ref{prop:circ-removal};
					\item when $\Re(\lambda) \relgg{X} 0$, the reduction of $\widehat{Z}$ with respect to $\text{ev}_\lambda$ is well-defined and coincides with the previous integral pairing. Hence we shall unambiguously denote the meromorphic (rational for non-Archimedean $F$) family deduced from $\widehat{Z}$ as $Z_\lambda$.
				\end{enumerate}
		\end{enumerate}
	\end{itemize}
\end{axiom}
If we change the relative invariants $f_i$ to $c_i f_i$ with $c_i \in F^\times$, and put
\[ |c|^\lambda := \prod_{i=1}^r |c_i|^{\lambda_i}, \quad \lambda = \sum_{i=1}^r \lambda_i \omega_i \in \Lambda_{\CC}, \]
then $Z_\lambda$ above will be replaced by $|c|^\lambda Z_\lambda$. The properties such as rationality, etc.\ remain unaffected. 
\begin{remark}\label{rem:zeta-distribution}
	The family $Z_\lambda$ may be called the \emph{zeta integrals}. As explained in Definition \ref{def:mero-rat}, $\widehat{Z}$ induces a meromorphic family $\widehat{\varphi}_\lambda \in \Hom_{G(F)}(\pi_\lambda, \Schw^\vee)$, and their restriction to $C^\infty_c(X^+)$ coincides with the original $\varphi_\lambda: \pi_\lambda \to C^\infty(X^+) \subset C^\infty_c(X^+)^\vee$. Indeed, $\widehat{\varphi}_\lambda|_{C^\infty_c(X^+)} = \varphi_\lambda$ for $\Re(\lambda) \relgg{X} 0$ by the construction of zeta integrals, thus they coincide for all $\lambda$ by meromorphy. Hereafter, we may safely write $\varphi_\lambda$ in place of $\widehat{\varphi}_\lambda$. Call $\varphi_\lambda$ the family of \emph{zeta distributions}.
\end{remark}

\begin{remark}
	One might be tempted to regard $\Schw$ as a space determined by the spherical embedding $X^+ \hookrightarrow X$. We are cautious about this point, since candidates of $\Schw$ may involve extra structures on $X$ such as affine bundles, etc.
\end{remark}

\begin{remark}\label{rem:Schw-fa}
	Let us turn to the functional-analytic aspects. The barreledness assumption on $\Schw$ is usually easy: by Proposition \ref{prop:bb-properties}, the Fréchet spaces, LF-spaces, and the algebraic spaces in Definition \ref{def:algebraic-tvs} are included.
	
	The separate continuity of the zeta integrals seems to be a minimal requirement. When $F$ is non-Archimedean, the continuity in $v \in V_\pi$ is automatic by Lemma \ref{prop:autocont}; if $\Schw$ is algebraic as well, then the separate continuity holds. For Archimedean $F$, the continuity will require some real effort. Note that one can actually deduce joint continuity under the assumptions of Proposition \ref{prop:joint-continuity}. Similar analytic subtleties are encountered in the study of Archimedean Rankin--Selberg integrals, eg.\ \cite{Ja09}.
\end{remark}

\begin{remark}
	Let us verify Axiom \ref{axiom:zeta} for $\Schw := C^\infty_c(X^+)$, in which case $\Schw = \Schw_\lambda$. The required properties of $\Schw$ have been verified in Remark \ref{rem:Schw-fa}. For every $\varphi \in \mathcal{N}_\pi$ the zeta integral
	\begin{gather*}
		v \otimes \xi \longmapsto \int_{X^+(F)} \varphi(v)|f|^\lambda \xi
	\end{gather*}
	converges for all $\lambda$. To show that it is deduced from some $\widehat{Z}$, observe that in the non-Archimedean case,
	\begin{inparaenum}[(i)]
		\item the compactness of $\Supp(\xi)$ and
		\item that $\varphi_\lambda(v)$, $\xi$ and $|f|$ are all stable under translation by some open compact subgroup $J \subset G(F)$,
	\end{inparaenum}
	implies that $\int_{X^+(F)} \varphi(v)|f|^\lambda \xi$ actually comes from an $\mathcal{O}$-valued finite sum via $\text{ev}_\lambda$. For the Archimedean case, it suffices to observe that $\int_{X^+(F)} \varphi(v)|f|^\lambda \xi$ is holomorphic in $\lambda$ and bounded independently of $\Im(\lambda)$.
\end{remark}

Returning to the general scenario, now we have the $\mathcal{K}$-linear map
\begin{equation}\label{eqn:T_pi} \begin{aligned}
	T_\pi: \mathcal{N}_\pi \dotimes{\CC} \mathcal{K} & \longrightarrow \mathcal{L}_\pi, \\
	\varphi \otimes a & \longmapsto a\widehat{Z}_\varphi.
\end{aligned}\end{equation}
Indeed, the denominator $t$ can be chosen to work for all $\varphi$ since $\dim_{\CC} \mathcal{N}_\pi < \infty$. 

\begin{lemma}\label{prop:uniqueness-embedding}
	The composite
	\[ \mathcal{N}_\pi \otimes \mathcal{K} \xrightarrow{T_\pi} \mathcal{L}_\pi \hookrightarrow \left\{\text{meromorphic families } \varphi_\lambda \in \Hom_{G(F)}(\pi_\lambda, \Schw^\vee) \right\} \]
	is injective, thus $T_\pi$ is injective as well. The last inclusion is explained in Definition \ref{def:mero-rat}.
\end{lemma}
\begin{proof}
	Upon clearing denominators, it suffices to show the injectivity on $\mathcal{N}_\pi \otimes \mathcal{O}$ which follows from the discussions in Remark \ref{rem:zeta-distribution}.
\MyQED\end{proof}

\begin{lemma}\label{prop:zeta-translation}
	Let $\lambda \in \Lambda_{\CC}$. We have a commutative diagram
	\[ \begin{tikzcd}
		\mathcal{L}_\pi \arrow{r}[above]{\sim} & \mathcal{L}_{\pi_\lambda} \\
		\mathcal{N}_\pi \arrow[hookrightarrow]{u}[left]{T_\pi} \arrow{r}[below]{\sim} & \mathcal{N}_{\pi_\lambda} \arrow[hookrightarrow]{u}[right]{T_{\pi_\lambda}}
	\end{tikzcd} \]
	in which the lower horizontal arrow is $\varphi \mapsto \varphi_\lambda = |f|^\lambda \varphi$, and the upper one is a transport by structure by translating the torus $\mathcal{T}$ by $|\omega|^\lambda$.
\end{lemma}
\begin{proof}
	It suffices to check this in the convergence range of zeta integrals.
\MyQED\end{proof}

Thus far, the convergence for zeta integrals for $\Re(\lambda) \relgg{X} 0$ is part of our assumptions. The following observation will be helpful for establishing the convergence in some cases. See Corollary \ref{prop:conv-gg0} for example.
\begin{proposition}\label{prop:zeta-convergence}
	Let $\varphi \in \mathcal{N}_\pi$, $v \in V_\pi$ and $\xi \in \Schw$. Given an open subset $\mathcal{U} \subset X(F)$ such that $\mathcal{U} \supset \partial X(F)$, we write $\mathcal{U}' := X(F) \smallsetminus \mathcal{U} = X^+(F) \smallsetminus \mathcal{U}$. If
	\begin{itemize}
		\item $\varphi_\lambda(v)$ is in $L^2(\mathcal{U})$ for $\Re(\lambda) \relgg{X} 0$;
		\item when $F$ is non-Archimedean, $\xi|_{\mathcal{U'}}=0$;
		\item when $F$ is Archimedean, $\varphi(v)$ has at most polynomial growth over $\mathcal{U}'$, i.e.\ there exists a semi-algebraic function $p$ on $X(F)$ such that $(1 + |p|)^{-1} \varphi(v)$ is in $L^2(\mathcal{U}')$;
	\end{itemize}
	then the zeta integrals $Z_{\lambda, \varphi}(v \otimes \xi)$ are convergent for $\Re(\lambda) \relgg{X} 0$.
\end{proposition}
\begin{proof}
	Decompose $\int_{X^+(F)} \varphi_\lambda \xi$ into $\int_{\mathcal{U}} + \int_{\mathcal{U}'}$. The first integral converges for $\Re(\lambda) \relgg{X} 0$ by our assumption, since $\Schw \subset L^2(X^+)$. The convergence for $\int_{\mathcal{U}'}$ follows immediately for non-Archimedean $F$.

	When $F$ is Archimedean, we write the integral over $\mathcal{U}'$ as
	\[ \int_{\mathcal{U'}} \frac{\varphi(v)}{1+|p|} \cdot (1+|p|) |f|^\lambda \xi. \]
	To show the convergence, it suffices to notice that $(1+|p|) |f|^\lambda \xi$ is $L^2$ if $\Re(\lambda) \relgeq{X} \lambda_0$, since $|f|^{2\lambda}\|\xi\|^2$ is of rapid decay.
\MyQED\end{proof}

\section{Model transitions: the local functional equation}\label{sec:model-transition}
To formulate the \emph{model transition}\index{model transition}, we fix $G$ and consider two sets of data, for $i=1,2$:
\begin{center}\begin{tabular}{cl}
	$X_i^+ \hookrightarrow X_i$ & affine spherical embedding \\
	$f^{(i)}_j \leftrightarrow \omega^{(i)}_j$, $1 \leq j \leq r^{(i)}$ & chosen relative invariants \\
	$\Lambda_{X_i} \subset \Lambda_i$ & monoids/lattices generated by $\omega^{(i)}_j$ \\
	$\Schw_i \subset C^\infty(X^+_i(F), \mathscr{E}_i)$ & Schwartz spaces
\end{tabular}\end{center}
each satisfying the Axioms \ref{axiom:geometric}, \ref{axiom:finiteness} and \ref{axiom:zeta}. Assume moreover that
\begin{gather}
	\Lambda_2 \subset \Lambda_1
\end{gather}
Hence the objects in Definition \ref{def:T} are related as
\begin{gather*}
	\mathcal{T}_2 \hookrightarrow \mathcal{T}_1, \\
	\mathcal{O}_1 \twoheadrightarrow \mathcal{O}_2.
\end{gather*}
We do not presume any inclusion between the monoids $\Lambda_{X_2}$ and $\Lambda_{X_1}$.

\begin{definition}\label{def:model-transition}
	By a \emph{model transition} between the data above, we mean a $G(F)$-equivariant isomorphism of topological vector spaces
	\[ \mathcal{F}: \Schw_2 \rightiso \Schw_1 \]
	that comes from the restriction of an isomorphism of unitary representations of $G(F)$
	\[ \mathcal{F}: L^2(X^+_2) \rightiso L^2(X^+_1). \]
	Again, the $L^2$-aspect will only be used in \S\ref{sec:connection-L2}.
\end{definition}

Some extra efforts are needed to cope with the possibility that $\Lambda_1 \neq \Lambda_2$. Define the multiplicative subset
\[ \mathcal{O}_1^\flat := \{a \in \mathcal{O}_1: a|_{\mathcal{T}_2} \neq 0 \} \subset \mathcal{O}_1 \smallsetminus \{0\} ; \]
denote the corresponding localization of $\mathcal{O}_1$ by $\mathcal{K}_1^\flat \subset \mathcal{K}_1$. The restriction homomorphism $\text{Res}^{\mathcal{T}_1}_{\mathcal{T}_2}: \mathcal{K}_1^\flat \to \mathcal{K}_2$ is well-defined.

Let $\pi$ be an irreducible nice representation of $G(F)$. Define the objects $\mathcal{N}^{(i)}_\pi$, $\mathcal{L}^{(i)}_\pi$, etc.\ as in \S\ref{sec:Schwartz}, for $i=1,2$. We define
\[ \mathcal{L}^{(1),\flat}_\pi = \varinjlim_{t \in \mathcal{O}_1^\flat} \mathcal{L}^{(1)}_{\pi,t}. \]
Pulling back by $\mathcal{F}$ followed by $\text{Res}^{\mathcal{T}_1}_{\mathcal{T}_2}$ gives a $\mathcal{K}_1^\flat$-linear map
\[ \mathcal{F}^\vee: \mathcal{L}^{(1), \flat}_\pi \to \mathcal{L}^{(2)}_\pi. \]

The construction below requires a further condition of \emph{avoidance of singularities}:
\begin{gather}\label{eqn:avoidance-singularities}
	\forall \varphi \in \mathcal{N}^{(1)}_\pi, \quad \widehat{Z}^{(1)}_{\pi, \varphi} \in \mathcal{L}^{(1), \flat}_\pi.
\end{gather}

\begin{definition}\label{def:lfe}\index{functional equation!local}\index{gamma(pi)@$\gamma(\pi)$}
	Let $\pi$ be an irreducible nice representation $\pi$ verifying \eqref{eqn:avoidance-singularities}. We say that the \emph{local functional equation} holds for  $\mathcal{F}$ at $\pi$, if there exists a $\mathcal{K}_1^\flat$-linear map
	\[ \gamma(\pi): \mathcal{N}^{(1)}_\pi \dotimes{\CC} \mathcal{K}_1^\flat \to \mathcal{N}^{(2)}_\pi \dotimes{\CC} \mathcal{K}_2 \]
	rendering the following diagram commutative
	\begin{equation}\label{eqn:gamma-def} \begin{tikzcd}
		\mathcal{L}^{(1), \flat}_\pi \arrow{r}{\mathcal{F}^\vee} & \mathcal{L}^{(2)}_\pi \\
		\mathcal{N}^{(1)}_\pi \otimes \mathcal{K}_1^\flat \arrow[hookrightarrow]{u} \arrow{r}[below]{\gamma(\pi)} & \mathcal{N}^{(2)}_\pi \otimes \mathcal{K}_2 \arrow[hookrightarrow]{u}
	\end{tikzcd}\end{equation}
	in which case $\gamma(\pi)$ is uniquely determined by applying Lemma \ref{prop:uniqueness-embedding} to $X_2$.
\end{definition}

We call $\gamma(\pi)$ the \emph{$\gamma$-factor} associated with the local functional equation.

By $\mathcal{K}_1^\flat$-linearity, $\gamma(\pi)$ is determined by its restriction $\mathcal{N}^{(1)}_\pi = \mathcal{N}^{(1)}_\pi \otimes 1 \to \mathcal{N}^{(2)}_\pi \otimes \mathcal{K}_2$, which may be regarded as a meromorphic/rational family of $\CC$-linear maps from $\mathcal{N}^{(1)}_\pi$ to $\mathcal{N}^{(2)}_\pi$ parametrized by $\mathcal{T}_2$. It makes sense to evaluate $\gamma(\pi)$ at a point of $\mathcal{T}_2$ off the singular locus. For the point corresponding to $\lambda \in \Lambda_{2, \CC}$, we denote the evaluation by\index{gamma(pi,lambda)@$\gamma(\pi,\lambda)$}
\begin{gather}\label{eqn:gamma-ev}
	\gamma(\pi, \lambda): \mathcal{N}^{(1)}_\pi \to \mathcal{N}^{(2)}_\pi.
\end{gather}

For $\lambda \in \Lambda_{2, \CC}$, recall the isomorphism $\mathcal{N}^{(i)}_\pi \rightiso \mathcal{N}^{(i)}_{\pi_\lambda}$ given by $\varphi \mapsto \varphi_\lambda = \varphi|f^{(i)}|^\lambda$, $i=1,2$. The $\gamma$-factor behaves well under this shift.

\begin{lemma}
	Suppose that $\lambda \in \Lambda_{2, \CC}$ and \eqref{eqn:avoidance-singularities} holds for $\pi$, then it also holds for $\pi_\lambda$. Moreover, the following diagram commutes:
	\[ \begin{tikzcd}
		\mathcal{N}^{(1)}_{\pi_\lambda} \arrow{r}[above]{\gamma(\pi_\lambda)} & \mathcal{N}^{(2)}_{\pi_\lambda} \otimes \mathcal{K}_2 \\
		\mathcal{N}^{(1)}_\pi \arrow{r}[below]{\gamma(\pi)} \arrow{u}[left]{\simeq} & \mathcal{N}^{(2)}_\pi \otimes \mathcal{K}_2 \arrow{u}[right]{\simeq}
	\end{tikzcd} \]
	where the vertical arrows are given by $\varphi \mapsto \varphi_\lambda$ and the automorphism of $\mathcal{K}_2$ of ``translation by $|\omega|^\lambda$'' in $\mathcal{T}_2$ (see Lemma \ref{prop:zeta-translation}). In the language of \eqref{eqn:gamma-ev}, this amounts to
	\[ \left[ \gamma(\pi, \lambda + \mu)(\varphi) \right]_\lambda = \gamma(\pi_\lambda, \mu)(\varphi_\lambda), \quad \varphi \in \mathcal{N}^{(2)}_\pi, \mu \in \Lambda_2. \]
\end{lemma}
\begin{proof}
	By Lemma \ref{prop:zeta-translation}, if $\varphi \in \mathcal{N}^{(1)}_\pi$ has image in $\mathcal{L}^{(1)}_{\pi,t}$ with $t \in \mathcal{O}_1^\flat$, then $\varphi_\lambda$ has image in $\mathcal{L}^{(1)}_{\pi, t'}$ where $t'$ is obtained from $t$ via a translation by $\lambda$, thus still in $\mathcal{O}^\flat_1$ since $\lambda \in \Lambda_{\CC,2}$. Furthermore, Lemma \ref{prop:zeta-translation} furnish the commutative diagrams
	\[ \begin{tikzcd}
		\mathcal{L}^{(1), \flat}_\pi \arrow{r}[above]{\sim} & \mathcal{L}^{(1), \flat}_{\pi_\lambda} \\
		\mathcal{N}^{(1)}_\pi \arrow[hookrightarrow]{u} \arrow{r}[below]{\sim} & \mathcal{N}^{(1)}_{\pi_\lambda} \arrow[hookrightarrow]{u}
	\end{tikzcd} \quad
	\begin{tikzcd}
		\mathcal{L}^{(2)}_\pi \arrow{r}[above]{\sim} & \mathcal{L}^{(2)}_{\pi_\lambda} \\
		\mathcal{N}^{(2)}_\pi \arrow[hookrightarrow]{u} \arrow{r}[below]{\sim} & \mathcal{N}^{(2)}_{\pi_\lambda} \arrow[hookrightarrow]{u}
	\end{tikzcd} \]
	The upper rows of both diagrams are connected by a roof by applying $\mathcal{F}^\vee$, which is obviously commutative. On the other hand, the left (resp.\  right) vertical arrows are connected by the commutative diagram \eqref{eqn:gamma-def} for $\pi$ (resp.\  $\pi_\lambda$). Assembling all these and bearing in mind that the vertical arrows are all injective, we obtain the required commutative diagram in the bottom.
\MyQED\end{proof}

The effect of changing relative invariants is easy to analyze.
\begin{lemma}\label{prop:change-rel-inv}
	Suppose that $f^{(i)}_j$ is changed to $c^{(i)}_j f^{(i)}_j$, for $i=1,2$ and $1 \leq j \leq r^{(i)}$. Define $\left| c^{(i)} \right|^\lambda = \prod_j \left| c^{(i)}_j \right|^{\lambda_j}$ if $\lambda = \sum_j \lambda_j \omega^{(i)}_j \in \Lambda_{i,\CC}$. Then $\gamma(\pi)$ will be replaced by
	\[ \left( \text{mult. by } |c^{(2)}|^\bullet \right) \circ \gamma(\pi) \circ \left( \text{mult. by } |c^{(1)}|^\bullet \right). \]
	Here $|c^{(i)}|^\bullet \in \mathcal{O}_i^\times$ stands for the function $|\omega|^\lambda \mapsto |c^{(i)}|^\lambda$ on $\mathcal{T}_i$. In particular, the evaluation $\gamma(\pi,0)$ at $\lambda=0$ is independent of the choice of relative invariants whenever it is well-defined.
\end{lemma}

\begin{remark}
	When $\Lambda_1 = \Lambda_2$, the superscripts $\flat$ will disappear and $\mathcal{K}_1 = \mathcal{K}_2 =: \mathcal{K}$, so that $\gamma(\pi)$ becomes a linear map between finite-dimensional $\mathcal{K}$-vector spaces. Once bases are chosen, it may also be called the \emph{$\gamma$-matrix}, or even a single meromorphic/rational function when both sides have dimension one.

	Furthermore, the roles of $X_1$, $X_2$ are symmetric in this case, hence $\gamma(\pi)$ is a $\mathcal{K}$-isomorphism. It follows that $\mathcal{N}^{(1)}_\pi$ and $\mathcal{N}^{(2)}_\pi$ are isomorphic. Indeed, one can exploit leading terms of $\gamma(\pi,\lambda)$ and $\gamma(\pi,\lambda)^{-1}$ at $\lambda=0$.
\end{remark}

\section{Connection with \texorpdfstring{$L^2$}{L2} theory}\label{sec:connection-L2}
Let $X^+ \hookrightarrow X$, $\Schw \subset L^2(X^+)$ be as in \S\ref{sec:Schwartz}, together with all the axioms therein. Choose a Plancherel measure $\mu$ for the spectral decomposition of $L^2(X^+)$.

Let $\tau \in \Pi_\text{unit}(G(F))$ be an irreducible unitary representation, realized on a Hilbert space. Its contragredient $\check{\tau}$ can be identified with its complex conjugate $\bar{\tau}$. Throughout this section, we adopt systematically the convention
\[ \pi := \check{\tau}^\infty. \]
The $L^2$-multiplicity space $\mathcal{M}_\tau$ defined in Theorem \ref{prop:direct-integral} for $\mu$-almost all $\tau$ is actually a Hilbert space. Put $\mathcal{M}_\tau^\vee := \Hom(\mathcal{M}_\tau, \CC)$.

\begin{proposition}
	For $\mu$-almost all $\tau \in \Pi_\mathrm{unit}(G(F))$, the space $\mathcal{M}_\tau^\vee$ is canonically a subspace of $\mathcal{N}_\pi$, hence of finite dimension.
\end{proposition}
\begin{proof}
	Apply Proposition \ref{prop:multiplicity-spaces} to $\tau$ and to the space $C^\infty_c(X^+)$ of test functions. Theorem \ref{prop:C-infty} asserts that
	\[ \Hom_{G(F)}(C^\infty_c(X^+), \tau) \simeq \Hom_{G(F)}(\pi, C^\infty(X^+)) = \mathcal{N}_\pi \]
	by adjunction. Finiteness results from Axiom \ref{axiom:finiteness}.
\MyQED\end{proof}

\begin{corollary}\label{prop:disintegration-mult-1}
	If $\dim \mathcal{N}_\pi \leq 1$ for $\mu$-almost all irreducible unitary representation $\tau$ of $G(F)$, then $\mathcal{M}_\tau$ is one-dimensional for $\mu$-almost all $\tau \in \Pi_\mathrm{unit}(G(F))$.
\end{corollary}

On the other hand, by Axiom \ref{axiom:zeta} we may also plug $\Schw$ into Proposition \ref{prop:multiplicity-spaces}. This yields
\begin{gather}\label{eqn:S-mult-space}
	\mathcal{M}^\vee_\tau \hookrightarrow \Hom_{G(F)}(\Schw, \tau) \quad \text{for $\mu$-almost all $\tau$.}
\end{gather}
The continuous inclusion $C^\infty_c(X^+) \hookrightarrow \Schw$ and Theorem \ref{prop:C-infty} give the commutative diagram
\begin{equation}\label{eqn:N-vs-S}\begin{tikzcd}
	\Hom_{G(F)}(\Schw, \tau) \arrow{rr} & & \mathcal{N}_\pi \\
	& \mathcal{M}_\tau^\vee \arrow[hookrightarrow]{lu} \arrow[hookrightarrow]{ru} &
\end{tikzcd}. \end{equation}

\begin{lemma}\label{prop:vee-inj}
	For any $\tau \in \Pi_{\mathrm{unit}}(G(F))$, taking adjoint gives an injective map
	\[ \Hom_{G(F)}(\Schw, \tau) \hookrightarrow \Hom_{G(F)}(\pi, \Schw^\vee). \]
\end{lemma}
\begin{proof}
	This map $\alpha \mapsto \alpha^\vee$ is characterized by
	\[ \angles{\alpha^\vee(v), \xi} = \angles{v, \alpha(\xi)}_{\bar{\tau} \otimes \tau \to \CC} = (\alpha(\xi) | v)_\tau \]
	for all $v$ in the underlying space of $\pi = \check{\tau}^\infty = \bar{\tau}^\infty$ and all $\xi \in \Schw$. It remains to recall that $\bar{\tau}^\infty \hookrightarrow \bar{\tau}$ is continuous of dense image.
\MyQED\end{proof}

The following hypothesis will be invoked in the comparison of Plancherel decompositions. Detailed discussions will be deferred to the end of this section.
\begin{hypothesis}\label{hyp:zeta-0}
	Assume that
	\begin{itemize}
		\item for $\mu$-almost all $\tau$, the zeta integrals $Z_{\lambda, \varphi}$ for $\varphi \in \mathcal{M}_\tau^\vee$ are holomorphic at $\lambda=0$, where we regard $\mathcal{M}_\tau^\vee$ as a subspace of $\mathcal{N}_\pi$;
		\item for every $\varphi \in \mathcal{M}_\tau^\vee$, the element of $\Hom_{G(F)}(\pi, \Schw^\vee)$ arising from $Z_{\lambda=0, \varphi}$ coincides with the image of $\mathcal{M}_\tau^\vee \hookrightarrow \Hom_{G(F)}(\Schw, \tau) \hookrightarrow \Hom_{G(F)}(\pi, \Schw^\vee)$.
	\end{itemize}
\end{hypothesis}

Henceforth, we work with two sets of data $X_i^+ \hookrightarrow X_i$, $\Schw_i$, etc. ($i=1,2$) as in \S\ref{sec:model-transition}, together with the model transition $\mathcal{F}: \Schw_2 \rightiso \Schw_1$ (recall Definition \ref{def:model-transition}). On the $L^2$-aspect, the existence of $\mathcal{F}$ has several immediate consequences.
\begin{enumerate}
	\item By the uniqueness part of Theorem \ref{prop:direct-integral}, we may use \emph{the same} Plancherel measure $\mu$ for decomposing both $L^2(X^+_1)$ and $L^2(X^+_2)$. In forming these spectral decompositions, we fix representatives of elements in $\Pi_\text{unit}(G(F))$ so that the multiplicity spaces $\mathcal{M}^{(1)}_\tau$, $\mathcal{M}^{(2)}_\tau$ are also uniquely defined as finite-dimensional Hilbert spaces, up to isometries.
	\item Corollary \ref{prop:disintegration-F} now gives a measurable family of isometries $\eta(\tau): \mathcal{M}^{(2)}_\tau \rightiso \mathcal{M}^{(1)}_\tau$ for $\mu$-almost all $\tau$, such that $\mathcal{F}$ disintegrates into the decomposable operator
	\begin{multline*}
		\int^\oplus_{\Pi_\text{unit}(G(F))} (\identity \otimes \eta(\tau)) \dd\mu(\tau): \\
		\int^\oplus_{\Pi_\text{unit}(G(F))} \tau \otimes \mathcal{M}^{(2)}_\tau \dd\mu(\tau) \rightiso  \int^\oplus_{\Pi_\text{unit}(G(F))} \tau \otimes \mathcal{M}^{(1)}_\tau \dd\mu(\tau).
	\end{multline*}
	\item Taking adjoint gives a measurable family  $\eta(\tau)^\vee: \mathcal{M}^{(1),\vee}_\tau \rightiso \mathcal{M}^{(2),\vee}_\tau$. It consists of isometries if we endow these spaces with their usual Hilbert structure. One can describe $\eta(\tau)^\vee$ even more explicitly. In view of the canonical inclusions \eqref{eqn:S-mult-space}, \eqref{eqn:N-vs-S}, and the fact that $\int^\oplus (\identity \otimes \eta(\tau)) \dd\mu(\tau)$ comes from a transport of structure by $\mathcal{F}: \Schw_2 \rightiso \Schw_1$ which extends to $L^2(X_2^+) \rightiso L^2(X_1^+)$, we arrive at the commutative diagram
	\begin{equation}\label{eqn:F-transport} \begin{tikzcd}
		\mathcal{N}^{(1)}_\pi & \mathcal{N}^{(2)}_\pi \\
		\Hom_{G(F)}(\Schw_1, \tau) \arrow{r}[above]{\mathcal{F}^*}[below]{\sim} \arrow{u} & \Hom_{G(F)}(\Schw_2, \tau) \arrow{u} \\
		\mathcal{M}^{(1),\vee}_\tau \arrow{r}[below]{\eta(\tau)^\vee}[above]{\sim} \arrow[hookrightarrow]{u} \arrow[hookrightarrow, bend left=90]{uu} &  \mathcal{M}^{(2),\vee}_\tau  \arrow[hookrightarrow]{u} \arrow[hookrightarrow, bend right=90]{uu}
	\end{tikzcd}\end{equation}
	for $\mu$-almost all $\tau$. Here $\mathcal{F}^*$ stands for pull-back by $\mathcal{F}$.
	
	Denote by $\mathcal{F}^\vee_*$ the push-forward via $\mathcal{F}^\vee: \Schw_1^\vee \rightiso \Schw_2^\vee$. The diagram
	\begin{equation}\begin{tikzcd}\label{eqn:F-Fvee}
		\Hom_{G(F)}(\Schw_1, \tau) \arrow{r}[above]{\mathcal{F}^*} \arrow[hookrightarrow]{d} & \Hom_{G(F)}(\Schw_2, \tau) \arrow[hookrightarrow]{d} \\
		\Hom_{G(F)}(\pi, \Schw_1^\vee) \arrow{r}[below]{\mathcal{F}^\vee_*} &  \Hom_{G(F)}(\pi, \Schw_2^\vee)
	\end{tikzcd}\end{equation}
	is easily seen to commute, the injectivity being assured by Lemma \ref{prop:vee-inj}.
	\item In view of \eqref{eqn:F-transport} and \eqref{eqn:F-Fvee}, our goal may be rephrased as:
		\begin{center}
			See the effect of $\mathcal{F}^\vee_*$ inside $\mathcal{N}^{(i)}_\pi$ using zeta integrals.
		\end{center}
		This is exactly the content of Hypothesis \ref{hyp:zeta-0}. Granting the local functional equation (Definition \ref{def:lfe}), $\eta(\tau)^\vee$ can be obtained by
		\begin{itemize}
			\item first restrict $\mathcal{F}^\vee_*: \Hom_{G(F)}(\pi_\lambda, \Schw_1^\vee) \to \Hom_{G(F)}(\pi_\lambda, \Schw_2^\vee)$ (for $\lambda \in \Lambda_{2,\CC}$) to $\gamma(\pi): \mathcal{N}^{(1)}_\pi \otimes \mathcal{K}_1^\flat \to \mathcal{N}^{(2)}_\pi \otimes \mathcal{K}_2$ using the injections of Lemma \ref{prop:uniqueness-embedding};
			\item secondly, restrict the domain of $\gamma(\pi)$ to $\mathcal{M}^{(1),\vee}_\tau = \mathcal{M}^{(1),\vee}_\tau \otimes 1$, thereby obtaining a meromorphic family in $\lambda \in \Lambda_{2,\CC}$ of linear maps $\mathcal{M}^{(1),\vee}_\tau \to \mathcal{N}^{(2)}_\pi$;
			\item the Hypothesis \ref{hyp:zeta-0} now implies that $\gamma(\pi)$ restricted to $\mathcal{M}^{(1),\vee}_\tau$ and evaluated at $\lambda=0$ gives $\eta(\tau)^\vee$.
		\end{itemize}
	\item Summing up, $\gamma(\pi)$ disintegrates $\mathcal{F}: L^2(X^+_2) \rightiso L^2(X^+_1)$ under Hypothesis \ref{hyp:zeta-0} and the local functional equation. Note that the evaluation at $\lambda=0$ is independent of the choice of relative invariants, by Lemma \ref{prop:change-rel-inv}.
\end{enumerate}

\begin{corollary}
	The model transition $\mathcal{F}$ is determined by its $\gamma$-factors.
\end{corollary}
\begin{proof}
	Indeed, $\mathcal{F}$ is determined by the associated isometry $L^2(X^+_2) \rightiso L^2(X^+_1)$.
\MyQED\end{proof}

\begin{corollary}
	Under the conditions of Corollary \ref{prop:disintegration-mult-1}, the $\gamma$-factors $\gamma(\pi)$ satisfy $|\gamma(\pi,0)|=1$ for $\mu$-almost all $\tau$.
\end{corollary}

In the rest of this section, we shall give some heuristics in support of Hypothesis \ref{hyp:zeta-0}. Retain the notations thereof, and recall that
\[ \Schw_\lambda := |f|^\lambda \Schw \hookrightarrow L^2(X^+), \quad \Re(\lambda) \relgeq{X} 0. \]
Alternatively, we may work with the family of non-equivariant embeddings $\alpha_\lambda: \Schw \to L^2(X^+)$ given by $\alpha_\lambda(\xi) = |f|^\lambda \xi$.

Since $\Schw_\lambda$ is still nuclear, the foregoing constructions carry over. Write $\tau_{-\lambda} := \tau \otimes |\omega|^{-\lambda}$ for $\tau \in \Pi_\text{unit}(G(F))$; this representation is still realized on a Hilbert space, albeit non-unitary in general. The arguments below apply to $\mu$-almost all $\tau$ as usual. It is routine to check the commutativity of the diagram below (except the dashed arrow):
\begin{equation}\label{eqn:dashed-diagram} \begin{tikzcd}[row sep=large]
	& \mathcal{M}_\tau^\vee \arrow[hookrightarrow]{d} \arrow[hookrightarrow]{rd} & \\
	\Hom_{G(F)}(\pi, (\Schw_\lambda)^\vee) \arrow{d}[left]{\simeq} & \Hom_{G(F)}(\Schw_\lambda, \tau) \arrow{r} \arrow[hookrightarrow]{l} & \mathcal{N}_\pi \arrow[dashed]{lld} \arrow{d}[left]{\simeq} \\
	\Hom_{G(F)}(\pi_\lambda, \Schw^\vee) & \Hom_{G(F)}(\Schw, \tau_{-\lambda}) \arrow[hookrightarrow]{l} \arrow{r} \arrow[crossing over, leftarrow]{u}[right]{\simeq} & \mathcal{N}_{\pi_\lambda}
\end{tikzcd}\end{equation}
where the rightmost vertical arrow is $\varphi \mapsto \varphi_\lambda$, and the other vertical arrows are the obvious ones. The dashed arrow is given by zeta integrals whenever they are defined at $\lambda$.

In order to proceed, we impose two assumptions for $\mu$-almost all $\tau$.
\begin{itemize}
	\item The family of maps $\mathcal{M}^\vee_\tau \to \Hom_{G(F)}(\Schw, \tau_{-\lambda})$ can be re-encoded into elements $\beta_{\lambda,\tau} \in \Hom_{G(F)}(\Schw, \tau_{-\lambda} \otimes \mathcal{M}_\tau)$, as $\dim_{\CC} \mathcal{M}_\tau < \infty$. Note that $\tau_{-\lambda}$ and $\tau$ share the same underlying Hilbert space, so it makes sense to talk about holomorphy in the sense of Definition \ref{def:L2-holomorphy}. We assume
		\begin{equation}\label{eqn:L2-holomorphy}\begin{gathered}
			\text{The family }\; \mathcal{M}^\vee_\tau \to \Hom_{G(F)}(\Schw, \tau_{-\lambda}) \text{ is holomorphic in $\lambda$ when }\; \Re(\lambda) \relgeq{X} 0.
		\end{gathered}\end{equation}
		Morally, this should be a consequence of the holomorphy of $\alpha_\lambda: \Schw \to L^2(X^+)$ (or some strengthening thereof); the author is unable to give a rigorous argument thus far.
	\item The restriction map is assumed to satisfy
		\begin{equation}\label{eqn:generic-inj}\begin{gathered}
			\Hom_{G(F)}(\pi_\lambda, \Schw^\vee) \hookrightarrow \Hom_{G(F)}(\pi_\lambda, C^\infty_c(X^+)^\vee) \\
			\text{when } \lambda \in \text{ some open subset } \mathcal{U} \neq \emptyset, \; \Re(\lambda) \relgg{X} 0.
		\end{gathered}\end{equation}
		This is intuitively plausible, since highly ``$X$-positive'' twists tend to suppress contributions from $\partial X$. We will establish \eqref{eqn:generic-inj} in the prehomogeneous, non-Archimedean case in Theorem \ref{prop:pvs-T_pi-isom} under certain geometric assumptions.
\end{itemize}

In the diagram \eqref{eqn:dashed-diagram}, composition of $\mathcal{M}_\tau^\vee \hookrightarrow \mathcal{N}_\pi$ with the zeta integral $\dashrightarrow$ furnishes a meromorphic family $\theta_\lambda: \mathcal{M}^\vee_\tau \to \Hom_{G(F)}(\pi_\lambda, \Schw^\vee)$. We want to compare it with the holomorphic family $\theta'_\lambda: \mathcal{M}_\tau^\vee \hookrightarrow \Hom_{G(F)}(\pi_\lambda, \Schw^\vee)$ derived from $\mathcal{M}^\vee_\tau \hookrightarrow \Hom_{G(F)}(\Schw_\lambda, \tau)$, under the assumption \eqref{eqn:L2-holomorphy}.

Recall that $\mathcal{N}_{\pi_\lambda} \subset \Hom_{G(F)}(\pi_\lambda, C^\infty_c(X^+)^\vee)$. By Remark \ref{rem:zeta-distribution}, the image of $\theta_\lambda$ in $\Hom_{G(F)}(\pi_\lambda, C^\infty_c(X^+)^\vee)$ equals $\varphi_\lambda \in \mathcal{N}_{\pi_\lambda}$ when $\Re(\lambda) \relgg{X} 0$. An easy diagram-chasing shows that $\theta'_\lambda$ also maps to $\varphi_\lambda$ in $\Hom_{G(F)}(\pi_\lambda, C^\infty_c(X^+)^\vee)$. By assumption \eqref{eqn:generic-inj} we have $\theta_\lambda = \theta'_\lambda$ for $\lambda$ in some nonempty open subset. Hence $\theta_\lambda = \theta'_\lambda$ for all $\Re(\lambda) \relgeq{X} 0$ by meromorphy.


Consequently, $Z_{\lambda, \varphi}$ is well-defined for all $\varphi \in \mathcal{M}^\vee_\tau$ whenever $\Re(\lambda) \relgeq{X} 0$. Indeed, it coincides with the map $\mathcal{M}^\vee_\tau \to \Hom_{G(F)}(\Schw_\lambda, \tau)$ that defines $\alpha_\lambda$ pointwise. All in all:

\begin{proposition}
	The Hypothesis \ref{hyp:zeta-0} holds under the assumptions \eqref{eqn:L2-holomorphy}, \eqref{eqn:generic-inj}.
\end{proposition}
\chapter{Convergence of some zeta integrals}\label{sec:convergence}
We will need the Hypothesis \ref{hyp:eigendensity} on the existence of eigenmeasures. This is always met in practice.

\section{Cellular decompositions}\label{sec:cellular-decomp}
In this section we consider
\begin{itemize}
	\item a finite-dimensional vector space $V$,
	\item an affine space $\mathcal{A}$ under $V$,
	\item a strictly convex cone $\mathcal{F}^\text{max}$ in $V$.
\end{itemize}
Furthermore, these objects are assumed to be rational, i.e.\ defined over $\Q$, and the same assumption pertains to all the constructions below.

By an \emph{embedded polytopal complex} in $\mathcal{A}$, we mean a compact subset of $\mathcal{A}$ obtained by gluing a finite family of polytopes in $\mathcal{A}$ intersecting in faces. In a similar vein, one may defined \emph{embedded polyhedral complexes} as well. For a rigorous definition we refer to \cite[Definition 1.39]{BG09}. We will not distinguish the embedded complex and its support, until necessary.

We shall make use of the decomposition \cite[Proposition 1.28]{BG09}: let $\mathcal{E}$ be a polyhedron in $\mathcal{A}$ that contains no affine subspaces of $\mathcal{A}$. There is then a decomposition into a Minkowski sum
\begin{gather}\label{eqn:rec-bottom}
	\mathcal{E} = B + \mathcal{F}
\end{gather}
where
\begin{itemize}
	\item $\mathcal{F} := \rec(\mathcal{E}) \subset V$ is the \emph{recession cone}\index{recession-cone@$\text{rec}(\mathcal{E})$}\index{recession cone}, i.e.\ the cone $\{v \in V: \mathcal{E} + \R_{\geq 0}v \subset \mathcal{E} \}$ of ``directions towards infinity'', and
	\item $B$ is an embedded polytopal complex --- there is actually a canonical choice, namely taking $B$ to be the \emph{bottom} of $\mathcal{E}$ (i.e.\ the union its bounded faces).
\end{itemize}
Such decompositions behave well when polyhedra are glued: if $\mathcal{E}$ is a face of $\mathcal{E}'$, then $\mathcal{F} \subset \mathcal{F}'$ and $B \subset B'$ by their definitions. Moreover, in any decomposition of the form \eqref{eqn:rec-bottom} we have $\mathcal{F} = \rec(\mathcal{E})$. Also notice that $\rec(\mathcal{E})$ can be described from presentations of $\mathcal{E}$:
\[ \mathcal{E} = \bigcap_{\alpha \in A} \{\alpha \geq 0\} \implies \rec(\mathcal{E}) = \bigcap_{\alpha \in A} \{\vec{\alpha} \geq 0\}, \]
where $A$ is a finite set of affine forms on $\mathcal{A}$ and $\vec{\alpha} \in V^\vee$ stands for the vectorial part of $\alpha$; see \cite[1.C]{BG09}. Let $\mathcal{F}$ be a face of $\mathcal{F}^\text{max}$ and $B \subset \mathcal{A}$ be a polytope. Suppose that we are given a predicate ``$B$ being \emph{$\mathcal{F}$-deep}'', subject to the conditions
\begin{enumerate}[(i)]
	\item if $B$ is $\mathcal{F}$-deep and $w \in \mathcal{F}$, then $B+w$ is also $\mathcal{F}$-deep;
	\item for every $B$ and $\mathcal{F}$, there exists $v \in \mathcal{F}$ such that $B+v$ is $\mathcal{F}$-deep;
	\item any $B$ is $\{0\}$-deep.
\end{enumerate}

It follows that being $\mathcal{F}$-deep can be viewed as a property of $B + \mathcal{F}$, which turns out to be an embedded polyhedral complex in $\mathcal{A}$: one reduces to the case where $B$ is a polytope, and use the fact \cite[Theorem 1.30]{BG09} that the Minkowski sum of polyhedra is still a polyhedron.

\begin{proposition}\label{prop:cellular-decomp}\index{cellular decomposition}
	Given a polyhedron $\mathcal{E}^\mathrm{max}$ in $\mathcal{A}$ containing no affine subspaces, such that $\mathrm{rec}(\mathcal{E}^\mathrm{max}) = \mathcal{F}^\mathrm{max}$, there exists a family of embedded polytopal complexes $B_{\mathcal{F}} \subset \mathcal{A}$ indexed by certain faces $\mathcal{F}$ of $\mathcal{F}^\mathrm{max}$, which verifies
	\begin{itemize}
		\item $B_{\mathcal{F}}$ is $\mathcal{F}$-deep for every $\mathcal{F}$,
		\item the following decomposition holds in $\mathcal{A}$:
			\[ \mathcal{E}^{\mathrm{max}} = \bigcup_{\mathcal{F}} \left( B_{\mathcal{F}} + \mathcal{F} \right); \]
		\item the subsets $B_{\mathcal{F}} + \mathcal{F}$ are embedded polyhedral complexes in $\mathcal{A}$, and they intersect only in faces.
	\end{itemize}
\end{proposition}

A typical situation is depicted below (cf.\ Example \ref{eg:typical-cone}): here we take $\mathcal{E}^\text{max} = \{\text{pt}\} + \mathcal{F}^\text{max}$.
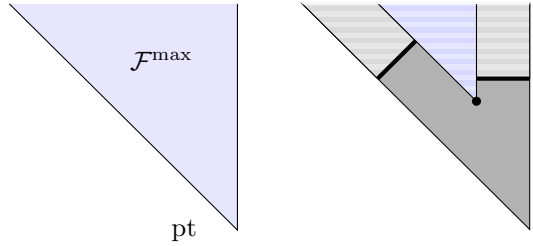
\begin{center}\begin{tikzpicture}[baseline]
	\fill[fill=blue!10!white] (-3,3) -- (0,0) -- (0,3) -- (-3,3);
	\draw (-3,3) -- (0,0) -- (0,3);
	\node at (-0.7,0) {pt};
	\node at (-1,2) [above] {$\mathcal{F}^\text{max}$};
	\end{tikzpicture} \qquad \begin{tikzpicture}
	\filldraw[fill=gray!60!white] (-3,3) -- (0,0) -- (0,3);
	\filldraw[pattern=horizontal lines light gray] (-2,2) -- (-3,3) -- (-2,3) -- (-1.5,2.5) -- (-2,2);  \draw[ultra thick] (-1.5,2.5) -- (-2,2);
	\filldraw[pattern=horizontal lines light gray] (-0.7,2) -- (0,2) -- (0,3) -- (-0.7,3) -- (-0.7,2);  \draw[ultra thick] (-0.7,2) -- (0,2);
	\filldraw[pattern=horizontal lines light blue] (-0.7,1.7) -- (-0.7,3) -- (-2,3) -- (-0.7,1.7);		\filldraw[black] (-0.7,1.7) circle (1.5pt);
	\draw[white, thick] (-3,3) -- (0,3);
\end{tikzpicture}\end{center}
In this case, the polytopal complexes $B_{\mathcal{F}}$ are
the two segments \begin{tikzpicture}[baseline=-3pt] \draw[ultra thick] (-1,0) -- (0,0); \end{tikzpicture},
the dot \begin{tikzpicture}[baseline=-3pt] \filldraw[black] (0,0) circle (1.5pt); \end{tikzpicture} and
the V-shaped area \begin{tikzpicture} \filldraw[fill=gray!60!white] (0,0) rectangle (0.5,0.5); \end{tikzpicture}.

\begin{proof}
	We shall begin with a decomposition
	\[ \mathcal{E}^\mathrm{max} = B^\mathrm{max} + \mathcal{F}^\mathrm{max} \]
	of type \eqref{eqn:rec-bottom}, and argue by induction on $\dim \mathcal{F}^\text{max}$.
	
	If $\mathcal{F}^\text{max} = \{0\}$, then $B_{\mathcal{F}^\text{max}} := B^\text{max}$ works well. In dimension $\geq 1$, take $v^\text{max} \in \relint(\mathcal{F}^\text{max})$ such that $B^\text{max} + v^\text{max}$ is $\mathcal{F}^\text{max}$-deep. There exists a decomposition
	\[ \mathcal{E}^\text{max} \smallsetminus \relint\left( \mathcal{E}^\text{max} + v^\text{max} \right) = \bigcup_{\mathcal{E}: \text{polyhedra}} \mathcal{E} \]
	as an embedded polyhedral complex in $\mathcal{A}$. Indeed, the presentations of polyhedra may be written as
	\begin{align*}
		\mathcal{E}^\text{max} & = \bigcap_{\alpha \in A} \{ \alpha \geq 0 \}, \\
		\relint\left( \mathcal{E}^\text{max} + v^\text{max} \right) & = \bigcap_{\alpha \in A} \{ \alpha > c_\alpha \}
	\end{align*}
	for some finite set $A$ of affine forms on $\mathcal{A}$, where $c_\alpha = \vec{\alpha}(v^\text{max})$. Assume that the presentation above for $\mathcal{E}^\text{max}$ is irredundant. It is then routine to decompose the difference set into polyhedra $\mathcal{E}$: in fact, they are defined by affine forms with vectorial parts equal to $\vec{\alpha}$, for various $\alpha \in A$.

	The construction above also entails that for each $\mathcal{E}$,
	\begin{itemize}
		\item $\mathcal{E}$ contains no affine subspace of $\mathcal{A}$;
		\item the recession cones $\mathcal{F} := \rec(\mathcal{E})$ are faces of $\mathcal{F}^\text{max}$;
		\item furthermore, $\dim \mathcal{F} < \dim \mathcal{F}^\text{max}$.
	\end{itemize}
	The second assertion above follows from the description of recession cones via presentations. Let us justify the last assertion. Assume on the contrary that $\mathcal{F}$ intersects $\relint(\mathcal{F}^\text{max})$, then for every $w \in \mathcal{E}$, there exists $v \in \mathcal{F}$ which is so deep in $\mathcal{F}$ that
	\[ w - v^\text{max} + v \in \relint\left( \mathcal{E}^\text{max} \right), \]
	equivalently,
	\[ w + v \in \mathcal{E} \cap \relint(\mathcal{E}^\text{max} + v^\text{max}) \]
	which leads to contradiction.

	All in all,
	\[ \mathcal{E}^\text{max} = \left( \bigcup {\mathcal{E}} \right) \cup \underbracket{ \left( \mathcal{E}^\text{max} + v^\text{max} \right)}_{\mathcal{F}^\text{max} \text{-deep}}. \]
	This is a decomposition as an embedded polyhedral complex. Since $\mathcal{F}$ is a proper face of $\mathcal{F}^\text{max}$, by induction hypothesis each $\mathcal{E}$ can be decomposed into the form $\bigcup_{\mathcal{F}' \subset \mathcal{F}} (B_{\mathcal{F}'} + \mathcal{F}')$ where $B_{\mathcal{F}'}$ is $\mathcal{F}'$-deep. The proof can be completed by collecting terms according to $\mathcal{F}'$ and gluing.
\MyQED\end{proof}

We take this opportunity to record an obvious result here.
\begin{proposition}\label{prop:face-trivial}
	Let $P \subset Q$ be polyhedra in an affine space $\mathcal{A}$. There exists a unique face $\mathcal{F}$ of $Q$ such that $\relint(P) \subset \relint(\mathcal{F})$.
\end{proposition}
\begin{proof}
	Establish the existence first. We may assume $\relint(P) \not\subset \relint(Q)$. We contend that $\relint(P) \subset \partial Q$. Indeed, if their relative interiors intersect, then $P$ and $Q$ have the same affine closure in $\mathcal{A}$. Since the relative interior equals the topological interior taken inside the affine closure, this would lead to $\relint(P) \subset \relint(Q)$ which is contradictory. Thus we have
	\[ \relint(P) \subset \partial Q = \bigcup_{H: \text{hyperplanes}} Q \cap H \quad \text{(finite union)}. \]
	The affine closure of $P$ can thus be covered by finitely many hyperplanes $H \subset \mathcal{A}$, from which we see $P \subset Q \cap H$ for some $H$. One can now proceed by induction on $\dim \mathcal{A}$.
	
	To show the uniqueness, suppose that $\mathcal{F}$, $\mathcal{F}'$ both satisfy our requirement. Then $\relint(\mathcal{F}) \cap \relint(\mathcal{F}') \neq \emptyset$, hence $\mathcal{F} = \mathcal{F}'$ since they are faces of $Q$.
\MyQED\end{proof}

\section{Smooth asymptotics}
Let $F$ be a non-Archimedean local field of characteristic zero. Consider a spherical homogeneous $G$-space $X^+$, with open $B$-orbit $\mathring{X}$, etc. Keep the notations from \S\ref{sec:geometric-bg}.

We begin by reviewing the \emph{exponential maps} constructed in \cite[\S 4.3]{SV17}. Let
\begin{itemize}
	\item $X^+ \hookrightarrow \bar{X}$ be a smooth complete toroidal compactification,
	\item $\Theta \subset \Delta_{X^+}$ corresponding to a $G$-orbit closure $Z$ in $\bar{X}$,
	\item $J \subset G(F)$ be an open compact subgroup.
\end{itemize}
In \textit{loc.\ cit.}, a canonical \emph{germ} of $F$-analytic morphisms from some open neighborhood of $Z(F)/J$ in $(N_Z \bar{X})(F)/J$ to $\bar{X}(F)/J$ is constructed, which is denoted by $\SVexp_{\Theta,J}$\index{exp@$\SVexp_\Theta$}. By a germ we mean an equivalence class of pairs $(U_Z, \theta)$, where
\begin{compactitem}
	\item $U_Z$ is an open neighborhood of $Z(F)/J$ in $(N_Z \bar{X})(F)/J$,
	\item $\theta: U_Z \to \bar{X}(F)/J$ is $F$-analytic,
	\item $(U_Z, \theta) \sim (U'_Z, \theta')$ if $\theta$ agrees with $\theta'$ over $U_Z \cap U'_Z$.
\end{compactitem}
The germ $\exp_{\Theta, J}$ is eventually equivariant: here ``equivariant'' refers to the the Hecke algebra $\mathcal{H}(G(F) \sslash J)$, and we say ``eventually'' because equalities are taken in the sense of germs. As usual, the construction of representatives of $\SVexp_{\Theta,J}$ boils down to the case of toric varieties via the Local Structure Theorem in Remark \ref{rem:local-structure-thm}; in particular, $\SVexp_{\Theta,J}$ is compatible with \eqref{eqn:iden-B-orbits}; cf.\ the proof of \cite[Lemma 4.3.2]{SV17}.

In \textit{loc.\ cit.} one prefers to work in $X_\Theta$ instead of $N_Z \bar{X}$, so that the open neighborhoods $U_Z$ are replaced by open neighborhoods $N_\Theta$ of the ``$\Theta$-infinity'' $\infty_\Theta$ inside $X_\Theta(F)/J$\index{infty-Theta@$\infty_\Theta$}. Specifically, we define $\infty_\Theta := \overline{\bigcup_{Z \leftrightarrow \Theta} Z} \subset \bar{X}$; a neighborhood of $\infty_\Theta$ signifies the intersection with $X^+(F)$ of a neighborhood of $\infty_\Theta(F)$ in $\bar{X}(F)$. This is justified as different $Z$ yields isomorphic $X_\Theta$, and it is shown in \cite[Proposition 4.3.3]{SV17} that $\exp_{\Theta,J}$ is independent of the choice of $\bar{X}$ and $Z$. We shall make free use of this formalism. To work with arbitrarily small $J$, one may take $\varprojlim_J$ and obtains a distinguished projective system of germs denoted by $\SVexp_\Theta$.

The next step is to discuss the smooth asymptotics map of \cite[\S 5]{SV17}. We begin by studying the effect of $\SVexp_\Theta$ on densities (Definition \ref{def:density}). Define the sheaves $\Omega_{X_\Theta/F}^\text{max}$, $\Omega_{X^+/F}^\text{max}$ of differential forms of top degree relative to $\Spec(F)$, in the usual manner.
\begin{proposition}
	Let $\omega$ be a rational section of $\Omega_{X^+/F}^\mathrm{max}$. There is a canonical way to associate a rational section $\tilde{\omega}$ of $\Omega_{X_\Theta/F}^\mathrm{max}$ such that
	\begin{itemize}
		\item if $\omega$ is regular everywhere, then so is $\tilde{\omega}$;
		\item if $\omega$ is a $G$-eigenform, then so is $\tilde{\omega}$ with the same eigencharacter;
		\item the recipe is compatible with \eqref{eqn:iden-B-orbits}, in particular $\omega \neq 0 \implies \tilde{\omega} \neq 0$;
		\item $\tilde{\omega}$ is an $A_\Theta$-eigenform for the action described in \S\ref{sec:boundary}.
	\end{itemize}
\end{proposition}
\begin{proof}
	See \cite[4.2.2, 4.2.4]{SV17}. The $G$-equivariance is clear from the construction via degeneration. The explicit form of $A_\Theta$-eigencharacter of $\tilde{\omega}$ can also be found there.
\MyQED\end{proof}

\begin{hypothesis}\label{hyp:eigendensity}
	As in \cite{SV17}, we assume henceforth that there exists $\omega \neq 0$ such that $|\omega|$ is a $G(F)$-eigendensity (or eigenmeasure), and so is $|\tilde{\omega}|$ with the same eigencharacter. Cf.\ the discussions in \cite[\S 4.1]{SV17}. The requirement will always be met for the examples under consideration.
\end{hypothesis}

\begin{theorem}[{\cite[Theorem 5.1.2]{SV17}}]\label{prop:smooth-asymptotics}\index{e-star-Theta@$e^*_\Theta$}
	Assume $X^+$ is wavefront. There is a $G(F)$-equivariant linear map
	\[ e^*_\Theta: C^\infty(X^+(F)) \to C^\infty(X_\Theta(F)) \]
	characterized as follows: for every compact open subgroup $J \subset G(F)$ and every representative of the germ $\SVexp_{\Theta,J}$ (denoted by the same symbol), there exists a small $J$-stable open neighborhood $N_\Theta$ of $\infty_\Theta$ such that
	\begin{gather}\label{eqn:asym-characterization}
		e^*_\Theta(a)|_{\bar{N}_\Theta} = (\SVexp_{\Theta, J})^* (a)|_{\bar{N}_\Theta}, \quad a \in C^\infty(X^+(F))^J.
	\end{gather}
	where $\bar{N}_\Theta := \exp_{\Theta,J}^{-1}(N_\Theta) \subset X_\Theta(F)$. In fact, it suffices to consider an open neighborhood $N_\Theta$ of some orbit closure $Z \leftrightarrow \Theta$. Here $C^\infty(\cdots)$ stands for the space of smooth functions.
\end{theorem}

This is expected to hold for non-wavefront $X^+$. In practice, one also has to allow an equivariant vector bundle $\mathscr{E}$. To simplify matters, we confine ourselves to the case of density bundles.

\begin{proposition}
	Let $s \in \R$. The map $e^*_\Theta$ defined above and the operation $\omega \mapsto \tilde{\omega}$ combine into a $G(F)$-equivariant morphism
	\begin{align*}
		e^*_\Theta: C^\infty(X^+(F), \mathscr{L}^s) & \longrightarrow C^\infty(X_\Theta(F), \mathscr{L}^s_\Theta) \\
		a|\omega|^s & \longmapsto e^*_\Theta(a)|\tilde{\omega}|^s
	\end{align*}
	where $\mathscr{L}^s$ (resp.\  $\mathscr{L}^s_\Theta$) stands for the bundle of $s$-densities on $X^+(F)$ (resp.\  $X_\Theta(F)$), and $\omega$ stands for a regular section of $\Omega_{X^+/F}^\mathrm{max}$. Moreover, it is compatible with the multiplication of densities \eqref{eqn:density-mult}:
	\[ e^*_\Theta(a) e^*_\Theta(a') = e^*_\Theta(aa') \]
	where $a \in C^\infty(X^+(F), \mathscr{L}^s)$, $a' \in C^\infty(X^+(F), \mathscr{L}^{s'})$, thus $aa' \in C^\infty(X^+(F), \mathscr{L}^{s+s'})$ for some $s, s' \in \R$.
\end{proposition}
Note that taking $s=0$ reverts to the asymptotics for smooth functions.
\begin{proof}
	We may define $e^*_\Theta$ as $a|\omega|^s \mapsto e^*_\Theta(a)|\tilde{\omega}|^s$ by fixing $|\omega|$, Its compatibility with respect to \eqref{eqn:density-mult} is then reduced to the case $s=s'=0$, which is clear by the construction in \cite[\S 5.2]{SV17}. The equivariance is evident.

	Now let $\omega_1$, $\omega_2$ be two nonzero sections giving rise to eigendensities. To show that they define the same map $e^*_\Theta$ for densities, all boils down to the equality
	\[ e^*_\Theta\left(\left| \frac{\omega_1}{\omega_2} \right|^s \right)  = \left| \frac{\tilde{\omega}_1}{\tilde{\omega}_2} \right|^s \]
	in $C^\infty(X_\Theta(F))$. This might be checked by inspecting the constructions of $e^*_\Theta$ and $\tilde{\omega}_i$; here we adopt the following indirect approach.
	
	Note that $\omega_1/\omega_2$ and $\tilde{\omega}_1/\tilde{\omega}_2$ are invertible regular functions on $X^+$ and $X_\Theta$, respectively. Therefore one can define the linear map
	\[ \bar{e}^*_\Theta(a) := \left| \frac{\tilde{\omega}_1}{\tilde{\omega}_2} \right|^{-s}  e^*_\Theta \left(a \left| \frac{\omega_1}{\omega_2} \right|^s \right) \]
	from $C^\infty(X^+(F))$ to $C^\infty(X_\Theta(F))$. It suffices to show $\bar{e}^*_\Theta = e^*_\Theta$ by putting $a = \mathbf{1}$, since $e^*_\Theta(\mathbf{1}) = \mathbf{1}$ is easily checked from construction. We invoke the characterization \eqref{eqn:asym-characterization} for this purpose.
	
	First, $\bar{e}^*_\Theta$ is $G(F)$-equivariant: suppose that $|\omega_i|$ has $G(F)$-eigencharacter $\chi_i$, then so does $|\tilde{\omega}_i|$ (for $i=1,2$). Thus for any $g \in G(F)$ and $a \in C^\infty(X^+(F))$,
	\begin{align*}
		\bar{e}^*_\Theta(ga) & = \left| \frac{\tilde{\omega}_1}{\tilde{\omega}_2} \right|^{-s}  e^*_\Theta \left(ga \cdot \left| \frac{\omega_1}{\omega_2} \right|^s \right) \\
		& = \frac{\chi_1(g)^s}{\chi_2(g)^s} \left| \frac{g\tilde{\omega}_1}{g\tilde{\omega}_2} \right|^{-s} \frac{\chi_2(g)^s}{\chi_1(g)^s} \; e^*_\Theta \left(ga \cdot \left| \frac{g\omega_1}{g\omega_2} \right|^s \right) \\
		& = \left| \frac{g\tilde{\omega}_1}{g\tilde{\omega}_2} \right|^{-s} e^*_\Theta \left(ga \cdot \left| \frac{g\omega_1}{g\omega_2} \right|^s \right) = \bar{e}^*_\Theta(a).
	\end{align*}
	Secondly, we verify \eqref{eqn:asym-characterization} for $\bar{e}^*_\Theta$. Let $a \in C^\infty(X^+(F))^J$. To show the equality of smooth functions $\bar{e}^*_\Theta(a)|_{N_\Theta} = (\SVexp_{\Theta, J})^* (a)|_{N_\Theta}$ for an open neighborhood $N_\Theta$ of $\infty_\Theta$, we may restrict both sides to the open dense subset $N_\Theta \cap \mathring{X}_\Theta(F)$. Recall that under the identification \eqref{eqn:iden-B-orbits}, we may identify $\omega_i$ and $\tilde{\omega}_i$ ($i=1,2$), hence $\bar{e}^*_\Theta(a)|_{N_\Theta} = e^*_\Theta(a)|_{N_\Theta} = (\SVexp_{\Theta, J})^* (a)|_{N_\Theta}$ as required.
\MyQED\end{proof}

Observe that $C^\infty(X^+(F), \mathscr{L}^s)$ is naturally endowed with an $A_\Theta(F) \times G(F)$-action (recall that $A_\Theta \simeq \mathcal{Z}(X_\Theta)$). An element of $C^\infty(X^+(F), \mathscr{L}^s)$ is called $A_\Theta(F)$-finite if it is contained in a finite-dimensional $A_\Theta(F)$-stable subspace.

\begin{theorem}\label{prop:finiteness-density}
	Let $\pi$ be an irreducible nice representation of $G(F)$. The space $\Hom_{G(F)}(\pi, C^\infty(X^+(F), \mathscr{L}^s))$ is finite-dimensional when $X^+$ is wavefront.
\end{theorem}
\begin{proof}
	Same as the proof of \cite[Theorem 5.1.5]{SV17}. There is no worry about continuity by Lemma \ref{prop:autocont}.
\MyQED\end{proof}
Again, the result is expected to hold without wavefront assumption.

\begin{theorem}\label{prop:asym-finite}
	Let $\pi$ and $X^+$ be as above. For every
	\[ \varphi \in \Hom_{G(F)}(\pi, C^\infty(X^+(F), \mathscr{L}^s)), \]
	its image is mapped to the space of $A_\Theta(F)$-finite elements in $C^\infty(X_\Theta(F), \mathscr{L}^s)$ under $e^*_\Theta$.
	
	Consequently, $\varphi(v)$ is a $\bar{X} \smallsetminus X$-finite function, cf.\ the discussion preceding \cite[Corollary 5.1.8]{SV17}.
\end{theorem}
\begin{proof}
	Same as the proof of \cite[Corollary 5.1.8]{SV17}. More precisely, every $v \in V_\pi$ lies in $V_\pi^J$ for some open compact subgroup $J$, therefore $\varphi(v)$ lies in the sum of finitely many copies of $V_\pi^J$ mapped into $C^\infty(X_\Theta(F), \mathscr{L}^s)$ by Theorem \ref{prop:finiteness-density}, which is finite-dimensional by admissibility. Since the $A_\Theta(F)$-translates of $\varphi(v)$ belong to the same subspace, the finiteness follows.
\MyQED\end{proof}

\section{Proof of convergence}
Let $F$ be a non-Archimedean local field of characteristic zero; denote by $q$ the cardinality of its residue field. We consider an affine spherical embedding $X^+ \hookrightarrow X$ satisfying the Axiom \ref{axiom:geometric}. The main results will be stated with the wavefront assumption; see the remarks after Theorem \ref{prop:smooth-asymptotics}.

In what follows, we adopt the formalism in \S\ref{sec:coefficients} with $\mathscr{E} := \mathscr{L}^{1/2}$, the hermitian pairing $\mathscr{E} \otimes \overline{\mathscr{E}} \to \mathscr{L}$ being that induced from \eqref{eqn:density-mult}. In particular, $C^\infty(X^+)$ (resp.\  $C^\infty(X_\Theta)$ for some $\Theta \subset \Delta_{X^+}$) now stands for the space of $\mathscr{L}^{1/2}$-valued smooth sections on $X^+(F)$ (resp.\  $\mathscr{L}^{1/2}_\Theta$-valued on $X_\Theta(F)$).

Recall that the image of $\HC: A_{X^+}(F) \to \mathcal{Q}$ is the lattice $X_*(A_{X^+}) = \Hom(\Lambda_{X^+}, \Z)$. We choose a uniformizer $\varpi \in F$ and define the homomorphism
\begin{align*}
	a: X_*(A_{X^+}) & \longrightarrow A_{X^+}(F) \\
	\check{\lambda} & \longmapsto \check{\lambda}(\varpi).
\end{align*}
The conventions in \S\ref{sec:Cartan} imply that
\begin{align*}
	\HC \circ a &= \identity, \\
	A_{X^+}(F)^{+} & = a(\mathcal{V} \cap X_*(A_{X^+})) \times A_{X^+}(\mathfrak{o}_F).
\end{align*}

\begin{notation}
	For any face $\mathcal{F}$ of $\mathcal{C}_X \cap \mathcal{V}$, write
	\[ \mathcal{F} \prec \Theta \]
	if $\mathcal{V} \cap \Theta^\perp$ is the unique face of $\mathcal{V}$ whose relative interior contains $\relint(\mathcal{F})$; see Proposition \ref{prop:face-trivial}.
\end{notation}

Recall that we have chosen $x_0 \in \mathring{X}(F)$. Choose a compact open subset $K \subset G(F)$ such that that the Cartan decomposition holds: $X^+(F) = A_{X^+}(F)^+ K$ where $A_{X^+} \simeq x_0 A \hookrightarrow X^+$. Enlarge $K$ appropriately so that Proposition \ref{prop:comparison-Cartan} holds.

\begin{lemma}\label{prop:compact-translate}
	Let $C$ be a compact subset of $X(F)$. There exist $v_0, \ldots, v_m \in \mathcal{V}$ verifying
	\[ C \cap X^+(F) \subset \bigcup_{i=0}^m \HC^{-1}\left( v_i + (\mathcal{C}_X \cap \mathcal{V})\right) K. \]
\end{lemma}
\begin{proof}
	Take the morphisms $p: \hat{X} \to X$ and $X^+ \hookrightarrow \bar{X}$ as in \eqref{eqn:hat-vs-bar}. By Corollary \ref{prop:compact-exhaustion} and the subsequent observation, $\hat{X}(F)$ may be covered by open subsets $C_{\hat{X}}$ such that $C_{\hat{X}} \cap X^+(F) \subset  \bigcup_{i=0}^m \HC^{-1}\left( v_i + (\mathcal{C}_X \cap \mathcal{V})\right) K$ for some $m$ and $v_0, \ldots, v_m \in \mathcal{V}$.
	
	Since $p$ is proper, $p^{-1}(C)$ is covered by finitely many open subsets $C_{\hat{X}}$ described above. Therefore $C \cap X^+(F) = p^{-1}(C) \cap X^+(F)$ is of the required form.
\MyQED\end{proof}

For the next results, we fix an open compact subgroup $J \subset G(F)$. For any $\Theta \subset \Delta_{X^+}$, we let $A_\Theta$ act on the left of $X_\Theta$.
\begin{lemma}\label{prop:eta-triv}
	For every $b \in X_\Theta(F)$, there exists $\eta \in C^\infty(A_\Theta(F) bJ)$ which is $A_\Theta(F) \times J$-invariant and non-vanishing.
\end{lemma}
\begin{proof}
	Let $S$ be the stabilizer of $b$ under $A_\Theta(F) \times J$-action. It acts on the fiber of $\mathscr{L}_\Theta^{1/2}$ by some continuous character $S \to \R^\times_{>0}$, and we have $A_\Theta(F)bJ \simeq S \backslash (A_\Theta(F) \times J)$. Thus it suffices to show $S$ is compact (cf.\ the criterion for the existence of invariant measures). Indeed, the projection of $S$ to $A_\Theta(F)$ equals $\{a \in A_\Theta(F) : ab \in bJ\}$; it must be bounded since $J$ is compact and $A_\Theta$ dilates the normal cone.
\MyQED\end{proof}
Note that such $\eta$ is unique up to $\CC^\times$.

\begin{definition}
	Given $\xi \in \Schw^J$ and a face $\mathcal{F}$ of $\mathcal{C}_X \cap \mathcal{V}$, an embedded polytopal complex $B \subset \mathcal{Q}$ is called \emph{$\mathcal{F}$-deep}, if $\mathcal{F} \prec \Theta$ and
	\begin{itemize}
		\item the germ $\SVexp_{\Theta, J}$ of exponential maps has a representative whose image contains $\HC^{-1}(B + \mathcal{F})K$, noting that the latter contains the $\Theta$-infinity since $\mathcal{F} \prec \Theta$;
		\item the property \eqref{eqn:asym-characterization} characterizing $e^*_\Theta$ holds for all $N_\Theta \subset \HC^{-1}(B + \mathcal{F})K$.
	\end{itemize}
\end{definition}

\begin{lemma}\label{prop:deep-property}
	The notion of $\mathcal{F}$-deepness conforms to the requirements in \S\ref{sec:cellular-decomp}. Consequently, Proposition \ref{prop:cellular-decomp} holds in this context.
\end{lemma}
\begin{proof}
	Given $B \subset \mathcal{Q}$, taking deep $v_0 \in \relint(\mathcal{F})$ has the effect of moving $\HC^{-1}(B + v_0 + \mathcal{F})K$ towards $\infty_\Theta$ --- here $K \subset G(F)$ is harmless by compactness; see \S\ref{sec:Cartan}. The conditions become vacuous when $\mathcal{F}=\{0\} \prec \Theta = \Delta_{X^+}$.
\MyQED\end{proof}

\begin{proposition}\label{prop:coeff-L2}
	Assume $X^+$ is wavefront. Let $\pi$ be an irreducible nice representation of $G(F)$ and $\varphi \in \mathcal{N}_\pi = \Hom_{G(F)}(\pi, C^\infty(X^+))$. Let $C^+ := \HC^{-1}(v_0 + \mathcal{C}_X \cap \mathcal{V}) K$, where $v_0 \in \mathcal{V}$. We have
	\[ \varphi_\lambda(v)|_{C^+} \in L^2(C^+), \quad v \in V_\pi^J \]
	whenever $\Re(\lambda) \relgg{X} 0$.
\end{proposition}
\begin{proof}
	In view of Lemma \ref{prop:deep-property}, we can apply Proposition \ref{prop:cellular-decomp} with $\mathcal{F}^\text{max} := \mathcal{C}_X \cap \mathcal{V}$ to decompose
	\[ v_0 + \mathcal{C}_X \cap \mathcal{V} = \bigcup_{\substack{\mathcal{F} \subset \mathcal{C}_X \cap \mathcal{V} \\ \text{face}}} (B_{\mathcal{F}} + \mathcal{F}), \]
	such that $B_{\mathcal{F}}$ is $\mathcal{F}$-deep for all faces $\mathcal{F}$. Thus we are reduced to showing that $\varphi_\lambda(v)$ is $L^2$ over $\HC^{-1}(B_{\mathcal{F}} + \mathcal{F}) K$ for all $\mathcal{F}$. Suppose that $\mathcal{F} \prec \Theta$, we may apply $e^*_\Theta$ (which is the same as $\SVexp_{\Theta,J}^*$ over $\HC^{-1}(B_{\mathcal{F}} + \mathcal{F}) K)$ to pass to the situation inside $X_\Theta(F)$. The integration over $ \HC^{-1}(B_{\mathcal{F}} + \mathcal{F})$ now comes from the $A_\Theta$-action on $X_\Theta$.

	By Theorem \ref{prop:asym-finite}, $e^*_\Theta(\varphi_\lambda(v))$ is $A_\Theta(F)$-finite. It suffices to show its square-integrability over open subsets of the form $A_{\mathcal{F}}(F)^{>0} bJ$. Observe that $\mathcal{F} \prec \Theta \implies A_{\mathcal{F}} \subset A_\Theta$. Upon trivializing $\mathscr{L}^{1/2}_\Theta$ over $A_\Theta(F) bJ$ by Lemma \ref{prop:eta-triv}, the function $a \mapsto e^*_\Theta(\varphi_\lambda(v))(ab)$ on $A_\Theta(F)$ becomes a linear combination of expressions $\chi_\lambda(a) Q(\HC(a))$, where $\chi_\lambda$ is a continuous character of $A_\Theta(F)$ and $Q: X_*(A_\Theta) \to \CC$ is a polynomial function.
	
	Since $\varphi_\lambda(v) = |f|^\lambda \varphi(v)$, we have
	\[ \chi_\lambda =  |\omega|^\lambda \chi_0, \quad \lambda \in \Lambda_{\CC}. \]
	It suffices to show the square-integrability of $\chi_\lambda(a) Q(\HC(a)) =|\omega|^\lambda(a) \chi_0(a) Q(\HC(a))$ over $A_{\mathcal{F}}(F)^{>0}$ for $\Re(\lambda) \relgg{X} 0$. Thanks to Lemma \ref{prop:Lambda-positivity}, we know that $\angles{\lambda, v} > 0$ for every $\lambda \in \relint(\Lambda_{X,\Q})$ and every extremal ray $\Q_{\geq 0} v$ of $\mathcal{F}$, and the square-integrability follows.
\MyQED\end{proof}

\begin{corollary}\label{prop:conv-gg0}
	Assume $X^+$ is wavefront. Given $v \in V_\pi$ and $\xi \in \Schw$, the zeta integral $Z_{\lambda,\varphi}(v \otimes \xi)$ converges whenever $\Re(\lambda) \relgg{X} 0$
\end{corollary}
\begin{proof}
	The Axiom \ref{axiom:zeta} asserts that $\Supp(\xi)$ has compact closure in $X(F)$. It remains to apply Lemma \ref{prop:compact-translate}, Proposition \ref{prop:coeff-L2} together with Proposition \ref{prop:zeta-convergence}, by working with a small enough $J$.
\MyQED\end{proof}

\chapter{Prehomogeneous vector spaces}\label{sec:pvs}
Unless otherwise specified, $F$ will denote a local field of characteristic zero. We also fix a nontrivial continuous unitary character $\psi: F \to \CC^\times$ and use the self-dual Haar measure on $F$; note that $\psi^{-1}$ leads to the same measure. The integration of densities on $F$-analytic manifolds is thus normalized.

The discussions in \S\ref{sec:lfe} on the non-Archimedean local functional equations will rely on the Hypothesis \ref{hyp:lfe}.

\section{Fourier transform of half-densities}\label{sec:Fourier}
The constructions below are largely extracted from \cite[\S 9.5]{SV17}.

Let $V$ be a finite-dimensional $F$-vector space. We consider an affine space $\mathcal{A}$ under $V$. The group of affine automorphisms $\Aut(\mathcal{A})$ sits in the short exact sequence
\[ 0 \to V \to \Aut(\mathcal{A}) \to \GL(V) \to 1; \]
here $V$ embeds as the normal subgroup of pure translations. It is always possible to choose a basepoint $0 \in \mathcal{A}$ to identify $\mathcal{A}$ with $V$, and this gives rise to a section $\GL(V) \to \Aut(\mathcal{A})$. Nevertheless, in view of future applications, we will stick to the affine set-up as far as possible.

Let $\mathscr{L}^s$ denote the sheaf of $s$-densities on $\mathcal{A}$, where $s \in \R$. It is trivializable as a $V$-equivariant sheaf: choose a basis of $e_1, \ldots, e_n$ of $V$ and let $x_1, \ldots, x_n \in V^\vee$ be the dual linear functionals, the density $|\dd x_1 \wedge \cdots \wedge \dd x_n|^s$ affords an everywhere nonvanishing, translation-invariant section for $\mathscr{L}^s$. We will mainly work with $s = \demi$.

\begin{definition}\index{Schwartz--Bruhat half-density}
	Choose a basepoint $0 \in \mathcal{A}$ and a basis of $V$ as above. Define the space of \emph{Schwartz--Bruhat half-densities} as
	\[ \Schw(\mathcal{A}) := \mathrm{SB}(V) \cdot |\dd x_1 \wedge \cdots \wedge \dd x_n|^{\demi} \]
	where $\mathrm{SB}(V)$ is the space of usual Schwartz--Bruhat functions on $V \simeq \mathcal{A}$; see \cite[\S 11]{Weil64}. Equip $\Schw(\mathcal{A})$ with the topology coming from $\mathrm{SB}(V)$.
\end{definition}
It is routine to show that the topology on $\Schw(\mathcal{A})$ is independent of the basepoint $0$ and the basis.

\begin{lemma}\label{prop:pvs-nuclear}
	The topological vector space $\Schw(\mathcal{A})$ is a separable, nuclear and barreled. For $F$ Archimedean it is Fréchet. For $F$ non-Archimedean it is algebraic in the sense of Definition \ref{def:algebraic-tvs}. The group $\Aut(\mathcal{A})$ acts continuously on the left of $\Schw(\mathcal{A})$.
\end{lemma}
\begin{proof}
	It suffices to establish these properties for $\mathrm{SB}(V)$. In the Archimedean case, it is well-known to be a separable Fréchet space, thus barreled; the nuclear property is proved in \cite[p.530]{Tr67}. In the non-Archimedean case, recall that $\mathrm{SB}(V)$ is defined to be the topological vector space
	\[ \varinjlim_{W \supset W'} \text{Maps}(W/W', \CC) \]
	where $W, W'$ range over compact open additive subgroups of $V$, i.e.\ lattices --- in fact, it suffices to take a countable family --- and $\text{Maps}(W/W', \CC)$ is finite-dimensional, thus endowed with the standard topology. Hence $\mathrm{SB}(V)$ is algebraic, the required topological properties are then assured by Lemma \ref{prop:algebraic-nuclear}. The assertion concerning $\Aut(\mathcal{A})$-action is evident.
\MyQED\end{proof}

Let $\check{\mathscr{L}}^s$ denote the sheaf of $s$-densities on $V^\vee$. We set out to define a canonical $\Ga$-torsor $\mathbf{R}$ over $V^\vee$, where $V^\vee$ is now viewed as an $F$-variety by slightly abusing notation. At each $\lambda \in V^\vee$, the fiber of $\mathbf{R}$ is the contracted product
\[ \mathbf{R}_{\lambda} := \mathcal{A} \utimes{V} \Ga, \]
i.e.\ its elements are equivalence classes $[a, t]$ such that $[a+v, t] = [a, \angles{\lambda, v} + t]$, where $a \in \mathcal{A}$, $v \in V$ and $t \in \Ga$. One readily verifies that various $\mathbf{R}_{\lambda}$ glue into a $\Ga$-torsor $\mathbf{R}$ over $V^\vee$. Every choice of base point $0 \in \mathcal{A}$ yields a section (thus a trivialization) of $\mathbf{R}$: simply take
\begin{gather}\label{eqn:R-trivialization}
	\lambda \longmapsto [0,0] \in \mathbf{R}_\lambda.
\end{gather}

\begin{notation}
	Being canonical, the torsor $\mathbf{R} \to V^\vee$ becomes equivariant under the action of $\Aut(\mathcal{A})$. Our convention is to let $\Aut(\mathcal{A})$ act \emph{on the right} of $V$, $V^\vee$ (see below) as well as on $\mathbf{R}$. In terms of equivalence classes, $g \in \Aut(\mathcal{A})$ acts on $\mathbf{R}$ as
	\begin{align*}
		\mathbf{R}_\lambda & \longrightarrow \mathbf{R}_{\lambda g} \\
		[a, t] & \longmapsto [ag, t].
	\end{align*}
	Note that $\angles{\lambda g, v} = \angles{\lambda, v\vec{g}^{-1}}$ for all $v \in V$, i.e.\ the contragredient, and $\vec{g} \in \GL(V)$ stands for the vectorial part of $g$. When $g$ is a pure translation $a \mapsto a+v$, we have $\lambda g = \lambda$, whereas
	\[ [a,t]g = [ag, t] = [a+v, t] = [a, t + \angles{\lambda, v}] \]
	on the fiber above $\lambda$.
\end{notation}

Using $\psi: \Ga(F) = F \to \CC^\times$, we obtain the $\CC^\times$-torsor $R_\psi$ over $V^\vee$ from $\mathbf{R}$, now regarded in the category of $F$-analytic manifolds. Again, any choice of base point trivializes $R_\psi$. Also, $R_\psi$ is equivariant under $\Aut(\mathcal{A})$. If $g$ is a pure translation $a \mapsto a+v$, then $g$ dilates each fiber $R_{\psi, \lambda}$ by the factor $\psi(\angles{\lambda, v})$. It is also clear from this description that
\begin{compactitem}
	\item $R_\psi \otimes R_{\psi^{-1}}$ is canonically trivialized, and
	\item $\overline{R_\psi} = R_{\psi^{-1}}$, thus $R_\psi$ carries a canonical hermitian pairing.
\end{compactitem}

The canonical pairing $\angles{\cdot, \cdot}: V \otimes V^\vee \to F$ induces a pairing between $\topwedge V$ and $\topwedge V^\vee$, which is still denoted by $\angles{\cdot,\cdot}$. Also, the bi-character $\psi(\angles{\cdot, \cdot})$ allows us to talk about duality between Haar measures on $V$ and $V^\vee$.

\begin{lemma}\label{prop:dual-measure}
	Let $\eta \in \topwedge V^\vee$ and $\eta' \in \topwedge V$. Suppose that $\eta, \eta' \neq 0$ so that they induce Haar measures on $V$ and $V^\vee$, respectively. Then the measures are dual with respect to $\psi$ if and only if $|\angles{\eta, \eta'}|=1$.
\end{lemma}
\begin{proof}
	It suffices to verify the case $\dim_F V = 1$. Then the assertion stems from the fact that our Haar measure on $F$ is self-dual with respect to $\psi$.
\MyQED\end{proof}

In general, any given vector bundle $\mathscr{E}$ over $V^\vee$ can be twisted by the $\CC^\times$-torsor $R_\psi$ by forming $\mathscr{E} \otimes R_\psi$. We retain the bracket notation for denoting elements in the fibers of $\mathscr{E} \otimes R_\psi$, namely:
\begin{gather}\label{eqn:bracket-notation}
	[a, s] = s \otimes [a, 1], \quad  a \in \mathcal{A}, s \in \mathscr{E}_\lambda \\
	[a+v, s] = \left[ a, \;\psi(\angles{\lambda, v})s \right], \quad v \in V.
\end{gather}
We also write $[\cdot, \cdot]_\lambda$ in order to indicate the $\lambda \in V^\vee$. If $\mathscr{E}$ is $\Aut(\mathcal{A})$-equivariant as well, then the action is $[a,s]_\lambda g = [ag, sg]_{\lambda g}$.

Now consider the sheaf $\check{\mathscr{L}}^s$ of $s$-densities over $V^\vee$, for $s \in \R$: it is also $\Aut(V^\vee)$-equivariant. Since $R_\psi$ is trivializable, we may define the Schwartz--Bruhat space $\Schw(V^\vee, R_\psi)$ as a subspace of $C^\infty(V^\vee, \check{\mathscr{L}}^{1/2} \otimes R_\psi)$, on which $\Aut(\mathcal{A})$ acts continuously on the left.

By the generalities from \S\ref{sec:integration-density}, one has the canonically defined unitary representations
\[ L^2(\mathcal{A}, \mathscr{L}^{1/2}), \quad L^2(V^\vee, \check{\mathscr{L}}^{1/2} \otimes R_\psi) \]
of $\Aut(\mathcal{A})$. They contain $\Schw(\mathcal{A})$ and $\Schw(V^\vee, R_\psi)$ as dense subspaces, respectively. Under this set-up, the elementary properties of Fourier transform can be summarized into a single statement as follows.

\begin{theorem}\label{prop:Fourier-def}\index{Fourier transform}
	There is a canonical $\Aut(\mathcal{A})$-equivariant continuous isomorphism, called the Fourier transform:
	\begin{align*}
		\mathcal{F}: \Schw(\mathcal{A}) & \stackrel{\sim}{\longrightarrow} \Schw(V^\vee, R_\psi) \\
		\xi & \longmapsto \left( \lambda \mapsto \int_{a \in \mathcal{A}} [a, \xi(a)]_{\lambda} \right)
	\end{align*}
	in the notation of \eqref{eqn:bracket-notation}; it extends to an $\Aut(\mathcal{A})$-equivariant isometry
	\[ \mathcal{F}: L^2(\mathcal{A}, \mathscr{L}^{\demi}) \stackrel{\sim}{\longrightarrow} L^2(V^\vee, \check{\mathscr{L}}^{\demi} \otimes R_\psi). \]
\end{theorem}
\begin{proof}
	Let us explain the integral over $\int_{\mathcal{A}}$ first. The integrand is an element in the fiber of $R_\psi \otimes \mathscr{L}^{\demi}$ at $\lambda$. To integrate it, we may multiply by some $|\eta|^{1/2}$ with $\eta \in \topwedge V^\vee$ nonzero, and divide it out after integration; such an operation is justified by the canonical pairing between $\topwedge V$ and $\topwedge V^\vee$.
	
	In down-to-earth terms, one chooses a base point $0 \in \mathcal{A}$ to identify $V$ with $\mathcal{A}$ and trivialize $R_\psi$. The notation \eqref{eqn:bracket-notation} leads to
	\begin{gather}\label{eqn:classical-Fourier}
		\lambda \mapsto \int_{v \in V} \xi(v) \psi\left( \angles{\lambda, v}\right).
	\end{gather}	
	This is almost the familiar Fourier transform $\mathrm{SB}(V) \rightiso \mathrm{SB}(V^\vee)$, except that we have to multiply some $|\eta|^{1/2}$ then divide it out, as has been performed previously. It follows that $\mathcal{F}$ is a continuous isomorphism. The fact that $\mathcal{F}$ is an $L^2$-isometry is basically a consequence of Lemma \ref{prop:dual-measure}.
	
	In the same manner, the $\Aut(\mathcal{A})$-equivariance can be deduced from the variance of usual Fourier transform under translations and linear transforms. Alternatively, it also follows from the formula $\lambda \mapsto \int_{a \in \mathcal{A}} [a, \xi(a)]_\lambda$ by formal manipulations.
\MyQED\end{proof}

\begin{corollary}\label{prop:Fourier-vector}
	By choosing a base point $0 \in \mathcal{A}$ so that $\mathcal{A} \simeq V$, the Fourier transform becomes an $\GL(V)$-equivariant continuous isomorphism
	\[ \mathcal{F}: \Schw(V) \stackrel{\sim}{\longrightarrow} \Schw(V^\vee) \]
	defined by \eqref{eqn:classical-Fourier}. It extends to an equivariant isometry $L^2(V, \mathscr{L}^{\demi}) \rightiso L^2(V^\vee, \check{\mathscr{L}}^{\demi})$.
\end{corollary}
\begin{proof}
	Only the equivariance requires explanation. The choice of $0 \in \mathcal{A}$ embeds $\GL(V)$ as a subgroup of $\Aut(\mathcal{A})$. It suffices to notice that the trivializing section \eqref{eqn:R-trivialization} of $\mathbf{R}$ is $\GL(V)$-stable.
\MyQED\end{proof}

A similar normalization of Fourier transforms has appeared in \cite[\S 7.10]{BR05} and \cite{St67}. For a version fibered in families, see \cite[9.5.8]{SV17}.

\section{Review of prehomogeneous vector spaces}
The first part of this section is algebro-geometric: we only require $F$ to be a field of characteristic zero with algebraic closure $\bar{F}$, and $G$ can be any connected affine $F$-group.
\begin{definition}\index{prehomogeneous vector space}
	A \emph{prehomogeneous vector space} over $F$ is a triplet $(G, \rho, X)$ where
	\begin{compactitem}
		\item $G$ is a connected affine algebraic $F$-group,
		\item $X$ is a finite-dimensional $F$-vector space on which $\GL(X)$ acts \emph{on the right},
		\item $\rho: G \to \GL(X)$ is an algebraic representation, making $V$ into a $G$-variety,
	\end{compactitem}
	such that there exists an open dense $G$-orbit $X^+$ in $X$. Equivalently, there exists $x_0 \in X(F)$ such that $\overline{x_0 \rho(G)} = V$. We shall often abbreviate the triplet by $X$ or by $\rho$.
\end{definition}
We adopt the shorthand $xg$ for $x\rho(g)$. The contragredient representation $\check{\rho}: G \to \GL(\check{X})$ furnishes the \emph{dual triplet} $(G, \check{\rho}, \check{X})$, but the latter is not prehomogeneous in general. The prehomogeneity of $\check{X}$ will be crucial in the study of local functional equations. For this purpose, we record the convenient notion of \emph{regularity} to ensure the prehomogeneity of $\check{\rho}$. As a byproduct, the density bundles are trivializable in the regular case.

\begin{definition}[Cf.\ {\cite[Chapter 2]{Ki03} or \cite[p.468]{Sa89}}]\label{def:pvs-regular}\index{relative invariant}\index{prehomogeneous vector space!regular}
	A \emph{relative invariant} over $F$ for a prehomogeneous vector space $(G, \rho, X)$ is an eigenvector $f \in F(X)^\times$ for the left $G$-action, say with the eigencharacter $\omega \in X^*(G)$. Call a relative invariant $f$ \emph{non-degenerate} if $f^{-1}\dd f: X \dashrightarrow \check{X}$ is a birational equivalence. Prehomogeneous vector spaces admitting some non-degenerate relative invariant over $F$ are called \emph{$F$-regular}, or simply \emph{regular}.
\end{definition}
Note that the rational map $f^{-1}\dd f$ is always $G$-invariant, see \cite[Proposition 2.13]{Ki03}.

\begin{remark}\label{rem:basic-relative-invariant}\index{relative invariant!basic}\index{X_rho(G)@$X^*_\rho(G)$}
	In this generality, by \cite[\S 1.1]{Sa89} there are relative invariants $f_1, \ldots, f_r \in F[X]$ such that
	\begin{compactitem}
		\item every relative invariant $f$ has a unique factorization into $c\prod_{i=1}^r f_i^{a_i}$ with $a_i \in \Z$ and $c \in F^\times$;
		\item their characters $\omega_1, \ldots, \omega_r$ are linearly independent;
		\item the zero loci of $f_1, \ldots, f_r$ are precisely the irreducible components of codimension one of $\partial X = X \smallsetminus X^+$.
	\end{compactitem}
	We may use the same indexing for $f_i$ and the codimension-one components of $\partial X$, so that each $f_i$ is unique up to $F^\times$. Call them the \emph{basic relative invariants}. The eigencharacters of relative invariants form a subgroup $X^*_\rho(G) \subset X^*(G)$.
\end{remark}

\begin{theorem}\label{prop:pvs-regular}
	If $(G, \rho, X)$ is regular, its dual $(G, \check{\rho}, \check{X})$ is prehomogeneous and regular as well. In the case, $(\det\rho)^2$ is the eigencharacter of some relative invariant $f$ of $(G, \rho, X)$. Moreover, in this case
	\begin{itemize}
		\item $X^*_\rho(G) = X^*_{\check{\rho}}(G)$;
		\item $f^{-1}\dd f$ is a $G$-equivariant isomorphism $X^+ \rightiso \check{X}^+$;
		\item $\partial X$ is a hypersurface if and only if $\partial \check{X}$ is.
	\end{itemize}
\end{theorem}
\begin{proof}
	See \cite[\S 1]{Sa89}; the case over $\bar{F}$ is discussed in detail in \cite[Theorem 2.16 and Corollary 2.17]{Ki03}.
\MyQED\end{proof}

\begin{notation}
	We denote the Schwartz--Bruhat spaces of half-densities attached to $X$ (resp.\  $\check{X}$) as $\Schw(X)$ (resp.\  $\Schw(\check{X})$) as in \S\ref{sec:Fourier}. The $L^2$-spaces associated to half-densities are abbreviated as $L^2(X)$ (resp.\  $L^2(\check{X})$), so that the Fourier transform in Corollary \ref{prop:Fourier-vector} gives the commutative diagram:
	\[ \begin{tikzcd}
		L^2(X) \arrow{r}[above]{\mathcal{F}}[below]{\sim} & L^2(\check{X}) \\
		\Schw(X) \arrow{r}[above]{\sim}[below]{\mathcal{F}} \arrow[hookrightarrow]{u} & \Schw(\check{X}) \arrow[hookrightarrow]{u}
	\end{tikzcd}\]
\end{notation}

Some easy observations:
\begin{compactitem}
	\item Since $(\det \rho)^2 \in X^*_\rho(G)$, the $G$-equivariant sheaf of $s$-densities $\mathscr{L}^s$ is trivializable for all $s \in \R$.
	\item $L^2(X^+) = L^2(X)$ (resp.\  $L^2(\check{X}^+) = L^2(\check{X})$ if $\check{X}$ is prehomogeneous) since the complement has lower dimension.
	\item The topological vector spaces displayed above are continuous $G(F)$-representations, and the Fourier transform $\mathcal{F}$ is $G(F)$-equivariant.
\end{compactitem}

Hereafter, we revert to the usual convention that $F$ is a local field of characteristic zero, and $G$ is split connected reductive. Consider a prehomogeneous vector space $(G, \rho, X)$; assume that the open $G$-orbit $X^+$ is an affine spherical homogeneous $G$-space. The first goal is to verify the geometric Axiom \ref{axiom:geometric} for $X^+ \hookrightarrow X$.

\begin{theorem}\label{prop:pvs-geometric-axiom}
	Under the conventions above, the Axiom \ref{axiom:geometric} is satisfied.
\end{theorem}
\begin{proof}
	The relative invariants for prehomogeneous vector spaces defined above is just a specialization for the general theory in \S\ref{sec:geometric-data}. The condition (G1) on $\Lambda_{X,\Q}$ follows from Remark \ref{rem:basic-relative-invariant}. Furthermore, since the zero loci of basic relative invariants cover the codimension-one components of $\partial X$, the condition (G2) is equivalent to that $\partial X$ is a hypersurface. This is in turn equivalent to $X^+$  being affine, by \cite[Theorem 2.28]{Ki03}.
\MyQED\end{proof}
When $X^+$ is a symmetric space, the conditions above hold true; moreover, $X^+$ is wavefront in that case.

Keeping the assumptions in Theorem \ref{prop:pvs-geometric-axiom}, the notation in \S\ref{sec:geometric-data} applies to $X^+ \hookrightarrow X$. We define accordingly $f_i \leftrightarrow \omega_i$, $\Lambda_X$, $\Lambda$, etc. In this section we work with half-densities, i.e.\ the case $\mathscr{E} = \mathscr{L}^{\demi}$.

\begin{theorem}\label{prop:pvs-Schw}
	Let $X^+ \hookrightarrow X$ be as above, and set $\Schw := \Schw(X)$.
	\begin{itemize}
		\item When $F$ is non-Archimedean and $X^+$ is wavefront, the Axioms \ref{axiom:finiteness} and \ref{axiom:zeta} are satisfied.
		\item When $F$ is Archimedean, the requirements on $\Schw$ in Axiom \ref{axiom:zeta} are satisfied. Those concerning the zeta integrals are temporarily unknown except when $\pi$ is the trivial representation, or when $X^+$ is an essentially symmetric $G$-space.
	\end{itemize}
	Here we say $X^+$ is essentially symmetric if it admits a finite equivariant étale covering by a symmetric $G$-space.
\end{theorem}
\begin{proof}
	The assertions on Axiom \ref{axiom:finiteness} follow directly from Remark \ref{rem:finiteness}. As to Axiom \ref{axiom:zeta}, the topological properties of $\Schw$ are ensured by Lemma \ref{prop:pvs-nuclear}.

	Obviously $C^\infty_c(X^+)$ injects continuously into $\Schw$. Let us show that
	\[ \Re(\lambda) \relgeq{X} 0 \implies \left[ \Schw_\lambda := |f|^\lambda \Schw \hookrightarrow L^2(X^+)\; \text{ continuously} \right]. \]
	Note that $|f|^\lambda$ extends to a continuous function on $X(F)$, thus bounded over compacta. The continuous injection follows easily for non-Archimedean $F$. On the other hand, the Archimedean case follows from the rapid decay of Schwartz functions, say in the sense of Definition \ref{def:rapid-decay} since $|f|$ is a semi-algebraic function over $X(F)$. Let $\alpha_\lambda: \Schw \to L^2(X^+)$ be the continuous injection obtained by composition with $m_\lambda: \Schw \rightiso \Schw_\lambda$. By similar arguments, it is routine to establish the holomorphy of $\alpha_\lambda$ in the sense of Definition \ref{def:L2-holomorphy}.

	Now turn to the ``smooth axioms'' for non-Archimedean $F$. Suppose that the residual field of $F$ is $\F_q$ and $X^+$ is wavefront.
	\begin{itemize}
		\item The support condition on $\xi \in \Schw$ is immediate from the definition of $\Schw$.
		\item The convergence of $Z_\lambda: \pi_\lambda \otimes \Schw \to \CC$ for $\Re(\lambda) \relgg{X} 0$ results from Corollary \ref{prop:conv-gg0}, and the separate continuity is automatic by Remark \ref{rem:Schw-fa}.
		\item It remains to check the rational continuation of $Z_\lambda(v \otimes \xi)$, where $v \in V_\pi$ and $\xi \in \Schw$ are fixed. For this purpose we extract the following fact from the proof of \cite[Proposition 15.3.6]{SV17} --- see also \cite[\S 4.5]{Sak08}.
		
		Let $Y$ be a normal $F$-variety with a divisor $D$ defined over $F$. Consider the data
		\begin{compactitem}
			\item $\Xi \in C^\infty_c(Y(F))$,
			\item $\theta$: a $D$-finite function -- the precise meaning is explained before \cite[Corollary 5.1.8]{SV17},
			\item $|f|^\lambda = \prod_{i=1}^r |f_i|^{\lambda_i}$ such that each $f_i \in F(Y)$ has polar divisor supported in $D$, and $\lambda_i \in \CC$,
			\item $\Omega$: a rational $F$-volume form on $Y$ with polar divisor supported in $D$,
			\item $I(\lambda_1, \ldots, \lambda_r) := \int_{Y(F)}\limits \theta \Xi |f|^\lambda |\Omega|$, viewed as an integral (if convergent) parametrized by points of $\mathcal{T} := (\CC^\times)^r$ via $(\lambda_i)_{i=1}^r \mapsto (q^{\lambda_i})_{i=1}^r$.
		\end{compactitem}
		Conclusion: if $I(\lambda_1, \ldots, \lambda_r)$ converges for $(\lambda_1, \ldots, \lambda_r)$ in some nonempty open subset, then it extends to a rational function over $\mathcal{T}$. Such a continuation is clearly unique. As indicated in \textit{loc.\ cit.}, this follows essentially from Igusa's theory on complex powers\index{Igusa's theory}. Specifically, it may be viewed as a generalization of the proof of \cite[Theorem 8.2.1]{Ig00}.

		Apply this to $Y = \hat{X}$ with $D = \hat{X} \smallsetminus X^+$, where $p: \hat{X} \twoheadrightarrow X$ is a proper birational $G$-equivariant morphism such that $\hat{X}$ is open in a smooth complete toroidal compactification $\bar{X}$ of $X^+$; it is constructed in \S\ref{sec:Cartan}, cf.\ \eqref{eqn:hat-vs-bar}. To ease notations, we assume $\mathscr{L}^{\demi}$ trivialized in what follows. Take $\Xi = \xi \circ p$, $\theta = \varphi_\lambda(v)$ and $|f|^\lambda$ be the familiar complex power of relative invariants. Let us check the premises:
		\begin{compactitem}
			\item $\Xi \in C^\infty_c(\hat{X}(F))$ since $\xi \in C^\infty_c(X(F))$ and $p$ is proper,
			\item the $D$-finiteness of $\theta$ follows from the $(\bar{X} \smallsetminus X^+)$-finiteness asserted in Theorem \ref{prop:asym-finite}, by pulling back via $\hat{X} \hookrightarrow \bar{X}$ (see the discussions preceding \cite[Corollary 5.1.8]{SV17}),
			\item take $\Omega$ to be any nonzero rational $G$-eigenform on $X$ of top degree.
		\end{compactitem}
		Since the convergence of $I(\lambda) = I(\lambda_1, \ldots, \lambda_r)$ for $\Re(\lambda) \relgg{X} 0$ is known, the rational continuation follows.
	\end{itemize}

	Finally, when $F$ is Archimedean and $\pi$ is trivial, everything reduces to the classical setting studied by M.\ Sato, T.\ Shintani \textit{et al.} See \cite[\S 1.4]{Sa89}. Here $v$ ranges over a $1$-dimensional vector space, therefore can be neglected; the continuity of $Z_\lambda$ in $\xi$ is contained in \textit{loc.\ cit.} For the Archimedean case with essentially symmetric $X^+$, see \cite{Li18}.
\MyQED\end{proof}

\section{Local functional equation}\label{sec:lfe}
Assume $F$ to be local non-Archimedean. In the first place, we consider a connected reductive $F$-group $G$ and a parabolic subgroup $P$, with Levi quotient $P/U_P \twoheadrightarrow M$.

For every smooth representation $\pi$ of $G(F)$, define the (unnormalized) Jacquet module $\pi_P$ as the vector space
\[ V_\pi\big/ \sum_{u \in U_P(F)} \Image(\pi(u)-\identity) \]
equipped with the natural, unnormalized $M(F)$-action. Observe that for every continuous character $\chi: G(F) \to \CC^\times$, we have the equality
\[ (\pi \otimes \chi)_P = \pi_P \otimes (\chi|_{M(F)}) \]
as smooth representations of $M(F)$: since $\chi(u)=1$ for any unipotent $u \in G(F)$, it factors through a continuous character $\chi|_{M(F)}: M(F) \to \CC$, which can be viewed as a true restriction once a Levi decomposition of $P$ is chosen.

Consider now a subgroup $\Lambda \subset X^*(G)$, giving rise to a complex torus $\mathcal{T}$ of unramified characters of $G(F)$. Suppose given a subgroup $H \subset G$ such that $U_P \subset H \subset P$. Denote $H_M := H/U_P \subset M$. The diagram
\[ \begin{tikzcd}
	G & H \arrow[hookrightarrow]{l} \arrow[twoheadrightarrow]{r} \arrow[hookrightarrow]{d} & H_M \arrow[hookrightarrow]{d} \\
	& P \arrow[hookrightarrow]{lu} \arrow[twoheadrightarrow]{r}[below, inner sep=1em]{\mod U_P} & M
\end{tikzcd}\]
is commutative with a cartesian square. By the previous observation, a continuous character $\chi: G(F) \to \CC^\times$ can thus be ``restricted'' to $H_M(F)$ in either way, with the same result. We may further restrict $\chi$ to the $(Z_M \cap H_M)(F)$; note that $Z_M \cap H_M$ is a diagonalizable $F$-group.

\begin{lemma}\label{prop:intertwining-finiteness}
	Let $\mathcal{T}$ be the complex torus of unramified characters associated to $\Lambda$ as above. Let $\pi$ be a smooth representation of $G(F)$ of finite length and $\sigma$ be a continuous character of $H(F)$. Suppose that for every $\chi \in \mathcal{T}$ in general position, there exists a non-trivial intertwining operator of smooth $G(F)$-representations
	\[ \pi \otimes \chi \hookrightarrow C^\infty(H(F) \backslash G(F), \sigma). \]
	Then the restriction of $\mathcal{T}$ to $(Z_M \cap H_M)(F)$ is finite.
\end{lemma}
\begin{proof}
	Denote the trivial representation by $\mathbf{1}$. For $\chi_0 \in \mathcal{T}$ in general position (i.e.\ in a Zariski-open subset), the assumption gives
	\begin{align*}
		\pi \otimes \chi_0 & \hookrightarrow C^\infty(H(F) \backslash G(F), \sigma) \simeq \Ind^G_H(\sigma) \\
		& = \Ind^G_P \Ind^P_H(\sigma) = \Ind^G_P \Ind^M_{H_M} (\sigma);
	\end{align*}
	we recall that $\Ind(\cdots)$ stands for the unnormalized smooth induction, and the representation $\Ind^M_{H_M} (\mathbf{1})$ is inflated to $P(F)$. By Frobenius reciprocity, this is equivalent to
	\[ \pi_P \otimes (\chi_0|_{M(F)}) = (\pi \otimes \chi_0)_P \xrightarrow{\neq 0} \Ind^M_{H_M}(\sigma). \]

	Consider the actions of $(Z_M \cap H_M)(F)$ on both sides in the displayed formula. On the right it acts by the character $\sigma$. On the left, $\pi_P \otimes \chi_0|_{M(F)}$ has a finite Jordan--Hölder series by \cite[VI.6.4]{Re10} whose subquotients admit central characters under $(Z_M \cap H_M)(F)$. These central characters form a set $\Xi(\chi_0)$ containing $\sigma$. For $\chi$ in general position, $\Xi(\chi\chi_0) = \chi|_{Z_M \cap H_M(F)} \Xi(\chi_0)$, whereas $\Xi(\chi\chi_0) \cap \Xi(\chi_0) \ni \sigma$.
	Hence $\chi$ lies in the inverse image of finitely many points of $\Hom(Z_M \cap H_M(F), \CC^\times)$, namely those $\omega/\omega'$ with $\omega, \omega' \in \Xi(\chi_0)$. This forces the restriction of $\mathcal{T}$ to $(Z_M \cap H_M)(F)$ to be finite.
\MyQED\end{proof}

Revert to the case that $G$ is a split connected reductive $F$-group and $X^+ \hookrightarrow X$ carries a prehomogeneous structure $(G, \rho, X)$. Assume that
\begin{compactitem}
	\item the premises of Theorem \ref{prop:pvs-Schw} hold, so that the conventions in \S\ref{sec:geometric-data} are in force;
	\item the prehomogeneous vector space $(G, \rho, X)$ is regular so that the dual triplet $(G, \check{\rho}, \check{X})$ is also prehomogeneous.
\end{compactitem}
Theorem \ref{prop:pvs-regular} implies that the premises of Theorem \ref{prop:pvs-Schw} also hold for $\check{X}$. All in all, the Fourier transform $\mathcal{F}: \Schw(X) \rightiso \Schw(\check{X})$ is a model transition according to Definition \ref{def:model-transition}. Here we take
\[  X_1 := \check{X}, \qquad X_2 := X. \]
The notations $C^\infty(X^+)$, $C^\infty_c(X^+)$, etc.\ have the same meaning as in \S\ref{sec:Schwartz}. Also note the situation is symmetric in $X$ and $\check{X}$: both $\Lambda_1$ and $\Lambda_2$ are equal to $\Lambda := X^*_\rho(G)$. Thus we have the same complex torus of unramified characters $\mathcal{T}_1 = \mathcal{T} = \mathcal{T}_2$, attached to the same $\mathcal{O}$, and the superscript $\flat$ in \S\ref{sec:model-transition} can be dropped.

\begin{hypothesis}\label{hyp:lfe}
	Suppose that for every $y \in (\partial X)(F)$ with stabilizer $H := \Stab_G(y)$, there exists a parabolic subgroup $P \subset G$ with Levi quotient $M := P/U_P$, such that
	\begin{itemize}
		\item $U_P \subset H \subset P$, so we can set $H_M := H/U_P$;
		\item the restriction of $\mathcal{T}$ to $(Z_M \cap H_M)(F)$ contains a complex torus of positive dimension.
	\end{itemize}
\end{hypothesis}

\begin{remark}
	When the spherical homogeneous $M$-space $H_M \backslash M$ is \emph{factorizable}, i.e.\ $\mathfrak{h}_M = (\mathfrak{h}_M \cap \mathfrak{z}_M) \oplus (\mathfrak{h}_M \cap \mathfrak{m}_\text{der})$ (cf.\ \cite[9.4.1]{SV17}), the second condition above is equivalent to: the restriction of $\mathcal{T}$ to $H_M(F)$ contains a complex torus of positive dimension.
\end{remark}

\begin{example}\label{eg:symm-forms}
	Let $X$ be the variety of symmetric bilinear forms on a finite-dimensional $F$-vector space $V$, whose elements may be identified with linear maps $b: V \to V^\vee$ such that the composite $V \rightiso (V^\vee)^\vee \xrightarrow{b^\vee} V^\vee$ equals $b$. Then $G := \GL(V)$ operates transitively on the right of $X$ with the dense open orbit $X^+$ consisting of non-degenerate forms, i.e.\ of invertible $b$. Therefore $X^+$ is a symmetric $G$-space and $\partial X$ is a prime divisor.

	The dual $\check{X}$ may be identified with the variety of symmetric bilinear forms on $V^\vee$. Indeed, the non-degenerate pairing is $(b, \check{b}) \mapsto \Tr(\check{b}b: V \to V)$, and the $G$-action on $\check{X}$ makes it invariant. In fact $X$ is regular as a prehomogeneous vector space; the proof is the same as that of Proposition \ref{prop:GJ-regular}, so we omit it.

	Let $x \in X(F)$ whose radical (that is, $V^\perp$ relative to $x$) we denote by $R_x$, and $H := \Stab_G(x)$. Set $P := \Stab_G(R_x)$. We have
	\begin{gather*}
		U_P \subset H \subset P, \\
		M = \GL(R_x) \times \GL(V/R_x), \\
		H_M = \GL(R_x) \times \mathrm{O}(V/R_x, x) \subset M.
	\end{gather*}
	Levi decompositions $P = M U_P$ arise from choices of an orthogonal decomposition $V = V_0 \oplus R_x$ relative to $x$. The Hypothesis \ref{hyp:lfe} for $X$ and $\check{X}$ is readily verified.
\end{example}

\begin{theorem}\label{prop:pvs-T_pi-isom}
	Under the Hypothesis \ref{hyp:lfe}, the $\mathcal{K}$-linear map $T_\pi: \mathcal{N}_\pi \otimes \mathcal{K} \to \mathcal{L}_\pi$ defined in \eqref{eqn:T_pi} for $X^+ \hookrightarrow X$ is an isomorphism.
	
	Moreover, \eqref{eqn:generic-inj} is satisfied in this case.
\end{theorem}
\begin{proof}
	First of all, there is no worry about topologies since $\Schw(X)$, $\Schw(\check{X})$ and the underlying spaces of irreducible nice representations are all algebraic. It suffices to show the surjectivity of $T_\pi$ by Lemma \ref{prop:uniqueness-embedding}.

	Let $\widehat{Z}: \widehat{\pi} \otimes \Schw(X) \to \mathcal{K}$ be a an element of $\mathcal{L}_{\pi,t}$ where $\pi$ is an irreducible nice representation of $G(F)$, and recall that $t \in \mathcal{O} \smallsetminus \{0\}$ is a denominator of the rational family $\widehat{Z}$. To show that $\widehat{Z}$ comes from $\mathcal{N}_\pi \otimes \mathcal{K}$, we may clear the denominator $t$ and assume $\widehat{Z}$ regular, i.e.\ $\widehat{Z} \in \mathcal{L}_{\pi,1}$.

	When restricted to $C^\infty_c(X^+)$, from $\widehat{Z}$ we deduce a family $\varphi_\lambda: \pi_\lambda \to C^\infty(X^+)$ by Theorem \ref{prop:C-infty}, where $\lambda \in \Lambda_{\CC}$. Twisting back gives a rational family $|f|^{-\lambda}\varphi_\lambda: \pi \to C^\infty(X^+)$; it is actually regular over $\mathcal{T}$ as $\varphi_\lambda$ is. As $\dim\mathcal{N}_\pi < \infty$, it is actually an element of $\mathcal{N}_\pi \otimes \mathcal{O}$. Therefore, the machine of zeta integrals, i.e.\ the map $T_\pi$, applies to yield $\widehat{Z}' \in \mathcal{L}_{\pi, t'}$ for some $t' \in \mathcal{O} \smallsetminus \{0\}$. By construction and Remark \ref{rem:zeta-distribution}, $\widehat{Z}'$ restricted to $C^\infty_c(X^+)$ coincides with $|f|^\lambda |f|^{-\lambda} \varphi_\lambda = \varphi_\lambda$. Hence $\widehat{Z}' - \widehat{Z}$ vanishes over $\pi \otimes C^\infty_c(X^+)$. We contend that $\widehat{Z}' = \widehat{Z}$.
	
	To show this, we may clear denominators by considering the regular family $t'(\widehat{Z}' - \widehat{Z})$ parametrized by $\mathcal{T}$. Denote by $T_\lambda: \pi_\lambda \to \Schw(X)^\vee$ the family of tempered distributions on $X$ associated to $t'(\widehat{Z}' - \widehat{Z})$. It boils down to showing $T_\lambda = 0$.

	Our spaces $C^\infty_c(X^+)$, $\Schw(X)$ are $\mathscr{L}^\demi$-valued, but it is easy to switch to the usual framework since $\mathscr{L}^\demi$ is equivariantly trivializable, although this is not strictly necessary. The theory of tempered distributions over non-Archimedean $X(F)$ implies that $T_\lambda$ is represented by $C^\infty$-functions (valued in half-densities) over the $G(F)$-orbits of lower dimension, for every $\lambda$. Suppose on the contrary that $T_\lambda \neq 0$ and let $y G(F)$ be a $G(F)$-orbit with maximal dimension such that $T_\lambda$ contains an intertwining operator $\pi_\lambda \xrightarrow{\neq 0} C^\infty(yG(F), \mathscr{L}^{\demi}|_{yG(F)})$. Here $y \in X(F)$ and we put $H := \Stab_G(y)$, everything varies algebraically over $\mathcal{T}$. Lemma \ref{prop:intertwining-finiteness} and Hypothesis \ref{hyp:lfe} applied to this $\pi$, $y$, $H$ and suitable $\sigma$ lead to contradiction.
	
	The arguments in the previous paragraph actually establishes \eqref{eqn:generic-inj} under Hypothesis \ref{hyp:lfe}. This completes the proof.
\MyQED\end{proof}

\begin{theorem}\label{prop:pvs-lfe}
	Under the Hypothesis \ref{hyp:lfe} for $X$, the local functional equation (Definition \ref{def:lfe}) holds for $\mathcal{F}: \Schw(X) \rightiso \Schw(\check{X})$.
\end{theorem}
\begin{proof}
	Recall that we put $X_1 = \check{X}$, $X_2 = X$. Fix $\pi$ and form $\mathcal{N}^{(i)}_\pi$ and $\mathcal{L}^{(i)}_{\pi}$ accordingly for $i=1,2$. We claim that every arrow in the diagram
	\[ \begin{tikzcd}
		\mathcal{L}^{(1)}_\pi \arrow{r}{\mathcal{F}^\vee} & \mathcal{L}^{(2)}_\pi \\
		\mathcal{N}^{(1)}_\pi \otimes \mathcal{K} \arrow{u} & \mathcal{N}^{(2)}_\pi \otimes \mathcal{K} \arrow{u}
	\end{tikzcd} \]
	is a $\mathcal{K}$-isomorphism. Indeed, the case for $\mathcal{F}^\vee$ stems from the facts that $\mathcal{T}_1 = \mathcal{T} = \mathcal{T}_2$ and $\mathcal{F}$ is an isomorphism, whereas the vertical arrows can be dealt by Theorem \ref{prop:pvs-T_pi-isom}. Hence it can be completed into the commutative diagram \eqref{eqn:gamma-def} for some $\gamma(\pi): \mathcal{N}^{(1)}_\pi \otimes \mathcal{K} \to \mathcal{N}^{(2)}_\pi \otimes \mathcal{K}$, as required.
\MyQED\end{proof}

The Hypothesis \ref{hyp:lfe} can be viewed as a variant of the condition \cite[p.474 (A.2)]{Sa89} in the study of \emph{prehomogeneous local zeta integrals}, which corresponds to the case of trivial representation $\pi = \mathbf{1}$.

The local functional equation for prehomogeneous zeta integrals is established for Archimedean $F$: see \cite[\S 1.4, Theorem $\mathbf{R}$]{Sa89}. Again, the $\gamma$-matrix appears in this framework.

\section{Local Godement--Jacquet integrals}\label{sec:GJ}\index{Godement--Jacquet integrals}
In what follows, $G := \GL(n) \times \GL(n)$ operates on the right of $X := \text{Mat}_{n \times n}$ via
\[ A(g_1, g_2) = g_2^{-1} A g_1, \quad A \in \text{Mat}_{n \times n}, \; (g_1, g_2) \in G. \]
This gives rise to a representation $\rho$ of $G$ on the vector space $X$. It is prehomogeneous: the open orbit is
\[ X^+ := \GL(n) = \{\det \neq 0\} \subset X \]
which falls under the group case of \S\ref{sec:group-case}: indeed, the $G$-action on $X^+$ coincides with \eqref{eqn:group-case-action1}.

Up to a multiplicative constant, the relative invariants are generated by the determinant $\det \in F[X]$. The eigencharacter of $\det$ is
\[ (g_1, g_2) \longmapsto \det g_2^{-1} \det g_1, \]
whose restriction to $Z_G \simeq \Gm^2$ is $(z_2, z_2) \mapsto (z_2^{-1} z_1)^n$. It generates the group $X^*_\rho(G)$. These descriptions work over any field $F$. Note that $\det$ plays two slightly different roles: (i) as an element of $F[X]$, and (ii) as a homomorphism $\GL(n) \to \Gm$.

By direct computation or Theorem \ref{prop:pvs-geometric-axiom}, these data conform to Axiom \ref{axiom:geometric}, and we may identify $\Lambda$ with $\Z$, $\Lambda_X$ with $\Z_{\geq 0}$.

There is another equally reasonable $G$-action on $X$ described as follows. As vector spaces, one can identify $X = \text{Mat}_{n \times n}$ with its dual $\check{X}$ by the perfect symmetric pairing
\[ \angles{\cdot,\cdot}: A \otimes B \longmapsto \Tr(AB). \]

\begin{lemma}
	Under the identification above, the contragredient triplet $(G, \check{\rho}, \check{X})$ is simply the flipped $G$-action on $X$
	\[ A(g_1, g_2) = g_1^{-1}A g_2, \]
	cf.\ the action \eqref{eqn:group-case-action2}.
\end{lemma}
\begin{proof}
	For all $A, B \in \text{Mat}_{n \times n}$, we have
	\begin{align*}
		\angles{A, B\rho(g_1^{-1}, g_2^{-1})} & = \angles{A, g_2 B g_1^{-1}} = \Tr(A g_2 B g_1^{-1}) \\
		& = \Tr(g_1^{-1} A g_2 B) = \angles{g_1^{-1}A g_2, B}.
	\end{align*}
	As $\angles{A, B\rho(g_1^{-1}, g_2^{-1})} = \angles{A\check{\rho}(g_1, g_2)}$, we deduce $A\check{\rho}(g_1, g_2) = g_1^{-1}A g_2$ for all $(g_1, g_2) \in G$.
\MyQED\end{proof}
This already implies the prehomogeneity of $(G, \check{\rho}, X)$: indeed, $\check{X}^+$ is identifiable with $X^+$. We can actually say more.

\begin{proposition}\label{prop:GJ-regular}
	The prehomogeneous vector space $(G, \rho, X)$ is regular.
\end{proposition}
\begin{proof}
	Apply Cramer's rule to show that the rational map
	\[ \frac{\dd(\det)}{\det}: X \dashrightarrow \check{X} \simeq X \]
	is actually $A \mapsto A^{-1}$, hence birational.
\MyQED\end{proof}

For $(G, \check{\rho}, X)$, the basic relative invariant is always $\det \in F[X]$, but its eigencharacter changes into
\[ (g_1, g_2) \longmapsto \det g_2 \det g_1^{-1}, \]
whose restriction to $Z_G = \Gm^2$ is $(z_1, z_2) \mapsto (z_2 z_1^{-1})^n$. Despite the identifications $X = \check{X}$ and $X^+ = \check{X}^+$, we will distinguish their $G$-actions by writing $\check{X}$ or $\check{\rho}$ whenever necessary. Over a local field $F$, they have the same Schwartz space $\Schw(X) = \Schw(\check{X})$ as defined in \S\ref{sec:Fourier}, but equipped with different $G(F)$-actions connected by flipping.

We proceed to describe the $G$-orbits in $X$. By linear algebra, they are parametrized by integers $0 \leq m \leq n$, with representatives
\[ e_m := \left( \begin{array}{c|c} 1_m & \\ \hline & 0_{n-m} \end{array}\right) \in X(F) \]
where $1_m$ (resp.\  $0_{n-m}$) is the identity $m \times m$-matrix (resp.\  zero $(n-m) \times (n-m)$-matrix). Put
\begin{align*}
	P_m & := \begin{pmatrix} \ast_{m \times m} & \\ \ast & \ast \end{pmatrix}, \quad P_m^- := \begin{pmatrix} \ast & \ast \\ & \ast_{m \times m} \end{pmatrix}, \\
	U_{P_m \times P_m^-} & := \text{the unipotent radical of }\; P_m \times P_m^- \subset G, \\
	L_m & := \begin{pmatrix} 1_m & \\ & \ast \end{pmatrix} \subset \GL(n).
\end{align*}
Then the stabilizer of $e_m$ is
\begin{align*}
	\Stab_G(e_m) & = \left\{ \left( \begin{pmatrix} a & \\ \ast & \ast \end{pmatrix}, \begin{pmatrix} a & \\ \ast & \ast \end{pmatrix} \right) : a \in \GL(m) \right\} \\
	& = \text{diag} \begin{pmatrix} \GL(m) & \\ & 1 \end{pmatrix} \cdot (L_m)^2 \cdot U_{P_m \times P_m^-},
\end{align*}

Denote by $M$ the Levi quotient of $P_m \times P_m^- \subset G$. Thus the orbit $e_m G$ is parabolically induced from some homogeneous $M$-space (recall \eqref{eqn:parabolic-induction}). Indeed, we have $\Stab_G(e_m) \subset P_m \times P_m^-$ and
\[ H_M := \Stab_G(e_m) \big/ U_{P_m \times P_m^-} \simeq \text{diag}(\GL(m)) \cdot \GL(n-m)^2. \]
Writing $M = (\GL(m) \times \GL(n-m))^2$, the inducing $M$-variety is readily seen to be
\[ \underbracket{\GL(m)}_{\GL(m) \times \GL(m)\text{-variety}} \times \underbracket{\{\text{pt}\}}_{\GL(n-m)^2\text{-variety}} \]
where the first slot is operated upon as in \eqref{eqn:group-case-action1}.

\begin{proposition}\label{prop:GJ-lfe-condition}
	For $0 \leq m < n$, the restriction of $(g_1, g_2) \mapsto \det g_2^{-1} \det g_1$ to the torus $Z_M \cap H_M$ has infinite order. Consequently, the family $|\det g_2|^{-s} |\det g_1|^s$ in $s \in \CC$ restricts to a complex torus of unramified characters on $(Z_M \cap H_M)(F)$ when $F$ is local non-Archimedean.
\end{proposition}
The ``restrictions'' here are performed as in Lemma \ref{prop:intertwining-finiteness} for both the algebraic and unramified characters: the unipotent radicals always factor out.
\begin{proof}
	We may embed $H_M$ in $G$ in the standard manner. The restriction of $\det g_2^{-1} \det g_1$ is trivial on $\text{diag}(\GL(m))$. It remains to observe that $\det g_2^{-1} \det g_1$ on the center of $\GL(n-m)^2$ has infinite order.
\MyQED\end{proof}

Now comes the zeta integrals. Our reference for the Godement--Jacquet local zeta integrals will be \cite{GH11-2}.

\begin{theorem}\label{prop:GJ-zeta}
	Let $F$ be a local field of characteristic zero. The Axioms \ref{axiom:geometric} and \ref{axiom:zeta} are verified for the regular prehomogeneous vector space $(G, \rho, X)$ and its dual together with their spaces of Schwartz--Bruhat half-densities.
\end{theorem}
\begin{proof}
	We have verified Axiom \ref{axiom:geometric} in Theorem \ref{prop:pvs-geometric-axiom}. By Theorem \ref{prop:pvs-Schw}, the Axiom \ref{axiom:zeta} is satisfied in the non-Archimedean case, and it remains to check the properties of Archimedean zeta integrals. It suffices to consider the triple $(G, \rho, X)$ with $F = \R$. The required properties in this case are all established in \cite{Li18}. Let us give some quick remarks and compare with the original results of Godement--Jacquet.
	\begin{itemize}
		\item The convergence of zeta integrals for $\Re(\lambda) \gg 0$ has been established for the $K \times K$-finite vectors of the form $v \otimes \check{v}$ in $\pi \boxtimes \check{\pi}$, where $K \subset \GL(n,F)$ is any maximal compact subgroup in good position relative to $A$ --- see \cite[\S 15.9]{GH11-2}. To deal with the general case and show continuity, one invokes the asymptotics of matrix coefficients obtained in \cite[Theorem 5.8 or Corollary 5.9]{KSS14}.
		\item Meromorphic continuation: Godement and Jacquet only considered a $K \times K$-invariant subspace $\Schw_0 \subset \Schw(X)$ spanned by ``Gaussian $\times$ polynomials''; the precise definition depends on $\psi$ as well and can be found in \cite[p.115]{GJ72}. This restriction on $\xi$ has been removed in \cite{Li18}. When restricted to the $K \times K$-finite vectors in $V_\Pi$, the zeta integrals admit meromorphic continuation by \cite[Theorem 15.9.1]{GH11-2} or \cite[Theorem 8.7]{GJ72}; the latter reference allows general $\pi$ and $F=\CC$. The point is to extend this to all vectors of $V_\Pi$ and obtain meromorphy together with continuity. This has also been achieved in \cite{Li18}: the basic tool is the \emph{method of analytic continuation} of Theorem \ref{prop:GS-principle}.
		\item  Finally, the $G(F)$-invariance follows from the case $\Re(\lambda) \gg 0$ by meromorphy.
	\end{itemize}
	This completes our proof modulo \cite{Li18}.
\MyQED\end{proof}

Recall that the Fourier transform of Schwartz--Bruhat half-densities $\mathcal{F}: \Schw(X) \rightiso \Schw(\check{X})$ affords a model transition.

\begin{corollary}\label{prop:GJ-lfe}
	Let $F$ be a local field of characteristic zero. The local functional equation holds for $\mathcal{F}: \Schw(X) \rightiso \Schw(\check{X})$ and every irreducible nice representation $\Pi$ of $G(F)$.
\end{corollary}
The proof is divided into two cases plus one necessary intermezzo.
\begin{proof}[Proof of Corollary \ref{prop:GJ-lfe}: non-Archimedean case]
	By Proposition \ref{prop:GJ-lfe-condition} we know Hypothesis \ref{hyp:lfe} is satisfied for both $X$ and $\check{X}$ by symmetry. It remains to apply Theorem \ref{prop:pvs-lfe}.
\MyQED\end{proof}

\begin{remark}
	Let us reconcile Corollary \ref{prop:GJ-lfe} with the usual Godement--Jacquet local functional equation treated in \cite{GH11-1}.
	
	Let the representation $\Pi = \pi \boxtimes \check{\pi}$ be irreducible and nice. Choose $\eta \in \topwedge \check{X} \smallsetminus \{0\}$, we obtain a Haar measure $|\Omega| := |\det|^{-n}|\eta|$ on $\GL(n,F)$, hence an invariant global section $|\Omega|^\demi$ of $\mathscr{L}^\demi$. For this choice and \S\ref{sec:group-case} we obtain generators $\varphi$ and $\check{\varphi}$ of $\mathcal{N}_\Pi$ and $\check{\mathcal{N}}_\Pi$, respectively. We may write
	\begin{align*}
		Z(\lambda, \xi, v \otimes \check{v}) & := Z_{\lambda, \varphi}\left( (v \otimes \check{v}) \otimes \xi \right) \\
		\check{Z}(\lambda, \xi, v \otimes \check{v}) & := Z_{\lambda, \check{\varphi}}\left( (v \otimes \check{v}) \otimes \xi \right)
	\end{align*}
	for $v \otimes \check{v} \in V_\pi \otimes V_{\check{\pi}}$, $\xi \in \Schw(X) = \Schw(\check{X})$ and $\lambda \in \Lambda_{\CC} = \CC$. By the previous conventions, $\Pi_\lambda = \pi|\det|^\lambda \boxtimes \check{\pi}|\det|^{-\lambda}$, hence
	\begin{gather*}
		\varphi_\lambda = |\det|^\lambda \varphi, \quad \check{\varphi}_\lambda = |\det|^{-\lambda} \check{\varphi}.
	\end{gather*}
	This behavior is compatible with the fact that the basic relative invariants of $X$ and $\check{X}$ are opposite: $\Lambda_X = \Z_{\geq 0}$ whereas $\Lambda_{\check{X}} = \Z_{\leq 0}$, if we let $1 \in \Z$ correspond to the eigencharacter $\det g_2^{-1} \det g_1$.

	By the choice of $\eta$, every $\xi \in \Schw(X)$ may be expressed as
	\[ \xi = \xi_0 |\eta|^{\demi} = \xi_0 |\det(\cdot)|^{\frac{n}{2}} \cdot |\Omega|^\demi, \quad \xi_0 \in \text{SB}(X). \]
	Simply put, $\Schw(X)$ may be identified with $|\det|^{\frac{n}{2}} \text{SB}(X)$ by choosing $\eta$; similarly, $\Schw(\check{X})$ is identifiable with the same function space, but the $G(F)$-action is flipped.
	
	To facilitate the transition to the Godement--Jacquet set-up, we shall assume that $|\eta|$ gives the self-dual Haar measure on $X(F)$ relative to $\psi$. Unwinding the constructions in Theorem \ref{prop:Fourier-def} (or its proof), we get the commutative diagram of continuous linear maps:
	\begin{equation}\label{eqn:two-Fourier} \begin{tikzcd}
		\xi \arrow[mapsto]{d} & \Schw(X) \arrow{r}[above]{\mathcal{F}} \arrow{d}[left]{\simeq} & \Schw(\check{X}) \arrow{d}[right]{\simeq} \\
		\xi_0 & \text{SB}(X) \arrow{r}[below, inner sep=1em]{\mathcal{F}_0 := \text{usual Fourier}} & \text{SB}(\check{X})
	\end{tikzcd}\end{equation}
	
	Denote by $Z^\text{GJ}(\cdots)$ the usual Godement--Jacquet integrals as defined in \cite[(15.4.3)]{GH11-2}.	Using $|\Omega| = |\det|^{-n}|\eta|$ to perform integration over $X^+(F) = \GL(n,F)$, we see that
	\begin{align*}
		Z(\lambda, \xi, v \otimes \check{v}) & = \int_{X^+(F)} |\det|^{\lambda + \frac{n}{2}} \angles{v, \pi(\cdot) v} \xi_0 |\Omega| \\
		& = Z^\text{GJ}\left( \lambda + \demi, \xi_0, \angles{v, \pi(\cdot)v} \right), \quad \Re(\lambda) \gg 0, \\
		\check{Z}(\lambda, \xi, v \otimes \check{v}) & = \int_{X^+(F)} |\det|^{-\lambda + \frac{n}{2}} \angles{\check{\pi}(\cdot)v, v} \xi_0 |\Omega| \\
		& = Z^\text{GJ}\left( -\lambda + \demi, \xi_0, \angles{\check{\pi}(\cdot)v, v} \right), \quad \Re(\lambda) \ll 0.
	\end{align*}
	By meromorphy, these relations between $Z$, $\check{Z}$ and $Z^\text{GJ}$ extends to all $\lambda$.
	
	Recall that $\mathcal{N}_\Pi$, $\check{\mathcal{N}}_\Pi$ have been identified with $\CC$. The final step is to reinterpret Definition \ref{def:lfe}. It asserts the existence of a factor $\gamma(\pi, \lambda)$, rational in $q^\lambda$, such that
	\[ \check{Z}(\lambda, \mathcal{F}(\xi), v \otimes \check{v}) = \gamma(\pi, \lambda)Z(\lambda, \xi, v \otimes \check{v}) \]
	as a meromorphic function in $\lambda \in \CC$, for all $v \otimes \check{v}$ and $\xi \in \Schw(X)$. Now we switch to $Z^\text{GJ}$. Put $\mu := \lambda + \demi$, $\gamma^\text{GJ}(\mu, \pi) := \gamma\left( \pi, \mu - \demi \right)$, $\beta := \angles{\check{v}, \pi(\cdot)v}$ and $\beta^\vee(x) = \beta(x^{-1})$. Diagram \eqref{eqn:two-Fourier} reads as $\mathcal{F}(\xi)_0 = \mathcal{F}_0(\xi_0)$. Therefore
	\begin{gather}\label{eqn:GJ-lfe-orig}
		Z^\text{GJ}( 1-\mu, \mathcal{F}_0(\xi_0), \beta^\vee ) = \gamma^\text{GJ}(\mu, \pi) Z^\text{GJ}( \mu, \xi_0, \beta ), \quad \xi_0 \in \text{SB}(X).
	\end{gather}
	This is exactly the Godement--Jacquet local functional equation stated in \cite[Theorem 15.4.4 (3)]{GH11-2}. The Archimedean case admits a similar interpretation.

	Remarkably, our formalism is free from mysterious shifts such as $(n-1)/2$: it is just a shadow of the half-densities.
\end{remark}

\begin{proof}[Proof of Corollary \ref{prop:GJ-lfe}: Archimedean case]
	Thanks to the foregoing remark, we can exploit the Archimedean case of \eqref{eqn:GJ-lfe-orig} established by Godement--Jacquet as in \cite[Theorem 15.9.1]{GH11-1}. However, the original arguments by Godement and Jacquet include two premises:
	\begin{itemize}
		\item the zeta integrals apply only to $K \times K$-finite pure tensors in $\pi \boxtimes \check{\pi}$;
		\item the Schwartz--Bruhat half-density $\xi$ is taken from the subspace $\Schw_0$ in the proof of Theorem \ref{prop:GJ-zeta}, which is known to be dense and preserved under Fourier transform.
	\end{itemize}

	These two premises are easily removed by the continuity of $Z_\lambda$ established in Theorem \ref{prop:GJ-zeta}, upon clearing denominators. 
\MyQED\end{proof}

\chapter{The doubling method}\label{sec:doubling}
In order to reconcile our framework with that of Braverman--Kazhdan \cite{BK02}, two Conjectures \ref{conj:integral-coinvariants} and \ref{conj:Schw-BK-mult1} will be postulated.

\section{Geometric set-up}\label{sec:geometric-setup}
Consider a field $F$; assume $\text{char}(F)=0$ unless otherwise specified. Let $G^\Box$ be a connected reductive $F$-group and consider a parabolic subgroup $P \subset G^\Box$ with Levi quotient $P/U_P = M$\index{G-Box@$G^\Box$}. To simplify matters, we assume that
\begin{gather}\label{eqn:H1-assumption}
	M(F) \twoheadrightarrow M_\text{ab}(F)
\end{gather}	

Consider the homogeneous $M_\text{ab} \times G^\Box$-space\index{X_P@$X_P$}
\begin{gather*}
	X_P := P_\text{der} \backslash G^\Box
\end{gather*}
endowed with the action
\begin{gather*}
	(P_\text{der} y) \cdot (\bar{m}, g) = P_\text{der}m^{-1}yg, \quad (\bar{m}, g) \in M_\text{ab} \times G^\Box
\end{gather*}
where $m \in P$ is any lifting of $\bar{m}$. From the Bruhat decomposition, one easily sees that $X_P$ is spherical.

By considering the stabilizer of the coset $P_\text{der} \cdot 1 \in X_P(F)$, we may also realize $X_P$ as $P \backslash (M_\text{ab} \times G^\Box)$, the embedding of $P$ being determined by $P \to P/P_\text{der} \rightiso M_\text{ab}$ (abelianization) and $P \hookrightarrow G^\Box$ (inclusion). Since $H^1(F, P_\text{der}) \to H^1(F, P)$ is injective by \eqref{eqn:H1-assumption} whereas the injectivity of $H^1(F, P) \to H^1(F, G^\Box)$ is well-known, we have
\[ X_P(F) = (P_\text{der} \backslash G^\Box)(F) = P_\text{der}(F) \backslash G^\Box(F). \]

\begin{proposition}\label{prop:X_P-LV}
	Over the algebraic closure, the Luna--Vust datum for $X_P$ is described as follows:
	\begin{itemize}
		\item $\Lambda_{X_P} = X_*(M_\mathrm{ab})$, $\mathcal{V} = \mathcal{Q} = X_*(M_\mathrm{ab})_{\Q}$;
		\item $\mathcal{D}^B$ is in natural bijection with $\check{\Delta}_P = \check{\Delta}_0 \smallsetminus \check{\Delta}^M_0$ where $\check{\Delta}_0$ (resp.\  $\check{\Delta}^M_0$) denotes the set of simple coroots of $G^\Box$ (resp.\  of $M$), with respect to compatibly chosen Borel subgroups;
		\item furthermore, $X_P$ is quasi-affine and $\overline{X_P}^\mathrm{aff}$ corresponds to the colored cone $(\mathcal{C}, \mathcal{F})$ with $\mathcal{F} = \mathcal{D}^B$ and $\mathcal{C}$ generated by $\{ \rho(D) :D \in \mathcal{D}^B\}$, i.e.\ by the images in $X_*(M_\mathrm{ab})$ of the coroots in $\check{\Delta}_P$.
	\end{itemize}
\end{proposition}
\begin{proof}
	The $M_\text{ab} \times G^\Box$-space $X_P$ is \emph{horospherical}\index{spherical variety!horospherical}, i.e.\ the stabilizer of any $x_0 \in X_P$ contains a maximal unipotent subgroup, hence $\mathcal{V} = \mathcal{Q}$ by \cite[Corollary 6.2]{Kno91}. The descriptions of $\Lambda_{X_P}$ and colors follow readily from Bruhat decomposition. The description of $(\mathcal{C}, \mathcal{F})$ is an easy consequence of Theorem \ref{prop:LV}; see also \cite[Remarks 4.3.3]{Sak12}.
\MyQED\end{proof}

Now we move to the case of a local field $F$. By the description in \S\ref{sec:integration-density} of $\mathscr{L}^s$, the line bundle on $X_P(F)$ of $s$-densities, $P(F)$ acts on the fiber of $\mathscr{L}^s$ at $P_\text{der} \cdot 1$ via the character
\[ t \mapsto \delta_P(t)^s = \left| \det(\Ad(t) | \mathfrak{u}_P ) \right|^s, \quad t \in P(F). \]

As usual, we use $\mathscr{L}^\demi$-valued sections to define
\[ L^2(X_P), \quad C^\infty_c(X_P) := C^\infty_c(X_P(F), \mathscr{L}^\demi), \quad C^\infty(X_P) := C^\infty(X_P(F), \mathscr{L}^\demi). \]
Therefore $(M_\text{ab} \times G^\Box)(F)$ acts unitarily on $L^2(X_P)$. Note that in \cite{BK02,Sak12} one adopts the usual $L^2$-space with respect to a $G^\Box(F)$-invariant measure $|\Xi|$; it transforms under $M_\text{ab}(F)$ as $|\bar{m}\Xi| = \delta_P(m) |\Xi|$ by the formula above, thus the unitary $M_\text{ab}(F)$-action on $L^2(X_P, |\Xi|^\demi)$ is $\xi_\text{BK}(y) \mapsto \xi_\text{BK}(\bar{m}^{-1}y) \delta_P(m)^{\demi}$. Summing up, a function $\xi_\text{BK}$ in the framework of \cite[(1.1)]{BK02} corresponds to a section $\xi = \xi_\text{BK} |\Xi|^\demi$. The conclusions remain unchanged.

\begin{remark}\label{rem:N_Pi-X_P}
	Let $\chi$ (resp.\  $\pi$) be a character of $M_\text{ab}(F)$ (resp.\  representation of $G^\Box(F)$) that is smooth and continuous, and put $\Pi := \chi \boxtimes \pi$. Frobenius reciprocity for continuous representations (see for example \cite[\S 2.5]{Be88}) implies
	\begin{align*}
		\Hom_{M_\text{ab} \times G^\Box(F)}(\chi \boxtimes \pi, C^\infty(X_P)) &= \Hom_{P(F)}(\chi \boxtimes \pi, \delta_P^\demi) = \Hom_{P(F)}(\pi, \check{\chi} \otimes \delta_P^\demi) \\
		& = \Hom_{G^\Box(F)}(\pi, I^{G^\Box}_P(\check{\chi})).
	\end{align*}
	Here the $\Hom$-spaces are continuous by default. Specializing to the case $\pi$ nice and irreducible, this yields a description for $\mathcal{N}_\Pi$. Unsurprisingly, we see that $X_P$ captures the part of the spectrum of $G^\Box(F)$ parabolically induced from various $\check{\chi}: M(F) \to M_\text{ab}(F) \to \CC^\times$.
\end{remark}

Now assume
\begin{compactitem}
	\item $F$: non-Archimedean local field of characteristic zero,
	\item $\psi: F \to \CC^\times$: nontrivial unitary character,
	\item $G^\Box$: split connected reductive group such that $G^\Box_\text{der}$ is simply connected.
\end{compactitem}

We normalize the Haar measure on $F$ by $\psi$. For any two parabolic subgroups $P, Q \subset G^\Box$ sharing a common Levi component $M$, Braverman and Kazhdan \cite[Theorem 1.4]{BK02} defined a normalized intertwining operator of $(M_\text{ab} \times G^\Box)(F)$-representations\index{Fourier-BK@$\mathcal{F}_{Q"|P}$}
\[ \mathcal{F}_{Q|P} = \mathcal{F}_{Q|P, \psi}: L^2(X_P) \to L^2(X_Q) \]
satisfying
\begin{compactenum}[(i)]
	\item $\mathcal{F}_{Q|P}$ is an isometry;
	\item $\mathcal{F}_{R|Q} \mathcal{F}_{Q|P} = \mathcal{F}_{R|P}$ for all $P, Q, R$ sharing the same $M$;
	\item $\mathcal{F}_{P|P} = \identity$, consequently $\mathcal{F}_{P|Q} \mathcal{F}_{Q|P} = \identity$ for all $P, Q$.
\end{compactenum}

Let us review the construction in \textit{loc.\ cit.} of $\mathcal{F}_{P|Q}$ briefly: it is the composite of two operators as follows.
\begin{enumerate}
	\item The intertwining operator $\mathcal{R}_{Q,P}: C^\infty_c(X_Q) \to C^\infty(X_P)$ defined by the same formula as the standard intertwining operator $J_{P|Q}(\cdot)$, essentially an integration over $(U_P \cap U_Q \backslash U_P)(F)$ (see \cite[(1.4)]{BK02}). It is convergent since unipotent group acts with Zariski-closed orbits on any affine variety. Since we are working with half-densities, by pairing $\mathfrak{u}_P/\mathfrak{u}_P \cap \mathfrak{u}_Q$ with $\mathfrak{u}_Q/\mathfrak{u}_P \cap \mathfrak{u}_Q$ via Killing form one can verify that $\mathcal{R}_{Q,P}$ is well-defined: it does not involve any choice of Haar measures, cf.\ the proof of Theorem \ref{prop:Fourier-def}.

	\item Define an action $\phi \mapsto \eta \ast \phi$ on the compactly supported functions or sections $\phi$ on $X_P(F)$, where $\eta$ is a density-valued function on $M_\text{ab}(F)$:
	\[ \eta \ast \phi : x \mapsto \int_{M_\text{ab}} \eta(\bar{m}) \phi(x \cdot \bar{m}). \]
		It is an authentic convolution if we let $\bar{m} \in M_\text{ab}$ act on the left of $X_P$ via $P_\text{der}y \mapsto P_\text{der}\bar{m}y$. The second operator is then given by left convolution with a suitable distribution $\eta_{Q, P, \psi}$ on $M_\text{ab}(F)$, but some regularization is needed: see \cite[\S 2]{BK02} for the precise meaning.
	\item Finally, \cite[Theorem 1.6]{BK02} asserts that $\mathcal{F}_{P|Q}(\xi) := \eta_{Q, P, \psi} \ast \mathcal{R}_{Q, P}(\xi)$ is well-defined on a dense subspace $C^0_c(X_Q) \subset C^\infty_c(X_Q)$ dense in $L^2(X_Q)$, with image in $C^\infty_c(X_P)$, and extends to an isometry $L^2(X_Q) \rightiso L^2(X_P)$.
\end{enumerate}

The \emph{Schwartz space} $\Schw(X_P)$\index{Schwartz-BK@$\Schw(X_P)$} is defined as the smooth $(M_\text{ab} \times G)(F)$-representation
\begin{gather}\label{eqn:Schw-BK}
	\Schw(X_P) := \sum_Q \mathcal{F}_{P|Q}(C^\infty_c(X_Q)),
\end{gather}
the sum ranging over all $Q$ with the same Levi component $M$. Therefore $\Schw(X_P)$ is a vector subspace of $L^2(X_P) \cap C^\infty(X_P)$; indeed, the smoothness follows from equivariance. In \textit{loc.\ cit.} it is denoted as $\Schw(G, M)$ since the spaces $\Schw(X_P)$ for different $P$ are canonically identified via the normalized intertwining operators; similarly, they write $L^2(X_P)$ as $L^2(G, M)$.

\begin{lemma}\label{prop:X_P-bdd}
	For every $\xi \in \Schw(X_P)$, the closure of $\Supp(\xi)$ in $\overline{X_P}^{\mathrm{aff}}(F)$ is compact.
\end{lemma}
\begin{proof}
	In view of Corollary \ref{prop:compact-exhaustion} and the description of Luna--Vust data for $\overline{X_P}^\text{aff}$ in Proposition \ref{prop:X_P-LV}, this is equivalent to \cite[Conjecture 5.6]{BK02}. The general case is established in \cite[Lemma 5.1]{GL17}; cf.\\ the discussion in the beginning of \cite[\S 4]{GL17}. 
\MyQED\end{proof}

\begin{remark}\label{rem:normalized-intop}
	Below is a summary of several facts from the constructions in \cite{BK02} (see the Introduction, \S 1.5 and \S\S 3--4 therein). Let $\chi$ be an irreducible nice representation of $M_\text{ab}(F)$.
	\begin{itemize}
		\item Choose a Haar measure of $M_\text{ab}(F)$ and form the $(M_\text{ab} \times G^\Box)(F)$-equivariant map
			\begin{align*}
				\check{\chi} \boxtimes C^\infty_c(X_P) & \longrightarrow C^\infty(X_P) \\
				1 \otimes \xi & \longmapsto \int_{M_\text{ab}(F)} \chi(m) (m\xi)(\cdot) \dd m
			\end{align*}
			According to Remark \ref{rem:N_Pi-X_P} with $\pi := C^\infty_c(X_P)$, the recipe gives a $G^\Box(F)$-equivariant map (extraction of the $\chi$-component) $C^\infty_c(X_P) \to I^{G^\Box}_P(\chi)$, which is easily seen to be surjective.
		\item This operation extends to a rational family of $(M_\text{ab} \times G^\Box)(F)$-equivariant maps $\check{\chi} \boxtimes \Schw(X_P) \to C^\infty(X_P)$, or equivalently the $G^\Box(X_P)$-equivariant maps
			\[ \Schw(X_P) \longrightarrow I^{G^\Box}_P(\chi). \]
			when $\chi$ varies in a $X^*(M_\text{ab})_{\CC}$-orbit. Rationality here is understood as in \S\ref{sec:auxiliary}, by working with a chosen $X^*(M_\text{ab})_{\CC}$-orbit $\mathcal{T}$ of unramified characters.
		\item For every pair of parabolic subgroups $P, Q \subset G^\Box$ sharing the same Levi $M$, there exists a rational family of normalized intertwining operators $R_{P|Q}(\chi): I^{G^\Box}_Q(\chi) \to I^{G^\Box}_P(\chi)$ with $\chi$ varying in $\mathcal{T}$, making the diagram
			\[ \begin{tikzcd}
				\Schw(X_Q) \arrow{r}[above, inner sep=1em]{\mathcal{F}_{P|Q}} \arrow{d} & \Schw(X_P) \arrow{d} \\
				I^{G^\Box}_Q(\chi) \arrow{r}[below, inner sep=1em]{R_{P|Q}(\chi)} & I^{G^\Box}_P(\chi)
			\end{tikzcd} \]
			commutative for $\chi$ in general position. Thus $R_{P|Q}(\chi)$ can be thought as the $\chi$-component of $\mathcal{F}_{P|Q}$. More precisely, $R_{P|Q}(\chi)$ equals the product of some rational function on $\mathcal{T}$ with the standard intertwining operator $J_{P|Q}(\chi): I^{G^\Box}_Q(\chi) \to I^{G^\Box}_P(\chi)$, which is also rational.
	\end{itemize}

	The statements concerning $\xi \in C^\infty_c(X_P)$ are easy. Now consider the rational continuation to $\Schw(X_P)$. We begin with the case $\xi = \mathcal{F}_{P|Q}(\xi_Q)$ where $\xi_Q \in C^\infty_c(X_Q)$. In the proof of \cite[Proposition 4.2]{BK02}, $C^0_c(X_Q)$ is defined as $\tau_{P,Q} \ast C^\infty_c(X_Q)$ where $\tau_{P,Q}$ belongs to Bernstein's center of $M_\text{ab}$. Moreover, it is shown that $\mathcal{R}_{Q, P}$ restricts to $C^0_c(X_Q) \to C^\infty_c(X_P)$, and ``covers'' $J_{P|Q}(\chi): I^{G^\Box}_Q(\chi) \to I^{G^\Box}_P(\chi)$ as in the diagram above. 

	On the other hand, there exists a rational function $\chi \mapsto r_{P, Q, \psi}(\chi)$ such that convolution by $\eta_{Q,P,\psi}$ on $C^\infty_c(X_P)$ covers the map $I^{G^\Box}_P(\chi) \xrightarrow{r_{P, Q, \psi}(\chi) \cdot \identity} I^{G^\Box}_P(\chi)$, for general $\chi$; this follows by \cite[2.3]{BK02}. Thus the case $\xi \mapsfrom \xi_Q \in C^0_c(X_Q)$ is established. Given the construction of $C^0_c(X_Q)$ alluded to above, the rational continuation extends to all $\xi \mapsfrom \xi_Q \in C^\infty_c(X_Q)$, thus to all $\xi \in \Schw(X_P)$. Thus $\mathcal{F}_{P|Q}$ covers a rational family $R_{Q|P}(\chi)$.
\end{remark}

It is unclear whether the following property is established in \cite{BK02}.
\begin{conjecture}\label{conj:integral-coinvariants}
	Choose a Haar measure on $M_\text{ab}(F)$. Upon twisting $\chi$ by a unramified character $M_\text{ab}(F)$ that is ``sufficiently negative'' relative to the roots in $U_P$, the map $\Schw(X_P) \to I^{G^\Box}_P(\chi)$ above is well-defined by the convergent integral $\int_{M_\text{ab}(F)} \chi(m) (m\xi) \dd m$.
\end{conjecture}

\section{The symplectic case}\label{sec:symplectic-case}
The general \emph{doubling method} was introduced in \cite{PSR86} and \cite[Part A]{GPSR87}. The original concern is to study certain $L$-functions attached to automorphic representations on classical groups or similitude groups without assuming genericity. Here we follow the reinterpretation of \cite[\S 7]{BK02}, which is closer in spirit to Godement--Jacquet theory and relies on the aforementioned Schwartz spaces. The underlying geometric mechanisms are the same.

For the sake of simplicity, in this work we treat mainly the case of symplectic groups for which \eqref{eqn:H1-assumption} is easy to verify. Let $(V, \angles{\cdot|\cdot})$ be a finite-dimensional symplectic $F$-vector space, $V \neq \{0\}$\index{V-space@$V$}. Let $G := \Sp(V)$ be the corresponding symplectic group. ``Doubling'' refers to the following construction \index{V-Box@$V^\Box$}\index{G-Box@$G^\Box$}
\begin{gather*}
	(V^\Box, \angles{\cdot|\cdot}^\Box) := (V, \angles{\cdot|\cdot}) \overset{\perp}{\oplus} (V, -\angles{\cdot|\cdot}), \\
	V^\Delta := \text{diag}(V) \subset V^\Box , \\
	G^\Box := \Sp(V^\Box).
\end{gather*}
Then $V^\Delta$ is a canonically defined Lagrangian subspace of $V^\Box$, i.e.\ a totally isotropic subspace of dimension $(\dim V^\Box)/2 = \dim V$. Hence we obtain the maximal parabolic subgroup
\[ P := \Stab_{G^\Box}(V^\Delta). \]
The Levi quotient $M := P/U_P$ is canonically isomorphic to $\GL(V^\Delta)$. The geometric quotient $X_P^\flat := P \backslash G^\Box$ classifies the Lagrangians in $V^\Box$: to each coset $Pg$ we attach the Lagrangian $V^\Delta g$ in $V^\Box$, where $G^\Box$ acts on the right of $V^\Box$ as usual. Form $X_P := P_\text{der} \backslash G^\Box$ as before. This yields an $M_\text{ab}$-torsor $X_P \twoheadrightarrow X_P^\flat$. The morphism is clearly $M_\text{ab} \times G^\Box$-equivariant. It can be made more precise as follows: fix a volume form $\Lambda_0 \in \topwedge V^\Delta$, $\Lambda_0 \neq 0$, then the diagram
\begin{equation}\label{eqn:iden-Grass} \begin{tikzcd}
	X_P \arrow[twoheadrightarrow]{d} \arrow{r}[above]{\sim} & \left\{ (\ell, \Lambda): \begin{array}{l} \ell \subset V^\Box: \text{Lagrangian} \\ \Lambda \in \topwedge \ell, \; \Lambda \neq 0 \end{array} \right\} \arrow[twoheadrightarrow]{d}[right]{\text{pr}_1} \\
	X_P^\flat \arrow{r}[below]{\sim} & \left\{ \ell \subset V^\Box : \text{Lagrangian} \right\}
	\end{tikzcd} \qquad \begin{tikzcd}
		P_\text{der}g \arrow[mapsto]{r} \arrow[mapsto]{d} & (\ell g, \Lambda_0 g) \arrow[mapsto]{d} \\
		P g \arrow[mapsto]{r} & \ell g
\end{tikzcd}\end{equation}
is commutative. The arrows above can be made $M_\text{ab} \times G^\Box$-equivariant if we let $M_\text{ab} \times G^\Box$ act on the right of the pairs $(\ell, \Lambda)$ as
\begin{gather*}
	(\ell, \Lambda)(\bar{m}, g) = (\ell g, t \cdot \Lambda g), \quad m \in \GL(V^\Delta), \; g \in \Sp(V^\Box)
\end{gather*}
with the conventions
\begin{compactitem}
	\item $\bar{m}$ denotes the image of $m$ in $M_\text{ab}$,
	\item $t := \det(m)^{-1}$,
	\item and $\Sp(V^\Box)$ acts on the right of the exterior algebra $\bigwedge V^\Box$.
\end{compactitem}

\begin{notation}\label{nota:det-Gm}
	Hereafter, we shall fix $\Lambda_0$ and identify $X_P \twoheadrightarrow X_P^\flat$ with the right column of \eqref{eqn:iden-Grass}. It is sometimes convenient to identify $M_\text{ab} \times G$ with $\Gm \times G$ using $\bar{m} \mapsto t = \det(m)^{-1}$ so that the action becomes
	\[ (\ell, \Lambda)(t, g) = (\ell g, t \cdot \Lambda g). \]
\end{notation}

On the other hand, $G \times G$ embeds naturally into $G^\Box$. Set $V^+ := V \oplus \{0\}$ and $V^- := \{0\} \oplus V$; they are Lagrangians of $V^\Box$.

\begin{theorem}[{\cite[I. Lemma 2.1]{GPSR87}}]\label{prop:GPSR-orbits}
	For every field extension $E$ of $F$, the $(G \times G)(E)$-orbits in $X_P^\flat(E)$, or equivalently the $(G \times G)(E)$-orbits of Lagrangians $\ell \subset V^\Box \dotimes{F} E$, are classified by the invariant
	\[ \kappa(\ell) := \kappa^\pm(\ell) = \dim_E  \left( \ell \cap (V^\pm \dotimes{F} E) \right). \]

	In particular, there exists an Zariski-open dense orbit $X^{+,\flat}$; it is characterized as $\{\ell: \kappa(\ell)=0\}$, and we have $\kappa(V^\Delta)=0$. Moreover, $P \cap (G \times G) = \mathrm{diag}(G)$.
	
	For any $\gamma \in (X^\flat \smallsetminus X^{+,\flat})(F)$, there exists a proper parabolic subgroup $R \subset G \times G$ such that $U_R \subset \Stab_{G \times G}(\gamma)$.
\end{theorem}
Implicit in the statements is the property that $\kappa^+(\ell) = \kappa^-(\ell)$. Let $G \times G$ act on $G$ as in \eqref{eqn:group-case-action1}. The theorem furnishes the $G \times G$-equivariant embedding
\begin{equation}\label{eqn:G-doubling-embedding} \begin{aligned}
	G & \hookrightarrow X_P^\flat \\
	g & \longmapsto V^\Delta (g,1): \text{ Lagrangian } \subset V^\Box.
\end{aligned}\end{equation}
It is an open immersion onto the open $G \times G$-orbit $X^{+,\flat}$. From the $M_\text{ab}$-torsor structure, the inverse image $X^+ \subset X_P$ of $X^{+,\flat}$ is readily seen to be an open dense $M_\text{ab} \times G \times G$-orbit in $X_P$.

\begin{lemma}\label{prop:GmG-embedding}
	Choose $\Lambda_0 \in \topwedge V^\Delta \smallsetminus \{0\}$ in \eqref{eqn:iden-Grass} as before. The composite below is an $M_\mathrm{ab} \times G \times G$-equivariant isomorphism (cf.\ (Notation \ref{nota:det-Gm}))
	\[\begin{tikzcd}[row sep=tiny]
		M_\mathrm{ab} \times G \arrow[r, "\sim"] & \Gm \times G \arrow{r} & X^+ \\
		(\bar{m}, g) \arrow[mapsto, r] & (\det(m)^{-1}, g) & \\
		& (t, g) \arrow[mapsto, r] & (V^\Delta (g, 1), t \Lambda_0 (g,1))
	\end{tikzcd}\]
	that covers \eqref{eqn:G-doubling-embedding}, the action on $M_\mathrm{ab} \times G$ being $(\bar{m}, g) \stackrel{(\bar{a}, g_1, g_2)}{\longmapsto} (\bar{m}\bar{a}, g_2^{-1} g g_1)$.
\end{lemma}
\begin{proof}
	Let $(\ell, \Lambda) \in X^+$ with $\ell \leftrightarrow g \in G$ under \eqref{eqn:G-doubling-embedding}, the inverse morphism is given by sending $(\ell, \Lambda)$ to $\left(\frac{\Lambda}{\Lambda_0 g}, g \right)$. The equivariance is routine to check.
\MyQED\end{proof}

For any $m \geq 0$, let $\pr^\vee(\bigwedge^m V^\Box)$ denote the space classifying lines in $\bigwedge^m V^\Box$. The identification in \eqref{eqn:iden-Grass} leads to the cartesian diagram
\begin{equation}\label{eqn:Plucker-embedding} \begin{tikzcd}
	X_P \arrow[hookrightarrow]{r} \arrow[twoheadrightarrow]{d} & \bigwedge^{\dim V} V^\Box \smallsetminus \{0\} \arrow[twoheadrightarrow]{d} \\
	X_P^\flat \arrow[hookrightarrow]{r} & \pr^\vee(\bigwedge^{\dim V} V^\Box)
\end{tikzcd}\end{equation}
where the lower horizontal arrow is the Plücker embedding: namely a Lagrangian $\ell$ is mapped to the line $\bigwedge^{\dim V} \ell$ in $\bigwedge^{\dim V} V^\Box$. All arrows are $M_\text{ab} \times G^\Box$-equivariant. As $\dim \bigwedge^{\dim V} V^\Box 
\geq 2$, one infers immediately that $X_P$ is quasi-affine: it is even possible to write down the defining equations of $X := \overline{X_P}^\text{aff}$ inside the affine space $\bigwedge^{\dim V} V^\Box$. We obtain $M_\text{ab} \times G \times G$-equivariant embeddings
\[ X^+ \xrightarrow{\text{open}} X_P \xrightarrow{\text{open}} X := \overline{X_P}^{\text{aff}} \xrightarrow{\text{closed}} \bigwedge^{\dim V} V^\Box. \]
Note that $X = X_P \sqcup \{\vec{0}\}$. \index{0@$\vec{0}$}

\begin{lemma}
	The homogeneous $M_\mathrm{ab} \times G \times G$-space $X^+$ is spherical and wavefront. Moreover, $X^+ \hookrightarrow X$ is an affine spherical embedding.
\end{lemma}
\begin{proof}
	Combining Lemma \ref{prop:GmG-embedding} with the discussions in \S\ref{sec:group-case} gives the first part. For the second part, it is routine to see that $X = \overline{X_P}^\text{aff}$ is geometrically integral; the normality follows from that of $X_P$. 
\MyQED\end{proof}

\begin{lemma}\label{prop:f-pm}\index{f-plus@$f^+$}
	Under the affine embedding $X^+ \hookrightarrow X$ above, the reduced closed subscheme $\partial X := X \smallsetminus X^+$ is the zero locus of the relative invariant under $M_\mathrm{ab} \times G \times G$
	\begin{align*}
		f^+: \bigwedge^{\dim V} V^\Box & \longrightarrow \topwedge V^\Box \simeq \Ga \\
		\Lambda & \longmapsto \Lambda \wedge v^+
	\end{align*}
	where $v^+ \in \topwedge V^+$, $v^+ \neq 0$ under the Plücker embedding. Furthermore, the restriction of $f^+$ to $\Gm \times G \rightiso X^+ \subset X$ (using Lemma \ref{prop:GmG-embedding}) is proportional to $\mathrm{pr}_1: \Gm \times G \twoheadrightarrow \Gm$. Consequently its eigencharacter is given by
	\begin{align*}
		\omega_1: M_\mathrm{ab} \times G \times G & \longrightarrow \Gm \\
		(\bar{m}, g_1, g_2) & \longmapsto t := (\det m)^{-1}.
	\end{align*}

	The same conclusion holds if we use $V^-, v^-$ and the corresponding regular function $f^-$.
\end{lemma}
\begin{proof}
	It suffices to deal with the case of $V^+$. By Theorem \ref{prop:GPSR-orbits}, the open orbit $X^+$ realized as the inverse image of $X^{+,\flat} \subset  X_P^\flat$ is characterized by $\ell \cap V^+ \neq \{0\}$. This is equivalent to $\Lambda \wedge v^+ \neq 0$ in $\topwedge V^\Box$, where $\Lambda \in \bigwedge^{\dim V} V^\Box$ lies over the image of $\ell$ under the Plücker embedding. It suffices to show that $f^+: \Lambda \mapsto \Lambda \wedge v^+$ pulled back to $\Gm \times G$ via Lemma \ref{prop:GmG-embedding} is proportional to $\text{pr}_1: \Gm \times G \to \Gm$.

	Fix $\Lambda_0 \in \topwedge V^\Delta$ as in \eqref{eqn:iden-Grass}. For any $(t,g) \in \Gm \times G$,
	\begin{align*}
		f^+(\Lambda_0 (t,g)) & = t \Lambda_0(g,1) \wedge v^+ = t \left( \Lambda_0 \wedge v^+(g^{-1}, 1)\right)(g, 1) \\
		& = (\det g^{-1}) t(\Lambda_0 \wedge  v^+)(g,1) \\
		& = (\det g^{-1}) \det(g,1) \cdot t \Lambda_0 \wedge v^+ = t \Lambda_0 \wedge v^+
	\end{align*}
	where we used the fact $V^+ \subset V^\Box$ is stable under $G \times \{1\}$.
\MyQED\end{proof}

\begin{lemma}\label{prop:doubling-irred}
	The closed subscheme $\partial X$ of $X$ is geometrically integral.

	Moreover, the valuation $v_{\partial X}: F(X) \to \Z$ satisfies $v_{\partial X}(f^\pm) = 1$ if it is normalized to have image $\Z$. Consequently, $\partial X$ is defined by the ideal generated by $f^\pm$.
\end{lemma}
\begin{proof}
	According to Theorem \ref{prop:GPSR-orbits}, $X_P^\flat$ is stratified into $G \times G$-orbits $\Omega_d = \{\ell: \kappa(\ell)=d\}$ with
	\[ \Omega_a \subset \overline{\Omega_b} \iff a \geq b. \]
	The closure relation entails that there is at most one $G \times G$-orbit of codimension one. By pulling back, the same property lifts to the $M_\text{ab} \times G \times G$-orbits in $X_P$, and then extends to $X$ since $X = X_P \sqcup \{\vec{0}\}$ by construction. In fact, our arguments below will give a recipe to steer an element of $\Omega_a$ towards $\Omega_{a+1}$ by one-parameter subgroups, $a=0,1,\ldots$.

	Being the zero locus of $f^\pm$, by Krull's theorem each irreducible component of $\partial X$ has codimension $\leq 1$. We infer that there exists a unique orbit $O$ in $X$ of codimension one, thus $\partial X = \overline{O}$ is irreducible since $O$ is. As our arguments are based on Theorem \ref{prop:GPSR-orbits} which works over any field extension of $F$, the irreducibility persists under field extensions. An immediate consequence is: any element of $X$ lying over a Lagrangian $\ell \in X^{+,\flat}$ with $\kappa(\ell)=1$ (i.e.\ $\ell \in \Omega_1$) belongs to $O$.
	
	It remains to show $v_{\partial X}(f^+) = 1$. We contend that for any $x_0 \in X^+(F)$, there exists a homomorphism of $F$-groups $\mu: \Gm \to G$ such that the morphism
	\begin{align*}
		c: \Gm & \longrightarrow X^+ \\
		t & \longmapsto t x_0 (1,\mu(t))
	\end{align*}
	extends to $c: \Ga \to X$ such that $c(0)$ lies in the $M_\text{ab} \times G \times G$-orbit of codimension one. As $(M_\text{ab} \times G \times G)(F)$ operates transitively on $X^+(F)$, we may assume that $x_0$ is any element lying over $V^\Delta \in X^{+,\flat}(F)$.
	
	Decompose $V$ into an orthogonal direct sum $V_0 \oplus V_1$ of symplectic vector spaces such that $\dim_F V_0 = 2$. Choose a basis $v, w$ of $V_0$ verifying $\angles{v|v'}=1$, and any basis $\{v_i : 1 \leq i \leq \dim V_1\}$ of $V_1$. Set
	\[ b := \bigwedge_{i=1}^{\dim V_1} (v_i, v_i) \in \Image [\; \topwedge V_1^\Delta \to \bigwedge^{\dim V_1} V^\Box \;] \]
	multiplied in any order, with the convention $b=1$ when $V_1 = \{0\}$. We may suppose that $x_0 = (v,v) \wedge (w,w) \wedge b$ under the embedding $X^+ \hookrightarrow \bigwedge^{\dim V} V^\Box \smallsetminus \{0\}$ in \eqref{eqn:Plucker-embedding}. Take
	\[ \mu(t) := \begin{cases}
		v \mapsto tv, \; w \mapsto t^{-1}w & \text{ on } V_0 \\
		\identity_{V_1} & \text{ on } V_1.
	\end{cases} \]
	In what follows, $o(t)$ stands for any expression valued in $\bigwedge V^\Box$ that involves only positive powers of $t$. We compute
	\begin{align*}
		c(t) & = t \left( (v,tv) \wedge (w, t^{-1}w) \wedge b \right) \\
		& = t\left( ((v,0) + (0,tv)) \wedge ((w,0) + (0, t^{-1}w)) \wedge b \right) \\
		& = t \left( \left( (v,0) \wedge (w,0) + t^{-1}(v,0) \wedge (0,w) + o(t) + (0,v) \wedge (0,w) \right) \wedge b \right) \\
		& = (v,0) \wedge (0,w) \wedge b + o(t).
	\end{align*}
	Hence $c$ extends to $\Ga$ by mapping $0$ to $(v,0) \wedge (0,w) \wedge b$. Observe that $(v,0) \wedge (0,w) \wedge b$ lies over the image of the Lagrangian $\ell = (Fv,0) \oplus (0, Fw) \oplus V_1^\Delta$ of $V_0^\Box \oplus V_1^\Box = V^\Box$ under \eqref{eqn:Plucker-embedding}. Since
	\[ \kappa(\ell) = \kappa^{V_0}((Fv,0) \oplus (0,Fw)) + \kappa^{V_1}(V_1^\Delta) = 1 + 0 = 1, \]
	it must lie in the codimension-one stratum by the arguments above. This establishes our claim.
	
	Now we are in position to invoke the techniques of \cite[\S 24.1]{Ti11}. Pick any $x_0$ and construct $c$ as above. By (24.1) in \textit{loc.\ cit.}, for $g \in M_\text{ab} \times G \times G$ in general position we have
	\[ \text{ord}_{t=0}(\underbracket{c^*(g f^+)}_{\in F[\Ga] \simeq F[t]}) = v_{\partial X}(f^+) \cdot i(c(0), \partial X \cdot c; X) \]
	where $i(c(0), \partial X \cdot c; X)$ is defined to be $\text{ord}_{t=0}(c^*(gh))$ for any local parameter $h$ of $\partial X$; it makes sense for generic $g$ since the orbit closure containing $c(0)$ equals $\partial X$.
	
	Evaluate the left hand side first. Since $f^+$ is a relative invariant, $g$ can be neglected. As $f^+$ has eigencharacter $\omega_1$, the regular function $c^*(f^+)(t) = f^+(c(t))$ is linear in $t$ by construction of $c$. So $\text{ord}_{t=0}(c^*(g f^+)) = 1$, which implies that $v_{\partial X}(f^+) = i(c(0), \partial X \cdot c; X) = 1$. This completes the proof.
\MyQED\end{proof}

\begin{theorem}\label{prop:geom-axiom-doubling}
	The relative invariants of $X$ are of the form $c (f^\pm)^n$ with $c \in F^\times$, $n \in \Z$, and the monoid $\Lambda_X$ is generated by $\omega_1$. Consequently, the geometric Axiom \ref{axiom:geometric} is satisfied by the spherical $M_\mathrm{ab} \times G \times G$-embedding $X^+ \hookrightarrow X$.
\end{theorem}
\begin{proof}
	Observe that $X^*(G)=\{0\}$. Apply Lemma \ref{prop:char-Lambda} together with Lemma \ref{prop:GmG-embedding} to deduce that $\Lambda$ is generated by $\omega_1$. By Lemma \ref{prop:f-pm} we have $\omega_1 \in \Lambda_X \subset \Lambda$ since it is the eigencharacter of $f^\pm$. 
	
	Since $(f^\pm)^{-1}$ has nontrivial polar divisor, namely $\partial X$, we see $-\omega_1 \notin \Lambda_X$. The only possibility is $\Lambda_X = \Z_{\geq 0} \cdot \omega_1$ and $f^\pm$ generates all relative invariants, as required.
\MyQED\end{proof}

\begin{remark}\label{rem:doubling-negligible}
	The following observation will be useful in the global setting: for every $\gamma \in \partial X(F)$, there exists a proper parabolic subgroup $R \subset M_\text{ab} \times G \times G$ such that $U_R$ stabilizes $\gamma$. Indeed, by construction $\partial X$ maps onto $\partial X^\flat = X^\flat \smallsetminus X^{\flat,+}$. Suppose $\gamma \neq \vec{0}$, otherwise our claim will be trivial. Theorem \ref{prop:GPSR-orbits} affords such a $U_R$ stabilizing the image $\bar{\gamma} \in \partial X^\flat(F)$ of $\gamma$. To ascend to $X$, it suffices to observe that the fiber over $\bar{\gamma}$ is a $\Gm$-torsor, on which the unipotent group $U_R$ can only act trivially. As in Theorem \ref{prop:GPSR-orbits}, this property holds true over any field $F$ of characteristic $\neq 2$.
	
	As one of the referees pointed out, this property holds for wavefront spherical varieties except for orbits ``along the center'': see \cite[Lemma 2.7.1]{SV17} for boundary degenerations, to which the general case can be related via blowups. We refer to \textit{loc.\ cit.} for details.
\end{remark}

\section{Doubling zeta integrals}\label{sec:doubling-zeta-integrals}
Retain the notations in \S\ref{sec:symplectic-case}; in particular $F$ is non-Archimedean, and $M_\text{ab} = \GL(V^\Delta)/\SL(V^\Delta)$. Since $X^+ \simeq \Gm \times G$ as $M_\text{ab} \times G \times G$-varieties, the density bundles $\mathscr{L}^s$ over $X^+$ are equivariantly trivializable --- this amounts to the existence of Haar measures on $(\Gm \times G)(F)$. Nevertheless, we retain the convention that $C^\infty(X^+)$, $C^\infty_c(X^+)$ and $L^2(X^+)$ all consist of $\mathscr{L}^\demi$-valued sections. Define $\mathcal{N}_\Pi := \Hom_{M_\text{ab} \times G \times G(F)}(\Pi, C^\infty(X^+))$.

\begin{lemma}\label{prop:N_Pi-doubling}
	Consider an irreducible nice representation $\Pi$ of $(M_\mathrm{ab} \times G \times G)(F)$. Then $\mathcal{N}_\Pi$ is nonzero if and only if $\Pi \simeq \chi \boxtimes \pi \boxtimes \check{\pi}$ for some continuous character $\chi$ of $\GL(V^\Delta, F)/\SL(V^\Delta, F)$ and irreducible nice representation $\pi$ of $G(F)$.
	
	In this case, $\mathcal{N}_\Pi$ is canonically isomorphic to the one-dimensional space of invariant sections of $\mathscr{L}$ over $(\Gm \times G)(F)$: to any Haar measure $|\Omega|$ on $(\Gm \times G)(F)$, we associate
	\begin{align*}
		\varphi: \chi \boxtimes \pi \boxtimes \check{\pi} & \longrightarrow C^\infty(X^+) \\
		1 \otimes v \otimes \check{v} & \longmapsto \left( \check{\chi} \circ \det \right) \; \angles{\check{v}, \pi(\cdot) v} |\Omega|^\demi
	\end{align*}
	where we identify $X^+$, $\Gm \times G$ and $M_\mathrm{ab} \times G$ by the recipe of Notation \ref{nota:det-Gm}.
\end{lemma}
\begin{proof}
	This is a special case of \S\ref{sec:group-case}; the contragredient comes from the way $M_\text{ab}$ acts on $X_P$.
\MyQED\end{proof}

By restriction to $X^+(F)$, we view $\Schw := \Schw(X_P)$ as a subspace of $C^\infty(X^+)$. In view of Theorem \ref{prop:geom-axiom-doubling}, hereafter we may take $f^+$ to be the generator of $\Lambda_X$ and work in the formalism of \S\ref{sec:Schwartz}. The datum $v^+$ can even be chosen so that $f^+(t,g)=t$ if we identify $\Gm \times G$ with $X^+$. Also, we identify $\Lambda$ with $\Z$, $\Lambda_{\CC}$ with $\CC$ in such a manner that $\Re(\lambda) \relgeq{X} 0 \iff \Re(\lambda) \geq 0$. The eigencharacter of $|f^+|^\lambda$ under $M_\text{ab}(F)$ now becomes $|\omega_1|^\lambda: \bar{m} \mapsto |\det m|^{-\lambda}$, where we denote by $m \in P(F)$ any inverse image of $\bar{m}$. In particular,
\[ \varphi_\lambda(1 \otimes v \otimes \check{v}) = \underbracket{(\check{\chi}|\cdot|^{-\lambda}) \circ \det}_{\text{on } M_\text{ab}} \cdot \underbracket{\angles{\check{v}, \pi(\cdot) v}}_{\text{on } G} |\Omega|^\demi. \]

To discuss zeta integrals, fix
\begin{compactitem}
	\item $\Pi = \chi \boxtimes \pi \boxtimes \check{\pi}$,
	\item $\varphi \in \mathcal{N}_\Pi$ associated to a $(M_\text{ab} \times G \times G)(F)$-invariant density $|\Omega|$ as in Lemma \ref{prop:N_Pi-doubling}.
\end{compactitem}
For $\lambda \in \CC$, $\Re(\lambda) \gg 0$, The zeta integral for $\xi \in \Schw$ and the vector $1 \otimes v \otimes \check{v}$ is
\begin{equation}\label{eqn:doubling-zeta} \begin{aligned}
	Z_{\lambda, \varphi}\left( (1 \otimes v \otimes \check{v}) \otimes \xi\right) & = \int_{X^+} \xi \varphi_\lambda(1 \otimes v \otimes \check{v}) \\
	& = \int_{\substack{g \in G(F) \\ \bar{m} \in M_\text{ab}(F)}} \xi_0 \left( \delta_P^{\demi} \cdot (\check{\chi}|\cdot|^{-\lambda}) \circ \det \right)(\bar{m}) \angles{\check{v}, \pi(g) v} |\Omega|,
\end{aligned}\end{equation}
where $\xi_0 := \xi \delta_P^{-\demi} |\Omega|^{-\demi}: X^+(F) \to \CC$.

Apart from the $\delta_P^\demi$-shift, we claim that it equals the integral in \cite[Theorem 7.5]{BK02} in their ``case 2''. The comparison goes as follows.
\begin{itemize}
	\item As remarked in \S\ref{sec:geometric-setup}, in \cite{BK02} one fixes a $G^\Box(F)$-invariant measure $|\Xi|$ on $X_P(F)$, the $M_\mathrm{ab}(F)$-action on $L^2(X_P(F))$ is twisted by $\delta_P^\demi$. Their Schwartz function $\xi_\text{BK}$ and our $\xi \in \Schw(X_P)$ are related via $\xi = \xi_\text{BK} |\Xi|^\demi$.
	\item We have seen in \S\ref{sec:geometric-setup} that $|\Xi|$ has eigencharacter $\delta_P$ under $M_\text{ab}(F)$-action. Hence $|\Omega| := \delta_P^{-1} \cdot |\Xi|$ is a $(M_\text{ab} \times G \times G)(F)$-invariant density on $X^+(F)$.
	\item Summing up, $\xi_0 = \xi \delta_P^{-\demi} |\Omega|^{-\demi} = \xi |\Xi|^{-\demi} = \xi_\text{BK}$.
\end{itemize}

\begin{proposition}\label{prop:doubling-zeta-axiom}
	The space $\Schw$ admits a natural structure of algebraic topological vector space. It satisfies Axiom \ref{axiom:zeta}.
\end{proposition}
\begin{proof}
	For every $Q$, the space $C^\infty_c(X_Q)$ is algebraic with a continuous $G^\Box(F)$-action. Since $\Schw$ is expressed as a finite $\varinjlim$ of those spaces in \eqref{eqn:Schw-BK}, $\Schw$ is algebraic as well. The required topological properties of $\Schw$ follows from Lemma \ref{prop:algebraic-nuclear}, and it is a continuous $G^\Box(F)$-representation.
	
	We also know that $\alpha: \Schw \hookrightarrow L^2(X_P) = L^2(X^+)$. Since $|f|^\lambda$ is a continuous on $X(F)$ when $\Re(\lambda) \geq 0$, we also have $\Schw_\lambda \subset L^2(X^+)$. The continuity of $\alpha_\lambda: \Schw \hookrightarrow L^2(X^+)$ is automatic by algebraicity. The holomorphy of $\alpha_\lambda$ over $\{\lambda \in \CC : \Re(\lambda) \geq 0\}$ boils down to that of $\int_{X_P} |f|^\lambda \xi \bar{b}$ for given $b \in L^2(X_P)$ and $\xi \in \Schw$, which is in turn guaranteed by Lemma \ref{prop:X_P-bdd} and dominated convergence.
	
	It remains to verify the conditions concerning zeta integrals. Convergence for $\Re(\lambda) \gg 0$ is implied by Corollary \ref{prop:conv-gg0}. Since $Z_\lambda$ equals the integral in \cite[\S 7]{BK02}, the rationality, etc.\ of $Z_\lambda$ are recorded in \textit{loc.\ cit.}; continuity does not matter since the topological vector spaces in view are all algebraic.
\MyQED\end{proof}

For discussions on the Archimedean case, see \cite[\S 6.1]{BK02}. One expects a suitable Fréchet space $\Schw = \Schw(X_P) \subset L^2(X_P)$. See also \cite[\S 5.2]{GL17}.

\begin{remark}\label{rem:doubling-classical}
	In order to get the usual doubling zeta integral, one has to choose a Haar measure on $M_\text{ab}(F)$, assume Conjecture \ref{conj:integral-coinvariants} and integrate $\xi\varphi_\lambda(1 \otimes v \otimes \check{v})$ along $M_\text{ab}(F)$-orbits for $\Re(\lambda) \gg 0$. This procedure has been summarized in Remark \ref{rem:normalized-intop}. The conclusion is that it yields a section $f_\lambda \in I^{G^\Box}_P(\chi \otimes |\det|^\lambda)$, rational in $q^\lambda$, and that \eqref{eqn:doubling-zeta} becomes
	\[ \int_G f_\lambda(g, 1) \angles{\check{v}, \pi(g) v} |\Omega''| \]
	for some invariant density $|\Omega''|$ on $G(F)$, where we used the natural inclusion $G \times G \hookrightarrow G^\Box$. This is the original integral in \cite[p. 3]{GPSR87} except that our parabolic induction is normalized. The sections $f_\lambda(g,1)$ so obtained are \emph{good sections} in the sense of \cite[p.110]{PSR87}, which play a pivotal role in the doubling method. Consequently, one should view $\Schw$ as a space of \emph{universal good sections}. See \cite{Sha17} for further details.
\end{remark}

Now move to the functional equations. Two types of model transitions are involved.
\begin{enumerate}
	\item Let $Q$ be another parabolic subgroup of $G^\Box$ sharing the same Levi component $M$ as $P$. In fact, the only choice in our framework is the opposite parabolic $Q = P^-$. The model transition is then $\mathcal{F}_{Q|P}: \Schw(X_P) \rightiso \Schw(X_Q)$.
	\item Let $w \in N_{G^\Box}(M)(F)/M(F)$ and $wQ := wQw^{-1}$. We have the isomorphism
		\begin{align*}
			X_{wQ} & \stackrel{\sim}{\longrightarrow} X_Q \\
			(wQ) g' & \longmapsto Q w^{-1} g', \quad g' \in G^\Box 
		\end{align*}
		between $G^\Box$-varieties over $F$. In contrast, acting by $w\bar{m}w^{-1} \in M_\text{ab}$ on $X_{wQ}$ corresponds to acting by $\bar{m}$ on $X_Q$; this is remedied by twisting the $M_\text{ab} \times G^\Box$-structure on $X_{wQ}$ accordingly. Therefore, transport of structure gives rise to the model transition
		\begin{align*}
			\mathcal{F}_w: L^2(X_Q) & \stackrel{\sim}{\longrightarrow} L^2(X_{wQ}) \\
			\Schw(X_Q) & \stackrel{\sim}{\longrightarrow} \Schw(X_{wQ}).
		\end{align*} 
		Usually we assume $wQ = P$, and such a $w$ exists in our setting. In this case, conjugation by $w$ induces the involution $t \mapsto t^{-1}$ of $M_\text{ab}$.
\end{enumerate}
We have $\Lambda_1 = \Z = \Lambda_2$ in the notation of \S\ref{sec:model-transition}. Thus $\mathcal{F} := \mathcal{F}_w \mathcal{F}_{Q|P}$ furnishes a model transition $\Schw(X_P) \rightiso \Schw(X_P)$ if we take $Q = P^-$ and $wQ = P$. More precisely, $X_1 = X = \overline{X_P}^\text{aff} =X_2$ and $\Schw_1 = \Schw(X_P) = \Schw_2$, but the objects with subscript $1$ are endowed with the twisted $M_\text{ab}$-action as above; $\mathcal{F}: \Schw_2 \rightiso \Schw_1$ is $M_\text{ab} \times G \times G(F)$-equivariant under the conventions above.

\begin{conjecture}\label{conj:Schw-BK-mult1}
	Set $\Schw := \Schw(X_P)$. Let $\chi \boxtimes \pi \boxtimes \check{\pi}$ be an irreducible nice representation of $(M_\text{ab} \times G \times G)(F)$. When we vary $\chi$ in a family $\chi \otimes |\det|^\lambda$, $\lambda \in \CC$, the space
	\[ \Hom_{M_\text{ab} \times G \times G(F)}\left( \chi \boxtimes \pi \boxtimes \check{\pi}, \Schw^\vee \right) \]
	is at most one-dimensional for $\chi$ in general position.
\end{conjecture}

\begin{theorem}
	Under the assumptions of Proposition \ref{prop:doubling-zeta-axiom} and admit either
	\begin{compactitem}
		\item Conjecture \ref{conj:integral-coinvariants}, or
		\item Conjecture \ref{conj:Schw-BK-mult1},
	\end{compactitem}
	then the local functional equation (Definition \ref{def:lfe}) holds for $\mathcal{F}$.
\end{theorem}
\begin{proof}
	Let $\Pi := \chi \boxtimes \pi \boxtimes \check{\pi}$ be an irreducible nice representation of $(M_\text{ab} \times G \times G)(F)$. The elements $\varphi^{(2)}$ of $\mathcal{N}^{(2)}_\Pi$ have been described as in Lemma \ref{prop:N_Pi-doubling}; as for $X_1 = X$, there is a twist
	\[ \varphi^{(1)}: 1 \otimes v \otimes \check{v} \longmapsto \chi(\cdot) \angles{\check{v}, \pi(\cdot) v} |\Omega|^\demi \]
	and $\varphi^{(1)}_\lambda = |\det|^\lambda \varphi^{(1)}$.

	Assume Conjecture \ref{conj:integral-coinvariants}.	In view of the twists and Remark \ref{rem:doubling-classical}, our assertion reduces to the doubling functional equation \cite[Theorem 2.1]{PSR86} provided that we can relate $\mathcal{F} = \mathcal{F}_w \mathcal{F}_{Q|P}$ to a rational family of intertwining operator $R(w, \lambda): I^{G^\Box}_P(\chi \otimes |\det|^\lambda) \to I^{G^\Box}_P(\check{\chi} \otimes |\det|^{-\lambda})$ such that the following diagram commutes for general $\lambda \in \CC$:
	\[ \begin{tikzcd}
		\Schw(X_P) \arrow{r}[above]{\mathcal{F}} \arrow{d} & \Schw(X_P) \arrow{d} \\
		I^{G^\Box}_P(\chi \otimes |\det|^\lambda) \arrow{r}[below, inner sep=1em]{R(w, \chi \otimes |\det|^\lambda)} & I^{G^\Box}_P(\check{\chi} \otimes |\det|^{-\lambda})
	\end{tikzcd} \]
	Apart from an easy $w$-twist, this is contained in Remark \ref{rem:normalized-intop}.
	
	Under Conjecture \ref{conj:Schw-BK-mult1}, the arguments are completely the same as \cite[Theorem 2.1]{PSR86}.
\MyQED\end{proof}

\begin{remark}
	Methodologically, the route via Conjecture \ref{conj:Schw-BK-mult1} seems more natural. It also implies the rationality of zeta integrals by Bernstein's principle \cite[Part B, \S 12]{GPSR87}, the isomorphism $\mathcal{N}_\Pi \otimes \mathcal{K} \rightiso \mathcal{L}_\Pi$, as well as the condition \eqref{eqn:generic-inj}; see the proof of Theorem \ref{prop:pvs-T_pi-isom}.
\end{remark}

\section{Relation to reductive monoids}\label{sec:relation-to-monoids}
Most of the following basic results are extracted from \cite[\S 27.1]{Ti11} and \cite{Re05}.
\begin{definition}\index{monoid}
	By an $F$-\emph{monoid} $Y$, we mean a monoid-like object in the category of algebraic $F$-varieties, or equivalently a representable functor $h_Y: \mathsf{Var}_F^\text{op} \to \mathsf{Monoid}$. The subfunctor of invertible elements is represented by an open subvariety $G := G(Y)$, which is actually an $F$-group called the \emph{unit group} of $Y$. Call $Y$ normal, affine, etc., if $Y$ has those properties as a variety. An $F$-monoid $Y$ with unit group $G$ is naturally a $G \times G$-variety under multiplication, cf.\ \eqref{eqn:group-case-action1}:
	\[ y(g_1, g_2) = g_2^{-1} y g_1. \]

	Fix the ground field $F$. Monoids with connected reductive unit groups are called \emph{reductive monoids}. Suppose that $Y$ is an affine, normal reductive monoid with unit group $G$. Then $G \hookrightarrow Y$ can be viewed as a spherical embedding under the $G \times G$-action.. The \emph{abelianization}\index{monoid!abelianization} of $Y$ is defined as the GIT quotient $\text{ab}: Y \to Y\sslash G _\text{der} \times G_\text{der}$.
	
	Following \cite{Ng14}, we say that $Y$ is \emph{flat}\index{monoid!flat} if $\text{ab}$ is a flat morphism with geometrically reduced fibers. Flat monoids are systematically studied by Vinberg and Rittatore. 
\end{definition}

\begin{example}\label{eg:GJ-monoid}
	The variety $\text{Mat}_{n \times n}$ is a normal affine monoid under matrix multiplication. Its unit group is $\GL(n)$. It is a flat monoid whose abelianization can be identified with $\det: \text{Mat}_{n \times n} \to \Ga$ (eg.\ by Igusa's criterion \cite[II. \S 4.5]{AG4}). This space has been considered in \S\ref{sec:GJ}.
\end{example}

\begin{theorem}[A.\ Rittatore]\label{prop:Rittatore}
	Let $G$ be a split connected reductive $F$-group. The category of normal affine reductive monoids with unit group $G$ is equivalent to the category of affine spherical embeddings of $G$.
\end{theorem}
More precisely, the monoid structure of a given spherical embedding $Y$ of $G$ is uniquely determined, since the multiplication must restrict to the group law over the open dense orbit $G$.

Consequently, such monoids are classified by Luna--Vust theory (Theorem \ref{prop:LV}) applied to the homogeneous $G \times G$-space $G$, at least over the algebraic closure. It turns out that their colored cones have full colors, and each color is defined over $F$.

\begin{corollary}[See {\cite[Theorem 27.12]{Ti11}}]\label{prop:full-color}
	The normal affine reductive monoids $Y$ are classified by colored cones of the form $(\mathcal{C}_Y, \mathcal{D}^B)$.
\end{corollary}

Now revert to the setting in \S\ref{sec:symplectic-case}. Recall that the affine variety $X$ carries a right $M_\text{ab}$-action commuting with $G^\Box$. Identify $M_\text{ab}$ and $\Gm$ as in Notations \ref{nota:det-Gm}, written as $\bar{m} \leftrightarrow t$. Then we can let $\Gm$ act on the left of $X$ via $tx = x \bar{m}^{-1}$. Also recall that there is a distinguished point $\vec{0} \in X(F)$. This endows $X$ with a conical structure. For every $x \in X(F)$ we have
\[ \lim_{t \to 0} tx = \vec{0} \in X(F), \]
the limit being taken in the algebraic sense: the orbit map $\Gm \to X$ extends to $\Ga \to X$, and $\vec{0} \in X(F)$ is the image of $0 \in \Ga$. This implies that $\lim_{t \to 0} f(tx) = f(\vec{0}) \in \pr^1(F)$, for every $f \in F(X)^\times$ and $x \in X$.

\begin{lemma}\label{prop:conical-function}
	Suppose that $f \in F(X)^\times$ satisfies $f(tx) = t^d f(x)$ under the $\Gm$-action, for some $d \in \Z$. If $d = 0$ then $f$ is constant.
\end{lemma}
\begin{proof}
	Assume $d=0$. For any $x \in X$ we have $f(x) = \lim_{t \to 0} f(tx) = f(\vec{0})$. Hence $f$ is constant.
\MyQED\end{proof}

Also recall that the open orbit $X^+ \simeq \Gm \times G$ is a $M_\text{ab} \times G \times G$-variety. By recalling the definition of the group actions together with the fact $M_\text{ab} \simeq \Gm$, this resembles the setting of monoids with unit group $M_\text{ab} \times G$. One exception: here the acting group is $M_\text{ab} \times G \times G$ instead of $(M_\text{ab} \times G)^2$; as remarked at the end of \S\ref{sec:group-case}, this is actually irrelevant. One retrieves the required $M_\text{ab}^2$-action by composing with $M_\text{ab}^2 \to M_\text{ab}$, $(a,b) \mapsto ab^{-1}$. We shall omit this formal step.

\begin{proposition}\label{prop:monoid}
	The $M_\mathrm{ab} \times G \times G$-variety $X$ admits a canonical structure of affine normal reductive monoid with unit group $M_\mathrm{ab} \times G$. Furthermore, it is a flat monoid with singular locus equal to $\{\vec{0}\}$.
\end{proposition} 
\begin{proof}
	Since $X^+ \hookrightarrow X$ is an affine spherical embedding, it suffices to apply Theorem \ref{prop:Rittatore} to deduce the monoid structure. On the other hand, the relative invariant $f^+$ defines a dominant morphism $X \to \Ga$, hence flat; we claim that $f^+$ is the abelianization of $X$. Since $f^+$ pulls back to $\text{pr}_1: \Gm \times G \to \Gm$ over $\Gm \subset \Ga$, it is indeed the GIT quotient by $G \times G$ by Igusa's criterion \cite[II. \S 4.5]{AG4}.

	The scheme-theoretic fibers of $f^+$ over $\Gm$ are all isomorphic to $G$, whereas the fiber over $0 \in \Ga$ is the geometrically integral scheme $\partial X$ by Lemma \ref{prop:doubling-irred}. Thus $X$ is a flat monoid.

	By \eqref{eqn:Plucker-embedding}, there are exactly two $G^\Box$-orbits in $X = \overline{X_P}^\text{aff}$, namely $X_P$ and $\{0\}$, thus it suffices to show $X$ is singular. The singularity of $X$ follows from \cite[Theorem 27.25]{Ti11}. Alternatively, one may apply \cite[Theorem 5.5]{Re05} using the fact that the monoid $X$ has a ``zero element'', namely $\vec{0}$.
\MyQED\end{proof}
Note that Lemma \ref{prop:doubling-irred} is necessary for the proof.

We are ready to describe the Luna--Vust datum for $X$. It is customary to work with a Borel subgroup of $M_\text{ab} \times G \times G$ of the form $M_\text{ab} \times B^- \times B$; let $A$ be the Levi quotient of $B$. Some generalities about the homogeneous space $X^+ \simeq \Gm \times G$ are needed; details can be found in \cite[p.161]{Ti11}.
\begin{itemize}
	\item Observe that
		\[ \Lambda_{X^+} = X^*(\Gm) \oplus X^*((A \times A)/ \text{diag}(A)) = \Z \oplus X^*(A) \]
		where the last equality comes from $\Z = X^*(\Gm) \simeq X^*(M_\text{ab})$ and the second projection $A \times A/\text{diag}(A) \rightiso A$.
	\item The colors for $X^+$ are exactly the products of $\Gm$ with the closures of the Bruhat cells $B^- w B$ in $X^+$, where $w$ represents the simple root reflections, indexed by either the simple coroots $\check{\alpha}$ or dually the fundamental weights $\varpi$. The color $D_\varpi$ corresponding to $\varpi$ equals the zero loci of the following matrix coefficient of the irreducible algebraic $G$-representation $(\rho, V_\rho)$ with highest weight $\varpi$
		\[ f_\varpi: (t,g) \mapsto \angles{\check{v}_{-\varpi}, v_\varpi \rho(g)}, \quad f_\varpi \in F(X^+), \]
		where $v_\varpi$ (resp.\  $\check{v}_{-\varpi}$) is a highest weight vector in in $V_\rho$ (resp.\  the lowest weight vector in $V_\rho^\vee$). The eigencharacter of $f_\varpi$ under $\Gm \times B^- \times B$ equals $(0, \varpi) \in \Lambda_{X^+} = \Z \oplus X^*(A)$. All are $F$-rational since $G$ is split.
	\item The valuation cone $\mathcal{V} \subset \mathcal{Q}$ equals the anti-dominant Weyl chamber
		\[ \bigcap_\alpha \{\angles{\alpha, \cdot} \leq 0\}, \quad \alpha \in X^*(A) \text{ ranges over the positive roots.} \]
	\item Choose a symplectic basis $e_1, \ldots, e_n, f_n, \ldots, f_1$ of the symplectic $F$-space $V$, i.e.\ $\angles{e_i|e_j} = \angles{f_i|f_j}=0$ and $\angles{e_i|f_j} = \delta_{i,j}$ (Kronecker's delta), so that $G = \Sp(V) \hookrightarrow \GL(2n)$. We can take $B$ to be the upper triangular subgroup intersected with $G$, thereby obtain the basis $\epsilon^*_1, \ldots, \epsilon^*_n$ of $X^*(A)$ given by
		\begin{gather*}
			\epsilon^*_i: \begin{pmatrix} t_1 & & & & &\\ & \ddots & & & & \\ & & t_n & & & \\ & & & t_n^{-1} & & \\ & & & & \ddots & \\ & & & & & t_1^{-1} \end{pmatrix} \mapsto t_i, \quad i=1, \ldots, n.
		\end{gather*}
		The dual basis of $X_*(A)$ is denoted by $\epsilon_1, \ldots, \epsilon_n$. The fundamental weights are given by
		\[ \varpi_i := \epsilon^*_1 + \ldots + \epsilon^*_i, \quad i=1, \ldots, n. \]
		On the other hand, the simple roots are $\epsilon^*_i - \epsilon^*_{i+1}$ with $1 \leq i < n$ and $2\epsilon^*_n$, relative to the Borel subgroup $B$ chosen above.
	\item For $i=1, \ldots, n$, the normalized discrete valuation corresponding to $D_{\varpi_i}$ is $\rho(D_{\varpi_i}) \in \mathcal{Q}$. Since $\partial X = \{f^+ = 0\}$ does not contain any color, we have $\rho(D_{\varpi_i})(f^+) = 0$ for all $i$. Also notice that $f^+$ has eigencharacter $(1, 0) \in \Z \oplus X^*(A)$, hence $(\rho(D_{\varpi_i}))_{i=1}^n$ lie in $\{0\} \oplus X_*(A)$ and form the dual basis of the fundamental weights, i.e.\ the simple coroots. Specifically, the simple coroot dual to $\varpi_i = \epsilon^*_1 + \cdots + \epsilon^*_i$ is
		\[ \check{\alpha}_i = \begin{cases} \epsilon_i - \epsilon_{i+1}, & 1 \leq i < n  \\ \epsilon_n, & i=n. \end{cases} \]
\end{itemize}

\begin{theorem}\label{prop:doubling-cone}
	The colored cone associated to $X^+ \hookrightarrow X$ is $(\mathcal{C}, \mathcal{D}^{M_\mathrm{ab} \times B^- \times B})$ where $\mathcal{C}$ is the cone generated by the colors
	\[ \rho(D_{\varpi_i}) = \check{\alpha}_i, \quad i=1, \ldots, n, \]
	together with
	\[ (1, -\epsilon_1) \in \Lambda_{X^+} \cap \mathcal{V}. \]
\end{theorem}
\begin{proof}
	By Corollary \ref{prop:full-color}, the colored cone $(\mathcal{C}, \mathcal{F})$ associated to $X^+ \hookrightarrow X$ satisfies $\mathcal{F} = \{\text{all colors}\}$, and we have seen that the corresponding valuations are given by simple coroots.
	
	Recall that $\partial X$ is a prime divisor by Lemma \ref{prop:doubling-irred}. Write $\varpi = \varpi_i$. The poles of the defining equations $f_\varpi \in F(X)$ for the colors $D_\varpi \subset X^+$ can only lie in $\partial X$. The defining equation of the prime divisor $\overline{D_\varpi}$ in $X$ is thus given by $t^d f_\varpi$, where $d = - v_{\partial X}(f_\varpi)$ (recall that $X$ is normal and $t \in F[X^+] \subset F(X)$ is a uniformizer for $\partial X$). We contend that $d=1$ for every fundamental weight $\varpi$; in particular $t f_\varpi \in F[X]$.
	
	Write $\mathbb{G}_m = \operatorname{Spec} F[s, s^{-1}]$. By the proof of Lemma \ref{prop:doubling-irred}, for $g \in M_{\mathrm{ab}} \times G \times G$ in general position we have
	\begin{align*}
		v_{\partial X}(f_\varpi) & = \mathrm{ord}_{s=0} \left( f_\varpi(c(s)g) \right) \cdot i(c(0), \partial X \cdot c; X) \\
		& = \mathrm{ord}_{s=0} \left( f_\varpi(x_0 (1, \mu(s))g ) \right) \cdot \underbracket{i(c(0), \partial X \cdot c; X)}_{= 1},
	\end{align*}
	where the base point $x_0 \in X^+(F)$, the morphism $c: \mathbb{G}_m \to X^+$ and $\mu: \mathbb{G}_m \to G$ are defined in the cited proof, namely $c(s) = s x_0 (1, \mu(s))$. Here we used the fact that $f_\varpi$ is invariant under $M_{\mathrm{ab}}$-action. For the same reason, it suffices to take $g = (1, g_1, g_2)$ with generic $g_1, g_2 \in G$.
	
	By Lemma \ref{prop:GmG-embedding}, we identify $X^+$ with $\mathbb{G}_m \times G$ so that $x_0 = (1, 1)$; the $\mathbb{G}_m \times G \times G$-action is also specified there. Let $\rho$ be the (right) $G$-representation of highest weight $\varpi$; take vectors $v_\varpi$ and $\check{v}_{-\varpi}$ in $\rho$ and $\check{\rho}$, with weights $\varpi$ and $-\varpi$ respectively such that $\langle \check{v}_{-\varpi}, v_\varpi \rangle = 1$. Then
	\[ f_\varpi(g) = \left\langle \check{v}_{-\varpi}, \; \rho(g) v_\varpi \right\rangle, \quad g \in G. \]
	Thus we have to calculate, for generic $g_1, g_2 \in G$:
	\[ \mathrm{ord}_{s=0} \left\langle \check{\rho}(g_2) \check{v}_{-\varpi}, \; \rho( \mu(s)^{-1}) \rho(g_1) v_{\varpi} \right\rangle. \]
	
	Consider the standard representation of $G$ on $V = \bigoplus_{i=1}^n (F e_k \oplus F e_{-k})$. Recall that \footnote{See N.\ Bourbaki, \emph{Lie groups and Lie algebras, Chapters 7---9}, pp.206--207.} the fundamental weights of $G$ are $\varpi_k = \epsilon^*_1 + \cdots + \epsilon^*_k$ with $k = 1, \ldots, n$; the corresponding highest weight representations $\rho_k$ are realized on
	\begin{align*}
		E_k & := \left\langle e_{a_1} \wedge \cdots e_{a_l} \wedge e_{-b_1} \wedge \cdots \wedge e_{-b_m} : 1 \leq a_1 < \cdots < b_m \leq n, \; l+m=k \right\rangle \\
		& \subset \wedge^k V, \quad k=1, \ldots, n.
	\end{align*}
	We then take
	\[ v_{\varpi_k} = e_1 \wedge \cdots \wedge e_k, \quad \check{v}_{-\varpi_k} = e^*_1 \wedge \cdots \wedge e^*_k \mod E_k^\perp . \]
	In the proof of Lemma \ref{prop:doubling-irred} we took
	\[ \mu(s): e_{\pm i} \mapsto \begin{cases}
		s^{\pm 1} e_{\pm i}, & i = 1, \\
		e_{\pm i}, & i > 1.
	\end{cases} \quad (1 \leq i \leq n) \]
	Using this realization, it is clear that
	\[ \delta(s, g_1, g_2) := \left\langle \check{\rho}_k(g_2) \check{v}_{-\varpi_k}, \; \rho_k( \mu(s)^{-1}) \rho_k(g_1) v_{\varpi_k} \right\rangle \; \in s^{-1} F[s], \]
	and $\mathrm{ord}_{s=0} \left( \delta(s, g_1, g_2) \right)$ for generic $(g_1, g_2)$ equals actually
	\[ \inf_{(g_1, g_2) \in G^2} \mathrm{ord}_{s=0} \left( \delta(s, g_1, g_2) \right), \quad \text{which is $\geq -1$}. \]
	If the right hand side is $\geq 0$ then $f_\varpi \in F[X]$ and is invariant under the $\mathbb{G}_m$-dilation on $X$. This would imply the constancy of $f_\varpi$ by Lemma \ref{prop:conical-function}. This leads to a contradiction. Hence $v_{\partial X}(f_\varpi) = -1$.
	
	Since $(t,g) \mapsto t$ has eigencharacter $(1,0)$, we see $\overline{D_\varpi}$ is defined by an eigenfunction of eigencharacter $(1, \varpi)$. If we express the normalized discrete valuation of the divisor $\partial X$ as
	\[ v_{\partial X} = \left(a_0, \sum_{i=1}^n a_i \epsilon_i \right) \in \Lambda_{X^+} \cap \mathcal{V}, \]
	then:
	\begin{enumerate}[(i)]
		\item since $f^\pm$ has eigencharacter $(1, 0, \ldots, 0)$, we infer $1 = v_{\partial X}(f^\pm) = a_0$ from Lemma \ref{prop:doubling-irred};
		\item since $\overline{D_{\varpi_i}} \not\supset \partial X$, we have $0 = v_{\partial X}(t f_{\varpi_i}) = a_0 + a_1 + \cdots + a_i$, for $i=1, \ldots, n$.
	\end{enumerate}
	This system of $(n+1)$ equations is easily solved to yield $v_{\partial X} = (1, -\epsilon_1)$. 
\MyQED\end{proof}

Consider Langlands' dual group $(M_\text{ab} \times G)^\wedge$ over $\CC$: it comes equipped with a dual Borel subgroup $\GmC \times \hat{B}$ and a dual maximal torus $\GmC \times \hat{A} \subset \GmC \times \hat{B}$, together with identifications $X^*(\hat{A}) = X_*(A)$ and $X_*(\hat{A}) = X^*(A)$, under which the based root datum of $(\hat{G}, \hat{B}, \hat{A})$ is dual to that of $G$. In particular, the roots/coroots, fundamental weights/coweights get exchanged.

We have $(M_\text{ab} \times G)^\wedge = \GmC \times \hat{G}$ and $\hat{G} \simeq \SO(2n+1, \CC)$. The standard representation $\text{std}: \hat{G} \to \GL(2n+1, \CC)$ is thus defined. Tensored with the identity representation of $\GmC$ on $\CC$, we get
\[ \identity \boxtimes \text{std}: (M_\text{ab} \times G)^\wedge \to \GL(2n+1, \CC). \]

This algebraic representation is irreducible since $\text{std}$ is; this standard fact will be reproved in what follows.
\begin{proposition}\label{prop:rho-weight}
	The representation $\identity \boxtimes \mathrm{std}$ has $\hat{B}$-lowest weight equal to $(1, -\epsilon_1) \in \Z \oplus X_*(A) = X^*(\GmC \times \hat{A})$, and $\hat{B}$-highest weight equal to $(1, \epsilon_1)$.
\end{proposition}
\begin{proof}
	The standard representation of $\SO(2n+1, \CC)$ has weights
	\[ -\epsilon_1, \ldots, -\epsilon_n, 0, \epsilon_n, \ldots, \epsilon_1 \]
	as elements of $X^*(\hat{A}) = X_*(A)$, which are strictly increasing in the order relative to $\hat{B}$. Hence $-\epsilon_1$ is the unique $B$-lowest weight; in particular $\text{std}$ is irreducible. The passage to $\identity \boxtimes \text{std}$ is straightforward.
\MyQED\end{proof}

We conclude this section by recalling the recipe of Braverman and Kazhdan \cite{BK00} that connects reductive monoids to $L$-functions. Here the group under inspection is $M_\text{ab} \times G \simeq \Gm \times G$. The inclusion $\GmC \hookrightarrow (\Gm \times \hat{G})$ is clearly dual to $\text{pr}_1: \Gm \times G \to \Gm$. Let $\rho := \identity \boxtimes \text{std}$ be the irreducible representation of $\GmC \times \hat{G}$; it restricts to $z \mapsto z\cdot\identity$ on $\GmC$.

Following \cite[\S 5.5]{BK00}, one attaches an irreducible affine normal monoid $X_\rho$ to $\rho$ with unit group $\Gm \times G$. L.\ Lafforgue explained their construction in \cite[Définition V.9]{Laf14}. In our setting where $\rho$ is irreducible and $(\Gm \times G)_\text{der} = G$ is simply connected, this construction has been rephrased by B.\ C.\ Ngô in \cite{Ng14} and \cite[\S 4]{BNS16} in terms of Vinberg's theory, which is essentially the dual version of Luna--Vust classification in the case of reductive monoids. In this situation $X_\rho$ will be a flat reductive monoid, called the \emph{$L$-monoid}\index{L-monoid@$L$-monoid} associated to $\rho$.

\begin{theorem}\label{prop:PSR-Ngo-compatible}
	The reductive monoids $X_\rho$ and $X$ are isomorphic.
\end{theorem}
\begin{proof}
	We shall adopt Lafforgue's description in \cite{Laf14}: the idea is to characterize the monoid in terms of the closure of a maximal torus with Weyl group action --- see \cite[Theorem 5.4]{Re05}. Fix a maximal torus $T \subset \Gm \times B$ of $\Gm \times G$. The closure $\bar{T}$ in $X$ is an affine normal toric variety under $T$, see \cite[6.2.14]{BK05}.
	Furthermore, \cite[Proposition 27.18]{Ti11} implies that as a toric variety $\bar{T}$ is determined by the cone $\mathcal{C}_{\bar{T}} \subset \mathcal{Q}$ generated by the Weyl translates of $\mathcal{C}_X \cap \mathcal{V}$. Combine Theorem \ref{prop:doubling-cone}, Proposition \ref{prop:rho-weight} and the standard properties of the weights of $\rho$, we see $\mathcal{C}_{\bar{T}}$ is the convex hull of the Weyl translates of the $\mathbb{G}_{m,\CC} \times \hat{B}$-lowest weight $(1, -\epsilon_1) \in \Z \oplus X_*(A)$ of $\rho$. This is exactly the characterization of $X_\rho$ in \cite[Définition V.9]{Laf14}.
\MyQED\end{proof}

As illustrated by \cite{PSR86,GPSR87}, the doubling zeta integrals or the integrals \eqref{eqn:doubling-zeta} point to the $L$-factor
\[ L\left( \cdot + \demi, \chi \times \pi \right) = L \left( \cdot + \demi, \chi \boxtimes \pi, \rho \right) \]
by taking the greatest common divisor for various $\xi$, where $\chi \boxtimes \pi$ is an irreducible nice representation of $M_\text{ab} \times G(F)$. The Braverman--Kazhdan program addresses the same issue by considering integrals of the same kind as \eqref{eqn:doubling-zeta}, namely: matrix coefficients integrated against ``Schwartz functions''. Morally, their approach is based on the following inputs
\begin{enumerate}[(i)]
	\item a certain Schwartz space that should be intimately related to the flat monoid $X_\rho$;
	\item a Fourier transform that should give rise to local functional equations;
	\item a Poisson summation formula in the global setting.
\end{enumerate}
The standard case is Godement--Jacquet theory whose corresponding monoid is described as in Example \ref{eg:GJ-monoid}. The preceding results affirm that
\begin{itemize}
	\item the geometry is correct: $X_\rho = X$ underlies both the $L$-monoid and the doubling constructions, and it conforms to our formalism in \S\ref{sec:geometric-data} and \S\ref{sec:Schwartz};
	\item the doubling integrals go beyond Godement--Jacquet --- in particular, the space $X$ is singular;
	\item we do have the Schwartz space and Fourier transform of at our disposal, and they also fit into our general setting \S\ref{sec:model-transition} of model transition.
\end{itemize}
A weak version of the Poisson formula for doubling integrals will be discussed in \S\ref{sec:Poisson}. Although our constructions still hinge on several conjectures, there seem to be good signs for further works.

\section{Remarks on the general case}\label{sec:doubling-general}
The reason for limiting to the symplectic case is threefold: (1) it is simpler; (2) in that case the formalisms of \cite{BK02} and \cite{PSR86, GPSR87} coincide; (3) the doubling method for unitary or orthogonal groups will involve non-split or disconnected groups which deviates from our general framework, albeit harmlessly.

We sketch the general case of doubling following \cite{PSR86, GPSR87}; see also \cite{Sha17}. Consider a field extension $E|F$ with $[E:F] \leq 2$. Fix $\epsilon \in \{\pm 1\}$ and a finite-dimensional $\epsilon$-Hermitian space $(V, (\cdot|\cdot))$ relative to $E|F$. Define $G = G(V)$ as the corresponding unitary group. Consider the $\epsilon$-Hermitian space $V^\Box := (V, (\cdot|\cdot)) \overset{\perp}{\oplus} (V, -(\cdot|\cdot))$ and the unitary group $G^\Box := G(V^\Box)$. The diagonal $V^\Delta$ is still a maximal totally isotropic subspace in $V^\Box$. Define the Siegel parabolic $P := \Stab_{G^\Box}(V^\Delta)$ of $G^\Box$, with Levi quotient $M$. Set $X_P := G^\Box/P_{\mathrm{der}}$ and $X^\flat_P := G^\Box / P$.

Note that $M \simeq \GL_E(V^\Delta)$, $M_{\mathrm{ab}} \simeq \mathbb{G}_{\mathrm{m},E}$ so our assumption \eqref{eqn:H1-assumption} is verified. When $E=F$ and $\epsilon = -1$, we revert to the symplectic case. When $E=F$ and $\epsilon = 1$, the groups $G^\Box$ and $G$ are disconnected.

\begin{itemize}
	\item By defining the Lagrangians of $V^\Box$ as maximal totally isotropic $E$-vector subspaces, the morphism $X_P \twoheadrightarrow X^\flat_P$ admits a linear-algebraic description parallel to \eqref{eqn:iden-Grass}, and we still have the Plücker embeddings.
	
	\item Lemma \ref{prop:GmG-embedding} still holds: we obtain an $M_{\mathrm{ab}} \times G \times G$-equivariant embedding $M_{\mathrm{ab}} \times G \hookrightarrow X_P$ with open image $X^+ \subset X_P$ by choosing any nonzero $\Lambda_0 \in \bigwedge^{\mathrm{max}} V^\Delta$. This embedding stems from Theorem \ref{prop:GPSR-orbits}, which holds in the general setting by \cite[I. Lemma 2.1]{GPSR87}.

	Note that $M_{\mathrm{ab}} \times G$ differs from the group $H$ considered in \cite[\S 7]{BK02} except in the symplectic case.

	\item As in \S\ref{sec:symplectic-case}, we take $X := \overline{X_P}^{\text{aff}} = X_P \sqcup \{0\}$, and deduce an $M_{\mathrm{ab}} \times G \times G$-equivariant open embedding $X^+ \hookrightarrow X$. The boundary $\partial X := X \smallsetminus X^+$ is again the zero locus of a relative invariant $f^+$ as in Lemma \ref{prop:f-pm}, with the same proof.

	\item Lemma \ref{prop:doubling-irred} (hence Theorem \ref{prop:geom-axiom-doubling}) also holds in this setting. It is a statement about geometry and one should check it on the algebraic closure $\overline{F}$. When $[E:F] = 2$, as $E \otimes_F \overline{F} \simeq \overline{F} \times \overline{F}$, we are reduced to the doubling method for general linear groups; the required geometric properties are furnished in \cite[I.4.2]{GPSR87}.

	\item Remark \ref{rem:doubling-negligible} also holds in the general setting, as proven in \textit{loc.\ cit.}

	\item The zeta integrals are as in \S\ref{sec:doubling-zeta-integrals} except that $\GL(V^\Delta) \xrightarrow{\det} \Gm$ becomes $\GL_E(V^\Delta) \xrightarrow{\det_E} \mathbb{G}_{\mathrm{m}, E}$, and $|\cdot|$ becomes $|\cdot|_E: E \to \R_{\geq 0}$. The result still takes the form \eqref{eqn:doubling-zeta}, namely
	\begin{equation*}\begin{aligned}
		Z_{\lambda, \varphi}\left( (1 \otimes v \otimes \check{v}) \otimes \xi\right) & = \int_{X^+} \xi \varphi_\lambda(1 \otimes v \otimes \check{v}) \\
		& = \int_{\substack{g \in G(F) \\ \bar{m} \in M_\text{ab}(F)}} \xi_0 \left( \delta_P^{\demi} \cdot (\check{\chi}|\cdot|_E^{-\lambda}) \circ \det_E \right)(\bar{m}) \angles{\check{v}, \pi(g) v} |\Omega|, \\
		& \Re(\lambda) \gg 0,
	\end{aligned}\end{equation*}
	with the same assumptions on $|\Omega|$, $1 \otimes v \otimes \check{v} \in \chi \otimes \pi \otimes \check{\pi}$, $\xi \in \Schw$ and $\xi_0 := \xi \delta_P^{-\demi} |\Omega|^{-\demi}: X^+(F) \to \CC$.

	\item The zeta integrals also have an interpretation as integration of matrix coefficients of $G$ against the \emph{good sections} in doubling method. There are subtleties related to normalization of intertwining operators; we refer to \cite{Sha17} for details.

	\item As for the relation to reductive monoids, cf.\ \S\ref{sec:relation-to-monoids}, we assume $E=F$ and $G$ is split in order to comply with \cite{BNS16}. Proposition \ref{prop:monoid} and its proof carry over, showing that $X^+ \hookrightarrow X$ is a flat reductive monoid. When $\epsilon = +1$, we choose a basis $e_1, \ldots, e_n, f_n, \ldots, f_1$ for $V$ such that $(e_i|f_j) = \delta_{i,j}$ and $(e_i|e_j) = (f_i|f_j) = 0$  for all $1 \leq i, j \leq j$. It yields a basis $\epsilon_1, \ldots, \epsilon_n$ for $X_*(A)$. The Luna--Vust datum $(\mathcal{C}, \mathcal{F})$ computed in this basis is still given by Theorem \ref{prop:doubling-cone}:
	\[ \mathcal{F} = \{\text{all colors}\}, \quad \mathcal{C} \text{ is generated by the simple coroots and } (1, -\epsilon_1). \] 

	This is to be compared with the irreducible representation $\rho := \identity \boxtimes \mathrm{std}$ of $\GmC \times \hat{G}$. Here the dual group $\hat{G}$ of the identity connected component of $G$ equals $\SO(2n+1, \CC)$ (resp.\ $\Sp(2n, \CC)$, $\SO(2n, \CC)$) if $G$ is symplectic (resp.\ odd orthogonal, even orthogonal) of rank $n$. The highest (resp.\ lowest) weight of $\rho$ is still given by Proposition \ref{prop:rho-weight}, namely $(1, \epsilon_1)$ (resp.\ $(1, -\epsilon_1)$).
	
	Now we can show that $X$ coincides with the monoid $X_\rho$ in \cite[Définition V.9]{Laf14} (or \cite{BNS16, Ng14}) as done in Theorem \ref{prop:PSR-Ngo-compatible}. The proof carries over verbatim.
\end{itemize}

\chapter{Speculation on the global integrals}\label{sec:global}
The constructions of basic vectors (also known as ``basic functions'') and $\vartheta$-distributions stem from \cite[\S 3]{Sak12}, to which we refer for further examples. In what follows, fix
\begin{compactitem}
	\item $F$: a number field with the ring of adèles $\A = \A_F$ and $F_\infty := \prod_{v|\infty} F_v$;
	\item $G$: split connected reductive $F$-group with a chosen Borel subgroup $B$;
	\item $\psi: F \backslash \A \to \CC^\times$: non-trivial unitary character, by which one normalizes the Haar measures on each $F_v$.
\end{compactitem}
We will consider affine spherical embeddings $X^+ \hookrightarrow X$ under $G$ subject to the same geometric Axiom \ref{axiom:geometric}, except that everything is now defined over the number field $F$. The objects $\Lambda$, $\Lambda_X$, etc.\ are thus defined. In particular, the chosen relative invariants $(f_i)_{i=1}^r$ and their eigencharacters $(\omega_i)_{i=1}^r$ are defined over $F$; the character $|\omega|^\lambda: G(\A) \to \CC^\times$ is automorphic for all $\lambda \in \Lambda_{\CC}$.

Assume furthermore that the monoids $\Lambda_{X_v}$ for each place $v$ equals $\Lambda_X$. Therefore $f_1, \ldots, f_r$ can be used in the local settings. We will choose an arbitrary $\mathfrak{o}_F$-model of $G$ and of the $G$-equivariant embedding $X^+ \hookrightarrow X$.

We will impose two Hypotheses \ref{hyp:evaluation-global}, \ref{hyp:theta} later on.

\section{Basic vectors}\label{sec:basic-vector}
Let $X^+ \hookrightarrow X$ be an affine spherical embedding under $G$-action.
\begin{itemize}
	\item We have the $G(\A)$-equivariant line bundle $\mathscr{L}$ of densities over the adélic space $X^+(\A)$. It is \textsc{not} defined as a restricted $\otimes$-product.
	\item Suppose that for each place $v$ of $F$, we are given:
		\begin{itemize}
			\item a $G(F_v)$-equivariant vector bundle $\mathscr{E}_v$ over the $F_v$-analytic manifold $X^+(F_v)$;
			\item a non-degenerate equivariant $\mathscr{E}_v \otimes \overline{\mathscr{E}_v} \to \mathscr{L}_v$ as in \eqref{eqn:hermitian-vb}, where $\mathscr{L}_v$ stands for the density bundle over $X^+(F_v)$;
			\item a smooth continuous representation $\Schw_v$ of $G(F_v)$ satisfying Axiom \ref{axiom:zeta} with respect to $\mathscr{E}_v$.
		\end{itemize}
		Moreover, we adopt the convention $L^2(X^+_v) := L^2(X^+(F_v), \mathscr{E}_v)$, $C^\infty_c(X^+_v) := C^\infty_c(X^+(F_v), \mathscr{E}_v)$ as in \S\ref{sec:coefficients}, so that $\Schw_v \subset L^2(X^+_v)$. The typical example is the bundle of half-densities: $\mathscr{E}_v = \mathscr{L}_v^\demi$.
	\item Observe that for almost all places $v \nmid \infty$, the $\mathfrak{o}_F$-model of the $G$-equivariant embedding $X^+ \hookrightarrow X$ and the relative invariants are \emph{unramified} at $v$, this means:
	\begin{itemize}
		\item $G$ is a smooth connected reductive group scheme over $\mathfrak{o}_v$, hence $G(\mathfrak{o}_v) \subset G(F_v)$ is a hyperspecial subgroup;
		\item $X^+$ is smooth over $\mathfrak{o}_v$;
		\item $f_i \in \mathfrak{o}_v[X] \cap \mathfrak{o}_v[X^+]^\times$ for $i=1, \ldots, r$.
	\end{itemize}
\end{itemize}

\begin{definition}\label{def:basic-vector}\index{basic vector}\index{xi-basic@$\xi_v^\circ$}
	By a system of \emph{basic vectors}, we mean a family of distinguished elements $\xi_v^\circ \in \Schw_v$ chosen for almost all $v \nmid \infty$ at which our data are unramified, such that
	\begin{enumerate}[(i)]
		\item $\xi_v^\circ$ is $G(\mathfrak{o}_v)$-invariant,
		\item $\Supp(\xi_v^\circ) \subset X(\mathfrak{o}_v) \cap X^+(F_v)$,
		\item $\xi_v^\circ$ is non-vanishing over $X^+(\mathfrak{o}_v)$,
		\item the global pairing $\mathscr{E} \otimes \overline{\mathscr{E}} \to \mathscr{L}$ is well-defined (see below).
	\end{enumerate}
	
	The \emph{adélic Schwartz space} with respect to the system of basic vectors is
	\[ \Schw := \Resotimes_v \Schw_v = \varinjlim_S \bigotimes_{v \in S} \Schw_v \]
	where $S$ ranges over finite sets of places of $F$ that contain $\{v: v \mid \infty\}$, and the transition map $\bigotimes_{v \in S} \Schw_v \to \bigotimes_{v \in V} \Schw_v$ is given by multiplication by $\prod_{v \in V \smallsetminus S} \xi^\circ_v$ whenever $V \supset S$.
	
	Given a system of basic vectors, define similarly the bundle $\mathscr{E}$ over $X^+(\A)$ whose fiber at $x = (x_v)_v \in X^+(\A)$ equals
	\[ \mathscr{E}_x := \Resotimes_v \mathscr{E}_{v, x_v}, \]
	the restricted product being taken with respect to the values of basic vectors $\xi_v^\circ$ at $x_v$. In a similar vein one defines $\overline{\mathscr{E}}_x$. \textit{A priori}, the local pairings $\mathscr{E}_v \otimes \overline{\mathscr{E}}_v \to \mathscr{L}_v$ do not necessarily give rise to $\mathscr{E} \otimes \overline{\mathscr{E}} \to \mathscr{L}$, due to the possible divergence of infinite products of local densities. Our last assertion ensures the convergence.
\end{definition}

\begin{remark}\label{rem:topology-global-S}
	The space $\Schw$ can be topologized in many ways. Here we adopt the \emph{inductive topology} for $\otimes$-products (see \cite[pp.92-93]{Gro52}) on each $\bigotimes_{v \in S} \Schw_v$, and then pass to $\varinjlim_S$.
\end{remark}

\begin{remark}\label{rem:global-L2}
	In view of the definition of $\Schw$ and the topology on $X^+(\A)$, the condition (iv) for basic vectors is equivalent to that every $\xi \in \Schw$ is locally $L^2$ over $X^+(\A)$.
\end{remark}

\begin{hypothesis}\label{hyp:evaluation-global}\index{ev-gamma@$\text{ev}_\gamma$}\index{chi-gamma@$\chi_\gamma$}
	We suppose that for every $\gamma \in X^+(F)$, a linear surjection
	\[ \text{ev}_\gamma: \mathscr{E}_\gamma \twoheadrightarrow \CC \]
	is given with the following properties.
	\begin{enumerate}[(i)]
		\item There exists an automorphic character $\chi_\gamma: H_\gamma(\A) \to \CC^\times$ where $H_\gamma := \Stab_G(\gamma)$, such that $\text{ev}_\gamma(\cdot\; h) = \chi_\gamma(h)\text{ev}_\gamma(\cdot)$ for all $h \in H_\gamma(\A)$, and $\chi_\gamma = 1$ on $(H_\gamma \cap Z_G)(\A)$.
		\item For every $g \in G(F)$, the composite of $\text{ev}_{\gamma g}$ with $g: \mathscr{E}_\gamma \rightiso \mathscr{E}_{\gamma g}$ equals $\text{ev}_\gamma$.
	\end{enumerate}
\end{hypothesis}
The maps $\text{ev}_\gamma$ will be used to define $\vartheta$-distributions in \S\ref{sec:theta-dist}. Furthermore, we will often extend it to the $F$-points of some subvariety larger than $X^+$, see \S\ref{sec:Poisson}.

\begin{example}\label{eg:ev-rational}
	Take $\mathscr{E}_v = \mathscr{L}^\demi_v$. Suppose that we are given a family of \emph{convergence factors}
	\[ \left\{ \lambda_v > 0 : v \text{ is a place of } F \right\} \]
	(terminology borrowed from \cite[\S 2.3]{Weil82}), such that for every $\gamma \in X^+(F)$ and every algebraic volume form $\omega$ of $X^+$ satisfying $\omega(\gamma) \neq 0$, the infinite product $\prod_{v \notin S} \lambda_v \xi_v^\circ(\gamma) |\omega(\gamma)|_v^{-\demi}$ converges, where $S$ is a large finite set of places and $\gamma \in X^+(\mathfrak{o}_v)$ for $v \notin S$. Note that this condition is independent of the choice of $\omega$. The conclusion is that one can define
	\[ \text{ev}_\gamma: \bigotimes_v t_v \longmapsto \prod_v \lambda_v \cdot \frac{t_v}{|\omega(\gamma)|_v^\demi}. \]
	It varies under $H_\gamma(\A)$ according to an automorphic character $\chi_\gamma$ by the discussions preceding \eqref{eqn:induced-density}, and $\chi_\gamma$ is easily seen to be trivial on $(H_\gamma \cap Z_G)(\A)$. The second property relating $\text{ev}_\gamma$ and $\text{ev}_{\gamma g}$ stems from Artin's product formula.
\end{example}

We refine Example \ref{eg:ev-rational} in two special cases below.

\begin{example}\label{eg:pvs-global}
	Suppose that $X$ is a prehomogeneous vector space under $G$-action (see \S\ref{sec:pvs}) and $\mathscr{E}_v = \mathscr{L}^\demi_v$ as above. We retain the assumptions in Theorem \ref{prop:pvs-geometric-axiom} so that $X^+ \hookrightarrow X$ satisfies all our premises. In this case we take $\Schw_v = \Schw(X_v)$ to be the space of Schwartz--Bruhat half-densities on $X(F_v)$. Take an $\mathfrak{o}_F$-structure of $G$ and $X$: concretely, this means that we fix an $\mathfrak{o}_F$-lattice $L \subset X(F)$ which is $G(\mathfrak{o}_F)$-invariant. Choose an algebraic volume form $\omega$ on $X$ that comes from $\bigwedge^{\dim X} L \smallsetminus \{0\}$. Then $X$ and $\omega$ have good reduction at almost all $v \nmid \infty$.

	At those places $v \nmid \infty$, the basic vectors are taken to be
	\[ \xi^\circ_v = \mathbf{1}_{X(\mathfrak{o}_v)} \lambda_v^{-1} |\omega|_v^\demi, \]
	where $\lambda_v > 0$ is taken such that $\|\xi^\circ_v\|_{L^2} = 1$; the square of $\lambda_v$ is essentially the convergence factor in the theory of Tamagawa measures.
	
	The global pairing $\mathscr{L}^\demi \otimes \mathscr{L}^\demi \to \mathscr{L}$ is well-defined by the choice of $\lambda_v$ and Remark \ref{rem:global-L2}.
	
	We claim that $\Schw$ with the topology from Remark \ref{rem:topology-global-S} is isomorphic to the usual space in \cite[\S 11]{Weil64} of adélic Schwartz--Bruhat half-densities on $X(\A)$ if there is only one Archimedean place; otherwise we must use $\widehat{\bigotimes}_{v \mid \infty} \Schw_v \otimes \bigotimes_{v \nmid \infty} \Schw_v$ instead. Notice that the inductive topology on $\bigotimes_{v \mid \infty} \Schw_v$ coincides with the projective topology since the Archimedean Schwartz spaces are nuclear Fréchet. The claim follows because for $v \nmid \infty$, the space $\Schw_v$ is a strict inductive limit of finite-dimensional spaces as in Lemma \ref{prop:pvs-nuclear}, and it remains to remark that inductive tensor product preserves such limits.

	The definition of $\text{ev}_\gamma$ is straightforward: indeed, the convergence factors $\lambda_v$ have just been defined. It corresponds to the usual evaluation map of Schwartz--Bruhat half-densities at rational points.

	Since $X$ is non-singular, we see that $\text{ev}_\gamma$ is defined for all $\gamma \in X(F)$ by exactly the same recipe of Example \ref{eg:ev-rational}.
\end{example}

\begin{example}\label{eg:X_P-global}
	Consider the space $X_P$ in \S\ref{sec:geometric-setup}; it is a spherical homogeneous space under $M_\text{ab} \times G^\Box$ with the notations thereof. The spaces $\Schw_v$ have been defined at non-Archimedean $v$: they take values in the bundle $\mathscr{E}_v := \mathscr{L}_v^\demi$. The basic vectors are defined in \cite[p.548]{BK02} as functions, denoted by $c_P = c_{P,0}$ therein. To obtain a section of $\mathscr{L}_v^\demi$, one has to choose a $G^\Box(F_v)$-invariant measure $|\Xi|_v$ on $X_P(F_v)$. By \cite[(5.4)]{BK02} we have $c_P = |\Xi|_v^\demi$ on the $G^\Box(\mathfrak{o}_F)$-orbit containing $P_\text{der} \cdot 1$. We prefer not to pin down the choice of $|\Xi|_v$ here: it is enough to satisfy (iv) of Definition \ref{def:basic-vector}.

	Granting the definition of $\Schw_v$ for $v \mid \infty$ (see \cite[\S 6.1]{BK02}), one can fix $|\Xi|_v$ at the remaining places $v$ and define $\Schw = \Resotimes_v \Schw_v$, whose elements may be viewed as functions. Let $\gamma \in X_P(F)$. In \textit{loc.\ cit.} the map $\text{ev}_\gamma$ is simply the evaluation of functions in $\Schw$ at $\gamma$. In terms of our formalism in Example \ref{eg:ev-rational}, it amounts to taking $\lambda_v = |\omega(\gamma)|_v^\demi \cdot |\Xi|_v^{-\demi}$.
	
	In the doubling method, eg.\ in \S\ref{sec:symplectic-case}, one usually works with a smaller group $G$ such that $G \times G$ is canonically embedded in $G^\Box$, and work with the spherical variety $X := \overline{X_P}^\text{aff}$ with open $M_\text{ab} \times G \times G$-orbit $X^+$. The global formalism developed in this section applies to $X^+ \hookrightarrow X$ instead of $X_P$. Nonetheless, as in the previous case of prehomogeneous vector spaces, the maps $\text{ev}_\gamma$ are actually defined on the larger variety $X_P$.
\end{example}

\section{Theta distributions}\label{sec:theta-dist}
Retain the conventions on $X^+ \hookrightarrow X$, $\Schw$, etc.\ from the previous section.

\begin{itemize}
	\item Choose $x_0 \in X^+(F)$ (Proposition \ref{prop:existence-point}) and put $Z[X] := \Stab_G(x_0) \cap Z_G$. This is a diagonalizable $F$-group which is independent of the choice of $x_0$. It is well-known that there is a closed subgroup $\mathfrak{a}[X] \subset Z[X]^\circ(F_\infty)$ isomorphic to $\R^{\dim Z[X]}$, such that $Z[X](\A)/\mathfrak{a}[X]Z[X](F)$ is compact; we recapitulate the construction as follows. The $F$-torus $Z[X]^\circ$ being split, it is isomorphic to $C \times_{\Q} F$ for a split $\Q$-torus $C$. From $C \hookrightarrow \Res_{F/\Q} (C \times_{\Q} F)$, with $\Res_{F/\Q}$ standing for Weil restriction, we deduce $C(\R) \hookrightarrow Z[X]^\circ(F_\infty) = \Res_{F/\Q}(C \times_{\Q} F)(\R)$. Now we can take $\mathfrak{a}[X]$ to be the identity connected component of $C(\R)$ under the usual topology, viewed as a subgroup of $Z[X]^\circ(F_\infty)$. \index{Z[X]@$Z[X]$}\index{a[X]@$\mathfrak{a}[X]$} 
	
	\item The \emph{automorphic quotient}\index{amquotient@$\amspace_{G,X}$} in our context is
		\[ \amspace_{G,X} := \mathfrak{a}[X] G(F) \backslash G(\A). \]
		Hereafter, we fix a Haar measure on $\mathfrak{a}[X]$ and use the Tamagawa measure on $G(\A)$, so that $\amspace_{G,X}$ is equipped with a Radon measure. The use of Tamagawa measure is immaterial for the constructions to follow: one can actually use any invariant measure.

		Denote by $\mathcal{A}_\text{cusp}(G, X)$\index{A-cusp-form@$\mathcal{A}_\text{cusp}(G,X)$} the space of cuspidal smooth automorphic forms living on the space $\amspace_{G,X}$. More precisely, our conventions are:
		\begin{compactitem}
			\item elements of $\mathcal{A}_\text{cusp}(G,X)$ are not required to have unitary central character,
			\item $K$-finiteness is not imposed on automorphic forms, and 
			\item the forms in $\mathcal{A}_\text{cusp}(G,X)$ have rapid decay and uniform moderate growth by \cite{MW12}.
		\end{compactitem}
		For $\phi \in \mathcal{A}_\text{cusp}(G,X)$ we put $\phi_\lambda := \phi |\omega|^\lambda$; it still lies in $\mathcal{A}_\text{cusp}(G,X)$ by Lemma \ref{prop:char-Lambda}. \index{phi-lambda@$\phi_\lambda$}
	\item The \emph{cuspidal automorphic representations}\index{representation!cuspidal automorphic} in our context are continuous $G(\A)$-subrepresentations $\pi$ together with a given automorphic realization $\pi \hookrightarrow \mathcal{A}_\text{cusp}(G,X)$. The representation $\pi$ decomposes as $\pi \simeq \widehat{\bigotimes}_{v | \infty} \pi_v \otimes  \Resotimes_{v \nmid \infty} \pi_v$ with each $\pi_v$ being an irreducible nice representation of $G(F_v)$. In particular, $\pi_v$ is an irreducible SAF representation when $v|\infty$, and $\pi$ is essentially unitary.
\end{itemize}

\begin{definition}\index{theta-distribution@$\vartheta$-distribution}
	By a \emph{$\vartheta$-distribution}, we mean a $G(F) \mathfrak{a}[X]$-invariant continuous linear functional
	\[ \vartheta: \Schw \longrightarrow \CC. \]
	By Frobenius reciprocity for smooth continuous representations, this is equivalent to giving a $G(\A)$-equivariant continuous linear map
	\begin{align*}
		\vartheta: \Schw & \longrightarrow C^\infty(\amspace_{G,X}) \\
		\xi & \longmapsto \left[ \vartheta_\xi(g) := \vartheta(g\xi) \right].
	\end{align*}
	For a precise description of the topology on $C^\infty(\amspace_{G,X})$, see \cite[p.677]{Be88}.
\end{definition}

Henceforth, we shall assume Hypothesis \ref{hyp:evaluation-global} and consider the $\vartheta$-distribution
\begin{gather}\label{eqn:theta-dist}
	\vartheta(\xi) = \sum_{\gamma \in X^+(F)} \text{ev}_\gamma(\xi), \quad \xi \in \Schw.
\end{gather}

\begin{hypothesis}\label{hyp:theta}
	The linear functional $\vartheta$ in \eqref{eqn:theta-dist} is indeed a $\vartheta$-distribution.
\end{hypothesis}
Granting its convergence, the required invariance of $\vartheta$ is automatic by the conditions in Hypothesis \ref{hyp:evaluation-global}.

Our aim is to integrate $\vartheta_\xi \in C^\infty(\amspace_{G,X})$, for a given $\xi \in \Schw$, against suitable functions of over $\amspace_{G,X}$. For this purpose, it seems necessary to assume that $\vartheta_\xi$ is of \emph{moderate growth}, cf.\ \cite[Proposition 3.1.3]{Sak12}. We skip this intermediate step and formulate the following axiom.

\begin{axiom}[Cf.\ {\cite[Conjecture 3.2.4]{Sak12}}]\label{axiom:theta}\index{Z_lambda@$Z_\lambda$}
	Assume Hypothesis \ref{hyp:theta}. For every cuspidal automorphic representation $\pi \hookrightarrow \mathcal{A}_\text{cusp}(G, X)$, we assume that there exists $\lambda(\pi) \in \Lambda_{\R}$ such that integral
	\[ Z_\lambda(\phi \otimes \xi) := \int_{\amspace_{G,X}} \vartheta_\xi \phi_\lambda \]
	is convergent (as a double integral $\int_{\amspace_{G,X}} \sum_\gamma \cdots$) for every $\xi \in \Schw$, $\phi \in V_\pi$ and $\Re(\lambda) \relgeq{X} \lambda(\pi)$, therefore $Z_\lambda$ defines a $G(\A)$-invariant bilinear form.

	Furthermore, we assume that $Z_\lambda$ admits a meromorphic continuation to all $\lambda \in \Lambda_{\CC}$ for any fixed $\phi$ and $\xi$. Call $Z_\lambda$ the \emph{global zeta integral} associated to $\vartheta$.
\end{axiom}

\begin{remark}
	The analogy with Axiom \ref{axiom:zeta} is clear. However, here we make no assumption on continuity. Let us indicate briefly what may be expected in general. Equip $V_\pi$ with the topology via its automorphic realization and suppose that the pairing $Z_\lambda: \pi_\lambda \otimes \Schw \to \CC$ is separately continuous for $\Re(\lambda) \relgeq{X} \lambda(\pi)$. If we topologize $\Schw$ as in Remark \ref{rem:topology-global-S}, the continuity in $\Schw$ can be checked on each $\Schw_v$.
	\begin{itemize}
		\item The automorphic realization makes $V_\pi$ into a $\varinjlim$ of Fréchet spaces (say by varying the level), thus barreled. It follows that $Z_\lambda$ is hypocontinuous for $\Re(\lambda) \relgeq{X} \lambda(\pi)$.
		\item If $\Schw$ is also barreled, then the hypocontinuity extends to all $\lambda \in \Lambda_{\CC}$ upon clearing denominators by the arguments of Lemma \ref{prop:circ-removal}, i.e.\ by the Theorem \ref{prop:GS-principle} of Gelfand--Shilov. This will yield a meromorphic family in $\Hom_{G(\A)}(\pi_\lambda, \Schw^\vee)$ as in the local case.
	\end{itemize}
	It is unclear if one can deduce the barreledness of $\Schw$ from those of $\Schw_v$.
\end{remark}

\begin{definition}\label{def:zeta-factorizable}
	We say that $Z_\lambda$ is \emph{factorizable} if it is proportional to an Euler product
	\[ Z^\text{loc}_{\lambda, (\varphi_v)_v} := \prod_v \left( Z_{\lambda, \varphi_v}: \pi_{v,\lambda} \otimes \Schw_v \to \CC \right) \]
	for some family $(\varphi_v \in \mathcal{N}_{\pi_v})_v$, whenever $\Re(\lambda) \relgg{X} 0$. Here $Z_{\lambda, \varphi_v}$ is the local zeta integral of Axiom \ref{axiom:zeta}. If $Z_\lambda$ is a linear combination of such Euler products, we say $Z_\lambda$ is \emph{essentially factorizable}.
\end{definition}

As is well-known, factorizability holds true if the space $\Hom_{G(F_v)}(\pi_{v,\lambda} \otimes \Schw_v, \CC)$ (say the algebraic $\Hom$) has dimension $\leq 1$ for generic $\lambda$ and every $v$. We prefer not to include this into the axiom.

\section{Relation to periods}\label{sec:periods}
Keep the setting of \S\ref{sec:basic-vector} and \S\ref{sec:theta-dist}. Specifically, we consider a paring $\mathscr{E}_v \otimes \overline{\mathscr{E}}_v \to \mathscr{L}_v$ at each place $v$ of $F$. In addition, a system of basic vectors $\xi_v^\circ$ is chosen, which gives rise to $\Schw = \Resotimes_v \Schw_v$ and $\mathscr{E} = \Resotimes_v \mathscr{E}_v$. Define
\[\begin{tikzcd}[row sep=tiny]
	\mathcal{J}^+ := \left\{ (x,g) \in X^+(\A) \times \amspace_{G,X} : xg^{-1} \in X^+(F) \right\} \arrow{r}[above]{\sim} & X^+(F) \times \amspace_{G,X} \\
	(x,g) \arrow[mapsto]{r} & (xg^{-1}, g).
\end{tikzcd}\]
Therefore $\mathcal{J}^+$ is a disjoint union of $\amspace_{G,X}$, from which it acquires the natural topology and Radon measure.

Consider the pull-push diagram
\begin{equation}\label{eqn:pull-push} \begin{tikzcd}[column sep=tiny, row sep=small]
	& \mathcal{J}^+ \arrow{ld}[left, inner sep=1em]{p_1} \arrow{rd}[right, inner sep=1em]{p_2} & \\
	X^+(\A) & & \amspace_{G,X}
\end{tikzcd}\end{equation}
made from the projection maps. We have seen that $p_2$ is surjective, whereas the image of $p_1$ is a union of $G(\A)$-orbits. Also notice that $p_1, p_2$ are both equivariant if we let $G(\A)$ act on the right of $\mathcal{J}^+$ by $(x,g)g' = (xg', gg')$.

\begin{notation}
	For every $x \in \mathcal{J}^+$ we set $H_x := \Stab_G(x)$, regarded as a scheme over $\A$. Observe that $H_x(\A) \supset \mathfrak{a}[X]$. For $\gamma \in X^+(F)$ we endow $H_\gamma(\A)$ with the right-invariant Tamagawa measure. Again, any coherent way of choosing Haar measures for various $H_\gamma$ will do the job.
\end{notation}

\begin{lemma}\label{prop:E-trivialization}
	Over $\mathcal{J}^+$, there is a canonical isomorphism $\tau = (\tau_{(x,g)})_{(x,g) \in \mathcal{J}} : p_1^{-1} \mathscr{E} \rightiso p_2^{-1} \CC$, where $\CC$ stands for the constant sheaf over $\amspace_{G,X}$. It satisfies
	\begin{align*}
		\tau_{(xg',gg')}(\xi g') & = \tau_{(x,g)}(\xi) \in \CC, \quad g' \in G(\A), \; \xi \in \mathscr{E}_x, \\
		\tau_{(x,g)}(\xi h) = \tau_{(x,gh^{-1})}(\xi) & = \chi_\gamma(ghg^{-1}) \tau_{(x,g)}(\xi), \quad h \in H_x(\A)
	\end{align*}
	where $\chi_\gamma: H_\gamma(\A) \to \CC^\times$ is the automorphic character postulated in Hypothesis \ref{hyp:evaluation-global} satisfying $\mathrm{ev}_\gamma(\cdot\;t) = \chi_\gamma(t)\mathrm{ev}_\gamma(\cdot)$.
\end{lemma}
\begin{proof}
	Let $(x,g) \in \mathcal{J}^+$, $\gamma := xg^{-1} \in X'(F)$. At the fiber over $(x,g)$, use the $G(\A)$-structure on $\mathscr{E}$ to set
	\begin{gather}\label{eqn:E-trivialization}
		\tau_{(x,g)}: p_1^{-1}(\mathscr{E})_{(x,g)} = \mathscr{E}_{\gamma g} \xrightarrow{g^{-1}} \mathscr{E}_\gamma \xrightarrow{\text{ev}_\gamma} \CC = p_2^{-1}(\CC)_{(x,g)}.
	\end{gather}
	It is routine to verify the required properties. In view of the topology on $\mathcal{J}^+$, the family $(\tau_{(x,g)})_{(x,g)}$ glues into an isomorphism $p_1^{-1} \mathscr{E} \simeq p_2^{-1} \CC$ between bundles.
\MyQED\end{proof}

Given $(x,g) \in \mathcal{J}^+$, we will also need the adjoint map $\check{\tau}_{(x,g)}: \mathscr{L}_x \rightiso \overline{\mathscr{E}}_x$. It is characterized by the commutativity of
\begin{equation}\label{eqn:tau-adjoint} \begin{tikzcd}[row sep=tiny]
	& \CC \otimes \mathscr{L}_x \arrow{rd}{\text{natural}} & \\
	\mathscr{E}_x \otimes \mathscr{L}_x \arrow{ru}{\tau_{(x,g) \otimes \identity}} & & \mathscr{L}_x \\
	& \mathscr{E}_x \otimes \overline{\mathscr{E}}_x \arrow{ru} \arrow[leftarrow]{lu}{\identity \otimes \check{\tau}_{(x,g)}} &
\end{tikzcd}\end{equation}
since both pairings towards $\mathscr{L}_x$ are non-degenerate. We still have $\check{\tau}_{(x,ghg^{-1})} = \chi(g^{-1}hg) \check{\tau}_{(x,g)}$ as in Lemma \ref{prop:E-trivialization}.

For any $x \in X^+(\A)$, set
\begin{gather*}
	\mathcal{Y}_x := \{ g \in \amspace_{G,X} : xg^{-1} \in X^+(F) \} = p_2(p_1^{-1}(x)).
\end{gather*}

\begin{lemma}\label{prop:Sha-finiteness}
	The set $\mathcal{Y}_x/H_x(\A)$ is finite.
\end{lemma}
\begin{proof}
	By mapping $g \in \mathcal{Y}_x$ to $\gamma := xg^{-1}$, the set $\mathcal{Y}_x/H_x(\A)$ is in bijection with the set of $G(F)$-orbits in the $(x G(\A)) \cap X^+(F)$: taking quotient by $\mathfrak{a}[X]$ has no effect here. Suppose that $\mathcal{Y}_x/H_x(\A) \neq \emptyset$ and choose $\gamma_0 \in (x G(\A)) \cap X^+(F)$.
	
	Every $\gamma \in (x G(\A)) \cap X^+(F)$ surely lies in the $G(\bar{F})$-orbit of $\gamma_0$, and the obstacle for upgrading this to $G(F)$-orbit is an element $\text{inv}(\gamma_0, \gamma) \in \Ker \left( H^1(F, H_x) \to H^1(F, G) \right)$. Furthermore, $\text{inv}(\gamma_0, \gamma)$ has trivial image in $H^1(F_v, H_x)$ at every place $v$ of $F$. We conclude by applying the finiteness of $\Ker\left( H^1(F, H_x) \to \prod_v H^1(F_v, H_x) \right)$: see \cite[Theorem 6.15]{PR94}, noting that in \textit{loc.\ cit.} an ``algebraic group'' means a possibly disconnected linear $F$-group.
\MyQED\end{proof}

For $\gamma \in X^+(F)$ and $\phi \in \mathcal{A}_\text{cusp}(G,X)$, the \emph{period integrals} $\int_{\mathfrak{a}[X] H_\gamma(F) \backslash H_\gamma(\A)} \phi$ are convergent: it suffices to notice that $\phi$ is rapidly decreasing modulo center, $Z[X] = H_\gamma \cap Z_G$ and $Z[X](F) \mathfrak{a}[X] \backslash Z[X](\A)$ is compact. Since one can view $H_\gamma(\A) \backslash G(\A)$ as an open subset of $X^+(\A)$ containing $\gamma$ via the action map, it makes sense to define\index{measure-gamma@$\mes_\gamma$}
\begin{gather}\label{eqn:mes-gamma}
	\mes_\gamma := (\text{the quotient of right Tamagawa measures }) \in \mathscr{L}_\gamma.
\end{gather}
This element can be transported: write $\mes_\gamma g \in \mathscr{L}_{\gamma g}$ for $g \in G(\A)$. We obtain an automorphic character $\delta_\gamma: H_\gamma(\A) \to \R^\times_{>0}$ characterized by \index{delta-gamma@$\delta_\gamma$}
\[ \mes_\gamma h = \delta_\gamma(h) \mes_\gamma. \]

\begin{definition}\index{period@$\mathcal{P}_x(\phi)$}
	For every $x \in p_1(\mathcal{J}^+)$ we put
	\begin{gather}\label{eqn:period}
	\mathcal{P}_x(\phi) := \sum_{\substack{g \in \mathcal{Y}_x / H_x(\A) \\ \gamma := xg^{-1} }} \; \check{\tau}_{(x,g)} \int_{\mathfrak{a}[X] H_\gamma(F) \backslash H_\gamma(\A)} (\delta_\gamma \chi_\gamma^{-1})(\cdot) \phi(\cdot \; g) \cdot \underbracket{\mes_\gamma g}_{\in \mathscr{L}_x} ,
	\end{gather}
	which is $\overline{\mathscr{E}}_x$-valued; the sum is finite by Lemma \ref{prop:Sha-finiteness}. Recall that $\chi_\gamma$, $\check{\tau}_{(x,g)}: \mathscr{L}_x \rightiso \overline{\mathscr{E}}_x$ are defined after Lemma \ref{prop:E-trivialization}; their variance properties entail that $\check{\tau}_{(x,g)}(\int \cdots) \mes_\gamma g$ depends only on the coset $g \in \mathcal{Y}_x/H_x(\A)$. Extension by zero defines $\mathcal{P}_x$ for all $x \in X^+(\A)$.
\end{definition}

\begin{proposition}\label{prop:period-equivariance}
	For any cuspidal automorphic representation $\pi \hookrightarrow \mathcal{A}_\mathrm{cusp}(G,X)$,
	\begin{enumerate}[(i)]
		\item the functionals $\mathcal{P}_x$ give rise to a $G(\A)$-equivariant linear map
		\begin{align*}
		\mathcal{P}: \pi & \longrightarrow C^\infty(X^+(\A)) := C^\infty(X^+(\A), \overline{\mathscr{E}}) \\
		\phi & \longmapsto [\mathcal{P}(\phi): x \mapsto \mathcal{P}_x(\phi)];
		\end{align*}
		\item for all $\lambda \in \Lambda_{\CC}$ and $\phi \in \pi$, we have $\mathcal{P}(\phi_\lambda) = |f|^\lambda \mathcal{P}(\phi)$. 
	\end{enumerate}
\end{proposition}
\begin{proof}
	Let $x \in X^+(\A)$ and $g' \in G(\A)$. From $\mathcal{Y}_{xg'} = \mathcal{Y}_x \cdot g'$ and $H_{xg'} = g'^{-1} H_x g'$, we deduce
	\begin{align*}
		\mathcal{P}_{xg'}(\phi) & = \sum_{ \substack{g'' \in \mathcal{Y}_{xg'}/H_{xg'}(\A) \\ \gamma := xg'g''^{-1} } } \check{\tau}_{(xg', g'')} \int_{\mathfrak{a}[X] H_\gamma(F) \backslash H_\gamma(\A)} (\delta_\gamma \chi_\gamma^{-1})(\cdot) \phi(\cdot \; g'') \cdot \underbracket{\mes_\gamma g''}_{\in \mathscr{L}_{xg'}} , \\
		& = \sum_{ \substack{g \in \mathcal{Y}_x/H_x(\A) \\ \gamma = xg^{-1} } } \check{\tau}_{(xg', gg')} \int_{\mathfrak{a}[X] H_\gamma(F) \backslash H_\gamma(\A)} (\delta_\gamma \chi_\gamma^{-1})(\cdot) \phi(\cdot \; gg') \cdot \mes_\gamma gg' \\
		& \quad (\text{setting } g'' = gg') \\
		& = \left[ \; \sum_{ \substack{g \in \mathcal{Y}_x/H_x(\A) \\ \gamma = xg^{-1} } } \check{\tau}_{(x, g)} \int_{\mathfrak{a}[X] H_\gamma(F) \backslash H_\gamma(\A)} (\delta_\gamma \chi_\gamma^{-1})(\cdot) (g'\phi)(\cdot \; g) \cdot \underbracket{\mes_\gamma g}_{\in \mathscr{L}_x} \;\right] g' \\
		& = \mathcal{P}_x(g'\phi) g' \; \in \overline{\mathscr{E}}_{xg'}.
	\end{align*}
	We have exploited the equivariance $\check{\tau}_{(xg',gg')}(\xi g') = \check{\tau}_{(x,g)}(\xi)g'$: it results from the corresponding property of $\tau$ in Lemma \ref{prop:E-trivialization}. This proves the first assertion.
	
	To prove the second assertion, observe that replacing $\phi$ by $\phi_\lambda$ amounts to insert a factor $|\omega|^\lambda(\cdot\; g)$ in the $\int_{\mathfrak{a}[X] H_\gamma(F) \backslash H_\gamma(\A)} \cdots$ in \eqref{eqn:period}. Lemma \ref{prop:char-Lambda} implies that $|\omega|^\lambda(\cdot\; g) = |\omega|^\lambda(g)$. It remains to notice that when $(x,g) \in \mathcal{J}^+$, we have
	\[ |f(x)|^\lambda = |f(\gamma g)|^\lambda = |f(\gamma)|^\lambda |\omega(g)|^\lambda = |\omega(g)|^\lambda \]
	by the product formula.
\MyQED\end{proof}

As a preparation for the next result, we inspect how the measure on $\mathcal{J}^+$ decomposes under the fibrations $p_1$ and $p_2$.
\begin{itemize}
	\item The fibers of $p_2: \mathcal{J}^+ \to \amspace_{G,X}$ are identifiable with the discrete space $X^+(F)$, equipped with the counting measure. Under this convention $p_2$ preserves measures locally.
	\item The fiber of $p_1: \mathcal{J} \to X'(\A)$ over $x$ is
		\[ \{x\} \times \bigsqcup_{\gamma \in X^+(F)} \{ g \in \amspace_{G,X} : \gamma g = x \}. \]
		Those $\gamma$ with $\{ g \in G(\A): \gamma g = x \} \neq \emptyset$ is in natural bijection with $\mathcal{Y}_x/H_x(\A)$. For these $\gamma$, the transporter $\{ g \in G(\A): \gamma g = x \}$ is an $(H_\gamma(\A), H_x(\A))$-bitorsor under bilateral translation. We shall equip $\{ g \in \amspace_{G,X} : \gamma g = x \}$ with the Tamagawa measure of $H_\gamma(F) \backslash H_\gamma(\A)$ divided by the Haar measure on $\mathfrak{a}[X]$ using the bitorsor structure.
\end{itemize}

\begin{theorem}\label{prop:theta-period}
	Let $\pi \hookrightarrow \mathcal{A}_\mathrm{cusp}(G,X)$ be a cuspidal automorphic representation and $\phi \in \pi$. The zeta integral $Z_\lambda(\phi \otimes \xi)$ in Axiom \ref{axiom:theta} satisfies
	\begin{align*}
		Z_\lambda(\phi \otimes \xi) & = \int_{\amspace_{G,X} \ni g} \; \sum_{\gamma \in X^+(F)} \mathrm{ev}_\gamma(g\xi) \phi_\lambda(g) = \int_{X^+(\A)} \xi \mathcal{P}(\phi_\lambda) \\
		& = \int_{X^+(\A)} \xi |f|^\lambda \mathcal{P}(\phi)
	\end{align*}
	for $\lambda \in \Lambda_{\CC}$ in its range of convergence. The integrands in the last two expressions are viewed as sections of $\mathscr{L}$ via $\mathscr{E} \otimes \overline{\mathscr{E}} \to \mathscr{L}$.
\end{theorem}
\begin{proof}
	The first equality follows from the definition of $\vartheta$-distribution.	For every $(x,g) \in \mathcal{J}^+$, Lemma \ref{prop:E-trivialization} yields
	\[ \begin{tikzcd}[row sep=tiny]
		\mathscr{E}_x \arrow{r}[above]{g^{-1}} \arrow[bend left=25]{rr}{\tau_{(x,g)}} & \mathscr{E}_{xg^{-1} =: \gamma} \arrow{r}[above]{\text{ev}_\gamma} & \CC \\
		p_1^*(\xi)(x,g) = \xi(x) \arrow[mapsto]{r} & \xi(x)g^{-1} \arrow[mapsto]{r} & \text{ev}_\gamma(\xi(x)g^{-1}) \\
		& (g\xi)(\gamma) \arrow[-,double equal sign distance]{u} & \text{ev}_\gamma(g\xi) \arrow[-,double equal sign distance]{u}.	
	\end{tikzcd} \]
	The pull-push diagram \eqref{eqn:pull-push} leads to
	\[ Z_\lambda(\phi \otimes \xi) = \int_{\amspace_{G,X}} p_{2 !} (p_1^* \xi) \phi_\lambda = \int_{\mathcal{J}^+} p_1^*(\xi) p_2^*(\phi_\lambda), \]
	where
	\begin{itemize}
		\item $p_1^*(\xi)$ is a section of $p_1^{-1} \mathscr{E}$ over $\mathcal{J}^+$, which is also viewed as a section of $\CC = p_2^{-1} \CC$ via $\tau$;
		\item $p_{2 !}$ denotes the summation over $\gamma = xg^{-1}$ of sections of $p_2^{-1} \CC$, along fibers of $p_2$.
	\end{itemize}
	
	The next step is to push the integral down to $X^+(\A)$ via $p_1$, namely
	\[ \int_{X^+(\A) \ni x} \; \int_{p_1^{-1}(x) \ni (x,g)} \tau_{(x,g)}(\xi) \phi_\lambda(g). \]
	According to the prior discussions, the properties in Lemma \ref{prop:E-trivialization} together with \eqref{eqn:mes-gamma}, it transforms into
	\[ \int_{X^+(\A) \ni x} \sum_{\substack{g \in \mathcal{Y}_x/H_x(\A) \\ \gamma := xg^{-1} }} \tau_{(x,g)}(\xi) \int_{\mathfrak{a}[X] H_\gamma(F) \backslash H_\gamma(\A) } (\delta_\gamma \chi_\gamma^{-1})(\cdot) \phi_\lambda(\cdot\; g) \cdot \underbracket{\mes_\gamma g}_{\in \mathscr{L}_x}. \]
	Note that $\int_{X^+(\A)}$ makes sense as the integrand is a section of $\mathscr{L}$. We may use \eqref{eqn:tau-adjoint} to rewrite it as
	\[ \int_{X^+(\A) \ni x} \sum_{\substack{g \in \mathcal{Y}_x/H_x(\A) \\ \gamma = xg^{-1} }} \xi(x) \; \check{\tau}_{(x,g)} \int_{\mathfrak{a}[X] H_\gamma(F) \backslash H_\gamma(\A) } (\delta_\gamma \chi_\gamma^{-1})(\cdot) \phi_\lambda(\cdot\; g) \cdot \mes_\gamma g; \]
	the integrand still takes value in $\mathscr{L} \leftarrow \mathscr{E} \otimes \overline{\mathscr{E}}$. All in all, we have arrived at
	\[ \int_{X^+(\A) \ni x} \xi(x) \mathcal{P}_x(\phi_\lambda) \]
	upon recalling \eqref{eqn:period}, whence the second equality. The last equality results from Proposition \ref{prop:period-equivariance} (ii).
\MyQED\end{proof}

\begin{corollary}\label{prop:period-dist}
	If $Z_\lambda(\phi \otimes \xi)$ is not identically zero for $\xi \in \Schw$, $\phi \in \pi$ and $\Re(\lambda) \relgg{X} 0$, then there exists $\gamma \in X^+(F)$ such that $\pi$ is globally $(H_\gamma, \delta_\gamma^{-1} \chi_\gamma)$-distinguished.
\end{corollary}

\begin{proposition}\label{prop:period-group-case}
	Suppose $G = H \times H$ where $H$ is a split connected reductive $F$-group, $X^+ = H$ under the action \eqref{eqn:group-case-action1}, and $\mathscr{E}_v = \mathscr{L}_v^\demi$ at each place $v$. Let $\Pi$ be a cuspidal automorphic representation realized on $\amspace_{G,X}$. If $Z_\lambda(\phi \otimes \xi)$ is not identically zero for $\phi \in \Pi$, then $\Pi \simeq \pi \boxtimes \check{\pi}$ for some cuspidal automorphic representation of $H(\A)$. In this case,
	\[ \mathcal{P}_x(\phi) = \angles{x\phi_1, \phi_2}_\mathrm{Pet} \cdot |\Omega|^\demi, \quad \phi = \phi_1 \otimes \phi_2 \in \Pi \]
	where $|\Omega|$ stands for some Haar measure on $H(\A)$, and $\angles{\cdot,\cdot}_\mathrm{Pet}$ is the Petersson pairing between $\pi$ and $\check{\pi}$. The global zeta integrals are factorizable in this case.
\end{proposition}
\begin{proof}
	For every $x \in X^+(\A) = H(\A)$:
	\begin{compactitem}
		\item $H \simeq H_x = \{ (x^{-1}hx, h): h \in H\} \subset H \times H$;
		\item $Z[X] = \text{diag}(Z_H)$, thus $\mathfrak{a}[X] H_x(F) \backslash H_x(\A)$ is just the usual automorphic quotient $\amspace_{H_x}$ associated to $H_x$;
		\item the set $\mathcal{Y}_x/H_x(\A)$ is a singleton with representative $g = (x,1) \in H \times H(\A)$;
		\item accordingly we may take $\gamma=1$ in \eqref{eqn:period}, so $H_\gamma = \text{diag}(H)$; note that $\delta_\gamma = \chi_\gamma = \mathbf{1}$.
	\end{compactitem}
	The first assertion follows immediately from Corollary \ref{prop:period-dist}. Assume hereafter $\Pi = \pi \boxtimes \check{\pi}$. Recall that we have put a Haar measure on $H_\gamma(\A) = H(\A)$, namely the Tamagawa measure. For $\phi = \phi_1 \otimes \phi_2$, the formula \eqref{eqn:period} reduces to
	\[ \mathcal{P}_x(\phi) = \check{\tau}_{(x,(x,1))} \int_{\amspace_{H_\gamma} \simeq \amspace_H \ni h} \phi_1(hx) \phi_2(h) \cdot \mes_1 (x,1). \]
	Note that the integral is just $\angles{x\phi_1, \phi_2}_\text{Pet}$ by the very definition of the Petersson pairing.
	
	It remains to clarify the term $T(x) := \check{\tau}_{(x,(x,1))}(\mes_1 (x,1)) \in \overline{\mathscr{E}}_x$. The equivariance of $\check{\tau}$ implies $T(x) = T(1)x$, hence $T: x \mapsto T(x)$ is a positive, right $H(\A)$-invariant section of $\mathscr{L}^\demi$. Therefore $T = |\Omega|^\demi$ for some Haar measure $|\Omega|$ on $H(\A)$. We shall not attempt to evaluate $|\Omega|$, as it depends on our initial choice of measures on $H(\A)$ and $\mathfrak{a}[X]$.
	
	Finally, the factorizability of $Z_\lambda$ follows from the well-known factorizability of Petersson pairings; more precisely, $\angles{\cdot, \cdot}_\text{Pet}$ factors into integral pairings of local matrix coefficients.
\MyQED\end{proof}

\section{Global functional equation and Poisson formula}\label{sec:Poisson}
In this section we work with two affine spherical embeddings $X_i^+ \hookrightarrow X_i$ for $i=1,2$, both subject to the previous conditions. Thus our data include the basic vectors $\xi_v^{\circ,(i)}$ for almost all $v \nmid \infty$ as well as the $\vartheta$-distributions $\vartheta^{(i)}: \Schw_i \to \CC$. Axiom \ref{axiom:theta} will be imposed later on.

In order to talk about model transition, hereafter we assume $\Lambda_2 \subset \Lambda_1$ as in the local setting.
%

Assume that we are given a family of model transitions $\mathcal{F}_v: \Schw_{2,v} \rightiso \Schw_{1,v}$ at each place $v$, as in \S\ref{sec:model-transition}. Moreover we assume
\[ \mathcal{F}_v \left( \xi_v^{\circ,(2)} \right) = \xi_v^{\circ,(1)} \quad \text{ for almost all } v \nmid \infty. \]
Consequently, the adélic model transition\index{model transition!global}
\[ \mathcal{F} := \bigotimes_v \mathcal{F}_v: \Schw_2 \rightiso \Schw_1 \]
is well-defined as a continuous $G(\A)$-equivariant linear map.

\begin{definition}\label{def:gfe}\index{functional equation!global}
	Given $\mathcal{F} = \bigotimes_v \mathcal{F}_v$ as above, assume the Axiom \ref{axiom:theta} so that $Z^{(i)}_\lambda$ is defined for $i=1,2$, we say that the \emph{global functional equation for zeta integrals} holds for a cuspidal automorphic representation $\pi \hookrightarrow \mathcal{A}_\text{cusp}(G, X_2) \cap \mathcal{A}_\text{cusp}(G, X_2)$, if
	\[ Z^{(1)}_\lambda(\phi \otimes \mathcal{F}(\xi)) = Z^{(2)}_\lambda(\phi \otimes \xi), \quad \xi \in \Schw_2, \; \phi \in V_\pi \]
	where both sides are interpreted as meromorphic families in $\lambda \in \Lambda_{2,\CC}$.	

	Here is an implicit assumption that $\Lambda_{2,\CC}$ is not contained in the singular locus of $\lambda \mapsto Z^{(1)}_\lambda(\phi \otimes \mathcal{F}(\xi))$.
\end{definition}

The  meromorphy and the global functional equation for $Z^{(i)}_\lambda$ ($i=1,2$) are often established at once, by invoking a \emph{Poisson formula}\index{Poisson formula} relating $\vartheta^{(1)}$ and $\vartheta^{(2)}$. The upshot is to work with the new distribution
\begin{gather*}
	\tilde{\vartheta}^{(i)}(\xi) := \sum_{\gamma \in X'_i(F)} \text{ev}_\gamma(\xi), \quad \xi \in \Schw_i
\end{gather*}
where $X'_i$ is subvariety lying between $X_i$ and $X^+_i$, such that one still have the maps $\text{ev}_\gamma$ for $\gamma \in X'_i(F)$ that satisfy the properties in Hypothesis \ref{hyp:evaluation-global}, for $i=1,2$. Granting the convergence of $\tilde{\vartheta}^{(1)}$, $\tilde{\vartheta}^{(2)}$, the Poisson formula would take the form
\begin{gather}\label{eqn:Poisson}
	\tilde{\vartheta}^{(1)}(\mathcal{F}(\xi)) = \tilde{\vartheta}^{(2)}(\xi), \quad \xi \in \Schw_2.
\end{gather}

To plug this into $Z_\lambda$, it is often necessary to assume that the orbits in $(X' \smallsetminus X)(F)$ are all \emph{negligible}\index{negligible orbit}. We say an element $\gamma \in X(F)$ is negligible if $\Stab_G(\gamma) \supset U_P$ for some parabolic subgroup of $P \subsetneq G$; its effect is that the integration of $\text{ev}_\gamma(g\xi)$ against a cusp form, whenever convergent, will be zero. Certainly, this assumptions excludes the case where $G$ is a torus and $X'(F) \neq X^+(F)$. Known instances include
\begin{itemize}
	\item Wavefront spherical varieties with a structure of preflag bundle \cite[\S 4.4]{Sak12}.
	\item Spaces arising in the doubling method, cf.\ Remark \ref{rem:doubling-negligible}.
	\item The Godement--Jacquet case with $n > 1$: $X = \text{Mat}_{n \times n}$ under $G = \GL(n) \times \GL(n)$-action. This is seen in \S\ref{sec:GJ}.
	\item The spaces of symmetric bilinear forms on an $F$-vector space of dimension $> 1$ (Example \ref{eg:symm-forms}).
\end{itemize}
We refer to Remark \ref{rem:doubling-negligible} for further discussions on this property.

We hesitate to formalize a general theory for deducing meromorphy and functional equation from \eqref{eqn:Poisson}, which would require finer properties of the Schwartz spaces and the geometry of $X_1, X_2$. Instead, let us inspect the mechanism in concrete examples in which $\Schw_i$, $\mathcal{F}$ and $\tilde{\vartheta}^{(i)}$ are already available.

In what follows, we fix a nontrivial unitary character $\psi = \prod_v \psi_v: \A/F \to \CC^\times$ and normalize the Haar measures on $\A$ and on each $F_v$ accordingly. Such a $\psi$ is unique up to dilation by $F^\times$. For almost all $v \nmid \infty$, we have $\psi_v|_{\mathfrak{o}_V}=1$ but $\psi_v|_{\mathfrak{m}_v^{-1}} \not\equiv 1$; such an additive character is called unramified. The self-dual Haar measure with respect to unramified $\psi_v$ is characterized by $\mes(\mathfrak{o}_v)=1$.

\begin{example}
	In the prehomogeneous case (Example \ref{eg:pvs-global})\index{prehomogeneous vector space}, we take $\tilde{\vartheta}$ to be the linear functional
	\[ \tilde{\vartheta} = \tilde{\vartheta}^X := \sum_{\gamma \in X(F)} \text{ev}_\gamma: \Schw(X) \to \CC. \]
	Its convergence and continuity are well-known. Hereafter, we assume that $X$ is $F$-regular so that the dual space $\check{X}$ is prehomogeneous as well.

	The local Fourier transforms $\mathcal{F}_v$ patch together to an adélic model transition $\mathcal{F}: \Schw(X) \rightiso \Schw(\check{X})$. Indeed, the basic vectors $\xi^\circ_v$ are preserved since $\psi_v$ is unramified almost everywhere and $\mathcal{F}_v$ is an isometry, so there is no worry about the convergence factors $\lambda_v$.
	
	The Poisson formula $\vartheta^{\check{X}}(\mathcal{F}\xi) = \vartheta^X(\xi)$ is almost classical: one can get rid of half-densities as in the proof of Theorem \ref{prop:Fourier-def}.

	The meromorphy and global functional equation for $Z^X_\lambda$ are currently unknown. However, if $\pi$ is allowed to be the trivial representation of $G(\A)$ (which is non-cuspidal), then this is known under mild conditions --- see \cite[Chapter 5]{Ki03} and the references therein.
\end{example}

\begin{example}\label{eg:GJ-global}\index{Godement--Jacquet integrals!global}
	As illustrated in \S\ref{sec:GJ}, Godement--Jacquet theory may be regarded as a special case of prehomogeneous vector spaces. This is also true in the global setting. Specifically, consider $X^+ = \GL(n)$, $X = \text{Mat}_{n \times n}$ with $G = \GL(n) \times \GL(n)$-action as in \eqref{eqn:group-case-action1}, a cuspidal representation $\Pi := \pi \boxtimes \check{\pi}$ of $G(\A)$ and $\phi = \phi_1 \otimes \phi_2 \in \Pi$. Identify $\Lambda_X$ with $\Z_{\geq 0}$ as in \S\ref{sec:GJ}, and define $\GL(n,\A)^1 := \{g \in \GL(n,\A) :|\det g| = 1 \}$ as usual.
	
	The meromorphy, etc.\ for zeta integrals can also be posed for the dual prehomogeneous vector space $\check{X}$ of $X$, and $\check{X}$ will play a crucial role in the arguments below.

	\begin{enumerate}
		\item Let us address the convergence of $Z_\lambda(\phi \otimes \xi)$ first. The stabilizer of $x_0 := 1 \in X^+(F)$ is $H = \text{diag}(\GL(n))$, so $Z[X] = \text{diag}(\Gm) \subset \Gm^2 = Z_G$. The construction in the beginning of \S\ref{sec:theta-dist} produces a subgroup $\mathfrak{a} \subset \Gm(F_\infty) = F_\infty^\times$, $\mathfrak{a} \simeq \R^\times_{>0}$, viewed also as a subgroup of $Z[X](F_\infty) \subset H(F_\infty)$. Hence
		\[ Z_\lambda(\phi \otimes \xi) = \int_{\mathfrak{a} G(F) \backslash G(\A) \ni (g_1, g_2)} \vartheta_\xi(g_1, g_2) \phi_1(g_1) \phi_2(g_2) |\det g_2^{-1} g_1|^\lambda \dd g_1 \dd g_2. \]
		The convergence for $\Re(\lambda) \gg 0$ follows because
		\begin{inparaenum}[(i)]
			\item $\vartheta_\xi$ has moderate growth,
			\item $\phi_1$, $\phi_2$ are both rapidly decreasing modulo center,
			\item $Z[X] \backslash Z_G$ acts on $X$ by dilation: $\det(g_2^{-1} g_1) \to 0$ implies $g_2^{-1}x g_1 \to 0$ when $(g_1, g_2) \in Z_G$. 
		\end{inparaenum}
		\item As a special case of Proposition \ref{prop:period-group-case},
			\[ Z_\lambda(\phi \otimes \xi) = \int_{\GL(n,\A) \ni x} |\det x|^\lambda \angles{x\phi_1, \phi_2}_\text{Pet} \underbracket{\xi(x) |\Omega|^\demi}_{\in \mathscr{L}_x} \]
			in its range of convergence, where
			\begin{compactitem}
				\item $\angles{\cdot, \cdot}_\text{Pet}$ stands for the Petersson bilinear pairing $\phi_1 \otimes \phi_2 \mapsto \int_{\amspace_{\GL(n)}} \phi_1 \phi_2$, where the space $\amspace_{\GL(n)} := \R^\times_{>0} \GL(n,F) \backslash \GL(n,\A)$ is naturally isomorphic to $\GL(n,F) \backslash \GL(n,\A)^1$ and carries a chosen Haar measure;
				\item $|\Omega|$ is some Haar measure on $\GL(n,\A)$.
			\end{compactitem}
			Up to a constant factor depending on measures, $Z_\lambda(\phi \otimes \xi)$ equals the adélic Godement--Jacquet integral $Z^\text{GJ}\left( \frac{1}{2} + \lambda, \xi_0, \beta \right)$ in \cite[Definition 15.2.3]{GH11-2}, where $\beta(x) := \angles{x\phi_1, \phi_2}_\text{Pet}$ and $\xi_0 := \xi |\det|^{-\frac{n}{2}} |\Omega|^{-\demi}$ are both $\CC$-valued smooth functions on $X^+(\A)$. The precise recipe can be found in \S\ref{sec:GJ}.
		\item Now consider the meromorphy and global functional equation for $Z_\lambda$. To simplify matters, we assume $n > 1$ . The arguments are the same as in \cite[Theorem 15.2.4]{GH11-2}; we translate them into our formalism as follows. Firstly, make the standard identification
		\[ \mathfrak{a} G(F) \backslash G(\A) = \left( \GL(n,F) \backslash \R^\times_{>0} \GL(n,\A)^1 \right) \times \left( \GL(n,F) \backslash \GL(n,\A)^1 \right) \]
			where $\R^\times_{>0}$ embeds into $\GL(n,\A)$ via
			$t \mapsto \left( \begin{smallmatrix} t & & \\ & \ddots & \\ & & t \end{smallmatrix} \right)$.
			Denote by $\omega_\pi$ the central character of $\pi$, and rewrite the defining integral for $Z_\lambda(\phi \otimes \xi)$ as
			\[ \int_{t \in \R^\times_{>0}} t^{n\lambda} \omega_\pi(t) \int_{\substack{g_1 \in \GL(n, F) \backslash \GL(n,\A)^1 \\ g_2 \in \GL(n,F) \backslash \GL(n,\A)^1}} \vartheta_\xi(tg_1, g_2) \phi_1(g_1) \phi_2(g_2) = \int^1_0 + \int^\infty_1 \cdots. \]
			Here $\R^\times_{>0}$ carries the measure $\dd^\times t = t^{-1} \dd t$. Thanks to the rapid decay of $\phi_1, \phi_2$ on $\GL(n,F) \backslash \GL(n,\A)^1$ and $\xi$ on $X(\A)$, the integral $\int^\infty_1$ is absolute convergent and defines an entire function in $\lambda$.
			
			As for $\int^1_0$, observe that when $\Re(\lambda) \gg 0$, we may replace $\vartheta_\xi(tg_1, g_2)$ by $\tilde{\vartheta}_\xi(tg_1, g_2) = \tilde{\vartheta}((tg_1,g_2)\xi)$ in the integrand, since the orbits in $(X \smallsetminus X^+)(F)$ are negligible, and the $\int_0^1 \int_{(g_1,g_2)} \sum_{\gamma \in X(F)} \cdots$ is still convergent. Now the Poisson formula can be applied to $\tilde{\vartheta}_\xi$ inside $\int^1_0 \cdots$, leading to the expression
			\[ \int_{t \in ]0,1[} t^{n\lambda} \omega_\pi(t) \int_{\substack{g_1 \in \GL(n, F) \backslash \GL(n,\A)^1 \\ g_2 \in \GL(n,F) \backslash \GL(n,\A)^1}} \tilde{\vartheta}_{\mathcal{F}(\xi)}(tg_1, g_2) \phi_1(g_1) \phi_2(g_2). \]
			Notice that $\tilde{\vartheta}_{\mathcal{F}(\xi)}$ is defined on the prehomogeneous vector space $\check{X}$ dual to $X$; thus the $G$-action here is $x(tg_1, g_2) = g_1^{-1} t^{-1} x g_2$, cf.\ \eqref{eqn:group-case-action2}. Upon unwinding it into $\int_0^1 \int_{(g_1, g_2)} \sum_{\gamma \in \check{X}(F)} \cdots$, the multiple integral is still convergent whenever $\Re(\lambda) \gg 0$. Indeed, the terms associated with $\gamma \neq 0$ have rapid decay as $t \to 0$ since the action is flipped, whereas the term $\text{ev}_0(\mathcal{F}(\xi))$ associated to $\gamma=0$ gives rise to a convergent integral since $\Re(\lambda) \gg 0$.

			By another argument involving negligible orbits, the $\tilde{\theta}_{\mathcal{F}(\xi)}$ in the resulting integral $\int^1_0$ can be replaced by $\theta_{\mathcal{F}(\xi)}$, giving rise to an expression that converges and is actually holomorphic for all $\lambda$.

			Evidently, the final expression $Z_\lambda = \int^\infty_1 + \int^1_0 \cdots$ is symmetric under exchanging $X \leftrightarrow \check{X}$, $\xi \leftrightarrow \mathcal{F}(\xi)$, and everything lies in its own domain of holomorphy. Hence we have established the entireness and the global functional equation for $Z_\lambda$ at once.
	\end{enumerate}
	When $n=1$, the boundary term $\gamma=0$ is no longer negligible, and some fractions involving $\text{ev}_0(\xi)$, $\text{ev}_0(\mathcal{F}(\xi))$ will appear in the arguments above. This is certainly familiar from Tate's thesis.
\end{example}

\begin{example}
	In the case $X = X_P$ of Example \ref{eg:X_P-global}, choose any other parabolic subgroup $Q \subset G^\Box$ sharing the same Levi component with $P$. Then we have the local model transition $\mathcal{F}_{Q|P,v}: \Schw(X_P)_v \rightiso \Schw(X_Q)_v$. If the basic vectors $c_{P,v}, c_{Q,v}$ (with the notations in \cite{BK02}) are chosen to have the same $L^2$-norm, then $\mathcal{F}_{Q|P,v}(c_{P,v}) = c_{Q,v}$ by \cite[p.548]{BK02}. Thus $\mathcal{F}_{Q|P} = \bigotimes_v \mathcal{F}_{Q|P,v}: \Schw(X_P) \rightiso \Schw(X_Q)$ is well-defined.
	
	Let $\xi = \bigotimes_v \xi_v \in \Schw(X_P)$ such that there exists places $w \neq w'$ of $F$ with $\xi_w \in C^\infty_c(X_{P,w})$ and $\mathcal{F}_{Q|P, w'}(\xi_{w'}) \in C^\infty_c(X_{Q,w'})$. For such $\xi$, define
	\[ \tilde{\vartheta}(\xi) := \sum_{\gamma \in X_P(F)} \text{ev}_\gamma(\xi). \]
	The Poisson formula in this context is recorded in \cite[Theorem 6.4]{BK02} (see also the proof). For the connections of $Z_\lambda$ to Eisenstein series, see \cite[\S 7.22]{BK99}.
\end{example}

\vspace{1em}
\begin{flushleft}
	Wen-Wei Li \\
	Academy of Mathematics and Systems Science \\
	Chinese Academy of Sciences \\
	55, Zhongguancun donglu, \\
	100190 Beijing, People's Republic of China. \\
	Current E-mail address: \href{mailto:wwli@pku.edu.cn}{\texttt{wwli@pku.edu.cn}}
\end{flushleft}


\bibliographystyle{abbrv}	
\bibliography{zeta}

\printindex

\end{document}